\documentclass[11pt]{amsart}
\usepackage{geometry}
\usepackage{graphpap}

\usepackage{amsmath}
\usepackage{cases}
\usepackage{amssymb}
\usepackage{hyperref}
\usepackage{xcolor}
\usepackage{color}

\newcommand{\noi}{\noindent}
\newcommand{\wt}{\widetilde}

\newtheorem{theorem}{Theorem}

\newtheorem{proposition}{Proposition}
\newtheorem{lemme}[proposition]{Lemma}

\newtheorem{question}[proposition]{Question}
\newtheorem{corollaire}[proposition]{Corollary}

\newtheorem{remarque}[proposition]{Remark}

\numberwithin{equation}{section}
\numberwithin{proposition}{section}

\def\ov{\overline} 
\def\11{{\rm 1~\hspace{-1.4ex}l} }
\def\R{\mathbb R}

\def\Z{\mathbb Z}
\def\N{\mathbb N}

\def\T{\mathbb T}

\def\P{\mathbf{P}}

\def\lg{\langle}
\def\rg{\rangle}
\def\wt{\widetilde}
\def\wh{\widehat}
\begin{document}
	\title[Refined probabilistic global well-posedness for the weakly dispersive NLS ]
	{
Refined probabilistic global well-posedness for the weakly dispersive NLS
	}
	\author{Chenmin Sun, Nikolay Tzvetkov}
	\address{
		Universit\'e de Cergy-Pontoise,  Cergy-Pontoise, F-95000,UMR 8088 du CNRS}
	\email{nikolay.tzvetkov@u-cergy.fr}
	\email{chenmin.sun@u-cergy.fr}
	\begin{abstract} 
We continue our study of the cubic fractional NLS with very weak dispersion $\alpha>1$ and data distributed according to the Gibbs measure. We construct the natural strong solutions for $\alpha>\alpha_0=\frac{31-\sqrt{233}}{14}\approx 1.124$ which is strictly smaller than $\frac{8}{7}$, the threshold beyond which the first nontrivial Picard iteration has no longer the Sobolev regularity needed for the deterministic well-posedness theory. This also improves our previous result in Sun-Tzvetkov \cite{Sun-Tz2}. We rely on recent ideas of Bringmann \cite{Bringmann} and Deng-Nahmod-Yue \cite{Deng2}.
In particular we adapt to our situation the new resolution ansatz in \cite{Deng2} which captures the most singular frequency interaction parts in the $X^{s,b}$ type space. To overcome the difficulties caused by the weakly dispersive effect, our specific strategy is to benefit from the ``almost" transport effect of these singular parts and to exploit their $L^{\infty}$ as well as the Fourier-Lebesgue property in order to inherit the random feature from the linear evolution of high frequency portions. 
\end{abstract}

	\maketitle
	\tableofcontents

\section{Introduction}


\subsection{Motivation}
In this article, we continue our study of the defocusing cubic fractional nonlinear Schr\"odinger equation (FNLS)
\begin{equation}\label{main-NLS}
i\partial_t u+|D_x|^\alpha u+|u|^{2}u=0, \quad (t,x)\in\R\times \T,
\end{equation}
where $u$ is complex-valued and $|D_x|^\alpha =(-\partial_x^2)^{\alpha/2}$ is defined as the Fourier-multiplier 
$\widehat{|D_x|^\alpha f}(k)=|k|^\alpha\widehat{f}(k)$. 
The parameter $\alpha$ measures the strength of the dispersion. In this article, we are always in the weak dispersive regime where $1<\alpha<2$.
The equation  \eqref{main-NLS} is a Hamiltonian system with  conserved energy functional
$$ H(u)=\int_{\T}||D_x|^{\frac{\alpha}{2}}u|^2dx+\frac{1}{2}\int_{\T}|u|^4dx\,.
$$ 
Moreover, the mass
$ M(u)=\int_{\T}|u|^2dx
$ 
is also conserved along the flow of \eqref{main-NLS}. 
The fractional Schr\"odinger equation was introduced in the
theory of the fractional quantum mechanics where the Feynmann path integrals
approach is generalized to $\alpha$-stable L\'evy process \cite{Laskin}. Also, it appears in the water
wave models (see \cite{Ionescu-Pusateri} and references therein). 
In addition, we refer to \cite{KLS} where the  fractional NLS on the line appears as a limit of the discrete NLS with long range interactions. 

The motivation in our previous work \cite{Sun-Tz2} is to provide macroscopic properties for the solutions of \eqref{main-NLS}, and in particular to detect the strength of the dispersion in the construction of the Gibbs measure. In that work, we construct global solutions on a full measure set with respect to the Gibbs measure by different methods, depending on the value of $\alpha$. More precisely, when $\alpha>\frac{6}{5}$, we construct the global strong solution satisfying the recurrence properties and show that the sequence of smooth solutions for FNLS with truncated initial data converges almost surely to the constructed strong solution. When $1<\alpha\leq \frac{6}{5}$, we rely on a simple method of Bourgain-Bulut \cite{BB1,BB2,BB3}  to prove the convergence of the Galerkine approximation scheme for the FNLS with truncated both data and nonlinearity. However,  we were not able to show that  the limit constructed by that method satisfies the flow property and therefore  it is a natural question to investigate whether there exists global strong solution in the full range $\alpha>1$ on the support of the Gibbs measure and if the strong solution coincides with the limit constructed by the Galerkine approximation scheme. 
\subsection{Setup and the main result}
To present the main result and to explain the different methods of constructing solutions, we recall the standard randomization procedure. Let $(g_k)_{k\in\Z}$ be a sequence of independent, standard complex-valued Gaussian random variables on a fixed probability space 
$(\Omega,\mathcal{F},\mathbb{P})$. Denote by $\mu$ the Gaussian measure on $H^{\frac{\alpha-1}{2}-\epsilon}(\T)$ for any $\epsilon>0$ induced by the map
\begin{equation}\label{induced-Gaussian}
\omega\longmapsto \phi^{\omega}(x):=\sum_{k\in\Z} \frac{g_k(\omega)}{[k]^{\frac{\alpha}{2}}}\mathrm{e}_k(x),
\end{equation}
where $\mathrm{e}_k(x)=\mathrm{e}^{ikx}$ and $[k]^{\frac{\alpha}{2}}=(1+|k|^{\alpha})^{\frac{1}{2}}$. Set $E_n=\textrm{span}\{\mathrm{e}_k: |k|\leq n
\}$.
We denote by 
$$
\Pi_n:  H^{\frac{\alpha-1}{2}-\epsilon}(\T) \longrightarrow E_n
$$ 
the corresponding spectral projection. When $\alpha>1$, it is well-known that for any $0\leq \sigma_0<\frac{\alpha-1}{2}$, $\||D|^{\sigma_0}u\|_{L^{\infty}(\T)}$ is $\mu$-almost surely finite. Then the Gibbs measure $\rho$ associated with \eqref{main-NLS} is $$d\rho(u)=\mathrm{e}^{-\frac{1}{2}\int|u|^4 }d\mu(u).$$
This measure can be viewed as the limit of $\rho_n$, the Gibbs measure associated with the truncated Hamiltonian 
$$ H_n(u)=\int_{\T}||D_x|^{\frac{\alpha}{2}}\Pi_nu|^2dx+\frac{1}{2}\int_{\T}|\Pi_nu|^4dx
$$
whose associated Hamiltonian flow is the truncated FNLS (ODE):
\begin{align}\label{FNLS-truncated} 
i\partial_tv_n+|D_x|^{\alpha}v_n+\Pi_n(|\Pi_nv_n|^2)=0,\quad v_n|_{t=0}=\Pi_n\phi.
\end{align}
Once the Gibbs measure $\rho$ is constructed, we need to construct the dynamics on the support of the measure, namely to solve \eqref{main-NLS} with randomized initial data \eqref{induced-Gaussian}. 

There are two ways to solve the dynamical problem, the first is to prove the convergence of \eqref{FNLS-truncated}, since for each fixed $n$, the truncated FNLS admits a global solution, as it is a Hamiltonian ODE on the finite dimensional space $E_n$. In \cite{Sun-Tz2}, using the Bourgain-Bulut argument, we have proved:
\begin{theorem}[\cite{Sun-Tz2}]\label{convergence-truncated}
Assume that $\alpha>1$ and $\sigma_0<\frac{\alpha-1}{2}$. The sequence $(v_n^{\omega})_{n\in\N}$ of solutions of \eqref{FNLS-truncated} with randomized initial data \eqref{induced-Gaussian} converges a.s. in $C(\R;H^{\sigma_0}(\T))$ to some limit $v$ which solves \eqref{main-NLS} in the distributional sense. 
\end{theorem}
The second approximation, more natural from the PDE view-point, is to consider the convergence of the sequence of smooth solutions $u_n$ of
\begin{equation}\label{FNLS:smooth}
i\partial_tu_n+|D_x|^{\alpha}u_n+|u_n|^2u_n=0,\quad u_n|_{t=0}=\Pi_n\phi.
\end{equation} 
Note that for each fixed $n$, the global well-posedness of \eqref{FNLS:smooth} is guaranteed, thanks to a theorem proved (in the range $\alpha>\frac{2}{3}$) in \cite{Thir} or \cite{Cho} (in the range $\alpha>1$). The major difference of the aforementioned approximations is that for the PDE approximation, we need to establish a probabilistic local well-posedness which provides us more information on the structure of the solution. While only to prove the convergence for the first approximation, some probabilistic compactness methods exploiting the invariance of the finite dimensional Gibbs measure $\rho_n$ are sufficient, see for example \cite{BB1},\cite{BB2},\cite{BB3},\cite{BTT-Toulouse},\cite{OThNLS} in the context of nonlinear Schr\"odinger and nonlinear wave equations. Therefore, a natural question can be formulated as follows:

\begin{question}\label{conjecture}
Can we show that for $\alpha>1$, the sequence $(u_n^{\omega})_{n\in\N}$ 
of solutions of \eqref{FNLS:smooth}
with randomized initial data \eqref{induced-Gaussian} converges a.s. in $C(\R;H^{\sigma_0}(\T))$ to some  unique limit $u$ which coincides with the limit obtained in Theorem \ref{convergence-truncated} ?  Moreover, can we define the solution map $\Phi(t)$ satisfying the flow property and Poincar\'e's recurrence property on a full measure set with respect to the Gibbs measure ?
\end{question}
We will call strong solutions those obtained when giving a positive answer of Question~\ref{conjecture}. The threshold $\alpha>1$ is designed for two reasons. Firstly, we do not need to renormalize the equation as the initial data lives in $L^{\infty}$ almost surely. Secondly, as we will see later, for $\alpha>1$, the second Picard's iteration enjoys some smoothing effect, due to the presence of the dispersion.  

 The main result of this article is the following partial answer of Question~\ref{conjecture} which improves our previous result in \cite{Sun-Tz2} for $\alpha>\frac{6}{5}$.
\begin{theorem}\label{thm5}
	Assume that $\alpha>\alpha_0=\frac{31-\sqrt{233}}{14}$ and $\sigma_0<\frac{\alpha-1}{2}$.
	Then the sequence of smooth solutions $(u_n^{\omega})_{n\in\N}$ of
	\begin{equation*}\label{FNLS-sequence}
	i\partial_tu_n+|D_x|^\alpha u_n+|u_n|^2u_n=0,\quad
	u_n|_{t=0}=\displaystyle{\sum_{|k|\leq n}\frac{g_k(\omega)}{[k]^{\frac{\alpha}{2}}}}\mathrm{e}_k,
	\end{equation*}
	 converges almost surely in $C(\R;H^{\sigma_0}(\T))$ to a limit which solves 
	\eqref{main-NLS}. 
\end{theorem}
Let us give a brief explanation about the number $\alpha_0$ appearing in the above statement.  The important feature is that the number $\alpha_0$ appearing in Theorem~\ref{thm5} is smaller than $8/7$ which is the threshold beyond which the first nontrivial Picard iteration has no longer the Sobolev regularity needed for the deterministic well-posedness theory (see the discussion below for more details). For this reason we find that the progress made in this paper is at a conceptual level.   

Following the argument in \cite{Sun-Tz2}, we are able to show that the unique limit satisfies the flow property and the Gibbs measure $\rho$ is invariant under the flow. The key point is to establish a  probabilistic local well-posedness result which provides a fine structure of the solution of \eqref{main-NLS}. Let us mention that when $\alpha>\frac{4}{3}$, the above theorem is proved in \cite{Dem} using only the deterministic theory without appealing to any random oscillation effect. In \cite{Sun-Tz2}, when $\frac{6}{5}<\alpha\leq \frac{4}{3}$, we go beyond the available deterministic theory by adapting the Da Prato-Debussche affine decomposition in conjugation with a gauge transformation to prove the probabilistic local well-posedness.
\subsection{Boundedeness of the Picard iterates in $L^\infty$}
To motive the necessity of a refined analysis and to compare with the context of parabolic equations, let us look at the formal Picard iteration scheme associated with our equation.  Denote by
$$ z_1^{\omega}(t)=\mathrm{e}^{it|D_x|^{\alpha}}\phi^{\omega}\,,
$$
where $\phi^{\omega}$ is given by \eqref{induced-Gaussian}. By formally expanding the solution of \eqref{main-NLS} as power series in terms of the initial data, we write 
$$ Z^{\omega}(t)=\sum_{j=0}^{\infty}z_{2j+1}^{\omega}(t).
$$
Formally inserting into the equation $(i\partial_t+|D_x|^{\alpha})Z^{\omega}+|Z^{\omega}|^2Z^{\omega}=0$ and comparing the coefficients, $z_{2j+1}^{\omega}$ should satisfy the equation:
$$ (i\partial_t+|D_x|^{\alpha})z_{2j+1}^{\omega}=-\!\!\!\!\!\!\!\!\!\!\!\sum_{\substack{j_1,j_2,j_3\geq 0\\
j_1+j_2+j_3=j-1 } }\!\!\!\!\!\!\!z_{2k_1+1}^{\omega}\ov{z}_{2j_2+1}^{\omega}z_{2j_3+1},\quad z_{2j+1}^{\omega}|_{t=0}=0.
$$
By induction we see that $z_{2j+1}^{\omega}$ is a $(2j+1)$-multilinear form of Gaussians:
\begin{align}\label{expressionz2k+1}
z_{2j+1}^{\omega}(t,x)=\sum_{k_1,\cdots,k_{2j+1}}c_{j}(t,k_1,\cdots,k_{2j+1})\frac{g_{k_1}\ov{g}_{k_2}\cdots \ov{g}_{k_{2j}}g_{k_{2j+1}}
}{[k_1]^{\frac{\alpha}{2}}\cdots [k_{2j+1}]^{\frac{\alpha}{2}}  }\mathrm{e}_{k_1-k_2+\cdots-k_{2j}+k_{2j+1}}(x).
\end{align}
The following proposition shows that every finite order of Picard's iteration is bounded in $C([0,T];L^{\infty}(\T))$:
\begin{proposition}\label{iteratedPicard} 
There exists $C_0>0$, such that for any $j\in\N$, $t\in\R, x\in\T$, we have
$$ \mathbb{E}[|z_{2j+1}^{\omega}(t,x)|^2]\leq C_0t^{2j}(2j+1)!\big(\frac{(2j-1)!!}{j!}\big)^2.
$$
In particular, for any $T>0$ sufficiently small, there exists $\Omega_T\subset\Omega$, with $\mathbb{P}[\Omega_T]=1$, such that for any $\omega\in\Omega_T$ and any $j$, the partial sum of the Picard iteration satisfies
$$ Z_{2j+1}^{\omega}(t):=\sum_{j'=0}^{j}z^{\omega}_{2j'+1}(t)\in C([0,T];L^{\infty}(\T)).
$$
\end{proposition}
Though the partial sum of the formal expansion $Z$ is bounded in $L^{\infty}$, this proposition does not tell anything about the convergence of the remainder in the formal expansion $Z(t)=\sum_{j=0}z_{2j+1}^{\omega}(t)$.  Much effort has to been addressed to in order to prove the convergence of the remainder.

Let us also mention a comparison with parabolic equations. For $\alpha>1$, a typical function with respect to $\mu$ is an $L^{\infty}$ function. As a consequence,
if we were dealing with a similar problem for a parabolic PDE then thanks to the nice
$L^{\infty}$ mapping properties of the heat flow, the analysis would become essentially trivial. On
the other hand, since we are dealing with a dispersive PDE, the linear problem is only
well-posed in $L^2$ 
in the scale of the $L^p$
spaces which makes that even at positive regularities, refined detereministic estimates and probabilistic considerations are essential in
the analysis. 

 \vspace{0.3cm}
 
\subsection{Difficulties and the Strategy}

Let us consider two extreme situations $\alpha=1$ and $\alpha=2$. When $\alpha=2$, the equation \eqref{main-NLS} is the classical cubic Schr\"odinger equation which has nice dispersive properties. In particular, the $L^4$ Strichartz estimate holds with no loss of spatial derivative. When $\alpha=1$, the equation \eqref{main-NLS} is the cubic half-wave equation. If we ignore the nonlocal issue and consider only the transport equation
\begin{align}\label{transport} 
i\partial_tu+i\partial_xu=|u|^2u
\end{align}
we can solve this equation simply in the space $L^{\infty}$. These facts indicate that in the intermediate case $1<\alpha<2$, we should balance the dispersive effect and the transport property of the solutions according to different regimes. However, when $\alpha$ is very close to 1, there are two major difficulties. Unlike the classical Schr\"odinger case $\alpha=2$, the $L^4$ Strichartz estimate loses almost $\frac{1}{8}$ derivatives (due to the degeneracy of the resonant function). Moreover, the fractional dispersion $|D_x|^{\alpha}$ is non-local which prevents us to use directly the transport property like \eqref{transport}.

Our strategy is based on the following observations. Firstly, the most singular parts in $X^{s,b}$ space come from the  high-low-low type frequency interactions. These parts satisfy morally the transport equation. Secondly, the loss of derivatives in the Strichartz inequality occurs in the high-high-high or high-high-low frequency interaction regimes. Hence we should place the most singular part in the space $L^{\infty}$ instead of $X^{s,b}$ in these regimes when estimating tri-linear expressions. To realize this strategy, we use the refined resolution ansatz introduced by Deng-Nahmod-Yue in \cite{Deng2}. Roughly speaking, it concerns refining the affine ansatz and decomposing the solution roughly as $e^{it|D_x|^{\alpha}}\phi^{\omega}+\Psi+w$ with a ``random averaging operator'' term $\Psi$ which captures the most singular frequency interactions. Additionally in our situation, the term $\Psi$ can be further decomposed into different parts carrying relatively ``good'' $L^{\infty}$ property and relatively ``good'' $X^{s,b}$ and Fourier-Lebesgue property.

\vspace{0.3cm}

{\bf $\!\!\!\!\!\!\!\!$The threshold $\alpha>\frac{6}{5}$ for the affine decomposition structure.}
To be more precise, we breifly recall the decomposition due to Bourgain \cite{B3} and Da Prato-Debussche \cite{DPD2} used in our previous work \cite{Sun-Tz2}. By using the gauge transformation
$$ v(t,x)=u(t,x)\mathrm{e}^{\frac{it}{\pi}\int_{\T}|u|^2dx },
$$
we transform the FNLS as
\begin{align}\label{FNLS-gauged}
 i\partial_tv+|D_x|^{\alpha}v=\mathcal{N}(v),\quad v|_{t=0}=\phi^{\omega}\,,
\end{align}
where the Wick-ordered nonlinearity is given by 
$$ \mathcal{N}(v):=-\mathcal{N}_3(v,v,v)+\mathcal{N}_0(v,v,v),
$$
and the trilinear forms $\mathcal{N}_3(\cdot,\cdot,\cdot)$ and $\mathcal{N}_0(\cdot,\cdot,\cdot)$ are defined as
\begin{equation}\label{trilinearform}
\begin{split}
&\mathcal{N}_3(f_1,f_2,f_3):=\sum_{\substack{k_1,k_2,k_3\\
		k_2\neq k_1,k_3 }} \widehat{f}_1(k_1)\ov{\widehat{f}_2}(k_2)\widehat{f}_3(k_3)\mathrm{e}_{k_1-k_2+k_3},\\
& \mathcal{N}_0(f_1,f_2,f_3):=\sum_{k\in\Z}\widehat{f}_1(k)\ov{\widehat{f}_2}(k)\widehat{f}_3(k)\mathrm{e}_k.
\end{split}
\end{equation}
To solve \eqref{FNLS-gauged}, we used the affine ansatz
$$ v(t)=S_{\alpha}(t)\phi^{\omega}+w(t),
$$
where $S_{\alpha}(t)=\mathrm{e}^{it|D_x|^{\alpha}}$ is the linear propagator. It turns out that the Duhamel's integration of the first Picard's iteration $\mathcal{I}\mathcal{N}\big(S_{\alpha}(t)\phi^{\omega}\big)$ has the spatial regularity $H^{\frac{3(\alpha-1)}{2}-}$. Yet, if we place it into the $X^{s,b}$ space, it is bounded a.s. in $X^{(\alpha-1)-,\frac{1}{2}+}$. In both these spaces, the spatial regularity for the first Picard's iteration is better than the initial data which merely lives in $H^{\frac{\alpha-1}{2}-}$. In order to close the fix-point argument, we should place the error term into some $X^{s,\frac{1}{2}+}$ space. Due to the weak dispersive effect, when $\alpha<2$, it was proved in \cite{Cho} (see also \cite{Sun-Tz2}) that the Duhamel's integration of the tri-linear operator is bounded on $X^{s,\frac{1}{2}+}$ only if $s\geq \frac{1}{2}-\frac{\alpha}{4}$. Therefore, this affine decomposition ansatz is suitable in $X^{s,b}$ type space only if $\alpha-1>\frac{1}{2}-\frac{\alpha}{4}$ which gives us the constraint $\alpha>\frac{6}{5}$.

Even if we do not place the first Picard's iteration in the $X^{s,b}$ type space, the other place that gives us the constraint $\alpha>\frac{6}{5}$ is the high-low-low frequency interaction for the crossing terms of the form $$f_N:=\mathcal{I}\mathcal{N}(S_{\alpha}(t)\mathbf{P}_N\phi^{\omega},\mathbf{P}_{\ll N}w,\mathbf{P}_{\ll N}w).$$
 By ignoring the issue of the modulation, the above term can be written formally as
 $$  \sum_{\substack{k,k_1,k_2,k_3\\
 k_1-k_2+k_3=k}}\mathbf{1}_{ |k_1|\sim N, |k_2|,|k_3|\ll N }\mathbf{1}_{|k_1|^{\alpha}-|k_2|^{\alpha}+|k_3|^{\alpha}-|k|^{\alpha}=O(1)}\frac{g_{k_1}}{[k_1]^{\frac{\alpha}{2}}
}\wh{w}(k_2)\wh{w}(k_3).
 $$
From a counting argument, the $X^{s,b}$ norm of the above quantity can be bounded by $$N^{(s-(\alpha-1))-}\|\wh{w}_2\|_{l^2}\|\wh{w}_3\|_{l^2},$$ hence we should require $s<\alpha-1$ to ensure that the above expression is bounded. It turns out that this high-low-low frequency interaction is the most singular part in the analysis. In order to improve the constraint of $\alpha$, a better understanding of this singular part is necessary.
\subsection{Refined resolution ansatz}
 Refined resolution ansatz to treat the singular high-low type interaction has been recently introduced by Bringmann \cite{Bringmann} for the wave equation and by Deng-Nahmod-Yue \cite{Deng2} for the 2D NLS in very different ways. The common feature in both these work is the observation that the low frequency component is independent with the high frequency linear evolution and the most singular interactions (high-low type) are removed by viewing them as part of the linear evolution for the high-frequency data and isolating them from $w(t)$ in the previous affine ansatz $u(t)=S_{\alpha}(t)\phi^{\omega}+w(t)$. More importantly, the authors in \cite{Deng2} exploits the fact that the low frequency components are also random, and this randomness of low frequency components is exactly what is captured by the matrix/operator norms introduced there.
 To better explain the idea in the context of FNLS, we need to introduce an extra term $\zeta$ such that $\Psi=S_{\alpha}(t)\phi^{\omega}+\zeta$ solves
 $$ i\partial_t\Psi+|D_x|^{\alpha}\Psi=\mathcal{I}\mathcal{N}(\mathbf{P}_{\text{high}}\Psi,\mathbf{P}_{\text{low}}u,\mathbf{P}_{\text{low}}u ),\quad \Psi|_{t=0}=\mathbf{P}_{\text{high}}\phi^{\omega}.
 $$
 Through this decomposition, on the one hand, the new remainder will solve some nonlinear equation with essentially no high-low-low type frequency interaction. On the other hand, since the isolated singular part $\Psi$ solves roughly a linear transport equation with some ``potential" independent of the high frequency initial data, it will inherit the randomness from the initial data $\phi^{\omega}$. Though $\Psi$ is no more regular than $X^{(\alpha-1)-,\frac{1}{2}+}$ in general, it has its own random structure though captured by certain matrix-norms. 

Now we recall the precise resolution ansatz of \cite{Deng2} in our context.  Set $y_N=v_N-v_{\frac{N}{2}}$.  Then $y_N$ solves the equation
\begin{equation*}
\begin{cases}
(i\partial_t+|D_x|^{\alpha})y_N=\mathcal{N}\big(v_{\frac{N}{2}}+y_N\big)-\mathcal{N}\big(v_{\frac{N}{2}}\big),\\
y_N|_{t=0}=\P_N\phi^{\omega}.
\end{cases}
\end{equation*}
For fixed $N$, we denote by $L_N$ the largest dyadic number $L$ such that $L<N^{1-\delta}$. For $L\leq L_N$, we introduce the function $\psi^N_L$ which captures the high-low-low frequency interaction:
\begin{equation}\label{psiNL}
\begin{cases}
(i\partial_t+|D_x|^{\alpha})\psi_L^N=-2\Pi_N\mathcal{N}_3\big(\psi^N_{L},\Pi_Lv_L,\Pi_Lv_L \big),\\
\psi^N_{L}|_{t=0}=\P_N\phi^{\omega}.
\end{cases}
\end{equation}
%
When $L=\frac{1}{2}$, we define $\psi^N_{\frac{1}{2}}:=S_{\alpha}(t)\P_N\phi^{\omega}$. Set $w_N=y_N-\psi^N_{L_N}$, then $w_N$ solves the equation
\begin{equation}\label{errorequation}
\begin{cases}
(i\partial_t+|D_x|^{\alpha})w_N=\mathcal{N}\big(w_N+\psi^N_{L_N}+v_{\frac{N}{2}} \big)-\mathcal{N}\big(v_{\frac{N}{2}} \big) 
+2\Pi_N\mathcal{N}_3\big(\psi^N_{L_N},\Pi_{L_N}v_{L_N},\Pi_{L_N}v_{L_N} \big),\\
w_N|_{t=0}=0.
\end{cases}
\end{equation}
%
Denote by $f_N=\psi_{\frac{1}{2}}^N$ the free evolution part and $\zeta_L^N=\psi_L^N-\psi_{\frac{L}{2}}^N$ if $\frac{1}{2}<L\leq L_N$, $\zeta_{\frac{1}{2}}^N=0.$ Then
\begin{equation}\label{y:ansatz}
y_N(t)=f_N(t)+\sum_{\frac{1}{2}<L\leq L_N}\zeta_L^N(t)+w_N(t),
\end{equation}
and the full resolution ansatz is
$$ v(t)=S_{\alpha}(t)\phi^{\omega}+\sum_{N}\sum_{\frac{1}{2}<L\leq L_N}\zeta_L^N(t)+\sum_{N}w_N(t).
$$
\begin{remarque}
	$v_{\frac{N}{2}}-v_{L_N}$ is pretended to have frequencies greater than $L_N$ and $w_N$ is pretended to have frequencies comparable to $N$, and $\zeta_L^N$ is pretended to contain the portion of frequency interaction from  $(\sim N) \times (\sim L)\times (\lesssim L)$. Therefore, by expanding the right hand side of \eqref{errorequation}, all the multi-linear forms essentially do not have bad frequency interaction of the form $(\sim N)\times (\ll L_N)\times (\ll L_N)$. 
\end{remarque}

The second parameter $L$ quantifies the range of ``low-frequency'' perturbation for the linear evolution of the high frequency data. It can be viewed as a deformation from the random oscillation effect to the time-oscillation effect (dispersive effect). When $L$ is relatively small, $\zeta_L^N$ behaves like the first Picard iteration of the linear evolution of Gaussian variables whose random effect is dominant. When $L$ is relatively large, $\zeta_L^N$ behaves like the error $w_N$ whose $X^{0,b}$-regularity is much better.

\vspace{0.3cm}
	
{\bf$\!\!\!\!\!\!\!\!\!$ Structure of $\psi_L^N$ in terms of operators.}
Given $v_L$, the equation of $\psi_L^N$ is linear with respect to the initial data. Therefore, we can write
$$ \psi^N_{L}=\mathcal{H}^{N,L}(\P_N\phi^{\omega}).
$$
The operator $\mathcal{H}^{N,L}$ is the random averaging operator and $\mathcal{F}_x\mathcal{H}^{N,L}\mathcal{F}_x^{-1}$ has kernel $\big(H_{kk*}^{N,L}(t)\big)$, thus we have
$$ \widehat{\psi_L^N}(t,k)=\sum_{\frac{N}{2}<|k^*|\leq N} H_{kk^*}^{N,L}(t)\frac{g_{k^*}(\omega)}{[k^*]^{\frac{\alpha}{2}}}.
$$
In other words, $H_{kk^*}^{N,L}(t)$ is the $k$-th Fourier mode of the solution to 
\begin{equation*}
\begin{cases}
(i\partial_t+|D_x|^{\alpha})\varphi=-2\Pi_N\mathcal{N}_3(\varphi,\Pi_Lv_L,\Pi_Lv_L),\\
\varphi|_{t=0}=\mathbf{1}_{\frac{N}{2}<|k^*|\leq N}\mathrm{e}_{k^*}.
\end{cases}
\end{equation*} 
Obviously, $$\text{supp}_{k,k^*}\big(H_{k,k*}^{N,L} \big)\subset\big\{(k,k^*): |k|\leq N, \frac{N}{2}<|k^*|\leq N \big\}.$$ When $L=\frac{1}{2}$, we use the convention $H_{kk^*}^{N,\frac{1}{2}}=\mathrm{e}^{it|k|^{\alpha}}\mathbf{1}_{k=k^*}$. Similarly, we denote by
$ h^{N,L}=\mathcal{H}^{N,L}-\mathcal{H}^{N,\frac{L}{2}},
$
hence $\zeta_L^N=h^{N,L}(\P_N\phi^{\omega})$. The kernel of $\mathcal{F}_xh^{N,L}\mathcal{F}_x^{-1}$ is denoted by $(h_{kk^*}^{N,L})$ and $h_{kk^*}^{N,L}$ has the same $k,k^*$ support property as $H_{kk^*}^{N,L}$.
The key point here is that
$ H_{kk^*}^{N,L},h_{kk^*}^{N,L} $ belong to the Borel $\sigma$-algebra $\mathcal{B}_{\leq L}$ generated by $\{g_k(\omega):|k|\leq L \}$, hence $H_{kk^*}^{N,L},h_{kk^*}^{N,L}$ are independent of $\sigma$-algebra $ \mathcal{B}_{>\frac{N}{2}}
$ generated by $\{g_k(\omega):|k|>N/2 \}$.
The random oscillation effect will be captured in terms of suitable norms for the operators $\mathcal{H}^{N,L}, h^{N,L}$, as explained in \cite{Deng2}. In this article, we need the Hilbert-Schmidt type norm to capture the $X^{s,b}$-regularity as well as the $L^{\infty}$ size of $\psi^N_{L}$ and a Fourier-Lebesgue type norm to  measure the size of the Fourier-coefficients of $\psi_L^N$.\\

\vspace{0.3cm}
{\bf$\!\!\!\!\!\!\!\!\!$ Probabilistic local convergence.}
To prove Theorem \ref{thm5}, the key point is a local convergence result for dyadic sequences which we will describe below:
\begin{theorem}\label{prob:LWP}
	Let $\alpha>\alpha_0$. Then there exist $C_0>0 $ and sufficiently small numbers $\theta>0, \epsilon>0$, such that for each sufficiently small $T>0$, there exists a set $\Omega_T\subset\Omega$ with the following properties:
	\begin{itemize}
		\item[(i)] $\mathbb{P}[\Omega_T^{c}]<C_0\mathrm{e}^{-T^{-\theta}}$.\\
		\item[(ii)] $\forall \omega\in\Omega_T$, the sequence of unique smooth solutions $(v_N)_{N\in 2^{\N}}$ of 
		$$ i\partial_tv_N+|D_x|^{\alpha}v_N=\mathcal{N}(v_N)
		$$
         with initial data $v_N|_{t=0}=\Pi_N\phi^{\omega}$ given by \eqref{induced-Gaussian} is a Cauchy sequence in
         $ C([-T,T];H^{\sigma_0}(\T))$. More precisely, for all $|t|\leq T$ and $N$, $v_N$ admits a decomposition
		$$ v_N(t)=S(t)\Pi_N\phi^{\omega}+\zeta^N+W_N,\text{ where }\zeta^N:=\sum_{M\leq N}\sum_{\frac{1}{2}<L\leq L_N }\zeta_L^N 
		$$
		with the property that $(W_N)_{N\in 2^{\N}}$ is a Cauchy sequence in $ X_T^{\frac{1}{2}-\frac{\alpha}{4}+\epsilon,\frac{1}{2}+\epsilon}$ and for each $N,L$,
		\begin{align*}
		& \|\zeta_L^N\|_{L_t^{4}([-T,T];L^{\infty}(\T)) }\leq C_0N^{-(\alpha-1)-\epsilon}L^{\frac{1}{2}-\nu},\\
		& \|\zeta_L^N\|_{L^{\infty}([-T,T];H^{(\alpha-1)-\epsilon}(\T) )}\leq C_0L^{-\nu}N^{-\frac{\epsilon}{2}},\\
		&\|\zeta_L^N\|_{L^{\infty}([-T,T];\mathcal{F}L^{\frac{\alpha}{2}-\epsilon,\infty}(\T)) }\leq C_0L^{-\nu}N^{-\frac{\epsilon}{2}}
		\end{align*}
		with $\nu=\min\{\frac{1}{2}-\frac{\alpha}{4},\frac{7(\alpha-1)}{4} \}-\epsilon$.
		\item[(iii)] The sequence $(v_n)_{n\in\N}$ of unique smooth solutions of
		$$ i\partial_tv_n+|D_x|^{\alpha}v_n=\mathcal{N}(v_n),\quad v_n|_{t=0}=\Pi_n\phi^{\omega}
		$$
		converges in $C([-T,T];H^{\sigma_0}(\T))$.
	\end{itemize} 
\end{theorem}
Note that by undoing the gauge transform
$$ u_n(t)=v_n(t)\mathrm{e}^{-\frac{it}{\pi}\int_{\T}|v_n|^2dx },
$$
 the unique solution $u_n\in C([-T,T];H^{\sigma_0}(\T))$ of the original FNLS equation is also a Cauchy sequence in $C([-T,T];H^{\sigma_0}(\T))$, which proves Theorem \ref{thm5} (locally in time).  
Unlike \cite{Deng2} where the dispersive effect is very strong while the nonlinearity can be arbitrarily large, we deal with the NLS model with a fixed nonlinearity but with very weak dispersion. Another different feature is that we do not need to renormalize the equation which makes the problem more natural from a purely PDE perspective.
Therefore the type of probabilistic well-posedness we get in this paper is close in spirit to the line of research initiated by Burq and the second author in \cite{BT1/2,BT-JEMS}. More importantly, we perform the multi-linear estimates in a very different manner compared with \cite{Deng2}. Indeed, in \cite{Deng2}, all the analysis was performed in the Fourier space, thanks to the strong linear and  multi-linear smoothing effect. However, in our situation, the deterministic smoothing is very weak (for Strichartz we loose almost $\frac{1}{8}$ derivative) and we rely more on the linear random oscillation effect. It is at this point that we need to define an extra Fourier-Lebesgue type operator norm $S^{b,q}$ in Section \ref{sec:keyiterative}. 

 We believe that the constraint $\alpha>\alpha_0=\frac{31-\sqrt{233}}{14}$ is technical. We point out again that $\alpha_0<\frac{8}{7}$ where $\frac{8}{7}$ is the threshold for the constraint $\frac{1}{2}-\frac{\alpha}{4}>\frac{3(\alpha-1)}{2}$ and $H^{\frac{3(\alpha-1)}{2}-}$ is the regularity of the first Picard's iteration $\mathcal{I}\mathcal{N}(S_{\alpha}(t)\phi^{\omega})$. The technical constraint $\alpha>\alpha_0$ is mainly caused by the condition  $\nu<\frac{7(\alpha-1)}{4}$. Indeed, this comes from the upper bound $N^{-(\alpha-1)+}L^{-\frac{7(\alpha-1)}{4}+}$ of the $X^{0,b}$ norm of the expression
$$ \mathcal{I}\mathcal{N}(\mathbf{P}_NS_{\alpha}(t)\phi,\mathbf{P}_LS_{\alpha}(t)\phi+\zeta_R^{L},\mathbf{P}_LS_{\alpha}(t)\phi+\zeta_R^L ).
$$
Note that $\zeta_R^L$ can be viewed as a size $R$ perturbation of Gaussians with Fourier support $\sim L$. Compared with the expression
$\mathcal{I}\mathcal{N}(\mathbf{P}_NS_{\alpha}(t)\phi,\mathbf{P}_LS_{\alpha}(t)\phi,\mathbf{P}_LS_{\alpha}(t)\phi ),
$
 the non-resonant relation $k_2\neq k_1,k_3$ will be destroyed and consequently, the estimate for terms like $$\mathcal{I}\mathcal{N}(\mathbf{P}_NS_{\alpha}(t)\phi,\zeta_R^N,\zeta_R^L ),\quad  \mathcal{I}\mathcal{N}(\mathbf{P}_NS_{\alpha}(t)\phi,\mathbf{P}_LS_{\alpha}(t)\phi,\zeta_R^L )$$ is worse than the former\footnote{ See (v) of Lemma \ref{2random-Zb} and (ii) of Lemma \ref{2random-Sb}}. 

Refined resolution ansatz in the context of nonlinear PDE in the presence of singular randomness were used in many previous works. 
In \cite{BOP}, \cite{OTW} ansatz taking contributions from possibly infinitely many Picard iterations are introduced. 
In \cite{H}, \cite{GIP}, in the context of parabolic equations, resolution ansatz exploiting randomness structure of certain terms beyond the affine ansatz are introduced (the randomness is captured using certain linearisation operators).
This type of ansatz was first introduced in the context of dispersive PDE in  \cite{GKO} and further developed in \cite{Bringmann2,OOT}. Different ansatz, which involve the randomness structure of operators and tensors, are introduced in \cite{Deng2},\cite{Deng3}.
It should be underlined that all these contributions are extensions of the ideas introduced in the fundamental papers by Bourgain \cite{B1,B2,B3}. 
\vspace{0.3cm}

{\bf $\!\!\!\!\!\!\!\!$ Organization of the article.} In this article, we only address the proof of Theorem~\ref{prob:LWP} since the remaining arguments of the proof of 
Theorem~\ref{thm5} follow from \cite{Sun-Tz2}.  In Section~2, we recall some preliminaries and define the functional spaces for functions and operators. In Section 3, following the iterative scheme in \cite{Deng2}, we first reduce the proof of Theorem \ref{prob:LWP} to an induction statement (Proposition \ref{keyinduction}). Then by assuming key multi-linear estimates summarized in Proposition \ref{Multilinearkey}, we prove the induction Proposition \ref{keyinduction}. The remaining sections are devoted to the proof of the statements in Proposition~\ref{Multilinearkey}. In Section 4, we deduce the $L^{\infty}$ and Fourier-Lebesgue property for the ``paracontrolled'' terms which will be used intensively. Next in Section 5, we prove the mapping properties of the \emph{ random averaging operators} leading to the self-closeness of the fix-point problem for $h^{N,L}$. Then in Section 6, we reduce the key multi-linear operators to the low-modulation cases in order to focus only on the discrete multi-linear summations later. In Section 7, we prove the bilinear estimates for the kernels of random averaging operators which helps us to control the source term of the fix-point problem for $h^{N,L}$. Finally in the remaining sections, we focus on the tri-linear estimates used to close the fix-point problem for the error $w_N$, in different frequency interactions regimes. In all multi-linear estimates, we always describe available algorithms first and then do the case-by-case analysis by implementing the algorithms.       

\subsection*{Acknowledgment}
We thank Tadahiro Oh for interesting discussions while the first author visiting the University of Edinburgh. We thank Yu Deng for valuable comments on the first version of the manuscript. The authors are supported by the ANR grant ODA (ANR-18-CE40- 0020-01).

\section{Notations and preliminaries}

\subsection{General notations}
The capital numbers $N,M,L,R$ represent dyadic numbers greater than $\frac{1}{2}$. For a finite collection of dyadic numbers $\{N_1,N_2,\cdot,N_k\}$, we denote by $N_{(1)}\geq N_{(2)}\geq N_{(3)}\geq\cdots N_{(k)}$ be the non-increasing rearrangement of it.

For two quantities $A,B$, the asymptotic notation $A\lesssim B$ ($A\gtrsim B$) means that there exists a constant $C$ such that $A\leq CB$($A\geq CB$). The notation $A\sim B$ means that $A\lesssim B$ and $A\gtrsim B$. The notation $A\lesssim_X B$($A\gtrsim_X B$) is used to specify that the constant $C$ depends on $X$.

For the Lebesgue exponents $1\leq p,q,r\leq \infty$, we always use $p',q',r'$ to denote their conjugate exponents such that  $\frac{1}{p}+\frac{1}{p'}=1$ with the canonical modification when $p=1$ or $\infty$.  For $1\leq j\leq n$, denote by $(Z_j,\mu_j)$ a finite sequence of measure space with the standard $L^{p_j}$ norm
$$ \|f(z_j)\|_{L_{z_j}^{p_j}}:=\Big(\int_{Z_j}|f(z_j)|^{p_j}d\mu_j\Big)^{\frac{1}{p_j}}.
$$ 
We will simply denote by $L_{z_1}^{p_1}L_{z_2}^{p_2}\cdots L_{z_j}^{p_j}$ to stand for $L^{p_1}(Z_1;L^{p_2}(Z_2;\cdots;L^{p_j}(Z_j)\cdots ) )$. For example, we denote by $L_t^qL_x^rl_k^q$ to stand for $L^q(\R_t;L^r(\T_x;l^q(\Z)))$. The Fourier-Lebesgue space $\mathcal{F}L^{s,q}(\T)$ is defined via the norm
\begin{align}\label{FL}
\|f\|_{\mathcal{F}L^{s,q}}:=\|\lg k\rg^s\wh{f}(k)\|_{l_k^q}.
\end{align}

 We denote by
$\Pi_N:=\mathcal{F}_x^{-1}\mathbf{1}_{|k|\leq N}\mathcal{F}_x$, and $\P_N:=\mathcal{F}_x^{-1}\mathbf{1}_{N/2<|k|\leq N}\mathcal{F}_x$ if $N\geq 1$ and $\P_{2^{-1}}=\Pi_{2^{-1}}$. 
$S_{\alpha}(t)=e^{it|D_x|^{\alpha}}$.
The twisted spacetime Fourier transform is defined as
$$ \widetilde{u}(\lambda,k):=(\mathcal{F}_{t,x}u)(\lambda+|k|^{\alpha},k).
$$
We also denote by
$$ \widetilde{h}_{kk^*}(\lambda):=(\mathcal{F}_th_{kk^*})(\lambda+|k|^{\alpha}),\quad \widetilde{\Theta}_{kk'}(\lambda,\lambda')=2\pi (\mathcal{F}_{t,t'}\Theta_{kk'})(\lambda+|k|^{\alpha},-\lambda'-|k'|^{\alpha}).
$$
The definition of $\widetilde{\Theta}_{kk'}(\lambda,\lambda')$ is such that if a operator is given by
$$ \mathcal{Q}(w)(t,k)=\sum_{k'}\int \Theta_{kk'}(t,t')(\mathcal{F}_xw)(t',k')dt',
$$
then
$$ \widetilde{\mathcal{Q}}(w)(\lambda,k)=\sum_{k'}\int\widetilde{\Theta}_{kk'}(\lambda,\lambda')\widetilde{w}(\lambda',k')d\lambda'.
$$

Define the affine space for a given number $k\in\Z$
$$ \Gamma(k):=\{(k_1,k_2,k_3)\in\Z^3: k_2\neq k_1, k_2\neq k_3, k_1-k_2+k_3=k \},
$$
and the resonant function on $\Gamma(k)$
$$ \Phi_{k_1,k_2,k_3}:=|k_1|^{\alpha}-|k_2|^{\alpha}+|k_3|^{\alpha}-|k|^{\alpha}.
$$


\subsection{Spaces for functions and operators}

Denote by $S_{\alpha}(t)=\mathrm{e}^{it|D_x|^{\alpha}}$. Recall that the Fourier restriction type space $X^{s,b}$ is defined with the associated norm
$$ \|u\|_{X^{s,b}}:=\|S_{\alpha}(-t)u\|_{H^s(\T_x;H^b(\R_t))}=\|\lg\lambda\rg^b\lg k\rg^s\wt{u}(\lambda,k)\|_{l_k^2L_{\lambda}^2}.
$$
Similarly, the Fourier-Lebesgue restriction space $X_{p,q}^{s,\gamma}$ is defined via the norm
$$ \|u\|_{X_{p,q}^{s,\gamma}}:=\|\lg\lambda\rg^{\gamma}\lg k\rg^s\wt{u}(\lambda,k) \|_{l_k^pL_{\lambda}^q}.
$$
Note that $X_{2,2}^{s,b}=X^{s,b}$.
For finite time interval $I\subset\R$, the localized restriction space $X_{I}^{s,b}$ is defined via the norm
$$ \|u\|_{X_I^{s,b}}:=\inf\{\|v\|_{X^{s,b}}:v|_{I}=u \}.
$$ 
For $X^{s,b}$ spaces, we have the following statements.
\begin{lemme}\label{localizationXsb} 
Let $\chi\in\mathcal{S}(\R)$. Then for $0<T<1$, $s\in\R$ and $-\frac{1}{2}<\wt{b}\leq b<\frac{1}{2}$, we have the estimate
$$ \|\chi(t/T)u\|_{X^{s,\wt{b}}}\lesssim T^{b-\wt{b}}\|u\|_{X^{s,b}}.
$$	
Moreover, if $u|_{t=0}=0$, then the above estimate holds for $0<\widetilde{b}\leq b<1$.
\end{lemme}
Note that the proof of the last statement can be found as Proposition 2.7 of \cite{Deng2}.
\begin{lemme}\label{linearestimate} 
Let $\chi\in \mathcal{S}(\R)$. Then for $s\in\R$, $\frac{1}{2}<b<1$, we have the estimate
$$ \Big\|\chi(t)\int_0^tS_{\alpha}(t-t')F(t')dt' \Big\|_{X^{s,b}}\lesssim \|F\|_{X^{s,b-1}}.
$$	
\end{lemme}

For $\mathcal{H}(t)$, time-dependent linear operator on $l^2$ with kernel $(H_{kk^*}(t))$, we introduce the norms:
\begin{equation}\label{matrixnorm}
\begin{split}
&\|\mathcal{H}\|_{Y^b}:=\|\langle\lambda\rangle^b\widetilde{H}_{kk^*}(\lambda)\|_{l_{k^*}^2\rightarrow L_{\lambda}^2l_k^2},\\
&\|\mathcal{H}\|_{Z^b}:=\|\langle\lambda\rangle^b\widetilde{H}_{kk^*}(\lambda)\|_{L_{\lambda}^2l_{k,k^*}^2},\\
&\|\mathcal{H}\|_{S^{b,q}}:=\|\langle\lambda\rangle^{\frac{2b}{q'}}\widetilde{H}_{kk^*}(\lambda)\|_{l_k^{\infty}L_{\lambda}^ql_{k^*}^2},
\end{split}
\end{equation}
where $\frac{1}{q}+\frac{1}{q'}=1$, $1\leq q\leq \infty$. For a linear operator $\Theta$ with kernel  $(\Theta_{kk'}(t,t'))$, we introduce the matrix norms:
\begin{equation}\label{matrixnorm2}
\begin{split} 
&\|\Theta\|_{Y^{b_1,b_2}}:=\|\langle\lambda\rangle^{b_1}\langle\lambda'\rangle^{-b_2}\widetilde{\Theta}_{kk'}(\lambda,\lambda')\|_{L_{\lambda'}^2l_k'^2\rightarrow L_{\lambda}^2l_{k}^2},\\ &\|\Theta\|_{Z^{b_1,b_2}}:=\|\langle\lambda\rangle^{b_1}\langle\lambda'\rangle^{-b_2}\widetilde{\Theta}_{kk'}(\lambda,\lambda')\|_{L_{\lambda,\lambda'}^2l_{k,k'}^2},\\
&\|\Theta\|_{S^{b_1,b_2,q}}:=\|\langle\lambda\rangle^{\frac{2b_1}{q'}}\langle\lambda'\rangle^{-b_2}\widetilde{\Theta}_{kk'}(\lambda,\lambda')\|_{l_k^{\infty}L_{\lambda}^qL_{\lambda'}^2l_{k'}^2}.
\end{split}
\end{equation}
Note that when we ignore the $k^*$ variable, the $S^{b,q}$ norm is just the restricted-type Fourier-Lebesgue norm $X_{\infty,q}^{0,\gamma}$ with $\gamma=\frac{2b}{q'}$.
The reason for introducing of the space $S^{b,q}$ is two-fold. First it characterizes the Fourier-Lebesgue norm of the para-controlled term $\psi_L^{N}$ which is morally $N^{-\frac{\alpha}{2}}$ for small $L$. This allows us to carry out many 
multi-linear estimates simply by Cauchy-Schwartz, as in our previous work \cite{Sun-Tz2}.
The second reason is technical. When we do the Wiener chaos estimate, in almost all the situations, leaving out $\|h^{N,L}\|_{S^{b,q}}$ is better than leaving out $\|h^{N,L}\|_{Z^{b}}$ since the later losses $N^{1-\frac{\alpha}{2}}$ factor.

Sometimes we will abuse the notation and write simply
$$ \|\Theta_{kk'}(\lambda,\lambda')\cdot m(k,k')\|_{\mathcal{X}^{b_1,b_2}}:=\|\mathcal{F}_{t,x}^{-1}(\widetilde{\Theta}_{kk'}(\lambda,\lambda')m(k,k') )\|_{\mathcal{X}^{b_1,b_2}},
$$
for $\mathcal{X}=Y,Z$ or $S$.

\begin{lemme}\label{Timelocalization}
	Let $\chi\in\mathcal{S}(\R)$, and recall that $\chi_T(t):=\chi(T^{-1}t)$ for $0<T\ll 1$. Then for $u(t,x)$ and operator $\Theta(t,t')=(\Theta_{kk'}(t,t'))$ satisfying $u(t=0,\cdot)=0,$ $\Theta(t=0,t')=0$, we have
	$$ \|\chi_T(t)u\|_{X_{\infty,q}^{0,\gamma}}\lesssim T^{\gamma_1-\gamma}\|u\|_{X_{\infty,q}^{0,\gamma_1}},\quad \|\chi_T(t)\Theta\|_{S^{b,q}}\lesssim T^{\frac{2(b_1-b)}{q'}}\|\Theta\|_{S^{b_1,q}},
	$$	
	with $1\leq q<\infty$, $ 0<\gamma<\gamma_1<1+\frac{1}{q'},$ and $ 0<b<b_1<1$.
\end{lemme}

\begin{proof}
	The proof is essentially the same as in \cite{Deng2}. We present a proof in the appendix.
\end{proof}

Fix a time cutoff $\chi\in C_c^{\infty}((-1,1))$, we define the time truncated Duhamel operator
\begin{equation}\label{truncatedDuhamel}
\mathcal{I}F(t):=\chi(t)\int_0^tS_{\alpha}(t-t')\big(\chi(t')F(t') \big)dt'.
\end{equation}
\begin{lemme}[\cite{Deng1}]\label{DuhamelKernel}
	The twisted space-time Fourier transformation is given be
	$$ \widetilde{\mathcal{I}F}(\lambda,k)=\int_{\R} K(\lambda,\mu)\wt{F}(\mu,k)d\mu,
	$$
	where
	$$ K(\lambda,\mu)=\int_{\R}\Big[\frac{\widehat{\chi}(\lambda-\sigma)\widehat{\chi}(\sigma-\mu)}{i\sigma}-\frac{\widehat{\chi}(\lambda)\widehat{\chi}(\sigma-\mu)}{i\sigma}\Big]d\sigma .
	$$ 	
	Moreover, for any $B>1$, we have
	$$ |K(\lambda,\mu)|\lesssim_B \Big(\frac{1}{\langle\lambda\rangle^B}+\frac{1}{\langle\lambda-\mu\rangle^B} \Big)\frac{1}{\langle\mu\rangle}.
	$$
\end{lemme}

We will need an elementary lemma:
\begin{lemme}\label{convolution}
	Let $0\leq \sigma\leq \beta$ and $\sigma+\beta>1$. Then for any $\epsilon>0$, we have
	$$ \int_{\R}\frac{dy}{\lg y-x\rg^{\sigma} \lg y\rg^{\beta} }\lesssim \frac{1}{\lg x\rg^{\gamma}}
	$$
	where
	\begin{align*} 
	\gamma=\begin{cases}
	& \sigma+\beta-1,\;\beta<1 \\
	& \sigma-\epsilon,\; \beta=1\\
	& \sigma,\; \beta>1\\
	\end{cases}
	\end{align*}
	uniformly in $x\in\R$.
\end{lemme}	
\begin{proof}
	See Lemma 2.2 of \cite{Sun-Tz2}.
\end{proof}

\subsection{Counting lemmas and the Strichartz inequality}

We need the following elementary counting principle:
\begin{lemme}\label{countingprinciple}
Let $I,J$ be two intervals and $\phi$ be a real-valued $C^1$ function defined on $I$, then
$$ \#\{k\in I\cap\Z:\phi(k)\in J \}\leq 1+\frac{|J|}{\inf_{\xi\in I}|\phi'(\xi)| }.
$$
\end{lemme}

\begin{lemme}\label{counting1}
Assume that $ N\gg N_2\vee N_3$, then for fixed $k_2,k_3$ such that $|k_2|\sim N_2, |k_3|\sim N_3$ and $k_2\neq k_3$, we have
$$ \sum_{\substack{|k_1|\sim N } }\mathbf{1}_{\Phi_{k_1,k_2,k_3}=\mu+O(N^{\epsilon})}\lesssim N^{\epsilon}\Big(1+\frac{N^{2-\alpha}}{\lg k_2-k_3\rg}\Big),
$$	
and the implicit constant is independent of $\mu$.
\end{lemme}
\begin{proof}
This follows from the fact that
$$ \big|\frac{\partial\Phi_{k_1,k_2,k_3}}{\partial k_1}\big|=\alpha\big|\text{sgn}(k_1)|k_1|^{\alpha-1}-\text{sgn}(k_1-k_2+k_3)|k_1-k_2+k_3|^{\alpha-1}  \big|\gtrsim |k_2-k_3||k_1|^{\alpha-2}, 
$$
if $|k_1|\sim N\gg |k_2|+|k_3|$. We conclude by the elementary counting principal.
\end{proof}

\begin{lemme}\label{counting1.5}
Assume that $N_1\sim N_2\sim N_3\sim N$, then for fixed $k_2,k_3$ such that $|k_2|\sim N_2, |k_3|\sim N_3$ and $k_2\neq k_3$, we have
$$ \sum_{|k_1|\sim N_1}\mathbf{1}_{\Phi_{k_1,k_2,k_3}=\mu+O(N^{\epsilon})}\lesssim N^{\epsilon}\Big(1+\frac{N^{2-\alpha}}{\lg k_2-k_3\rg}\Big),
$$
and the implicit constant is independent of $\mu$.	
\end{lemme}
\begin{proof}
Arguing as in the proof of the previous lemma, when sgn$(k_1)\neq$sgn$(k_1-k_2+k_3)$, we have $|\partial_{k_1}\Phi|\sim N^{\alpha-1}$. When sgn$(k_1)=$sgn$(k_1-k_2+k_3)$, we may assume that $k_1>0$, hence
$$ \big|\frac{\partial\Phi_{k_1,k_2,k_3}}{\partial k_1}\big|=\alpha(\alpha-1) \int_{\min\{k_1,k_1-(k_2-k_3)\} }^{\max\{k_1,k_1-(k_2-k_3)\} }\frac{d\xi}{|\xi|^{2-\alpha}}\gtrsim |k_2-k_3|\min\big\{\frac{1}{|k_1|^{2-\alpha}},\frac{1}{|k_1-k_2+k_3|^{2-\alpha}} \big\}.
$$
This completes the proof of Lemma \ref{counting1.5}.
\end{proof}
We need also the following Lemma proved in \cite{Sun-Tz2}.
\begin{lemme}\label{counting2ndorder}
	Denote by
	$$ A_{a,l,M_1,M_2}:=\{k\in\Z: M_1\leq |k|\leq 2M_1, M_2\leq |a-k|\leq 2M_2, \big||k|^{\alpha}+|a-k|^{\alpha}-l \big|\leq r\}.
	$$
	Then for $r\geq \frac{1}{100}, 1<\alpha<2$, we have
	$$ \#A_{a,l,M_1,M_2}(r)\lesssim \min\{M_1,M_2\}^{1-\frac{\alpha}{2}}r^{\frac{1}{2}},
	$$
	where the implicit constant is independent of $a,l,r, M_1$ and $M_2$.
\end{lemme}
Next we recall the following bilinear Strichartz inequality:
\begin{lemme}[\cite{Sun-Tz2}]\label{bilinearStrichartz}
	Let $1<\alpha\leq 2$ and $s\geq \frac{1}{2}-\frac{\alpha}{4}$. Then for any $N\geq M$, we have
	\begin{align}\label{eq:bilinear}
	\|\mathbf{P}_Nf\cdot \mathbf{P}_M g\|_{L_{t,x}^2}\lesssim M^s\|\mathbf{P}_Nf\|_{X^{0,\frac{3}{8}}}\|\mathbf{P}_Mg\|_{X^{0,\frac{3}{8}}}.
	\end{align}	
\end{lemme}

\subsection{Estimates for operators}
\begin{lemme}\label{matrix-bound}
	Let $\Lambda:l^2\rightarrow l^2$ be a bounded operator with kernel $(\sigma_{kk'})_{k,k'\in\Z}$. Then
	$$ \|\Lambda\|_{l^2\rightarrow l^2}\leq \sup_{k}|\sigma_{kk}|+\Big(\sum_{k,k':k\neq k'}|\sigma_{kk'}|^2\Big)^{\frac{1}{2}}.
	$$
\end{lemme}
\begin{proof}
For any $a\in l^2,d\in l^2$, 
$$  (\Lambda a,d)_{l^2}=\sum_{k\neq k'}\sigma_{kk'}a_{k'}\ov{d}_k+\sum_{k}\sigma_{kk}a_k\ov{d}_k.
$$ 
By Cauchy-Schwartz, we have
\begin{align*}
 |(\Lambda a,d)_{l^2}|\leq &\Big(\sum_{k}|d_k|^2\Big)^{\frac{1}{2}}\Big(\sum_{k}\Big|\sum_{k,k':k'\neq  k'}|\sigma_{kk'}||a_{k'}|\Big|^2\Big)^2+\sup_{k}|\sigma_{kk}|\|a\|_{l^2}\|d\|_{l^2}\\
 \leq & \big[\sup_k|\sigma_{kk}|+\Big(\sum_{k,k':k\neq k'}|\sigma_{kk'}|^2\Big)^{\frac{1}{2}} \big]\|a\|_{l^2}\|d\|_{l^2}.
\end{align*}
In view of the duality, this completes the proof of Lemma \ref{matrix-bound}.
\end{proof}

The same argument yields:
\begin{lemme}\label{matrixboundoffdiagonal}
Let $\Lambda: l^2\rightarrow l^2$ be a bounded operator with kernel $(\sigma_{kk'})_{k,k'\in\Z}$. Then for any $L>0$,
$$ \|\Lambda\|_{l^2\rightarrow l^2}\leq L\sup_{k,k':|k-k'|< L}|\sigma_{k,k'}|+\Big(\sum_{k,k':|k-k'|\geq L}|\sigma_{kk'}|^2 \Big)^{\frac{1}{2}}.
$$
\end{lemme}
\begin{proof}
The only difference is the estimate for the quantity
$$ \Big|\sum_{k,k':|k-k'|< L}\sigma_{k,k'}a_k d_{k'} \Big|.
$$
We first pull out $\sup_{k,k':|k-k'|< L}|\sigma_{k,k'}|$ and then use Cauchy-Schwartz and Young's convolution inequality to estimate $\sum_{k,k'}\mathbf{1}_{|k-k'|< L}a_k d_{k'}$ as $$\|a\|_{l^2}\|d\|_{l^2}\|\mathbf{1}_{|\cdot|<L}\|_{l^1}\leq L\|a\|_{l^2}\|d\|_{l^2}.$$
This completes the proof of Lemma \ref{matrixboundoffdiagonal}.
\end{proof}

\begin{lemme}\label{matrix-bound2}
Let $\mathcal{G}:l_k^1l_{k_1}^2\rightarrow l_{k_2}^2$ is a bounded operator defined via
$$ b_{k,k_1}\mapsto \sum_{k_1,k}\sigma_{k,k_1}^{k_2}b_{k,k_1}.
$$
Then
$$ \|\mathcal{G}\|_{l_k^1l_{k_1}^2\rightarrow l_{k_2}^2}\leq \sup_{k,k_1}\Big(\sum_{k_2}|\sigma_{k,k_1}^{k_2}|^2 \Big)^{\frac{1}{2}}+\sup_{k,k'}\Big(\sum_{\substack{k_1,k_1'\\ (k,k_1)\neq (k',k_1')}}\Big|\sum_{k_2}\sigma_{k',k_1'}^{k_2}\ov{\sigma}_{k,k_1}^{k_2}\Big|^2 \Big)^{\frac{1}{4}}.
$$
\end{lemme}
\begin{proof}
One verifies directly that $\mathcal{G}^*:l_{k_2}^2\rightarrow l_k^{\infty}l_{k_1}^2$ is given by
$$ a_{k_2}\mapsto \sum_{k_2}\ov{\sigma}_{k,k_1}^{k_2}a_{k_2}
$$
and the matrix element of $\mathcal{G}^*\mathcal{G}$ is
$$ \sigma_{k,k_1}^{k',k_1'}=\sum_{k_2}\ov{\sigma}_{k,k_1}^{k_2}\sigma_{k',k_1'}^{k_2}.
$$ 
For $b\in l_k^{1}l_{k_1}^2, d\in l_{k}^1l_{k_1}^2$, we have
\begin{align*}
&|\langle\mathcal{G}^*\mathcal{G}b,d \rangle|\leq \Big|\sum_{k,k_1}\sigma_{k,k_1}^{k,k_1}b_{k,k_1}d_{k,k_1} \Big|+\Big|\sum_{\substack{k,k_1,k',k_1'\\
(k,k_1)\neq (k',k_1') } }\sigma_{k,k_1}^{k',k_1'}b_{k',k_1'}d_{k,k_1} \Big|\\
\leq &\sup_{k,k_1}|\sigma_{k,k_1}^{k,k_1}|\cdot \|b_{k,k_1}\|_{l_k^1l_{k_1}^2}\|d_{k,k_1}\|_{l_k^{\infty}l_{k_1}^2}+\sum_{k',k}\Big(\sum_{\substack{k_1,k_1'\\ (k',k_1')\neq (k,k_1)}}|\sigma_{k,k_1}^{k',k_1'}|^2\Big)^{\frac{1}{2}}\Big(\sum_{k_1,k_1'}|b_{k',k_1'}|^2|d_{k,k_1}|^2 \Big)^{\frac{1}{2}}\\
\leq &\sup_{k,k_1}|\sigma_{k,k_1}^{k,k_1}|\cdot \|b_{k,k_1}\|_{l_k^1l_{k_1}^2}\|d_{k,k_1}\|_{l_k^{\infty}l_{k_1}^2}+\sup_{k',k}\Big(\sum_{\substack{k_1,k_1'\\ (k',k_1')\neq (k,k_1)} }|\sigma_{k,k_1}^{k',k_1'}|^2\Big)^{\frac{1}{2}}\cdot \|b_{k',k_1'}\|_{l_{k'}^1l_{k_1'}^2}\|d_{k,k_1}\|_{l_k^1l_{k_1}^2}.
\end{align*}
Using the fact that $l^1\hookrightarrow l^{\infty}$, this implies that
$$ \|\mathcal{G}^*\mathcal{G}\|_{l_{k}^1l_{k_1}^2\rightarrow l_k^{\infty}l_{k_1}^2}\leq \sup_{k,k_1}\sum_{k_2}|\sigma_{k,k_1}^{k_2}|^2+\sup_{k,k'}\Big(\sum_{\substack{k_1,k_1'\\ (k',k_1')\neq (k,k_1) } } \Big|\sum_{k_2}\sigma_{k',k_1'}^{k_2}\ov{\sigma}_{k,k_1}^{k_2} \Big|^2\Big)^{\frac{1}{2}}
$$
From $\|\mathcal{G}\|_{l_k^1l_{k_1}^2\rightarrow l_{k_2}^2}=\|\mathcal{G}^*\mathcal{G}\|_{l_{k}^1l_{k_1}^2\rightarrow l_k^{\infty}l_{k_1}^2}^{\frac{1}{2}}$, we complete the proof of Lemma \ref{matrix-bound2}.	
\end{proof}

Given $h(\lambda)$, a $\lambda$-dependent family of linaer operators on $L^2(\T)$, we may identify it as a $\lambda$-dependent family of linear operators on $l^2$ with kernel $(h_{kk^*}(\lambda))$, where
$$ h_{kk^*}(\lambda)=(h(\lambda)(\mathrm{e}_{k^*}),\mathrm{e}_{k}).
$$
We will need two technical lemmas (in the proof of Proposition \ref{keyinduction}) concerning some estimates of the kernel related to $h(\lambda)$: 
\begin{lemme}\label{lemma:operator1}
Let $h$ be a $\lambda$-dependent family of operators with kernel $(h_{kk^*}(\lambda))$.	Consider the operator $\widetilde{\mathcal{H}}$ with kernel
	$$ \wt{H}_{kk^*}(\lambda):=\sum_{k'}\int_\R \wt{\Theta}_{kk'}(\lambda,\lambda')h_{k'k^*}(\lambda')d\lambda',
	$$
	where $\wt{\Theta}_{kk'}(\lambda,\lambda')$ is supported in $|k-k'|\lesssim L$, then for any $\beta\geq 0$, we have
	$$ \Big\|\Big\lg\frac{|k-k^*|}{L}\Big\rg^{\beta}\wt{H}_{kk^*}(\lambda) \Big\|_{L_{\lambda}^2l_{k}^2l_{k^*}^2}\lesssim_{\beta} \|\wt{\Theta}_{kk'}(\lambda,\lambda')\|_{L_{\lambda'}^2l_{k'}^2\rightarrow L_{\lambda}^2l_k^2}\cdot \Big\|\Big\lg\frac{|k'-k^*|}{L}\Big\rg^{\beta}h_{k'k^*}(\lambda') \Big\|_{L_{\lambda'}^2l_{k'}^2l_{k^*}^2}.
	$$
\end{lemme}
\begin{proof}
	See Proposition 2.5 of \cite{Deng2}.
\end{proof}

\begin{lemme}\label{lemma:operator2}
	Consider the operator $\wt{\mathcal{H}}$ with kernel
	$$ \wt{H}_{kk^*}(\lambda):=\sum_{k'}\int_\R \wt{\Theta}_{kk'}(\lambda,\lambda')h_{k'k^*}(\lambda')d\lambda',
	$$
	then we have
	\begin{align*}
	& \mathrm{(i)}\quad \|\wt{H}_{kk^*}(\lambda)\|_{l_k^{\infty}L_{\lambda}^ql_{k^*}^2}\lesssim  \big\|\wt{\Theta}_{kk'}(\lambda,\lambda') \big\|_{l_k^{\infty}L_{\lambda}^qL_{\lambda'}^2l_{k'}^2}\|h_{k'k^*}(\lambda')\|_{l_{k^*}^2\rightarrow L_{\lambda}^2l_{k'}^2};\\
	&\mathrm{(ii)}\quad  \|\widetilde{H}_{kk^*}(\lambda)\|_{L_{\lambda}^2l_k^2l_{k^*}^2}\lesssim \|\widetilde{\Theta}_{k,k'}(\lambda,\lambda')\|_{L_{\lambda,\lambda'}^2l_{k,k'}^2 }\|h_{k'k^*}(\lambda')\|_{l_{k^*}^2\rightarrow L_{\lambda'}^2l_{k'}^2};\\
	&\mathrm{(iii)}\quad 
	\|\widetilde{H}_{kk^*}(\lambda)\|_{L_{\lambda}^2l_k^2l_{k^*}^2}\leq \|\wt{\Theta}_{k,k'}(\lambda,\lambda')\|_{L_{\lambda'}^2l_{k'}^2\rightarrow L_{\lambda}^2l_k^2}\|h_{k'k^*}(\lambda')\|_{L_{\lambda'}^2l_{k',k^*}^2}, 
	\end{align*}
	where we mean
	$$ \|h_{k'k^*}(\lambda')\|_{l_{k^*}^2\rightarrow L_{\lambda'}^2l_{k'}^2}:=\|h(\lambda')\|_{L_x^2\rightarrow L_{\lambda'}^2L_x^2}.
	$$
\end{lemme}
\begin{proof}
	(iii) is relatively simple. For fixed $k^*$, viewing $\wt{\Theta}_{kk'}(\lambda,\lambda')$ as the kernel of the operator from $L_{\lambda'}^2l_{k'}^2$ to $L_{\lambda}^2l_k^2$, we have
	$$ \|\wt{H}_{kk^*}(\lambda)\|_{L_{\lambda}^2l_k^2}\leq \|\wt{\Theta}_{k,k'}(\lambda,\lambda')\|_{L_{\lambda'}^2l_{k'}^2\rightarrow L_{\lambda}^2l_k^2}\|h_{k'k^*}(\lambda')\|_{L_{\lambda'}^2l_{k'}^2}.
	$$
	Taking $l_{k^*}^2$ to both sides and by Fubini, we obtain (iii).
	
	To prove (i) and (ii), recall that   $h_{k'k^*}(\lambda')=\big( h(\lambda')(\mathrm{e}_{k^*}),\mathrm{e}_{k'}\big)_{L_x^2}$. Denote by $h^*(\lambda')$ the adjoint of $h(\lambda')$. By linearity, we have
	\begin{align*}
	\wt{H}_{kk^*}(\lambda)=&\sum_{k'}\int \wt{\Theta}_{kk'}(\lambda,\lambda')\big(h(\lambda')(\mathrm{e}_{k^*}),\mathrm{e}_{k'} \big)_{L_{x}^2}d\lambda'\\
	=&\int \big(\mathrm{e}_{k^*},\sum_{k'}\ov{\wt{\Theta}}_{kk'}(\lambda,\lambda')h^*(\lambda')(\mathrm{e}_{k'}) \big)_{L_x^2}d\lambda'\\
	=&\Big(\mathrm{e}_{k^*}, \int h^*(\lambda')\big(\sum_{k'}\ov{\wt{\Theta}}_{kk'}(\lambda,\lambda')\mathrm{e}_{k'} \big)d\lambda \Big)_{L_{x}^2}.
	\end{align*}
	Here we omit the issue of the legality of changing the order of the integration and the summation, which can be justified by a standard density argument. 
	
	Denote by $F_k(\lambda,\lambda')=\sum_{k'}\ov{\wt{\Theta}}_{kk'}(\lambda,\lambda')\mathrm{e}_{k'}$. For fixed $k,\lambda$, viewing $\int h^*(\lambda')F_k(\lambda,\lambda')d\lambda'$ as a function in $L^2(\T)$, we have
	\begin{align*}
	\|\wt{H}_{kk^*}(\lambda)\|_{l_{k^*}^2}^2:=\sum_{k^*}\big|\big(\mathrm{e}_{k^*},\int h^*(\lambda')F_k(\lambda,\lambda')d\lambda' \big)_{L_x^2} \big|^2
	=\Big\|\int h^*(\lambda')F_k(\lambda,\lambda')d\lambda' \Big\|_{L_x^2}^2,
	\end{align*}
	thanks to Plancherel. Now viewing $h$ as a linear operator from $L_{x}^2$ to $L_{\lambda'}^2L_x^2$ with kernel $h_{k'k^*}(\lambda')$, hence $h^*$ is from $L_{\lambda'}^2L_x^2$ to $L_x^2$ and
	$$ h^*(G):=\int h^*(\lambda')(G(\lambda'))d\lambda'.
	$$
	By viewing $F_k(\lambda,\cdot)$ as a function (for fixed $k,\lambda$) in $L_{\lambda'}^2L_x^2$, we have
	$$ \int h^*(\lambda')F_k(\lambda,\lambda')d\lambda'=h^*\big(F_k(\lambda,\cdot)\big).
	$$
	Therefore, 
	\begin{align*}
	\|\wt{H}_{kk^*}(\lambda)\|_{l_{k^*}^2}=&\big\|h^*\big(F_k(\lambda,\cdot) \big)\big\|_{L_x^2}\leq \|h^*\|_{L_{\lambda'}^2L_x^2\rightarrow L_x^2}\|F_k(\lambda,\lambda')\|_{L_{\lambda'}^2L_x^2}\\
	=&\|h\|_{L_x^2\rightarrow L_{\lambda'}^2L_x^2}\cdot\Big\|\sum_{k'}\ov{\wt{\Theta}}_{kk'}(\lambda,\lambda')\mathrm{e}_{k'} \Big\|_{L_{\lambda'}^2L_x^2}\\=&\|h\|_{L_x^2\rightarrow L_{\lambda'}^2L_x^2}\cdot\big\|\wt{\Theta}_{kk'}(\lambda,\lambda') \big\|_{L_{\lambda'}^2l_{k'}^2}.
	\end{align*}
	Taking  $L_{\lambda}^q$ and $l_k^{\infty}$ or $L_{\lambda}^2l_k^{2}$ to both sides, we obtain (i) and (ii), with respectively.
	The proof of Lemma \ref{lemma:operator2} is complete.
\end{proof}

\subsection{Probability tool-box}

We denote by $\mathcal{B}_{\leq N}$ ($\mathcal{B}_{>N}$), the Borel $\sigma$ algebra generated by $\{g_k(\omega):|k|\leq N \}$( $\{g_k(\omega):|k|>N \}$), and $\mathcal{B}_N$ be the Borel $\sigma$ algebra generated by $\big\{g_k(\omega): \frac{N}{2}<|k|\leq N \big\}$. For a $\sigma$ algebra $\mathcal{B}$, we use the notation $X\in\mathcal{B}$ to mean that $X$ is $\mathcal{B}$-measurable and $X\perp\mathcal{B}$ to mean that $X$ is independent of $\mathcal{B}$.

Let $(\Omega,\mathcal{F},\mathbb{P})$ be a probability space, $(\mathcal{X},\mu)$ be a measure space and $\mathcal{G}\subset\mathcal{F}$ be a sub $\sigma$-algebra. Let $X,Y$ be two random variables and let $f(x,\omega)$ be a random function with value on $L^r(\mathcal{X},\mu)$. We recall the following classical inequalities for the conditional expectations:
\begin{align*}
&\text{(i) H\"older: }\quad \mathbb{E}[|XY||\mathcal{G}]\leq \big(\mathbb{E}[|X|^p|\mathcal{G}]\big)^{\frac{1}{p}}\cdot 
\big(\mathbb{E}[|Y|^{p'}|\mathcal{G}]\big)^{\frac{1}{p'}};\\
&\text{(ii) Minkowski: If $p\geq r\geq 1$, }
\mathbb{E}[\|f(x,\omega)\|_{L_x^r}^p|\mathcal{G} ]\leq \big\|\big(\mathbb{E}[|f(x,\omega)|^p|\mathcal{G}]
\big)^{\frac{1}{p}}
\big\|_{L_x^r};\\
&\text{(iii) Chebyshev: For any } \lambda>0 \text{ and } p\in(0,\infty), \quad 
\mathbb{P}[|X|>\lambda|\mathcal{G}]\leq \frac{1}{\lambda^p}\mathbb{E}[|X|^p|\mathcal{G} ].
\end{align*}

\begin{lemme}[Conditional Wiener Chaos]\label{chaosWiener}
	Let $(g_j(\omega))_{j\in E}$ be a inpendent, identically distributed complex Gaussians and $E$ is a finite index set. Let $\mathcal{C}$ be a $\sigma$-algebra independent of $(g_j(\omega))_{j\in E}$. Assume that $(c_{k_1,k_2,\cdots,k_m}(\omega))_{(k_1,\cdots,k_n)\in\N^m}$ is a sequence of $\mathcal{C}$-measurable random variables. 
	Then for any finite subset $S\subset E^m$ and $p\geq 1$, we have
	\begin{align*}
	&\Big(\mathbb{E}\Big[\Big|\sum_{(k_1,k_2,\cdots, k_m)\in S} c_{k_1,k_2,\cdots,k_m}(\omega)\prod_{j=1}^{m}g_{k_j}^{\iota_j}(\omega) \Big|^{2p}\Big|\mathcal{C}\Big]\Big)^{\frac{1}{2p}}\\ \leq &C_0(2p-1)^{\frac{m}{2}} \Big(\mathbb{E}\Big[\Big|\sum_{(k_1,k_2,\cdots, k_m)\in S} c_{k_1,k_2,\cdots,k_m}(\omega)\prod_{j=1}^{m}g_{k_j}^{\iota_j}(\omega) \Big|^2\Big|\mathcal{C}\Big]\Big)^{\frac{1}{2}},
	\end{align*}
	where  $g_{k_j}^{\iota_j}(\omega)=g_{k_j}(\omega)$ or $\ov{g}_{k_j}(\omega)$ and the uniform constant $C_0$ is independent of the set $S$, the $\sigma$-algebra $\mathcal{C}$ and the number $m$ and $p\geq 1$.
\end{lemme}

\begin{proof}
	Since the conditional expectation can be viewed as partial integration for the product probability space $\Omega=\Omega_1\times\Omega_2$, the conclusion follows from the usual Wiener chaos estimate.
\end{proof}

\begin{corollaire}\label{largedeviation}
	Assume that $1\leq p_1,p_2,\cdots, p_n<\infty$ and $E\subset\N$ is a finite index set. Let $(g_j)_{j\in E}$ be a sequence of independent standard complex Gaussians. Let $\mathcal{C}$ be a $\sigma$-algebra independent of the $\sigma$-algebra generated by $(g_j)_{j\in E}$. Let $(c_{k_1^*,\cdots,k_m^*}(z_1,\cdots,z_n;\omega))_{(k_1^*,\cdots,k_m^*)\in E^m}$ be a sequence of $\mathcal{C}$-measurable random variables with values in $L^{p_1}(Z_1)\times\cdots L^{p_n}(Z_n)$. Consider the function
	$$ F(z_1,\cdots,z_n,\omega):=\sum_{(k_1^*,\cdots, k_m^*)\in E^m }c_{k_1^*,\cdots, k_{m}^*}(z_1,\cdots,z_n;\omega)\prod_{j=1}^mg_{k_j}^{\iota_j}(\omega).
	$$
	Assume that there exists some constant $\Lambda_0>0$, such that 
	\begin{align}\label{conditionalbound}
	\| \big(\mathbb{E}[|F|^2|\mathcal{C} ]\big)^{\frac{1}{2}} \|_{L_{z_1}^{p_1}\cdots L_{z_n}^{p_n}}\leq \Lambda_0,
	\end{align}
	then for any $R>0$, outside an exceptional set of probability\footnote{This exceptional set depends on the random functions $c_{k_1^*,\cdots,k_m^*}$} $<C\mathrm{e}^{-cR^{\frac{2}{m}}}$, 
	we have
	$$ \|F\|_{L_{z_1}^{p_1}\cdots L_{z_n}^{p_n} }\leq R\Lambda_0.
	$$
\end{corollaire}

\begin{proof}
	Let $p\geq 2\max\{p_1,\cdots,p_n\}$. By the conditional Chebyshev and the Minkowski inequalities, 
	$$ \mathbb{P}[\|F\|_{L_{z_1}^{p_1}\cdots L_{z_n}^{p_n} }>\lambda |\mathcal{C} ]\leq \frac{1}{\lambda^{2p}}\big\|
	(\mathbb{E}[|F(z_1,\cdots,z_n,\omega)|^{2p}|\mathcal{C}])^{\frac{1}{2p}} 
	\big\|_{L_{z_1}^{p_1}\cdots L_{z_n}^{p_n}
	}^{2p}. 
	$$
	By Lemma \ref{chaosWiener} and the assumption \eqref{conditionalbound}, we have
	$$ \big\|
	(\mathbb{E}[|F(z_1,\cdots,z_n,\omega)|^{2p}|\mathcal{C}])^{\frac{1}{2p}} 
	\big\|_{L_{z_1}^{p_1}\cdots L_{z_n}^{p_n}
	}^{2p}\leq C^{2p}(2p)^{mp}\Lambda_0^{2p}.
	$$
	By choosing $\lambda=R\Lambda_0$ and optimizing the choice of $p$, we obtain that
	$$ \mathbb{P}[\|F\|_{L_{z_1}^{p_1}\cdots L_{z_n}^{p_n}}>R\Lambda_0|\mathcal{C}]\leq C\mathrm{e}^{-cR^{\frac{2}{m}}}.
	$$
	By taking the expectation once to the inequality above, the proof of Corollary \ref{largedeviation} is complete.
\end{proof}


\section{Key iterative steps}\label{sec:keyiterative}

Though the smooth solutions $(v_N)_{N\in 2^{\N}}$ of \eqref{FNLS-gauged} with initial data $\Pi_N\phi^{\omega}$ already exist, the proof of Theorem \ref{prob:LWP} will be achieved by solving local-in-time fix-point problems.   
To this end, we recall the iteration scheme introduced in \cite{Deng2} with a slightly different setting, in order to solve the non-truncated equation. 
\subsection{Rigorous resolution scheme}
Here we need to take into consideration of the time-restriction issue. Recall that $\chi\in C_c^{\infty}(\R)$ is a bump function which is $1$ on $[-\frac{1}{2},\frac{1}{2}]$ and is zero outside $[-1,1]$. For $0<T<1$, we define $\chi_T(t):=\chi(t/T)$. We now rewrite the ansatz with time localization. We will use the notation $A^{\#}$ to define the quantities after time-localization procedure. In what follows, we describe inductive definition for the time-restriction that gives the rigorous resolution scheme\footnote{Since for the truncated initial data, the global smooth solution exists and is unique, the fix-point procedure used here is only to establish required bounds in suitable function spaces.}. This contains four steps, including two fix-point problems:

\noi
$\bullet${\bf Initial step:} We define 
$$w_1^{\#}(t):=w_1(t)\chi_T(t),\quad \psi_{\frac{1}{2}}^{M,\#}(t):=\chi(t)S_{\alpha}(t)\mathbf{P}_M\phi^{\omega},\quad H_{kk^*}^{M,\frac{1}{2},\#}=\chi(t)\mathrm{e}^{it|k|^{\alpha}}\mathbf{1}_{k=k^*},$$
for any dyadic number $M\geq \frac{1}{2}$.

\noi
$\bullet${\bf Induction assumption}:
Suppose that $M\geq 1$ is a given dyadic number and we have defined:
$$w_N^{\#}(t), \; \forall N\leq M; \quad  h^{N',L,\#},\;\forall L\leq L_{N'}, L<M.$$
We need to define $w_{2M}^{\#}$ and $h^{N',M,\#}$ for all $N'>M^{\frac{1}{1-\delta}}$ (i.e. $M\leq L_{N'}$). First, note that for all $L\leq M, R<L^{1-\delta}<M$, $\zeta_R^{L,\#}:=\psi_R^{L,\#}-\psi_{\frac{R}{2}}^{L,\#}$ are well-defined, since they can be written in terms of the operators $H^{L,R,\#}$ (hence in terms of the operators $h^{L,R,\#}$). Thus for all $L\leq N$, the following functions 
$$ v_L^{\#}:=\sum_{\frac{1}{2}\leq L'\leq L}y_{L'},\quad \text{ where }\quad  y_{L'}^{\#}:=\chi(t)S_{\alpha}(t)\mathbf{P}_{L'}\phi^{\omega}+\sum_{(L',R'):R'<L'^{1-\delta}}\zeta_{R'}^{L',\#} +w_{L'}(t)
$$
are also well-defined.

\noi
$\bullet${\bf Uniform bounds 1: Estimates for the linear operator:}
Next for $L=M$ and $N'>M^{\frac{1}{1-\delta}}$, we define $H_{kk*}^{N',M,\#}$ by taking the $k$-th Fourier mode of the solution $\varphi^{\#}$ to the equation
\begin{equation}\label{induction:tosolve1}
\varphi^{\#}(t)=\chi(t)S_{\alpha}(t)(\mathrm{e}_{k^*})+2i\chi_T(t)\cdot\Pi_{N'}\mathcal{I}\mathcal{N}_3(\varphi^{\#},\Pi_Mv_M^{\#},\Pi_Mv_M^{\#}),\quad \text{ where } \frac{N'}{2}<|k^*|\leq N'.
\end{equation}
We note that knowing $v_M$, for fixed $N'$, the solution $\varphi^{\#}(t)$ exists and is unique, from a simple Grownwall type argument\footnote{More precisely, we may multiply both sides by $\ov{\varphi^{\#}}$ and doing the integration by part. This allows us to control $\frac{d}{dt}\|\varphi^{\#}(t)\|_{L^2(\T)}^2$ by $C(N',v_M)\|\varphi^{\#}(t)\|_{L^2(\T)}^2$. }. Since it is a linear equation, this will turn out to be true, if $T>0$ is sufficiently small.

\noi
$\bullet${\bf Uniform bounds 2: Estimates for the smooth remainder:}
We finally write down the equation of $w_{2M}^{\#}$ to finish the induction step. Since $L_{2M}<(2M)^{1-\delta}\leq M$ (true for large $M$, the case that we concern), the function $\psi_{L_{2M}}^{2M,\#}:=\mathcal{H}^{2M,L_{2M},\#}(\mathbf{P}_{2M}\phi^{\omega})$ is well-defined. Now we define $w_{2M}^{\#}$ by solving the following equation:
\begin{equation}\label{induction:tosolve2}
\begin{split}
w_{2M}^{\#}(t)=&-i\chi_T(t)\mathcal{I}\big[\mathcal{N}\big(w_{2M}^{\#}+\psi_{L_{2M}}^{2M,\#}+v_M^{\#} \big)-\mathcal{N}\big(v_{M}^{\#}\big) \big]\\
&-2i\chi_T(t)\Pi_{2M}\mathcal{I}\mathcal{N}_3\big(\psi_{L_{2M}}^{2M,\#},\Pi_{L_{2M}}v_{L_{2M}}^{\#}, \Pi_{L_{2M}}v_{L_{2M}}^{\#} \big).
\end{split}
\end{equation}

Note that
to solve $w_{2M}^{\#}$ through \eqref{induction:tosolve2}, we expand the right side of \eqref{induction:tosolve2}, the resulting terms can be grouped as follows:
\begin{enumerate}
 	\item[1)] At least two entries in $\mathcal{I}\mathcal{N}_3\big(\cdot,\cdot,\cdot \big)$ are $\psi_{L_{2M}}^{2M,\#}$ and the other one (if not the same) is $v_{M}^{\#}$;
	\item[2)] Exactly one entry in $\mathcal{I}\mathcal{N}_3\big(\cdot,\cdot,\cdot \big)$ is $\psi_{L_{2M}}^{2M,\#}$, at least one $v_{M}^{\#}-\Pi_{L_{2M}}v_{L_{2M}}^{\#}$, and the last one (if it is different from the two) is $v_{M}^{\#}$;
	\item[3)] $\mathcal{N}_3\big(v_{M}^{\#},\psi_{L_{2M}}^{2M,\#},v_{M}^{\#} \big)$;
	\item[4)] At least two $w_{2M}^{\#}$ or exactly one $w_{2M}^{\#}$ and at least one $\psi_{L_{2M}}^{2M}$.
	\item[5)] Exactly one $w_{2M}^{\#}$ and all others equal $v_{M}^{\#}$;
	\item[6)] The projective term: $\Pi_{2M}^{\perp}\mathcal{I}\mathcal{N}_3(\psi_{L_{2M}}^{2M,\#},v_M^{\#},v_M^{\#} )$;
	\item[7)] The diagonal nonlinear term $\mathcal{I}\mathcal{R}_0:=\mathcal{I}\mathcal{N}_0(w_{2M}^{\#}+\psi_{L_{2M}}^{2M,\#}+v_M^{\#})-\mathcal{I}\mathcal{N}_0(v_M^{\#})$. 
\end{enumerate}

\subsection{Key multi-linear terms}
First we describe the key multi-linear terms in order to estimate the linear operators through \eqref{induction:tosolve1}, 
Since the solution $\varphi^{\#}$ of \eqref{induction:tosolve1} is $\mathcal{H}^{N',M,\#}(\mathrm{e}_{k^*})$, by taking the difference $\mathcal{H}^{N',M,\#}(\mathrm{e}_{k^*})-\mathcal{H}^{N',\frac{M}{2},\#}(\mathrm{e}_{k^*})$, we have
\begin{equation*}
\begin{split}
h^{N',M,\#}(\mathrm{e}_{k^*})=&2i\chi_T(t)\mathcal{I}\Pi_{N'}\mathcal{N}_3\big(h^{N',M,\#}(\mathrm{e}_{k^*}),\Pi_Mv_M^{\#},\Pi_Mv_M^{\#} \big)+2i\mathcal{P}_{N',M}^{+}\big(H^{N',\frac{M}{2}}(\mathrm{e}_{k^*})\big),
\end{split}
\end{equation*} 
where the operator $\mathcal{P}_{N',L}^{+}$ is defined by
\begin{equation}\label{Pml+}
\mathcal{P}_{N',L}^{+}(w):=\chi_T(t)\cdot\mathcal{I}\Pi_{N'}\big[\mathcal{N}_3\big(w,\Pi_Lv_L^{\#},\Pi_Lv_L^{\#}\big)-\mathcal{N}_3\big(w,\Pi_{\frac{L}{2}}v_{\frac{L}{2}}^{\#}, \Pi_{\frac{L}{2}}v_{\frac{L}{2}}^{\#}  \big)\big]
\end{equation}
for $L>\frac{1}{2}$ and
$$ \mathcal{P}_{N',\frac{1}{2}}^+(w):=\chi_T(t)\cdot\mathcal{I}\Pi_{N'}\mathcal{N}_3(w,\Pi_{\frac{1}{2}}v_{\frac{1}{2}}^{\#},\Pi_{\frac{1}{2}}v_{\frac{1}{2}}^{\#}).
 $$ 
We will need an analogue of the operator above later:
\begin{equation}\label{Pml-}
\mathcal{P}_{N',L}^{-}(w):=\chi_T(t)\cdot\mathcal{I}\big[\Pi_{M'}\mathcal{N}_3\big(\Pi_Lv_L^{\#},w,\Pi_Lv_L^{\#}\big)-\Pi_{N'}\mathcal{N}_3\big(\Pi_{\frac{L}{2}}v_{\frac{L}{2}}^{\#},w,\Pi_{\frac{L}{2}} v_{\frac{L}{2}}^{\#}  \big)\big]
\end{equation}
for $L>\frac{1}{2}$ and
$$ \mathcal{P}_{N',\frac{1}{2}}^-(w):=\chi_T(t)\cdot\mathcal{I}\Pi_{N'}\mathcal{N}_3(\Pi_{\frac{1}{2}}v_{\frac{1}{2}}^{\#},w,\Pi_{\frac{1}{2}}v_{\frac{1}{2}}^{\#}).
$$ 
Now \eqref{induction:tosolve1} is reduced to the following equation:
\begin{equation}\label{induction:tosolve1.5}
\begin{split}
h^{N',M,\#}(\mathrm{e}_{k^*})=&2i\sum_{L\leq M}\mathcal{P}_{N',L}^{+}\big(h^{N',M,\#}(\mathrm{e}_{k^*})\big)+2i\sum_{L<M}\mathcal{P}_{N',M}^{+}\big(h^{N',L,\#}(\mathrm{e}_{k^*}) \big)\\
+&2i\mathcal{P}_{N',M}^{+}\big(H^{N',\frac{1}{2},\#}(\mathrm{e}_{k^*})\big).
\end{split}
\end{equation}

To deal with \eqref{induction:tosolve2}, we need to treat the terms of type 1)-7). By definition, we have the decompositions
$$ v_M^{\#}=\sum_{N\leq M}\big(\psi_{L_N}^{N,\#}+w_N^{\#}\big),\quad v_M^{\#}-\Pi_{L_{2M}}v_{L_{2M}}^{\#}=\sum_{L_{2M}<R\leq M}\big(\psi_{L_R}^{R,\#}+w_R^{\#} \big)+\Pi_{L_{2M}}^{\perp}v_{L_{2M}}^{\#},
$$
$$ \psi_{L_N}^{N,\#}=\chi(t)S_{\alpha}(t)\mathbf{P}_N\phi^{\omega}+\sum_{L\leq L_N}\zeta_L^{N,\#}.
$$
We now precise the multi-linear terms according to their types. To simplify the notation, we will not write $(\cdot)^{\#}$ for the time-restriction here, and we mean $\zeta^{2M}$ by a term of the form $\zeta_{L}^{2M}$ for some $L\leq L_{2M}$ or $\psi_{\frac{1}{2}}^{2M}$.
\begin{enumerate}
	\item[1)] $\mathcal{I}\mathcal{N}_3\big(\zeta^{2M},\zeta^{2M},* \big),\quad \mathcal{I}\mathcal{N}_3\big(\zeta^{2M},*,\zeta^{2M}\big)$;\\
	\item[2)] $\mathcal{I}\mathcal{N}_3\big(\zeta^{2M},w_N+\psi_{L_N}^{N}+\Pi_{L_{2M}}^{\perp}v_{L_{2M}}^{\#},* \big),\quad \mathcal{I}\mathcal{N}_3\big(\zeta^{2M},*, w_N+\psi_{L_N}^N+\Pi_{L_{2M}}^{\perp}v_{L_{2M}}^{\#} \big),\\ \mathcal{I}\mathcal{N}_3\big(w_M+\psi_{L_N}^N+\Pi_{L_{2M}}^{\perp}v_{L_{2M}}^{\#},\zeta^{2M},* \big)$ with some $N\geq M^{1-\delta}$;
	\item[3)] $\mathcal{I}\mathcal{N}_3\big(*,\zeta^{2N},* \big)$;\\
	\item[4)] $\mathcal{I}\mathcal{N}_3\big(w_{2M}\big),\quad \mathcal{I}\mathcal{N}_3\big( w_{2M},*, w_{2M}\big),\quad \mathcal{I}\mathcal{N}_3\big( w_{2M},w_{2M},*\big),$\\
	$\mathcal{I}\mathcal{N}_3\big(w_M,\zeta^{2M},* \big),\quad \mathcal{I}\mathcal{N}_3\big( \zeta^{2M},w_{2M},*\big),\quad \mathcal{I}\mathcal{N}_3\big( \zeta^{2M},*,w_{2M}\big)$;\\
		\item[5)] $\mathcal{I}\mathcal{N}_3\big(w_{2M},v_M,v_M \big),\quad \mathcal{I}\mathcal{N}_3\big(v_M,w_{2M},v_M \big)$;
	\item[6)] $\Pi_{2M}^{\perp}\mathcal{I}\mathcal{N}_3(\zeta^{2M},v_M,v_M)$;
	\item[7)] Trivial resonances: $\mathcal{I}\mathcal{R}_0$.
\end{enumerate}
where the input $*$ stands for a term belonging to the set of functions
$$\big\{ v_M,\; \Pi_{L_{2M}}v_{L_{2M}},\; \zeta^{2M}\big\}.
$$

In summary, the \emph{only possible high-low-low interactions} appear in the following situations:
\begin{itemize}
	\item Case 3), but $\zeta^{2M}$ is in the "good" position.
	\item Case 5), but with $w_{2M}$ who has the dominated frequency. These terms can be viewed as errors of certain linearization procedure and will be treated by the operator $\mathcal{P}^{\pm}$ defined later.
	\item Pseudo high-low interactions\footnote{This is the main different issue compared to the truncated FNLS.}: Terms in  Case 2) and Case 5) involving the entry $\Pi_{L_{2M}}^{\perp}v_{L_{2M}}$ or $\Pi_M^{\perp}v_M$. Though these are not high-low interaction, when we decompose $v_M=\sum_{M'\leq M}y_{M'}$, the portions $y_{M'}$ coming from $M'\ll M$ may not have sufficient decay in the estimates, they behave like just $y_{M'}$ in a priori. This is caused by the fact that the Fourier support of $v_M$ is not bounded, when we truncate only the initial data. Extra estimates for $\Pi_M^{\perp}v_M$ is needed. 
\end{itemize} 

\subsection{Induction step}
Now we summarize the induction step. First we define the linear operators (with $L<N^{1-\delta}$):
\begin{equation}\label{PNL+}
\mathcal{P}_{N,L}^{+}(w):=\chi_T(t)\cdot\mathcal{I}\Pi_N\big[\mathcal{N}_3\big(w,\Pi_Lv_L^{\#},\Pi_Lv_L^{\#}\big)-\mathcal{N}_3\big(w,\Pi_{\frac{L}{2}}v_{\frac{L}{2}}^{\#}, \Pi_{\frac{L}{2}}v_{\frac{L}{2}}^{\#}  \big)\big]
\end{equation}
and
\begin{equation}\label{PNL-}
\mathcal{P}_{N,L}^{-}(w):=\chi_T(t)\cdot\mathcal{I}\Pi_N\big[\mathcal{N}_3\big(\Pi_Lv_L^{\#},w,\Pi_Lv_L^{\#}\big)-\mathcal{N}_3\big(\Pi_{\frac{L}{2}}v_{\frac{L}{2}}^{\#},w, \Pi_{\frac{L}{2}}v_{\frac{L}{2}}^{\#}  \big)\big]
\end{equation}

for $L\leq L_N=\max\{L':L'<N^{1-\delta} \}$. Denote by $\Theta_{k,k'}^{N,L}(t,t')$ the kernel of the operator $\mathcal{P}_{N,L}^{+}$. Note that on the support of $\Theta_{kk'}^{N,L}(t,t')$, $|k-k'|\lesssim L$.
Following \cite{Deng2}, for a given dyadic number $M$, we call Loc($M$) the following uniform bounds:
for all $(L,N)$, such that $\frac{1}{2}<L<M$, $L<N^{1-\delta}$,
\begin{align*}
&(\mathrm{i})\quad  \|h^{N,L,\#}\|_{Y^b}\leq L^{-\delta_0^2},\;\|h^{N,L,\#}\|_{S^{b,q}}\leq N^{\epsilon_1}L^{-\nu},\;\|h^{N,L,\#}\|_{Z^b}\leq N^{1-\frac{\alpha}{2}+\epsilon_1}L^{-\nu}; \\
&(\mathrm{ii})\quad  \|\zeta_L^{N,\#}\|_{X_{\infty,q}^{0,\frac{2b_0}{q'}}}\leq N^{\epsilon_2}L^{-\nu},\; \|\zeta_L^{N,\#}\|_{X^{0,b_0}}\leq N^{-(\alpha-1)+\epsilon_2}L^{-\nu}; \\
&(\mathrm{iii})\quad
\|\zeta_L^{N,\#}\|_{L_t^{4}L_x^{\infty} }\leq N^{-(\alpha-1)+\epsilon_2}L^{\frac{1}{2}-\nu}
 \\
&(\mathrm{iv})\quad
\Big\|\Big\lg\frac{|k-k^*|}{L}\Big\rg^{\kappa} h^{N,L,\#}_{kk^*} \Big\|_{Z^b}\leq N;\\
&(\mathrm{v})\quad
\|\mathcal{P}_{N,L}^{\pm}\|_{X^{0,b}\rightarrow X^{0,b}}\leq T^{\theta}L^{-\delta_0}; 
\\
&(\mathrm{vi})\quad
\big\|\mathbf{1}_{|k|,|k'|\geq \frac{N}{4}}\cdot\Theta_{kk'}^{N,L}(t,t') \big\|_{Z^{b,b}}\leq T^{\theta}N^{1-\frac{\alpha}{2}+\epsilon_1}L^{-\nu};\\
&(\mathrm{vii})\quad
\big\|\mathbf{1}_{|k|,|k'|\geq \frac{N}{4}}\cdot\Theta_{kk'}^{N,L}(t,t') \big\|_{S^{b,b}}\leq T^{\theta}N^{\epsilon_1}L^{-\nu};\\
&(\mathrm{viii})\quad
\|w_N^{\#}\|_{X^{0,b}}\leq N^{-s}, \|\Pi_{N_0}^{\perp}w_N^{\#}\|_{X^{0,b}}\leq N_0^{-s}  \quad \forall N\leq M, N_0\geq N.
\end{align*}
\begin{remarque}\label{numerical}
Hierarchy of numerical constants: Let $\sigma>0$ be the free small parameter to choose. 
\begin{align*}
& b_0=0.5+\sigma^{200},\; b=0.5+2\sigma^{200},\; b_1=0.5+3\sigma^{200},\;\theta=0.01\sigma^{200},\; q^{-1}=\sigma^{50},\; \kappa=\sigma^{-500};\\
& \epsilon_1=\sigma^2,\;\epsilon_2=\epsilon_1+100\sigma^{5},\; \delta=\sigma^{20},\;\delta_0=\sigma^{10},\; 
s=\frac{1}{2}-\frac{\alpha}{4}+\sigma;\\ &\nu=\min\big\{\frac{1}{2}-\frac{\alpha}{4},\;\frac{7(\alpha-1)}{4} \big\}-\sigma.
\end{align*} 
With these choices, for $\sigma>0$ small enough, we have
\begin{align*}
&\nu\leq \min\big\{s,\;\frac{7(\alpha-1)}{4} \big\}-100(\epsilon_1+\epsilon_2),\; \frac{b_1-\frac{1}{2}}{q'-1}\ll q'-2b_1\ll 1;\\
& \epsilon_1>100(b_1-\frac{1}{2}+\frac{1}{q}+\theta),\quad (\alpha-1)+2s\nu-s>0.
\end{align*}
In the remaining part of this article, all these numerical constants are reserved to depend only on the free small parameter $\sigma$, which will be chosen small enough if necessary.
\end{remarque}
\begin{remarque}
In the induction argument, the condition (i) for the next step is inherent from (v), the condition (ii) is inherent from (vii), and the condition (iii) is inherent from (vi).		The condition (iv) means that the support of $h_{k,k^*}^{N,L}$ is essentially restricted on $|k-k^*|\lesssim L$. 
\end{remarque}

The key inductive proposition is the following:
\begin{proposition}\label{keyinduction}
Assume that $\alpha\in(\alpha_0,\frac{6}{5}]$. There exists $\sigma>0$, sufficiently small and we fix the numerical constants as in Remark \ref{numerical}. Suppose that \text{Loc}($M$) is true for all $\omega\in \Omega^*$. Then there exists a measurable set $\Omega^*_1\subset \Omega^*$ with the property that (for some $\theta>0$) $\mathbb{P}[\Omega^*\setminus \Omega_1^*]<C_{\theta}\mathrm{e}^{-T^{-\theta}M^{\theta}}$, such that for all $\omega\in\Omega_1^*$, the statement \text{Loc}($2M$) is true. Consequently, outside a exceptional set of probability $<C_{\theta}\mathrm{e}^{-T^{-\theta}}$, the statement Loc($M$) holds true for every dyadic number $M\geq \frac{1}{2}$.
\end{proposition}
\begin{remarque}
The main reason for the constraint $\alpha>\alpha_0$ is the condition
\begin{align}\label{constraint-main}
(\alpha-1)+2s\nu-s>0.
\end{align}
By numerical computation, one verifies easily that, for sufficiently small choice of the free parameter $\sigma>0$, the above condition holds if $\alpha>\alpha_0$.
\end{remarque}
Using Proposition \ref{keyinduction}, we can easily deduce Theorem \ref{prob:LWP}, (i) and (ii). Indeed, we first delete a set of probability smaller than $C_{\theta}\mathrm{e}^{-T^{-\theta}}$ such that the statements Loc($M$) are true for all dyadic $M\geq \frac{1}{2}$ on the interval $[-T,T]$\footnote{Note that $T$ is involved in the time cutoff functions to define $\zeta_L^{N,\#}$, $h^{N,L,\#}, w_N^{\#}$. }. In particular, for each dyadic number $M$, we have
$$ v_M^{\#}=\sum_{N\leq M} (\chi(t)S_{\alpha}(t)\mathbf{P}_N\phi^{\omega}+w_N^{\#})+\sum_{N\leq M}\sum_{\frac{1}{2}<L\leq L_N}\zeta_L^{N,\#},
$$
satisfying the estimates (i) to (viii) listed in the hypothesis Loc($M$). Therefore, $w^{\#}=\sum_{M}w_M^{\#}$ is a convergent sequence in $X_T^{\frac{1}{2}-\frac{\alpha}{4}+\epsilon,\frac{1}{2}+\epsilon}$, and $\zeta^{\#}=\sum_{N,L}\zeta_L^{N,\#}$ is a convergent sequence in $L^{\infty}([-T,T];H^{\alpha-1-\epsilon}(\T) )\cap L^{\infty}([-T,T];\mathcal{F}L^{\frac{\alpha}{2}-\epsilon,\infty}(\T) )$. Now for fixed $N$, the smooth solution $v_M$ and $v_M^{\#}$ are both solutions of \eqref{FNLS-gauged} with \emph{the same} initial data $\Pi_N\phi^{\omega}$. By uniquenss of the smooth solution, when restricting to a smaller time interval, say $[-T/2,T/2]$, we have $v_M^{\#}=v_M$. This allows us to decompose $v_N$ similarly as sums of $S(t)\Pi_N\phi^{\omega}, \zeta_L^N$ and $w_N$. Moreover,  the same equations \eqref{induction:tosolve1} and \eqref{induction:tosolve2} hold for if we drop the $\#$ notation and the time truncation $\chi_T(t)$. This shows that $\zeta_L^N, w_N$ coincide with $\zeta_L^{N,\#}, w_N^{\#}$ on $[-T/2,T/2]$. This proves Theorem \ref{prob:LWP}, (i) and (ii).

\subsection{The key multilinear estimate}
We make the following assumptions on $\wt{v}_j(\lambda_j,k_j)$:
\\
\noi
$\bullet${\bf Type (G)} Gaussian: $$\wt{v}_j(\lambda_j,k_j)=\wt{\psi}^{N_j}_{\frac{1}{2}}(\lambda_j,k_j):=\mathbf{1}_{\frac{N_j}{2}<|k_j|\leq N_j}\frac{g_{k_j}(\omega)}{[k_j]^{\frac{\alpha}{2}}}\widehat{\chi}(\lambda_j)
 $$
 with the bounds
 \begin{align}\label{Gaussianbounds} 
 \|v_j\|_{X^{0,b}}\leq N_j^{-\frac{\alpha-1}{2}+\epsilon_1},\; \|v_j\|_{X_{\infty,q}^{0,\frac{2b}{q'}}}\leq N_j^{-\frac{\alpha}{2}+\epsilon_1}
 \end{align}
\noi
$\bullet${\bf Type (C)} Colored:
$$ \wt{v}_j(\lambda_j,k_j)=\wt{\zeta}_{L_j}^{N_j}(\lambda_j,k_j)=\sum_{\frac{N_j}{2}<|k_j^*|\leq N_j}\wt{h}_{k_jk_j^*}^{N_j,L_j}(\lambda_j)\frac{g_{k_j^*}(\omega)}{[k_j^*]^{\frac{\alpha}{2}}},
$$
where $1\leq L_j\leq L_{N_j}<N_j^{1-\delta}$, 
$$ \mathrm{supp}(\wt{h}_{k_jk_j^*}^{N_j,L_j})\subset\{ |k_j|\leq N_j, \frac{N_j}{2}<|k_j^*|\leq N_j \} 
$$
and $\wt{h}_{k_jk_j^*}^{N_j,L_j}$ is $\mathcal{B}_{\leq L_j}$ measurable. Moreover, we assume that
\begin{align}
Y^b \text{ norm: } &\big\|\lg\lambda_j\rg^b\wt{h}_{k_jk_j^*}^{N_j,L_j}(\lambda_j) \big\|_{l_{k_j^*}^2\rightarrow l_{k_j}^2L_{\lambda_j}^2}\leq L_j^{-\delta_0^2},\label{eq(C):hypothesis1}\\
S^{b,q} \text{ norm: } & \big\|\lg\lambda_j\rg^{\frac{2b}{q'}}\wt{h}_{k_jk_j^*}^{N_j,L_j}(\lambda_j) \big\|_{l_{k_j}^{\infty}L_{\lambda_j}^ql_{k_j^*}^2}\leq N_j^{\epsilon_1}L_j^{-\nu}\label{eq(C):hypothesis2}\\
Z^b \text{ norm: } &\big\|\lg\lambda_j\rg^b\wt{h}_{k_jk_j^*}^{N_j,L_j}(\lambda_j) \big\|_{L_{\lambda_j}^2l_{k_j}^{2}l_{k_j^*}^2}\leq N_j^{1-\frac{\alpha}{2}+\epsilon_1}L_j^{-\nu}\label{eq(C):hypothesis3}  \\
L^{\infty} \text{ norm: } &\|v_j\|_{L_t^4L_x^{\infty}}\leq N_j^{-(\alpha-1)+\epsilon_2}L_j^{\frac{1}{2}-\nu} \label{eq(C):pointwise}\\
X^{0,b_0} \text{ norm: }
&\|v_j\|_{X^{0,b_0}}\leq N_j^{-(\alpha-1)+\epsilon_2}L_j^{-\nu}\\
X_{\infty,q}^{0,\frac{2b_0}{q'}}
\text{ norm: } 
&\|v_j\|_{X_{\infty,q}^{0,\frac{2b_0}{q'} }}\leq N_j^{-\alpha+\epsilon_2}L_j^{-\nu}
\end{align}
and the almost localization condition:
\begin{align}\label{almostlocal}
\Big\|\langle\lambda_j\rangle^b\Big\lg\frac{|k_j-k_j^*|}{L_j} \Big\rg^{\kappa}\wt{h}_{k_jk_j^*}^{N_j,L_j}(\lambda_j) \Big\|_{L_{\lambda_j}^2l_{k_jk_j^*}^2}\leq N_j.
\end{align}
\noi
$\bullet${\bf Type (D)} Deterministic: 
\begin{align}\label{TypeD}
 \big\|\langle\lambda_j\rangle^b\wt{v}_j(\lambda_j,k_j) \big\|_{l_{k_j}^2L_{\lambda_j}^2}\leq N_j^{-s},\; \|\langle\lambda_j\rangle^b\widetilde{v}_j(\lambda_j,k_j)\mathbf{1}_{|k_j|>N_0} \|_{l_{k_j}^2L_{\lambda_j}^2 }\leq N_0^{-s},\;\forall N_0\geq N_j.
\end{align}
For functions $v_j$ of the form (G),(C) or (D), they are all associated with dyadic numbers $N_j$ (or $(N_j,L_j)$ for the type (C)). In order to organize the terms in a unified way, we will call that $(N_j,L_j)$ a \emph{characterized pair }for a function $v_j$ of the form (G), (C) or (D), where we use the convention $L_j=\frac{1}{2}$ if $v_j$ is of type (G) while $L_j=2L_{N_j}$ if $v_j$ is of type (D). We will also call $N_j$ the \emph{characterized frequency} of $v_j$, in the sense that the Fourier support of $v_j$ is essentially localized at scale $N_j$. Recall that $N_{(1)}\geq N_{(2)}\geq\cdots$ is the non-increasing rearrangement of $N_1,N_2,\cdots$, we will denote by $v_{(j)}$, the function in the set $\{v_1,v_2,\cdots\}$ with the characterized frequency $N_{(j)}$. Yet, the order of $L_{(1)}, L_{(2)},\cdots$ is not specified.  

 The proof of Proposition \ref{keyinduction} consists of solving the fix-point problem 1 and the fix-point problem 2. The following two propositions are crucial for solving the first fix-point problem. Recall that $\Theta_{k,k_1}^{N,L}(t,t')$ is the kernel of the operator $\mathcal{P}_{N,L}^+$ defined in \eqref{PNL+}.
\begin{proposition}\label{KernelSb}
Assume that $\Omega_0\subset\Omega$ is such that \eqref{Gaussianbounds}-\eqref{TypeD} holds for all $N_j\leq L$. Then  outside a set of probability $<\mathrm{e}^{-cL^{\theta}R }$ and any $N>L^{\frac{1}{1-\delta}}$,  we have
	$$ \|\Theta_{k,k_1}^{N,L}(t,t_1)\mathbf{1}_{|k|,|k_1|\geq\frac{N}{4}}\|_{S^{b_1,b,q}}\leq RN^{100(\frac{1}{q}+\kappa^{-0.1}+b_1-\frac{1}{2})}L^{-\nu},
	$$
provided that $\nu\leq \min\{s,\frac{7(\alpha-1)}{4} \}-10(\epsilon_1+\epsilon_2)$.
\end{proposition}

\begin{proposition}\label{kernelZb}
Assume that $\Omega_0\subset\Omega$ is such that \eqref{Gaussianbounds}-\eqref{TypeD} holds for all $N_j\leq L$. Then  outside a set of probability $<\mathrm{e}^{-cL^{\theta}R }$ and any $N>L^{\frac{1}{1-\delta}}$,  we have
	$$ \|\Theta_{k,k_1}^{N,L}(t,t_1)\mathbf{1}_{|k|,|k_1|\geq \frac{N}{4}}\|_{Z^{b_1,b}}\leq R N^{1-\frac{\alpha}{2}}\cdot N^{100(\frac{1}{q}+\kappa^{-0.1}+b_1-\frac{1}{2})}L^{-\nu},
	$$
provided that $\nu\leq \min\{s,\frac{7(\alpha-1)}{4} \}-10(\epsilon_1+\epsilon_2)$.
\end{proposition} 
The following proposition is crucial to solve the fix-point problem 2:
\begin{proposition}\label{Multilinearkey}
Assume that $\alpha>\alpha_0.$  There exists $\sigma>0$, sufficiently small in the definition of the numerical constants in Remark \ref{numerical}, such that the following holds true:
Suppose that $v_1,v_2,v_3$ are of type (G), (C) or (D) with characterized parameters $(N_j,L_j)$, $j=1,2,3$, with respectively. Then outside an exceptional set of probability $<\mathrm{e}^{-N_{(1)}^{\theta}R^{\frac{2}{3}}}$, independent of the choice of functions $v_j$ of type (D), we have 
\begin{itemize}
	\item[(1)] If $v_{(1)}$ is of type (G) or (C) and $N_{(2)}\gtrsim N_{(1)}^{1-\delta}$, we have
	$$ \|\mathcal{I}\mathcal{N}_3(v_1,v_2.v_3)\|_{X^{0,b_1}}\lesssim RN_{(1)}^{-s}N_{(2)}^{-\delta_0}.
	$$
	\item[(2)] If $N_2\gg N_1, N_3$ and $v_2$ is of type (G) or (C), we have
	 	$$ \|\mathcal{I}\mathcal{N}_3(v_1,v_2.v_3)\|_{X^{0,b_1}}\lesssim RN_{(1)}^{-s}N_{(2)}^{-\delta_0}.
	 $$
	 \item[(3)] For any $N_0\gg N_{(1)}$, 
	 	$$ \|\Pi_{N_0}^{\perp}\mathcal{I}\mathcal{N}_3(v_1,v_2.v_3)\|_{X^{0,b_1}}\lesssim RN_{0}^{-s}N_{(2)}^{-\delta_0}.
	 $$
	\item[(4)] If $N_1\gg N_2,N_3$ and $v_1$ is of type (G) or (C), then 
	$$ \|\Pi_{N_1}^{\perp}\mathcal{I}\mathcal{N}_3(v_1,v_2,v_3)\|_{X^{0,b_1}}\lesssim RN_1^{-s}N_{(2)}^{-\delta_0}.
	$$
	\item[(5)] The operator	
	$$ \mathcal{P}^{+}: v\mapsto \mathcal{I}\mathcal{N}_3(v,v_2,v_3)
	$$
	satisfies
	$$ \|\mathcal{P}^+\|_{X^{0,\frac{3}{8}}\rightarrow X^{0,b_1}}\lesssim R(N_2\vee N_3)^{-\delta_0},
	$$
	and similarly,
	the operator	
	$$ \mathcal{P}^{-}: v\mapsto \mathcal{I}\mathcal{N}_3(v_1,v,v_3)
	$$
	satisfies
	$$ \|\mathcal{P}^-\|_{X^{0,\frac{3}{8}}\rightarrow X^{0,b_1}}\lesssim R(N_1\vee N_3)^{-\delta_0},
	$$
		and the implicit constants are independent of $v_2,v_3$.
		
 \item[(6)] When $L<N^{1-\delta}$, we have
   $$ \|\mathcal{P}_{N,L}^+\circ (\widetilde{\mathbf{P}}_{N}\mathcal{H})\|_{S^{b_1,q}}\lesssim RL^{-\delta_0}\|\mathcal{H}\|_{S^{b,q}},
   $$
   where $\widetilde{\mathbf{P}}_{N}=\mathcal{F}_x^{-1}\mathbf{1}_{|k_1|\sim N}\mathcal{F}_x$ is a Fourier projector similar as $\mathbf{P}_N$,
   and $\mathcal{P}_{N,L}^+$ is given by \eqref{PNL+}.
	\item[(7)] For all $L<N^{1-\delta}$,
	$$ \|\mathcal{P}_{N,L}^{\pm}\|_{X^{0,\frac{3}{8}}\rightarrow X^{0,b_1}}\lesssim RL^{-\delta_0}.
	$$
	\item[(8)] For the resonant terms, we have
	$$ \|\mathcal{I}\mathcal{N}_0(v_1,v_2,v_3)\|_{X^{0,b_1}}\lesssim RN_{(1)}^{-s}N_{(2)}^{-\delta_0},\quad \|\Pi_{N_0}^{\perp}\mathcal{I}\mathcal{N}_0(v_1,v_2,v_3)\|_{X^{0,b_1}}\lesssim RN_0^{-s}N_{(2)}^{-\delta_0}
	$$
	for all $N_0\gg N_{(1)}$.
\end{itemize}
\end{proposition}

\subsection{Proof of the main theorem}
In this section, we assume Proposition \ref{KernelSb}, Proposition \ref{kernelZb} and Proposition \ref{Multilinearkey} and proceed to prove Proposition \ref{keyinduction}.

\begin{proof}[Proof of Proposition \ref{keyinduction}]

We assume Loc($M$) is true for some large dyadic number $M$, and we will show that Loc($2M$) holds. To make the argument clean, first we delete a set of probability $<\mathrm{e}^{-{2M}^{\theta}R^{\frac{2}{3}}}$ and we do not explicitly make any claim when there is necessary to delete some exceptional set of of the same size of the probability.  

\noi
$\bullet$ {\bf Step 1:} For $L=M, M<N^{1-\delta}$, we first show that (v),(vi), (vii) holds. From the decomposition
$$ v_M^{\#}=\psi^{M,\frac{1}{2},\#}+\sum_{\frac{1}{2}<L<M^{1-\delta}}\zeta_L^{M,\#}+w_M^{\#},
$$ 
for any $w$, $\mathcal{P}_{N,M}^{+}(w)$ is a sum of $\chi_T(t)\mathcal{I}\Pi_N\mathcal{N}_3(w,v_2,v_3)$ for $v_2,v_3$ run over all terms of type (G), (C) or (D) with characterized parameters $(N_2,L_2), (N_3,L_3)$ satisfying
$$ \frac{M}{2}\leq N_2\vee N_3\leq M,\quad L_j<N_j^{1-\delta}<M,\; j=1,2.
$$ 
Hence by (5) of Proposition \ref{Multilinearkey}, we have
\begin{align*}
\|\mathcal{P}_{N,M}^+\|_{X^{0,b}\rightarrow X^{0,b}}\lesssim T^{b_1-b}R\sum_{\substack{N_{(2)}=\frac{M}{2},M\\
N_{(3)}\leq M } }M^{-\delta_0}\lesssim T^{\frac{b_1-b}{2}}M^{-\delta_0^2},
\end{align*}
provided that $T\ll R^{-\frac{2}{b_1-b}}$ is chosen small enough. The kernel estimates (vi), (vii) are direct consequences of Proposition \ref{kernelZb} and Proposition \ref{KernelSb}, with respectively.

\noi
$\bullet${\bf Step 2:} Next we prove (i),(ii),(iii) and (iv) by using \eqref{induction:tosolve1.5}. Note that (ii),(iii) is a direct consequence of (i) and (iv), see Section \ref{Sec:pointwisebound} for details. From \eqref{induction:tosolve1.5}, we have 
\begin{align*}
h^{N,M,\#}=&2i\sum_{L\leq M} \mathcal{P}_{N,L}^+\circ h^{N,M,\#}+2i\sum_{L<M}\mathcal{P}_{N,M}^+\circ h^{N,L,\#}\\
+&2i\mathcal{P}_{N,M}^+\circ H^{N,\frac{1}{2},\#}.
\end{align*}
Therefore,
\begin{align*}
\|h^{N,M,\#}\|_{Y^b}\lesssim &\sum_{L\leq M}\|\mathcal{P}_{N,L}^+\|_{X^{0,b}\rightarrow X^{0,b}}\|h^{N,M,\#}\|_{Y^b}+\sum_{L<M}\|\mathcal{P}_{N,M}^+\|_{X^{0,b}\rightarrow X^{0,b}}\|h^{N,L,\#}\|_{Y^b}\\
+&\|\mathcal{P}_{N,M}^+\|_{X^{0,b}\rightarrow X^{0,b}}\|H^{N,\frac{1}{2},\#}\|_{Y^b}\\
\lesssim &\|h^{N,M,\#}\|_{Y^b}\sum_{L\leq M}T^{\theta}L^{-\delta_0^2}+\sum_{L<M}T^{\theta}M^{-\delta_0^2}L^{-\delta_0^2}+T^{\theta}M^{-\delta_0^2}.
\end{align*}
This implies that 
$$ \|h^{N,M,\#}\|_{Y^b}\leq M^{-\delta_0^2},
$$
provided that $T>0$ is small enough. This proves the first inequality of (i). Next we prove (iv).  From \eqref{induction:tosolve1.5}, we have
\begin{align}\label{formulafix-point1}
\widetilde{h}_{kk^*}^{N,M,\#}(\lambda)=&2i\sum_{L\leq M}\int\sum_{k'} \Theta_{kk'}^{N,L}(\lambda,\lambda')\widetilde{h}_{k'k^*}^{N,M,\#}(\lambda')d\lambda'+2i\sum_{L<M}\int \sum_{k'} \Theta_{kk'}^{N,M,}(\lambda,\lambda')\widetilde{h}_{k'k^*}^{N,L,\#}(\lambda')d\lambda'\notag \\
+&2i\int \sum_{k'} \Theta_{kk'}^{N,M}(\lambda,\lambda')\widetilde{H}_{k'k^*}^{N,\frac{1}{2},\#}(\lambda')d\lambda'.
\end{align}
Note that $\Theta_{kk'}^{N,L}(t,t')$ is supported at $|k-k'|\lesssim L$, from Lemma \ref{lemma:operator1}, we have
\begin{align*}
&\Big\|\Big\langle\frac{|k-k^*|}{M}\Big\rangle^{\kappa}\langle\lambda\rangle^b\widetilde{h}_{kk^*}^{N,M,\#}(\lambda)\Big\|_{L_{\lambda}^2l_{k,k^*}^2}\\ \lesssim &T^{b_1-b}\sum_{L\leq M}\|\langle\lambda\rangle^{b_1}\langle\lambda'\rangle^{-b}\Theta_{kk'}^{N,L}(\lambda,\lambda')\|_{L_{\lambda'}^2l_{k'}^2\rightarrow L_{\lambda}^2l_k^2}\Big\|
\Big\langle\frac{|k'-k^*|}{M} \Big\rangle^{\kappa}\lg\lambda'\rg^b
 \widetilde{h}_{k'k^*}^{N,M,\#}(\lambda')\Big\|_{L_{\lambda'}^2l_{k',k^*}^2}\\
 +&T^{b_1-b}\sum_{L<M}
 \|\langle\lambda\rangle^{b_1}\langle\lambda'\rangle^{-b}\Theta_{kk'}^{N,M}(\lambda,\lambda')\|_{L_{\lambda'}^2l_{k'}^2\rightarrow L_{\lambda}^2l_k^2}\Big\|
 \Big\langle\frac{|k'-k^*|}{L} \Big\rangle^{\kappa}\lg\lambda'\rg^b
 \widetilde{h}_{k'k^*}^{N,L,\#}(\lambda')\Big\|_{L_{\lambda'}^2l_{k',k^*}^2}\\
 +&T^{b_1-b}\|\langle\lambda\rangle^{b_1}\langle\lambda'\rangle^{-b}\Theta_{kk'}^{N,M}(\lambda,\lambda')\|_{L_{\lambda'}^2l_{k'}^2\rightarrow L_{\lambda}^2l_k^2}\|
\langle k'-k^*\rangle^{\kappa}\lg\lambda'\rg^b
 \widetilde{H}_{k'k^*}^{N,\frac{1}{2},\#}(\lambda')\|_{L_{\lambda'}^2l_{k',k^*}^2}.
\end{align*}
By using the induction hypothesis and the boundeness of $\mathcal{P}_{N,L'}^+$ (the property (v) that we have just proved) for all $L'\leq M$, we obtain (iv), provided that $T$ is chosen small enough.

Next we prove the third inequality of (i).
From Lemma \ref{lemma:operator2}, we have
\begin{align*}
\|h^{N,M,\#}\|_{Z^b}\lesssim &T^{b_1-b}\sum_{L\leq M} \|\mathcal{P}_{N,L}^{+} \|_{X^{0,b}\rightarrow X^{0,b_1}}\|h^{N,M,\#}\|_{Z^b}\\
+&T^{b_1-b}\sum_{L<M}\min\{\|\Theta_{kk'}^{N,M}(t,t')\|_{Z^{b_1,b}}\|h^{N,L,\#}\|_{Y^b}, \|\mathcal{P}_{N,M}^{+}\|_{X^{0,b}\rightarrow X^{0,b_1}  }\|h^{N,L,\#}\|_{Z^b} \}.
\end{align*}
Note that for the term $$\sum_{L<M}\int\sum_{k'}\Theta_{kk'}^{N,M}(\lambda,\lambda')\widetilde{h}_{k'k^*}^{N,L,\#}(\lambda')d\lambda' $$
in \eqref{formulafix-point1}, we may assume that $|k-k^*|\ll N$ and $|k'-k^*|\ll N$, since otherwise the bound follows trivially from (vi) that we have just proved. In particular, we have $|k|,|k'|\geq \frac{N}{4}$ as $|k^*|\geq \frac{N}{2}$. Using (vi) that we have just proved, we obtain the third inequality of (i).

Finally in this step, we prove the second inequality of (i). Again in \eqref{formulafix-point1}, we may assume that $|k|,|k'|\geq \frac{N}{4}$ since otherwise we can use (vii) to obtain a better bound. After this reduction, we could apply (6) of Proposition \ref{Multilinearkey} to treat the term
$$ \int \sum_{L\leq M}\sum_{k}\Theta^{N,L}_{kk'}(\lambda,\lambda')\wt{h}_{k'k^*}^{N,M,\#}(\lambda')d\lambda'=2i\sum_{L\leq M}\mathcal{P}_{N,L}^{+}\circ(\wt{\mathbf{P}}_Nh^{N,M,\#})_{kk^*}(\lambda).
$$
Combining with Lemma \ref{Timelocalization}, Lemma \ref{lemma:operator1} and (6) or Proposition \ref{Multilinearkey}, we have 
\begin{align*}
\|h^{N,M,\#}\|_{S^{b,q}}\lesssim & RT^{\frac{2(b_1-b)}{q'}}\sum_{L\leq M} L^{-\delta_0}\|\wt{h}^{N,M,\#}\|_{S^{b,q}}+\sum_{L<M}\|\Theta^{N,M}\|_{S^{b,b,q}}\|h^{N,L,\#}\|_{Y^b}\\
+&\|\Theta^{N,M}\|_{S^{b,b,q}}\|H^{N,\frac{1}{2},\#}\|_{Y^b},
\end{align*}
and this is conclusive when $T\ll 1$ is chosen small enough.
\vspace{0.3cm}

\noi
$\bullet${\bf Step 3:} We prove (viii) by solving the equation \eqref{induction:tosolve2}. We will construct the fix-point of the equation \eqref{induction:tosolve2} in the set $\mathcal{Z}_{2M}$, where $$\mathcal{Z}_{N}:=\{w:\;\|w\|_{X^{0,b}}\leq N^{-s}, \|\Pi_{N_0}^{\perp}w\|_{X^{0,b}}\leq N_0^{-s},\;\forall N_0\geq N \}$$
for dyadic numbers $N$. By hypothesis, we already know that $w_N^{\#}\in\mathcal{Z}_N$ for all $N\leq M$.
For $N$, we define the norm
$$ \|w\|_{\mathcal{W}_N}:=\max\big\{ N^s\|w\|_{X^{0,b}},\sup_{N_0\geq N}N_0^s\|\Pi_{N_0}^{\perp}w\|_{X^{0,b}}\big\},
$$
where the sup is taken over all dyadic integers greater than $N$. Then finding a fix-point in the set $\mathcal{Z}_{2M}$ is equivalent to find a fix-point in the unit ball of $\mathcal{W}_{2M}$. Since it is not difficult to verify that (with $M$ fixed) $\|\cdot\|_{\mathcal{W}_{2M}}$ is a norm on some Banach space embedded in $X^{0,b}$ (see Lemma \ref{Banachnorm}), we can still apply the Banach fix-point theorem (contraction principle).

First we verify that the mapping induced by the right side of \eqref{induction:tosolve2} sends a unit ball of $\mathcal{W}_{2M}$ to itself, provided that $T$ is sufficiently small (recall that $T$ is involved to define $w^{\#}=w(t)\chi_T(t)$). Thanks to Lemma \ref{localizationXsb}, it suffices to estimate the $X^{0,b_1}$ norm for the multi-linear terms \emph{without} the time cutoff $\chi_T(t)$ factor in front of the Duhamel operator $\mathcal{I}$. From \eqref{induction:tosolve2} (changing $w_{2M}^{\#}$ there to $w^{\#}$), the right side of the integration equation of $w^{\#}(t)$ is a linear combination of multi-linear terms of types 1)-7) below \eqref{induction:tosolve2}. Since for $\delta\ll 1$, $M<(2M)^{1-\delta}$, all the conditions \eqref{eq(C):hypothesis1}, \eqref{eq(C):hypothesis2}, \eqref{eq(C):hypothesis3} and \eqref{almostlocal} are satisfies for type (C) terms with characterized parameters $(N_j,L_j)$ satisfying $N_j\leq 2M$ and $L_j<N_j^{1-\delta}$. Moreover,
$$ \|w_{N_j}^{\#}\|_{X^{0,b}}\leq N_j^{-s},  \;\|\Pi_{N_0}^{\perp}w_{N_j}\|_{X^{0,b}}\leq N_0^{-s}
$$
for all $N_j\leq M$ and $N_0\geq N_j$. Then the rest argument is a direct application of the statements in Proposition \ref{Multilinearkey}. Next, to verify that the mapping defined by the right side of \eqref{induction:tosolve2} is a contraction, the argument is similar. Indeed, we pick two different $w,w'\in\mathcal{Z}_{2M}$, due to the tri-linearity of the right side of \eqref{induction:tosolve2}, there must be $w-w'$ appearing in at least one place in each multi-linear expression $\mathcal{N}(\cdot,\cdot,\cdot)$. Then applying (5) of Proposition \ref{Multilinearkey} and Lemma \ref{localizationXsb}, we are able to leave out a factor $T^{b_1-b}\|w-w'\|_{\mathcal{W}_{2M}}$ when estimating the $\mathcal{W}_{2M}$ norm of the difference. From the Banach fix-point theorem, we are able to find the unique fix-point $w^{\#}(t)=w(t)\chi_{2T}(t)$ in $\mathcal{Z}_{2M},$ supported in $|t|\leq 2T$.

The proof of Proposition \ref{keyinduction} is now complete. 

\end{proof}

\begin{lemme}\label{Banachnorm}
	Assume that $\mathcal{W}$ is a Banach space with the norm $\|\cdot\|$and $(T_j)_{j\in\N}$ is a sequence of bounded linear operators on $\mathcal{X}$ and $T_1=\mathrm{Id}$. Consider another space
	$$ \mathcal{W}_*:=\{w\in \mathcal{W}:\; \|w\|_*<+\infty \},
	$$
	where
	$$ \|w\|_*:=\sup_{j\in\N}\|T_jw\|.
	$$
	Then with $(\mathcal{W}_*,\|\cdot\|_*)$ is a Banach space.
\end{lemme}
\begin{proof}
	The triangle inequality is trivial. We only need to show that $\mathcal{W}_*$ is complete. Take a Cauchy-sequence $(w^{(k)})\subset \mathcal{W}_*$ such that
	$$ \lim_{k,k'\rightarrow\infty}\|w^{(k)}-w^{(k')}\|_*=0.
	$$
	In particular, since $T_1=\mathrm{Id}$ and $\mathcal{W}$ is complete, there exists $w\in\mathcal{W}$ such that $\|w^{(k)}-w\|\rightarrow 0$. Since for any $\epsilon>0$, there exists $k_0=k_0(\epsilon)$, such that for all $k,k'\geq k_0$,
	$$ \sup_{j\in\N}\|T_jw^{(k)}-T_jw^{(k')}\|<\epsilon.
	$$
	Thus for each fixed $j$, passing $k'\rightarrow+\infty$, we have
	$\|T_jw^{(k)}-T_jw\|\leq \epsilon.$ This implies that
	$$ \lim_{k\rightarrow\infty}\sup_{j\in\N}\|T_jw^{(k)}-T_jw\|=0.
	$$
	The proof of Lemma \ref{Banachnorm} is complete.
\end{proof}

\subsection{Sketch of the convergence of the whole sequence} 
We now explain briefly how to modify the arguments in this section to prove the convergence for the whole sequence $(v_n)_{n\in\N}$, satisfying
$$  i\partial_tv_n+|D_x|^{\alpha}v_n=\mathcal{N}(v_n),\quad v_n|_{t=0}=(\Pi_n-\Pi_{\frac{N}{2}})\phi^{\omega},
$$
where $\frac{N}{2}<n<N$. For this, we first define similar random averaging operator (as well as their kernels):
$  H^{n,L},\; h^{n,L},\; \mathcal{P}_{n,L}^{\pm} ,\; \Theta_{kk'}^{n,L}  
$
and the corresponding ``para-controlled'' objects $\psi_L^{n}, \zeta_L^n=\psi_L^n-\psi_{\frac{L}{2}}^n$
by changing $N$ to $n$ while keeping the constraint $L<N^{1-\delta}$. Then we add the same bounds for these objects as $H^{N,L}, h^{N,L}, \mathcal{P}_{N,L}^{\pm}, \Theta_{kk'}^{N,L}, \psi_L^N, \zeta_L^{N}$ in the definition of Loc($M$) for all $(L,N)$ such that $L<M, L<N^{1-\delta}$ and all $\frac{N}{2}<n<N$. We need also to add $X^{0,b}$ bounds of $w_{n'}$ and $\Pi_{N_0}^{\perp}w_{n'}$ for all $\frac{N'}{2}<n'< N', N_0\geq N'$ and $N'\leq M$ in the definition of Loc($M$). Then to pass from Loc($M$) to Loc($2M$), we make use of Proposition \ref{KernelSb}, Proposition \ref{kernelZb} and Proposition \ref{Multilinearkey}. Note that here we should prove stronger statements in these propositions accordingly, providing estimates of $S^{b_1,b,q}$ and $Z^{b_1,b}$ norms of the kernel $\Theta_{k,k_1}^{n,L}$. Here the observation is that, the proof of Proposition \ref{KernelSb} and Proposition \ref{kernelZb} (in Section \ref{section:kernel}) is not specific to dyadic numbers $N$ (in the definition of $\Theta^{n,L}$, the letter $n$ appears only in the frequency truncation $\Pi_n$ in front of the multi-linear expression $\mathcal{N}_3$) and the probability of the exceptional set that we delete each time can depend only on the \emph{dyadic parameters} $L,N$ such that $L<N^{1-\delta}$. Therefore, the results of Proposition \ref{KernelSb} and Proposition \ref{kernelZb} are also true for all $\Theta_{k,k_1}^{n,L}$ such that $\frac{N}{2}<n<N$ and $L<N^{1-\delta}$.
Finally, to get desired bounds for $w_m$ if $M<m<2M$, the analysis is similar as solving the Fix-point 2 for $w_{2M}$. Therefore, if $\frac{N}{2}<n<N$, through the decomposition
$$ v_n(t)=v_{\frac{N}{2}}(t)+S(t)(\Pi_n-\Pi_{\frac{N}{2}})\phi^{\omega}+\sum_{\frac{1}{2}<L\leq L_{N/2}}\zeta_L^n+w_n(t),
$$
we deduce that $(v_n(t))_{n\in\N}$ is also a Cauchy sequence in $C([-T,T];H^{\sigma_0}(\T))$.

Once Theorem \ref{prob:LWP} is proved, we are able to deduce Theorem \ref{thm5} as in \cite{Sun-Tz2}, and we omit the detail.
\section{$L^{\infty}$ and Fourier-Lebesgue property for  paracontrolled objects}\label{Sec:pointwisebound}
In this section, we prove (ii) (iii) of the statement Loc(2$M$). Note that $\epsilon_2>\theta+2\epsilon_1$, the key for the proof of (iii) is the following probabilistic pointwise bound:
\begin{lemme}\label{pointwisebound}
Assume that $0<T<1, \frac{1}{2}<L_1<N_1^{1-\delta}$ and $$  \|h^{N_1L_1}\|_{Z^b}\leq N_1^{1-\frac{\alpha}{2}+\epsilon_1}L_1^{-\nu},\quad  \Big\|\Big\lg\frac{|k-k^*|}{L_1}\Big\rg^{\kappa}\lg\lambda\rg^b\wt{h}_{kk^*}^{N_1L_1}(\lambda) \Big\|_{L_{\lambda}^2l_{k,k^*}^2}\leq N_1,
$$	
then for any $R\gg \epsilon_1^{-\frac{1}{2}}$, outside a set of probability $<\mathrm{e}^{-cN_1^{\theta}R^2}$, we have
	$$\Big\|\sum_{|k_1|\leq N_1}\sum_{|k_1^*|\sim N_1}\chi_T(t)h_{k_1k_1^*}^{N_1L_1}(t)\frac{g_{k_1^*}}{[k_1^*]^{\frac{\alpha}{2}}}\mathrm{e}^{ik_1x} \Big\|_{L_t^{4}L_x^{\infty} }\leq C_{\epsilon}T^{\frac{1}{8}}RN_1^{-(\alpha-1)+\theta+2\epsilon_1}L_1^{\frac{1}{2}-\nu}.
	$$
\end{lemme}

\begin{proof}
By abusing notation, we still denote by $h^{NL}(t)=\chi_T(t)h^{NL}(t)$. Assume that $\epsilon_1^{-1}<q_1<\infty$, then from Chebyshev's inequality, Sobolev embedding $W_x^{\frac{2}{q_1},q_1}\hookrightarrow L_x^{\infty}$ and Minkowski's inequality, we have
	\begin{align*}
	&\mathbb{P}\Big[\Big\|\sum_{|k_1|\leq N_1, |k_1^*|\sim N_1}h_{k_1k_1^*}^{N_1L_1}(t)\frac{g_{k_1^*}}{[k_1^*]^{\frac{\alpha}{2}}}\mathrm{e}^{ik_1x}  \Big\|_{L_t^{4}L_x^{\infty} }>A \Big] \\ \leq &\frac{1}{A^{q_1}} 
	\Big\|\sum_{|k_1|\leq N_1, |k_1^*|\sim N_1}h_{k_1k_1^*}^{N_1L_1}(t)\frac{g_{k_1^*}}{[k_1^*]^{\frac{\alpha}{2}}}\mathrm{e}^{ik_1x}  \Big\|_{L_{\omega}^{q_1}L_t^{4}L_x^{\infty} }^{q_1}\\
	\leq &\frac{1}{A^{q_1}} 
	\Big\|\sum_{|k_1|\leq N_1, |k_1^*|\sim N_1}h_{k_1k_1^*}^{N_1L_1}(t)\frac{\langle k_1\rangle^{\frac{2}{q_1}} g_{k_1^*}}{[k_1^*]^{\frac{\alpha}{2}}}\mathrm{e}^{ik_1x}  \Big\|_{L_t^{4}L_x^{q_1}L_{\omega}^{q_1} }^{q_1}.
	\end{align*}
	Since $g_{k_1^*}(\omega)$ and $h_{k_1k_1^*}^{N_1L_1}(\omega)$ are independent, we may write them as $g_{k_1^*}(\omega_1)$ and $h_{k_1k_1^*}^{N_1L_1}(\omega_2)$. Then for fixed $t$ and $x$, we have from Lemma \ref{chaosWiener} that
	\begin{align*}
	\Big\|\sum_{|k_1^*|\sim N_1}\frac{g_{k_1^*}(\omega) }{[k_1^*]^{\frac{\alpha}{2}} }\sum_{|k_1|\leq N_1}\langle k_1\rangle^{\frac{2}{q_1}}h_{k_1k_1^*}^{N_1L_1}(t)\mathrm{e}^{ik_1x}   \Big\|_{L_{\omega_1}^{q_1} }\lesssim &q_1^{\frac{1}{2}}\Big\|\sum_{|k_1^*|\sim N_1}\frac{g_{k_1^*}(\omega) }{[k_1^*]^{\frac{\alpha}{2}} }\sum_{|k_1|\leq N_1}\langle k_1\rangle^{\frac{2}{q_1}}h_{k_1k_1^*}^{N_1L_1}(t)\mathrm{e}^{ik_1x}   \Big\|_{L_{\omega_1}^2 }\\
	\lesssim &q_1^{\frac{1}{2}}N_1^{-\frac{\alpha}{2}+\frac{2}{q_1}}\Big(\sum_{|k_1^*|\sim N_1}\Big|\sum_{|k_1|\leq N_1}h_{k_1k_1^*}^{N_1L_1}(t)  \Big|^2\Big)^{\frac{1}{2}}\\
	\end{align*} 
	Note that 
	\begin{align*}
	\Big(\sum_{|k_1^*|\sim N_1}\Big|\sum_{|k_1|\leq N_1}h_{k_1k_1^*}^{N_1L_1}(t)  \Big|^2\Big)^{\frac{1}{2}} \leq &(L_1N_1^{\epsilon_1})^{\frac{1}{2}}\Big(\sum_{|k_1^*|\sim N_1, |k_1-k_1^*|\leq L_1N_1^{\epsilon_1}}|h_{k_1k_1^*}^{N_1L_1}(t)|^2\Big)^{\frac{1}{2}}\\+ &N_1^{1-\epsilon_1\kappa}\Big(\sum_{\substack{ |k_1^*|\sim N_1, |k_1|\leq N_1,\\ |k_1^*-k_1|>L_1N_1^{\epsilon_1} } } \Big\langle\frac{|k_1-k_1^*|}{L}\Big\rangle^{2\kappa}|h_{k_1k_1^*}^{N_1L_1}(t) |^2 \Big)^{\frac{1}{2}}.
	\end{align*}
	Therefore,
	\begin{align*}
	&\mathbb{P}\Big[\Big\|\sum_{|k_1|\leq N_1, |k_1^*|\sim N_1}h_{k_1k_1^*}^{N_1L_1}(t)\frac{g_{k_1^*}}{|k_1^*|^{\frac{\alpha}{2}}}\mathrm{e}^{ik_1x}  \Big\|_{L_t^{4}L_x^{\infty} }>A \Big]\\ \lesssim &\frac{{q_1}^{\frac{q_1}{2}}N_1^{-\frac{q_1\alpha}{2}+2 }}{A^{q_1}} \Big[(L_1N_1^{\epsilon_1})^{\frac{q_1}{2}} \|h_{k_1k_1^*}^{N_1L_1}(t)\|_{L_t^{4}l_{k_1,k_1^*}^2}^{q_1}
	+ N_1^{\frac{q_1}{2}-q_1\epsilon_1\kappa}\Big\|\Big\langle\frac{|k_1-k_1^*|}{L} \Big\rangle^{\kappa} h_{k_1k_1^*}^{N_1L_1}(t) \Big\|_{L_t^{4}l_{k_1,k_1^*}^2 }^{q_1}
	\Big].
	\end{align*}
using again the Minkowski inequality and the Sobolev embedding $H_t^{\frac{1}{4}}\hookrightarrow L_t^{4}$, for any  $h_{k_1k_1^*}(t)$, we have
	\begin{align*}
	\|h_{k_1k_1^*}^{N_1L_1}(t)\|_{L_t^{4}l_{k_1,k_1^*}^2 }=&\|\mathrm{e}^{it|k_1|^{\alpha}}h_{k_1k_1^*}^{N_1L_1}(t)\|_{L_t^{4}l_{k_1,k_1^*}^2 } \\ \leq &\|\mathrm{e}^{it|k_1|^{\alpha}}h_{k_1k_1^*}^{N_1L_1}(t)\|_{l_{k_1,k_1^*}^2H_{t}^{\frac{1}{4}} }= \|\langle\lambda_1\rangle^{\frac{1}{4}}\widetilde{h}^{N_1L_1}_{k_1k_1^*}(\lambda_1) \|_{l_{k_1,k_1^*}^2L_{\lambda}^2}\lesssim T^{\frac{1}{8}} \|h_{k_1,k_1^*}^{N_1L_1}\|_{Z^{b}},
	\end{align*}
where to the last step, we use the fact that $h_{k_1k_1^*}^{N_1L_1}(t)=\chi_T(t)h_{k_1k_1^*}^{N_1L_1}(t)$ and the time-localization property (Lemma \ref{Timelocalization}). Similarly,
$$ \Big\|\Big\lg\frac{|k_1-k_1^* |}{L} \Big\rg^{\kappa}h_{k_1k_1^*}^{N_1L_1}(t) \Big\|_{L_t^4l_{k_1,k_1^*}^2}\lesssim
T^{\frac{1}{8}}\Big\|\Big\lg\frac{|k_1-k_1^* |}{L} \Big\rg^{\kappa}h_{k_1k_1^*}^{N_1L_1}(t) \Big\|_{Z^b} 
$$
	Therefore,
	\begin{align*}
&	\mathbb{P}\Big[\Big\|\sum_{|k_1|\leq N_1, |k_1^*|\sim N_1}h_{k_1k_1^*}^{N_1L_1}(t)\frac{g_{k_1^*}}{|k_1^*|^{\frac{\alpha}{2}}}\mathrm{e}^{ik_1x}  \Big\|_{L_t^{4}L_x^{\infty} }>A \Big]\\ \lesssim &\frac{q_1^{\frac{q_1}{2}}}{A^{q_1}}N_1^{-\frac{q_1\alpha}{2}+2}T^{\frac{q_1}{8}}\big[L_1^{\frac{q_1}{2}-q_1\nu}N_1^{q_1(1-\frac{\alpha}{2}+2\epsilon_1)}+N_1^{q_1(1-\epsilon_1\kappa)} \big]\lesssim \frac{q_1^{\frac{q_1}{2}}}{A^{q_1}}N_1^{q_1(1-\alpha)+2q_1\epsilon}T^{\frac{q_1}{8}}L_1^{q_1(\frac{1}{2}-\nu)}.
	\end{align*}
	Since $\kappa\epsilon_1\gg 1$, by choosing $R=RN_1^{\theta}T^{\frac{1}{8}}N_1^{-(\alpha-1)+2\epsilon_1}L_1^{\frac{1}{2}-\nu}$ and optimizing the choice of $q_1\sim R^2$ (thanks to the fact that $R\gg \epsilon_1^{-1/2}$), we obtain the desired estimate. This completes the proof of Lemma \ref{pointwisebound}.
\end{proof}
Similarly, to prove (ii) for Loc(2$M$), it suffices to prove:
\begin{lemme}\label{paracontrolregularityX}
Assume that $0<T<1, \frac{1}{2}<L_1<N_1^{1-\delta}$ and $$  \|h^{N_1L_1}\|_{Z^b}\leq CN_1^{1-\frac{\alpha}{2}+\epsilon_1}L_1^{-\nu},\quad  \Big\|\Big\lg\frac{|k-k^*|}{L_1}\Big\rg^{\kappa}\lg\lambda\rg^b\wt{h}_{kk^*}^{N_1L_1}(\lambda) \Big\|_{L_{\lambda}^2l_{k,k^*}^2}\leq N_1^{\frac{1}{2}}
$$	
and
$$ \|h^{N_1L_1}\|_{S^{b,q}}\leq N_1^{-\frac{\alpha}{2}+\epsilon_1}L_1^{-\nu}.
$$
Let
$$ v_1(t,x)=\sum_{|k_1|\leq N_1}\sum_{|k_1^*|\sim N_1}\chi_T(t)h_{k_1k_1^*}^{N_1L_1}(t)\frac{g_{k_1^*}}{[k_1^*]^{\frac{\alpha}{2}}}\mathrm{e}^{ik_1x}.
$$
 Then for any $R\gg \epsilon_1^{-\frac{1}{2}}$,  outside a set of probability $<\mathrm{e}^{-cN_1^{\theta}R^2}$, we have
	$$ \|v_1\|_{X_{\infty,q}^{0,\frac{2b_0}{q'}}}=\|\langle\lambda\rangle^{\frac{2b_0}{q'}}\widetilde{v}(\lambda,k)\|_{l_k^{\infty}L_{\lambda}^q}\lesssim T^{b-b} RN_1^{-\frac{\alpha}{2}+\epsilon_1+\theta+\frac{1}{q}}L_1^{-\nu} 
	$$
	and
		$$ \|v_1\|_{X^{0,b_0}}\lesssim T^{b-b_0}RN_1^{-(\alpha-1)+\epsilon_1+\theta}L_1^{-\nu}.
	$$
\end{lemme}

\begin{proof}
We only prove the bound $X_{\infty,q}^{0,\frac{2b_0}{q'}}$ and the bound for the norm $X^{0,b_0}$ can be obtained in the similar way. Since
$$ \wt{v}_1(\lambda_1,k_1)=\sum_{k_1^*:|k_1^*|\sim N_1}\wt{h}_{k_1k_1^*}^{N_1L_1}(\lambda_1)\frac{g_{k^*}}{[k_1^*]^{\frac{\alpha}{2}}}.
$$
Note that for fixed $|k_1|\leq N_1$, applying Corollary \ref{largedeviation}, outside a set $\Omega_{k_1}$ (depending on $k_1$) of probability $<\mathrm{e}^{-cN_1^{\theta}R_1^2}$, 
\begin{align}\label{eq:FLbound} \Big|\sum_{k_1^*:|k_1^*|\sim N_1}\wt{h}_{k_1k_1^*}^{N_1L_1}(\lambda_1)\frac{g_{k_1^*}}{[k_1^*]^{\frac{\alpha}{2}}}\Big|\leq CN^{\theta}R\cdot \Big(\mathbb{E}^{\mathcal{B}_{\leq L}}\Big[\Big|\sum_{k_1^*:|k_1^*|\sim N} \wt{h}_{k_1k_1^*}^{N_1L_1}(\lambda)\frac{g_{k_1^*}}{[k_1^*]^{\frac{\alpha}{2}}} \Big|^2 \Big]\Big)^{\frac{1}{2}}.
\end{align}
By deleting the union $\cup_{|k_1|\leq N_1}\Omega_{k_1}$ for which the probability is smaller than $$\sum_{|k_1|\leq N_1}\mathrm{e}^{-cN_1^{\theta}R^2}<\mathrm{e}^{-c'N_1^{\theta}R^2},$$ above bound \eqref{eq:FLbound} is uniform for $|k_1|\leq N_1$.
Using the independence of $h_{k_1k_1^*}^{N_1L_1}$ and $g_{k_1^*}$, the conditional expectation can be bounded by
$$ N_1^{-\alpha+\theta}R\|\wt{h}_{k_1k_1^*}^{N_1L_1}(\lambda)\|_{l_{k_1}^{\infty}l_{k_1^*}^2}\leq N_1^{-\alpha+\theta}R\|\wt{h}_{k_1k_1^*}^{N_1L_1}(\lambda)\|_{l_k^{q}l_{k^*}^2}.
$$ 
Noticing that $k_1$ is constraint in the set $|k_1|\leq N_1$, multiplying by $\lg\lambda\rg^{\frac{2b_0}{q'}}$ to both sides of \eqref{eq:FLbound} and taking the $L_{\lambda_1}^{q}$ and then $l_{k_1}^{\infty}$, we have
$$ \|v_1\|_{X_{\infty,q}^{0,\frac{2b_0}{q'}}}\leq N_1^{-\alpha+\theta+\frac{1}{q}}R\|\lg\lambda\rg^{\frac{2b_0}{q'}}\wt{h}_{k_1k_1^*}^{N_1L_1}(\lambda_1)\|_{l_{k_1}^{\infty}L_{\lambda_1}^ql_{k_1^*}^2 }=  N_1^{-\alpha+\theta+\frac{1}{q}}R\|h^{N_1L_1}\|_{S^{b_0,q}}.
$$
Using Lemma \ref{Timelocalization}, we have $\|h^{N_1L_1}\|_{S^{b_0,q}}\lesssim T^{b-b_0}\|h^{N_1L_1}\|_{S^{b,q}}$ since $h^{N_1L_1}|_{t=0}=0$. 
The proof of Lemma \ref{paracontrolregularityX} is now complete.
\end{proof}

\section{Mapping properties of the operator $\mathcal{P}_{N,L}^+$ }

In this section, we will prove (3),(5),(6),(7) of Proposition \ref{Multilinearkey}.

For given space-time functions $v_2, v_3$, consider the operator
$$ \mathcal{Q}_{3,N}(w):=\mathcal{I}\Pi_{N}\mathcal{N}_3\big(w, v_2,v_3\big)
$$
and we denote by $\Theta_{kk'}(t,t')$ its kernel. Note that the operator $\mathcal{Q}_{3,N}$ depends on the functions $v_2,v_3$ and $N$. By implicitly inserting $w=\chi(t)w$, we have from Lemma \ref{DuhamelKernel} that
\begin{align*}
\widetilde{\mathcal{Q}_{3,N}}(w)(\lambda,k)=&\int_{\R}K(\lambda,\mu)\widetilde{\big(\mathcal{N}_3}(w,v_2,v_3) \big)(\mu,k)d\mu.
\end{align*}
Note that
\begin{align*}
&\widetilde{\mathcal{N}_3}(w,v_2,v_3)(\lambda,k)\\
=&\frac{1}{(2\pi)^3}\!\!\!\!\!\sum_{\substack{|k|\leq N,k_1,k_2,k_3\\
		(k_1,k_2,k_3)\in\Gamma(k) } }\!\!\!\!\!
\int \widehat{\chi}(\lambda-\lambda_1+\lambda_2-\lambda_3-\Phi_{k_1,k_2,k_3})\widetilde{w}(\lambda_1,k_1)\ov{\widetilde{v_2}}(\lambda_2,k_2)\widetilde{v_3}(\lambda_3,k_3)d\lambda_1 d\lambda_2 d\lambda_3.
\end{align*}
Denote by
\begin{align}\label{interkernel1}
\Xi_{kk'}(\lambda,\lambda'):=\frac{\mathbf{1}_{k\neq k',|k'|\leq N}}{(2\pi)^3}\!\!\!\!\!\sum_{\substack{k_2,k_3\\
		k_2\neq k_3, k_2-k_3=k-k' }}\!\!\!\!\!\!\!\int\widehat{\chi}(\lambda-\lambda'+\lambda_2-\lambda_3-\Phi_{k',k_2,k_3})\ov{\widetilde{v_2}}(\lambda_2,k_2)\widetilde{v_3}(\lambda_3,k_3)d\lambda_2 d\lambda_3,
\end{align}
then from Lemma \ref{DuhamelKernel}, the kernel of $\mathcal{Q}_{3,N}$ is given by
\begin{align}\label{kernel1}
\widetilde{\Theta}_{kk'}(\lambda,\lambda')=\int_{\R} K(\lambda,\mu)\Xi_{kk'}(\mu,\lambda')d\mu.
\end{align}

\subsection{$S^{b,q}$-mapping properties of the operator $\mathcal{P}_{N,L}^+$}
In this subsection, we prove (6) of Proposition \ref{Multilinearkey}.	By decomposing $v_L^{\#}$ as sums of type (G), (C) and (D) terms in $\mathcal{P}_{N,L}^+$, it suffices to prove the estimate by changing $\mathcal{P}_{N,L}^+$ to some operator
$$ w\mapsto \Theta_N(w):=\Pi_N\mathcal{I}\mathcal{N}_3(w,v_2,v_3)
$$	
for functions $v_2=\Pi_{L}v_2, v_3=\Pi_{L}v_3$ of type (G),(C) or (D) with characterized parameters $(N_2,L_2),(N_3,L_3)$ satisfying $N_2\vee N_3=L$ or $\frac{L}{2}$ and $N_2,N_3\ll N$. We denote by $(\Theta_{kk_1}(t,t'))$ the kernel of $\Theta_N$. Let $\mathcal{H}(t)$ be a linear operator with kernel $(H_{kk^*}(t))$. By abusing the notation, we still denote by $q_{kk^*}(t)$ the matrix-element of the operator $\Theta_N\circ(\widetilde{\mathbf{P}}_N\mathcal{H}(\lambda))$. Therefore,
$$ \widetilde{q}_{kk^*}(\lambda)=\int_{\R}\sum_{|k_1|\sim N}\widetilde{\Theta}_{kk_1}(\lambda,\lambda_1)\widetilde{H}_{k_1k^*}(\lambda_1)d\lambda_1.
$$
Note that on the support of $\widetilde{H}_{k_1k^*}(\lambda_1)$, we have $|k_1|\sim N$. Inserting \eqref{interkernel1} and \eqref{kernel1} into the expression above, we have
\begin{align*}
\widetilde{q}_{kk^*}(\lambda)=&\sum_{\substack{(k_1,k_2,k_3)\in\Gamma(k)\\
		|k_1|\sim N\\
		|k_j|\leq N_j,j=2,3 }}
\!\!\!\!\!\!\!\int K(\lambda,\mu) \widetilde{H}_{k_1k^*}(\lambda_1) d\mu d\lambda_1\\
&\hspace{1.5cm}\times \int \widehat{\chi}(\mu-\lambda_1+\lambda_2-\lambda_3+\Phi_{k_1,k_2,k_3})\ov{\widetilde{v}_2}(\lambda_2,k_2)\widetilde{v}_3(\lambda_3,k_3)d\lambda_2d\lambda_3,
\end{align*}
where the kernel $|K(\lambda,\mu)|\lesssim \langle\mu\rangle^{-1}(\lg\lambda\rg^{-10}+\lg\lambda-\mu\rg^{-10} )$. Our goal is to estimate 
 $$\|\lg\lambda\rg^{\frac{2b_1}{q'}}\widetilde{q}_{kk^*}(\lambda)\|_{l_k^{\infty}L_{\lambda}^ql_{k^*}^2}.$$ We will control it in two ways, according to the size of $L$.

\noi
$\bullet${\bf Case 1: $L^{\frac{2b_1}{q'}-\frac{\alpha-1}{2}+\epsilon_2+\delta_0}\lesssim N^{\frac{(\alpha-1)(q'-2b_1)}{q'}}$} 

\vspace{0.3cm}	
In this case, we will \emph{integrate high modulations first}.

By taking $l_{k^*}^2$ and using Minkowski, H\"older, we have (here we hide the constraints on $k_j$ by implicitly inserting some indicators to $\widetilde{H}_{k_1k^*}, \widetilde{v}_2,$ and $\widetilde{v}_3$ )
\begin{align*}
\|\widetilde{q}_{kk^*}(\lambda)\|_{l_{k^*}^2}\lesssim &\!\!\!\sum_{(k_1,k_2,k_3)\in\Gamma(k)}\!
\int \frac{1}{\lg\mu\rg}\Big(\frac{1}{\lg\lambda\rg^{10} }+\frac{1}{\lg\lambda-\mu\rg^{10}}\Big)\frac{a_1(k_1)a_2(k_2)a_3(k_3)}{\lg\mu-\Phi_{k_1,k_2,k_3}\rg^{b_0}} d\mu\end{align*}
since $\frac{2b}{q'}>b_0>\frac{1}{2}$,
where
$$ a_1(k_1)=\|\lg\lambda_1\rg^{\frac{2b}{q'}}H_{k_1k^*}(\lambda_1)\|_{L_{\lambda_1}^ql_{k^*}^2},\quad a_j(k_j)=\|\lg\lambda_j\rg^{b_0}\widetilde{v}_j(\lambda_j,k_j)\|_{L_{\lambda_j}^2},\;j=2,3.
$$
Note that here we used the fact that 
$$ |\widehat{\chi}(\mu-\lambda_1+\lambda_2-\lambda_3-\Phi_{k_1,k_2,k_2})|\lesssim \lg\mu-\lambda_1+\lambda_2-\lambda_3-\Phi_{k_1,k_2,k_2}\rg^{-10}
$$
and Lemma \ref{convolution}. Using Lemma \ref{convolution} again, we have
$$ \int_{|\mu-\lambda|\geq 2|\lambda|}\frac{d\mu}{\lg\mu\rg \lg\lambda-\mu\rg^{10} \lg\mu-\Phi_{k_1,k_2,k_3}\rg^{b_0} }\lesssim \frac{1}{\lg\lambda\rg\lg\lambda-\Phi_{k_1,k_2,k_3}\rg^{b_0}}.
$$
By writing $$\int_{|\mu-\lambda|<2|\lambda|}\leq \int_{|\mu|\leq \frac{\lambda}{2}}+\int_{|\mu-\lambda|\leq \frac{|\lambda|}{2}}+\int_{|\mu-\lambda|\geq \frac{|\lambda|}{2},\frac{|\lambda|}{2}\leq |\mu|\leq 3|\lambda|},
$$
we deduce from Lemma \ref{convolution} that
$$\int_{|\mu-\lambda|<2|\lambda|} \frac{d\mu}{\lg\mu\rg \lg\lambda-\mu\rg^{10} \lg\mu-\Phi_{k_1,k_2,k_3}\rg^{b_0} }\lesssim_{\epsilon} \frac{1}{\lg\lambda\rg\lg\lambda-\Phi_{k_1,k_2,k_3}\rg^{b_0-\epsilon}}.
$$ 
Thus
$$ \int\frac{1}{\lg\mu\rg}\Big(\frac{1}{\lg\lambda\rg^{10} }+\frac{1}{\lg\lambda-\mu\rg^{10}}\Big)\cdot\frac{1}{\lg\mu-\Phi_{k_1,k_2,k_3}\rg^{b_0}}d\mu\lesssim_{\epsilon} \frac{1}{\lg\lambda\rg\lg\lambda-\Phi_{k_1,k_2,k_3}\rg^{b_0-\epsilon}}.
$$
Next, by H\"older's inequality and Lemma \ref{convolution}, we have
\begin{align*}
&\|\lg\lambda\rg^{\frac{2b_1}{q'}}\widetilde{q}_{kk^*}(\lambda)\|_{L_{\lambda}^ql_{k^*}^2}\lesssim 
\sum_{(k_1,k_2,k_3)\in\Gamma(k)}\frac{\prod_{j=1}^3a_j(k_j)}{\lg\Phi_{k_1,k_2,k_3}\rg^{1-\frac{2b_1}{q'}}}\\
\lesssim &\sum_{|m|\leq 10L}\frac{1}{1+(mN^{\alpha-1})^{1-\frac{2b_1}{q'} }} \sum_{(k_1,k_2,k_3)\in\Gamma(k)}a_1(k_1)a_2(k_2)a_3(k_3)\mathbf{1}_{\Phi_{k_1,k_2,k_3}\in I_m^{(N)}},
\end{align*}
provided that $q'>2b_1$, where $I_m^{(N)}=[(m-1)N^{\alpha-1}, mN^{\alpha-1}]$, since for $N\gg L$, the value of $\Phi_{k_1,k_2,k_3}$ is constraint in $\cup_{|m|\leq 10L}I_m^{(N)}$. Since $|\partial_{k_2}\Phi_{k_1,k_2,k_3}|,|\partial_{k_3}\Phi_{k_1,k_2,k_3}|\gtrsim N^{\alpha-1}$, we have 
$$ \sum_{k_2}\mathbf{1}_{\Phi_{k_1,k_2,k_3}\in I_m^{(N)}}\lesssim 1,\quad \sum_{k_3}\mathbf{1}_{\Phi_{k_1,k_2,k_3}\in I_m^{(N)}}\lesssim 1,
$$
uniformly in $m$. Then by Schur's test, for fixed $k$,
$$ \sum_{k_2,k_3}a_1(k+k_2-k_3)a_2(k_2)a_3(k_3)\mathbf{1}_{\Phi_{k+k_2-k_3,k_2,k_3}\in I_m^{(N)}}\lesssim \|a_1(k_1)\|_{l_{k_1}^{\infty}}\|a_2(k_2)\|_{l_{k_2}^2}\|a_3(k_3)\|_{l_{k_3}^2}.
$$
Thus from the elementary inequality
$$ \sum_{|m|\leq 2L}\frac{1}{1+(mN^{\alpha-1})^{\beta}}\leq 1+2\int_0^{10L}\frac{dz}{1+(N^{\alpha-1}z)^{\beta}}\lesssim_{\beta} 1+\frac{L^{1-\beta}}{N^{(\alpha-1)\beta}}
$$
for $0<\beta<1$, we have (with $\beta=1-\frac{2b_1}{q'}$)
\begin{align}\label{method1} 
\|\lg\lambda\rg^{\frac{2b_1}{q'}}\widetilde{q}_{kk^*}(\lambda)\|_{l_k^{\infty}L_{\lambda}^ql_{k^*}^2 }\lesssim &\Big(1+L^{\frac{2b_1}{q'}}N^{-(\alpha-1)\big(1-\frac{2b_1}{q'}\big)}\Big)\|\mathcal{H}\|_{S^{b,q}}\|v_2\|_{X^{0,b_0}}\|v_3\|_{X^{0,b_0}}  \notag \\
\lesssim &\big(L^{-\frac{\alpha-1}{2}+\epsilon_2}+L^{\frac{2b_1}{q'}-\frac{\alpha-1}{2}+\epsilon_2}N^{-(\alpha-1)\big(1-\frac{2b_1}{q'}\big) }\big)\|\mathcal{H}\|_{S^{b,q}},
\end{align}
since $N_2\vee N_3\sim L$ and at least (when $v_2,v_3$ are both of type (G)) 
$$ \|v_2\|_{X^{0,b_0}}\lesssim N_2^{-\frac{\alpha-1}{2}+\epsilon_2},\; \|v_3\|_{X^{0,b_0}}\lesssim N_3^{-\frac{\alpha-1}{2}+\epsilon_2}. 
$$
Therefore, whenever
\begin{align*}
L^{\frac{2b_1}{q'}-\frac{\alpha-1}{2}+\epsilon_2+\delta_0}\lesssim N^{\frac{(\alpha-1)(q'-2b_1)}{q'}},
\end{align*}
the upper bound \eqref{method1} is $L^{-\delta_0}\|\mathcal{H}\|_{S^{b,q}}$, which is conclusive.

\vspace{0.3cm}

\noi
$\bullet${\bf Case 2: $L^{\frac{2b_1}{q'}-\frac{\alpha-1}{2}+\epsilon_2+\delta_0}\gg N^{\frac{(\alpha-1)(q'-2b_1)}{q'}}$} 

\vspace{0.3cm}

In this case, we will \emph{reduce to multi-linear sums of low modulations. }

By duality, it suffices to estimate
\begin{align*}
&\int \sum_{\substack{k,k_1,k_2,k_3,k^*\\
		(k_1,k_2,k_3)\in\Gamma(k) } }
\lg\lambda\rg^{\frac{2b_1}{q'}}K(\lambda,\mu)y_{kk^*}(\lambda)\wt{H}_{k_1k^*}(\lambda_1)\ov{\widetilde{v}}_2(\lambda_2,k_2)\widetilde{v}_3(\lambda_3,k_3)\\
&\hspace{2cm} \times \widehat{\chi}(\mu-\lambda_1+\lambda_2-\lambda_3-\Phi_{k_1,k_2,k_3})d\lambda_1 d\lambda_2 d\lambda_3 d\mu d\lambda
\end{align*}
where $\|y_{kk^*}(\lambda)\|_{l_k^1L_{\lambda}^{q'}l_{k^*}^2}=1$.
Summing over $k^*$ and using Cauchy-Schwartz, it suffices to estimate the expression
\begin{align}\label{eqsubsec5.1case2}
&\sum_{\substack{k,k_1,k_2,k_3\\
		(k_1,k_2,k_3)\in\Gamma(k) } } \int \frac{\lg\lambda\rg^{\frac{2b_1}{q'}} }{\lg\mu\rg}\Big(\frac{1}{\lg\lambda\rg^A}+\frac{1}{\lg\mu-\lambda\rg^A} \Big)\cdot \frac{g_k(\lambda)f_{k_1}(\lambda_1)w_2(\lambda_2,k_2)w_3(\lambda_3,k_3) }{\lg\lambda_1\rg^{\frac{2b}{q'}}\lg\lambda_2\rg^{b_0}\lg\lambda_3\rg^{b_0} }\\
&\hspace{3cm} \times |\widehat{\chi}(\mu-\lambda_1-\lambda_2+\lambda_3+\Phi_{k_1,k_2,k_3})|d\lambda_1 d\lambda_2 d\lambda_3 d\mu d\lambda, \notag
\end{align}
where $A\gg 1$, $g_k(\lambda)=\|y_{kk^*}(\lambda)\|_{l_{k^*}^2}$, $f_{k_1}(\lambda_1)=\lg\lambda_1\rg^{\frac{2b}{q'}}\|\wt{H}_{k_1k^*}(\lambda_1)\|_{l_{k^*}^2}$ and $w_j(\lambda_j,k_j)=|\lg\lambda_j\rg^{b_0}\widetilde{v}_j(\lambda_j,k_j)|$ for $j=2,3$. From the rapid decay of $\widehat{\chi}$, the contribution from $|\mu-\lambda_1+\lambda_2-\lambda_3|\gg N^{\alpha-1}L$ is negligible and can be simply controlled by \footnote{see the proof of Proposition \ref{modulationreduction1} later for details.} 
\begin{align}\label{error}
N^{-10}\|g_k(\lambda)\|_{l_k^1L_{\lambda}^{q'} }\|f_{k_1}(\lambda_1)\|_{l_{k_1}^{\infty}L_{\lambda_1}^{q} } \|w_2(\lambda_2,k_2)\|_{L_{\lambda_2}^2l_{k_2}^2}\|w_3(\lambda_3,k_3)\|_{L_{\lambda_3}^2l_{k_3}^2},
\end{align}
hence we may assume that the multiple integration is taken over $|\mu-\lambda_1+\lambda_2-\lambda_3|\lesssim N^{\alpha-1}L$. Denote by
$$ \mathcal{T}_{\mu,\lambda,\lambda_1,\lambda_2,\lambda_3}:=\sum_{\substack{k,k_1,k_2,k_3\\
		(k_1,k_2,k_3)\in\Gamma(k) } }
g_k(\lambda)f_{k_1}(\lambda_1)|\widehat{\chi}(\mu-\lambda_1+\lambda_2-\lambda_3-\Phi_{k_1,k_2,k_3})|w_2(\lambda_2,k_2)w_3(\lambda_3,k_3).
$$
Note that for fixed $|k|\sim N$ and $\mu_0\in\R$,
$$ \sup_{k_2}\sum_{k_3}f_{k+k_2-k_3}(\lambda_1)\mathbf{1}_{|k_2|,|k_3|\leq L}|\widehat{\chi}(\mu_0+\Phi_{k+k_2-k_3,k_2,k_3})|\lesssim \|f_{k+k_2-k_3}(\lambda_1)\mathbf{1}_{|k_2|,|k_3|\leq L}\|_{l_{k_2,k_3}^{\infty}},
$$
and
$$\sup_{k_3}\sum_{k_2}f_{k+k_2-k_3}(\lambda_1)\mathbf{1}_{|k_2|,|k_3|\leq L}|\widehat{\chi}(\mu_0+\Phi_{k+k_2-k_3,k_2,k_3})|\lesssim \|f_{k+k_2-k_3}(\lambda_1)\mathbf{1}_{|k_2|,|k_3|\leq L}\|_{l_{k_2,k_3}^{\infty}}
$$
since $|\partial_{k_j}\Phi_{k+k_2-k_3,k_2,k_3}|\gtrsim N^{\alpha-1}\gg 1$ for $j=2,3$.
Applying Schur's test, we have
\begin{equation}\label{Schur}
\begin{split}  
&\sum_{\substack{k_2,k_3\\ 
		|k_j|\sim N_j,j=2,3} }f_{k+k_2-k_3}(\lambda_1)|\widehat{\chi}(\mu_0+\Phi_{k+k_2-k_3,k_2,k_3})|w_2(\lambda_2,k_2)w_3(\lambda_3,k_3)\\ \lesssim &\big\|f_{k+k_2-k_3}(\lambda_1)\mathbf{1}_{|k_2|\leq L,
		|k_3|\leq L }\big\|_{l_{k_2,k_3}^{\infty}}\|w_2(\lambda_2,k_2)\|_{l_{k_2}^2}\|w_3(\lambda_3,k_3)\|_{l_{k_3}^2}\\
\leq &\big\|f_{k+k_2-k_3}(\lambda_1)\mathbf{1}_{|k_2|\leq L,|k_3|\leq L}\big\|_{l_{k_2,k_3}^{q}}\|w_2(\lambda_2,k_2)\|_{l_{k_2}^2}\|w_3(\lambda_3,k_3)\|_{l_{k_3}^2},
\end{split}
\end{equation}
where to the last inequality, we use the embedding $l^q\rightarrow l^{\infty}$. Writing $\varphi_j(\lambda)=\|w_j(\lambda_j,k_j)\|_{l_{k_j}^2}$ for $j=2,3$, by H\"older we estimate the contribution from $\frac{1}{\lg\mu\rg\lg\lambda\rg^A}$ in \eqref{eqsubsec5.1case2} by
\begin{align*}
&\int \mathbf{1}_{|\mu-\lambda_1+\lambda_2-\lambda_3|\lesssim N^{\alpha-1}L} \lg\mu\rg^{-1}\lg\lambda\rg^{-A+\frac{2b_1}{q'}}
\lg\lambda_1\rg^{-\frac{2b}{q'}}
\lg\lambda_2\rg^{-b_0}
\lg\lambda_3\rg^{-b_0} \mathcal{T}_{\mu,\lambda,\lambda_1,\lambda_2,\lambda_3}d\mu d\lambda d\lambda_1 d\lambda_2 d\lambda_3\\
\lesssim_{\epsilon} &N^{\epsilon}\sum_{k}\int g_k(\lambda)\lg\lambda\rg^{-A+\frac{2b_1}{q'}}\lg\lambda_1\rg^{-\frac{2b}{q'}}\lg\lambda_2\rg^{-b_0}\lg\lambda_3\rg^{-b_0}\varphi_2(\lambda_2)\varphi_3(\lambda_3)\big\|f_{k+k_2-k_3}(\lambda_1)\mathbf{1}_{\substack{|k_2|\leq L\\
		|k_3|\leq L }}\big\|_{l_{k_2,k_3}^q}d\lambda d\lambda_1 d\lambda_2 d\lambda_3\\
\lesssim & N^{\epsilon}\sum_{k}\|g_k(\lambda)\|_{L_{\lambda}^{q'}}\|\varphi_2(\lambda_2)\|_{L_{\lambda_2}^2}\|\varphi_3(\lambda_3)\|_{L_{\lambda_3}^2}\big\|f_{k+k_2-k_3}(\lambda_1)\mathbf{1}_{\substack{|k_2|\leq L\\
		|k_3|\leq L }}\big\|_{l_{k_2,k_3}^qL_{\lambda_1}^q}\\
\lesssim &N^{\epsilon}\|g_k(\lambda)\|_{l_k^1L_{\lambda}^{q'}}\|\varphi_2(\lambda_2)\|_{L_{\lambda_2}^2}\|\varphi_3(\lambda_3)\|_{L_{\lambda_3}^2}\big\|f_{k+k_2-k_3}(\lambda_1)\mathbf{1}_{\substack{|k_2|\leq L\\
		|k_3|\leq L }}\big\|_{l_k^{\infty}l_{k_2,k_3}^qL_{\lambda_1}^q}\\
\lesssim &N^{\epsilon}L^{\frac{2}{q}}\|g_k(\lambda)\|_{l_k^1L_{\lambda}^{q'}}\|\varphi_2(\lambda_2)\|_{L_{\lambda_2}^2}\|\varphi_3(\lambda_3)\|_{L_{\lambda_3}^2}\|f_{k_1}(\lambda_1)\|_{l_k^{\infty}L_{\lambda_1}^q}.
\end{align*}
It remains to estimate the last contribution from $\frac{1}{\lg\mu\rg\lg\mu-\lambda\rg^A}$ in \eqref{eqsubsec5.1case2}. Note that the integration over $|\lambda-\mu|\gg N^{\alpha-1}L$ gives us an error like \eqref{error}, hence we can assume that $|\lambda-\mu|\lesssim N^{\alpha-1}L$, and in particular, $|\lambda-\lambda_1+\lambda_2-\lambda_3|\lesssim N^{\alpha-1}L$, since the region of integration is $|\mu-\lambda_1+\lambda_2-\lambda_3|\lesssim N^{\alpha-1}L.$ From Lemma \ref{convolution},
$$  \int \frac{\lg\lambda\rg^{\frac{2b_1}{q'}} }{\lg\mu\rg\lg\mu-\lambda\rg^A}d\mu\lesssim \lg\lambda\rg^{-1+\frac{2b_1}{q'}}.
$$
Using \eqref{Schur}, the inequality above, and then H\"older's inequality for the integration in $\lambda$, we have
\begin{align*}
&\int \mathbf{1}_{\substack{|\mu-\lambda_1+\lambda_2-\lambda_3|\lesssim N^{\alpha-1}L\\
		|\lambda-\lambda_1+\lambda_2-\lambda_3|\lesssim N^{\alpha-1}L } } \lg\mu\rg^{-1}\lg\lambda-\mu\rg^{-A}\lg\lambda\rg^{\frac{2b_1}{q'}}
\lg\lambda_1\rg^{-\frac{2b}{q'}}
\lg\lambda_2\rg^{-b_0}
\lg\lambda_3\rg^{-b_0} \mathcal{T}_{\mu,\lambda,\lambda_1,\lambda_2,\lambda_3}d\mu d\lambda d\lambda_1 d\lambda_2 d\lambda_3\\
\lesssim &\sum_{k}\int \mathbf{1}_{|\lambda-\lambda_1+\lambda_2-\lambda_3|\lesssim N^{\alpha-1}L}\cdot g_k(\lambda)\lg\lambda\rg^{-\frac{q'-2b_1}{q'}}\lg\lambda_1\rg^{-\frac{2b}{q'}}\lg\lambda_2\rg^{-b_0}\lg\lambda_3\rg^{-b_0}\varphi_2(\lambda_2)\varphi_3(\lambda_3)\\
&\hspace{4cm}\times 
\big\|f_{k+k_2-k_3}(\lambda_1)\mathbf{1}_{|k_2|\leq L,|k_3|\leq L}\big\|_{l_{k_2,k_3}^{q}}d\lambda d\lambda_1 d\lambda_2 d\lambda_3\\
\lesssim &\sum_{k}\|g_k(\lambda)\|_{L_{\lambda}^{q'}}\|\mathbf{1}_{|\lambda-\cdot|\lesssim N^{\alpha-1}L}\|_{L_{\lambda}^{\frac{q'}{2b_1-1+\epsilon}}}\|\lg\lambda\rg^{-\frac{q'-2b_1}{q'}}\|_{L_{\lambda}^{\frac{q'}{q'-2b_1-\epsilon}}}\\
\times &\int\lg\lambda_1\rg^{-\frac{2b}{q'}}\lg\lambda_2\rg^{-b_0}\lg\lambda_3\rg^{-b_0}\varphi_2(\lambda_2)\varphi_3(\lambda_3)
\big\|f_{k+k_2-k_3}(\lambda_1)\mathbf{1}_{|k_2|\leq L,|k_3|\leq L}\big\|_{l_{k_2,k_3}^{q}}d\lambda_1 d\lambda_2 d\lambda_3,
\end{align*} 
for $\epsilon>0$ small and to be chosen later.
Next we use H\"older's inequality for the integration in $\lambda_1,\lambda_2,\lambda_3$, the above quantity can be bounded by
\begin{align*}
&N^{\frac{(\alpha-1)(2b_1-1+\epsilon)}{q'}}L^{\frac{2b_1-1+\epsilon}{q'}}\sum_{k}\|g_k(\lambda)\|_{L_{\lambda}^{q'}}\|\varphi_2(\lambda_2)\|_{L_{\lambda_2}^2}\|\varphi_3(\lambda_3)\|_{L_{\lambda_3}^2}\big\|f_{k+k_2-k_3}(\lambda_1)\mathbf{1}_{\substack{|k_2|\leq L\\|k_3|\leq L}}\big\|_{l_{k_2,k_3}^{q}L_{\lambda_1}^{q}}\\
\leq &N^{\frac{(\alpha-1)(2b_1-1+\epsilon)}{q'}}L^{\frac{2b_1-1+\epsilon}{q'}+\frac{2}{q}}\|g_k(\lambda)\|_{l_k^1L_{\lambda}^{q'}}\|\varphi_2(\lambda_2)\|_{L_{\lambda_2}^2}\|\varphi_3(\lambda_3)\|_{L_{\lambda_3}^2}\|f_{k_1}(\lambda_1)\|_{l_k^{\infty}L_{\lambda_1}^q}.
\end{align*} 
Since 
$$ \|\varphi_j(\lambda_j)\|_{L_{\lambda_j}^2}=\|v_j\|_{X^{0,b_0}}\lesssim N_j^{-\frac{\alpha-1}{2}+\epsilon_2},\; j=2,3
$$
and $N_2\vee N_3\sim L$,
we finally have (fixing $\epsilon=\epsilon_2$, say)
\begin{align}\label{method2}
\|\lg\lambda\rg^{\frac{2b_1}{q'}}\widetilde{q}_{kk^*}(\lambda)\|_{l_k^{\infty}L_{\lambda}^ql_{k^*}^2}\lesssim & N^{\frac{(\alpha-1)(2b_1-1+\epsilon)}{q'}}L^{-\frac{\alpha-1}{2}+\frac{2b_1-1+\epsilon_2}{q'}+\frac{2}{q}+\epsilon_2+\delta_0}\cdot L^{-\delta_0}\|\mathcal{H}\|_{S^{b,q}}.
\end{align}
Recall that for the case 2, 
$$ L^{\frac{2b_1}{q'}-\frac{\alpha-1}{2}+\epsilon_2+\delta_0}\gg N^{\frac{(\alpha-1)(q'-2b_1)}{q'}},
$$
then
$$ L^{\frac{\alpha-1}{2}-\frac{2b_1-1}{q'}-\frac{2}{q}-\epsilon_2-\delta_0}\gg N^{\frac{(\alpha-1)(q'-2b_1)}{q'} \cdot \big(\frac{2b_1}{q'}-\frac{\alpha-1}{2}+\epsilon_2+\delta_0\big)^{-1}}.
$$
By our definition of numerical parameters in \eqref{numerical}, if the free parameter $\sigma$ is chosen small enough, we have
$\frac{(\alpha-1)(q'-2b_1)}{q'}\cdot \big(\frac{2b_1}{q'}-\frac{\alpha-1}{2}+\epsilon_2+\delta_0\big)^{-1}\geq \frac{(\alpha-1)(2b_1-1)}{q'}
 ,$
in particular,
$$ L^{\frac{\alpha-1}{2}-\frac{2b_1-1+\epsilon}{q'}-\frac{2}{q}-\epsilon_2-\delta_0 }\geq  N^{\frac{(\alpha-1)(2b_1-1+\epsilon_2)}{q'}}.
$$
This completes the proof of (6) of Proposition \ref{Multilinearkey}.

\subsection{$X^{0,b}$-mapping property of the operator norm of $\mathcal{P}_{L}^{\pm}$}
In this subsection, we prove (3),(5),(7) of Proposition \ref{Multilinearkey}. The key point is the following lemma:
\begin{lemme}\label{lem:mappingXsb}
	Assume that $v_1, v_2, v_3$ are of type (G), (C) or (D) with characterized parameters $(N_j,L_j), j=1,2,3$, then
	$$ \|\mathcal{N}_3(v,v_2,v_3)\|_{X^{0,b_1-1}}\lesssim (N_2\vee N_3)^{-\delta_0 }\|v\|_{X^{0,\frac{3}{8}}}
	$$
	and
	$$ \|\mathcal{N}_3(v_1,v,v_3)\|_{X^{0,b_1-1}}\lesssim (N_1\vee N_3)^{-\delta_0}\|v\|_{X^{0,\frac{3}{8}}}.
	$$
	Moreover, the same estimate holds, with uniform implicit constants on the r.h.s., if we replace $v_1,v_2,v_3$ by $\Pi_{M_1}v_1$, $\Pi_{M_2}v_2$ and $\Pi_{M_3}v_3$ for any dyadic numbers $M_1,M_2,M_3$. 
\end{lemme}

\begin{proof}
	We will only prove the estimate for $\mathcal{N}_3(v,v_2,v_3)$, since the other follows from the same argument.
	By duality, it suffices to show that\footnote{We omit the estimate for the diagonal nonlinearities $\mathcal{N}_0(v,v_2,v_3)$ here, since this term is always better and will be treated in Section \ref{sec:high-high-high}. }, for every $w\in X^{0,1-b_1}, \|w\|_{X^{0,1-b_1}}\leq 1$, we have
	$$ \Big|\iint v\ov{v}_2v_3\cdot\ov{w}dtdx \Big|\lesssim (N_2\vee N_3)^{-\delta_0}\|v\|_{X^{0,\frac{3}{8}}}.
	$$
	Splitting the functions as Littlewood-Paley pieces, $v_j=\sum_{M_j}\mathbf{P}_{M_j}v_j, j=2,3$ and $v=\sum_{M_1}\mathbf{P}_{M_1}v$, $w=\sum_{M}\mathbf{P}_{M}w$. 
	
	\vspace{0.3cm}
	\noi
	$\bullet${\bf Case 1: $v_2,v_3$ are both of type (D)}
	
	In this case, we have
	$$ \|\mathbf{P}_{M_j}v_j\|_{X^{0,b}}\lesssim (M_j\vee N_j)^{-s},\quad j=2,3.
	$$
	By inserting $\chi(t)$ into $w$ and using the bilinear Strichartz (Lemma \ref{bilinearStrichartz}), we have
	\begin{align*}
	&\Big|\iint v\ov{v}_2v_3\cdot\ov{w}dtdx\Big|\\ \lesssim&\sum_{\substack{M_1,M_2,M_3,M\\
			M\lesssim M_1\vee M_2\vee M_3 } } (M_{(2)}M_{(3)})^{\frac{1}{2}-\frac{\alpha}{4}}\|\mathbf{P}_{M_1}v\|_{X^{0,\frac{3}{8}}}
	\|\mathbf{P}_{M_2}v_2\|_{X^{0,\frac{3}{8}}}\|\mathbf{P}_{M_3}v_3\|_{X^{0,\frac{3}{8}}}
	\|\mathbf{P}_{M}w\|_{X^{0,\frac{3}{8}}},
	\end{align*}
	where $M_{(1)}\geq M_{(2)}\geq M_{(3)}$ is the non-increasing re-ordering of $M_1,M_2,M_3$. Without loss of generality, we may assume that $M_2\geq M_3$. Note that when $M_1\gg M_2\vee M_3$, we must have $M\sim M_1$ in the sum , otherwise $\int \mathbf{P}_{M_1}v\cdot\mathbf{P}_{M_2}\ov{v}_2\mathbf{P}_{M_3}v\cdot\mathbf{P}_{M}\ov{w}dx=0.$ We estimate this contribution as
	\begin{align*}
	&\sum_{\substack{M,M_1,M_2,M_3\\
			M\sim M_1\gg M_2\geq M_3 } }(M_{(2)}M_{(3)})^{\frac{1}{2}-\frac{\alpha}{4}}\|\mathbf{P}_{M_1}v\|_{X^{0,\frac{3}{8}}}
	\|\mathbf{P}_{M_2}v_2\|_{X^{0,\frac{3}{8}}}\|\mathbf{P}_{M_3}v_3\|_{X^{0,\frac{3}{8}}}
	\|\mathbf{P}_{M}w\|_{X^{0,\frac{3}{8}}}\\
	\lesssim &\sum_{\substack{M,M_1\\
			M\sim M_1 } }\|\mathbf{P}_{M_1}v\|_{X^{0,\frac{3}{8}}}\|\mathbf{P}_Mw\|_{X^{0,\frac{3}{8}}}\sum_{\substack{M_2,M_3\\
			M_3\leq M_2\ll M_1 } }(M_2M_3)^{\frac{1}{2}-\frac{\alpha}{4}}(N_2\vee M_2)^{-s}(N_3\vee M_3)^{-s} \\
	\lesssim &(N_2\vee N_3)^{\frac{1}{2}-\frac{\alpha}{4}-s}\|v\|_{X^{0,\frac{3}{8}}},
	\end{align*}
	where to the final inequality, we use the fact that $s-(\frac{1}{2}-\frac{\alpha}{4})=\sigma>0$, $\mathrm{min}\{b,1-b_1 \}>\frac{3}{8}$ and Cauchy-Schwartz for the sum $\sum_{M\sim M_1}$. For the contribution of $M_1\lesssim M_2$, we estimate it as
	\begin{align*}
	&\sum_{\substack{M,M_1,M_2,M_3\\
			M_3\leq M_2\\
			M,M_1\lesssim M_2 } }(M_{(2)}M_{(3)})^{\frac{1}{2}-\frac{\alpha}{4}}\|\mathbf{P}_{M_1}v\|_{X^{0,\frac{3}{8}}}
	\|\mathbf{P}_{M_2}v_2\|_{X^{0,\frac{3}{8}}}\|\mathbf{P}_{M_3}v_3\|_{X^{0,\frac{3}{8}}}
	\|\mathbf{P}_{M}w\|_{X^{0,\frac{3}{8}}}\\
	\lesssim &\sum_{\substack{M,M_1,M_2,M_3\\
			M_3\leq M_2\\
			M,M_1\lesssim M_2 } } (M_2M_3)^{\frac{1}{2}-\frac{\alpha}{4}}(M_2\vee N_2)^{-s}(N_3\vee M_3)^{-s} \|\mathbf{P}_{M_1}v\|_{X^{0,\frac{3}{8}}}\|\mathbf{P}_Mw\|_{X^{0,\frac{3}{8}}}\\
	\lesssim &\sum_{M_2,M_3:M_3\leq M_2}(M_2M_3)^{\frac{1}{2}-\frac{\alpha}{4}}\log(M_2)^2(M_2\vee N_2)^{-s}(M_3\vee N_3)^{-s}\cdot\|v\|_{X^{0,\frac{3}{8}}}\\ \lesssim_{\epsilon} & (N_2\vee N_3)^{\frac{1}{2}-\frac{\alpha}{4}-s+\epsilon}\|v\|_{X^{0,\frac{3}{8}}}\leq (N_2\vee N_3)^{-\delta_0}\|v\|_{X^{0,\frac{3}{8}}},
	\end{align*}
	by choosing $0<\epsilon<\sigma-\delta_0$ (which is positive if $\sigma\ll 1$)  here.
	Therefore,
	$$ \Big|\iint  v\ov{v}_2v_3\cdot\ov{w}dxdt \Big|\lesssim (N_2\vee N_3)^{-\delta_0}\|v\|_{X^{0,\frac{3}{8}}}.
	$$

	\noi
	$\bullet$ {\bf Case 2: One of $v_2,v_3$ is of type (D) and the other is of type (G) or (C) }
	
	Without loss of generality, we may assume that $v_2$ is of type (D) and $v_3$ is of type (G) or (C), thus
	$$ \|\mathbf{P}_{M_2}v_2\|_{X^{0,b}}\lesssim (M_2\vee N_2)^{-s}.
	$$
	Using the bilinear Strichartz inequality, we have
	\begin{align*}
	&\Big|\iint \mathbf{P}_{M_1}v\mathbf{P}_{M_2}\ov{v}_2 v_3\cdot\mathbf{P}_M\ov{w}dxdt \Big| \leq \|\mathbf{P}_{M_1}v\mathbf{P}_{M_2}v_2\|_{L_{t,x}^2}\|v_3\|_{L_t^4L_x^{\infty}}\|\mathbf{P}_Mw\|_{L_t^4L_x^2}\\
	\lesssim &(M_1\wedge M_2)^{\frac{1}{2}-\frac{\alpha}{4}}\|\mathbf{P}_{M_1}v\|_{X^{0,\frac{3}{8}}}\|\mathbf{P}_{M_2}v_2\|_{X^{0,\frac{3}{8}}}
	\|\mathbf{P}_{M}w\|_{X^{0,\frac{1}{4}}}\|v_3\|_{L_t^4L_x^{\infty}}.
	\end{align*}
	We then take the dyadic summation in $M_1,M_2$ and $M$. Since the Fourier support of $v_3$ is constraint at $|k_3|\lesssim N_3$, the contribution for $M\vee M_1\vee M_2\lesssim N_3$ is bounded by
	\begin{align*}
	\sum_{\substack{M,M_1,M_2\\
			M\vee M_1\vee M_2\lesssim N_3 }} (M_1\wedge M_2)^{\frac{1}{2}-\frac{\alpha}{4}} \|\mathbf{P}_{M_1}v\|_{X^{0,\frac{3}{8}}}\|\mathbf{P}_{M}w\|_{X^{0,\frac{1}{4}}}\|\mathbf{P}_{M_2}v_2\|_{X^{0,\frac{3}{8}}}
	\lesssim &N_2^{\frac{1}{2}-\frac{\alpha}{4}-s}\log(N_3)^2\|v\|_{X^{0,\frac{3}{8}}}.
	\end{align*}
	When $M\vee M_1\vee M_2\gg N_3$, then one of the situations must happen: $M\sim M_1\sim M_2$, or $M\sim M_1\gg M_2$, or $M\sim M_2\gg M_1$, or $M_1\sim M_2\gg M$. Therefore, we have
	\begin{align*}
	\sum_{\substack{M,M_1,M_2\\
			M\vee M_1\vee M_2\gg N_3 } }(M_1\wedge M_2)^{\frac{1}{2}-\frac{\alpha}{4}}
	\|\mathbf{P}_{M_1}v\|_{X^{0,\frac{1}{4}}}\|\mathbf{P}_{M}w\|_{X^{0,\frac{1}{4}}}\|\mathbf{P}_{M_2}v_2\|_{X^{0,\frac{3}{8}}}
	\lesssim_{\epsilon} &N_2^{\frac{1}{2}-\frac{\alpha}{4}-s+\epsilon}\|v\|_{X^{0,\frac{3}{8}}}.
	\end{align*}
	Choosing $0<\epsilon<\delta_0-\delta$, we obtain that
	$$ \Big|\iint v\ov{v}_2 v_3\cdot\ov{w}dxdt \Big|\lesssim N_2^{-\delta_0}\|v_3\|_{L_t^4L_x^{\infty}}\|v\|_{X^{0,\frac{3}{8}}}.
	$$
	
	Alternatively, we use the bilinear Strichartz inequality as for the Case 1 and obtain that
	\begin{align*}
	\sum_{M_1,M_2,M, M_3\lesssim N_3}\Big|\iint \mathbf{P}_{M_1}v\mathbf{P}_{M_2}\ov{v}_2\mathbf{P}_{M_3}v\cdot \mathbf{P}_{M}\ov{w}dxdt\Big|
	\lesssim & N_2^{-\delta_0}N_3^{\frac{1}{2}-\frac{\alpha}{4}}\|v_3\|_{X^{0,\frac{3}{8}}}\|v\|_{X^{0,\frac{3}{8}}}.
	\end{align*}
	Therefore, we have
	$$ \Big|\iint v\ov{v}_2v_3\cdot\ov{w}dxdt\Big|
	\lesssim N_2^{-\delta_0}\|v\|_{X^{0,\frac{3}{8}}}\min\big\{\|v_3\|_{L_t^4L_x^{\infty}}, N_3^{\frac{1}{2}-\frac{\alpha}{4}}\|v_3\|_{X^{0,\frac{3}{8}}}\big\}.
	$$
	When $v_3$ is of type (G), then $\|v_3\|_{L_t^4L_x^{\infty}}<N_3^{-\frac{\alpha-1}{2}+\epsilon_2}$ which is conclusive. When $v_3$ is of type (C), we have
	\begin{align*}
	\min\big\{\|v_3\|_{L_t^4L_x^{\infty}}, N_3^{\frac{1}{2}-\frac{\alpha}{4}}\|v_3\|_{X^{0,\frac{3}{8}}}\big\}\lesssim &\min\{N_3^{-(\alpha-1)+\epsilon_2}L_3^{\frac{1}{2}-\nu}, N_3^{\frac{1}{2}-\frac{\alpha}{4}-(\alpha-1)+\epsilon_2}L_3^{-\nu}  \}\\ \lesssim &N_3^{-(\alpha-1)+s-2s\nu+\epsilon_2-(1+2\nu)\sigma} 
	\end{align*}
	which is conclusive since $\epsilon_2<(1+2\nu)\sigma$ and
	\begin{align}\label{constraint-operatorbound}
	(\alpha-1)+2s\nu>s,
	\end{align}
	thanks to the choice of numerical parameters.

	\noi
	$\bullet${\bf Case 3: $v_2,v_3$ are both of type (G) or (C)}
	
In this case, $\mathbf{P}_{M_j}v_j=0$ when $M_j\gg N_j$. Therefore, by splitting as
$$ \iint v\ov{v}_2v_3\ov{w}dtdx=\sum_{M_1,M} \iint \mathbf{P}_{M_1}v\ov{v}_2v_3\mathbf{P}_M\ov{w}dtdx,
$$
we may assume that either $M\sim M_1\gg N_2\vee N_3$ or $M,M_1\lesssim N_2\vee N_3$. By H\"older,
	\begin{align*}
&	\Big|\iint \mathbf{P}_{M_1}v\ov{v}_2v_3\cdot\mathbf{P}_M\ov{w}dxdt \Big|\leq \|\mathbf{P}_{M_1}v\|_{L_t^4L_{x}^2}\|\mathbf{P}_Mw\|_{L_t^4L_{x}^2}\|v_2\|_{L_t^4L_x^{\infty}}\|v_3\|_{L_t^4L_x^{\infty}}\\
	\lesssim &\|\mathbf{P}_{M_1}v\|_{X^{0,\frac{1}{4}}}\|\mathbf{P}_Mw\|_{X^{0,\frac{1}{4}}}\|v_2\|_{L_t^4L_x^{\infty}}\|v_3\|_{L_t^4L_x^{\infty}}.
	\end{align*} 
Therefore, we have
	\begin{align*}
	\sum_{\substack{M_1,M\\
			M_1\sim M\gg N_2\vee N_3  }} \|\mathbf{P}_{M_1}v\|_{X^{0,\frac{1}{4}}}\|\mathbf{P}_Mw\|_{X^{0,\frac{1}{8}}} \|v_2\|_{L_t^4L_x^{\infty}}
	\|v_3\|_{L_t^4L_x^{\infty}} 
	\lesssim &\|v\|_{X^{0,\frac{1}{4}}}\|v_2\|_{L_t^4L_x^{\infty}}
	\|v_3\|_{L_t^4L_x^{\infty}},
	\end{align*}
and
	\begin{align*}
	\sum_{M_1,M\lesssim N_*  } \|\mathbf{P}_{M_1}v\|_{X^{0,\frac{1}{4}}}\|\mathbf{P}_Mv\|_{X^{0,\frac{1}{8}}}\|v_2\|_{L_t^4L_x^{\infty}}
	\|v_3\|_{L_t^4L_x^{\infty}} \lesssim &\log(N_2\vee N_3)^2\|v\|_{X^{0,\frac{1}{4}}}\|v_2\|_{L_t^4L_x^{\infty}}
	\|v_3\|_{L_t^4L_x^{\infty}}.
	\end{align*}
	Alternatively, using the bilinear Strichartz, we have
	\begin{align*}
	&\sum_{M,M_1,M_2,M}\Big|\iint \mathbf{P}_{M_1}v\mathbf{P}_{M_2}\ov{v}_2 \mathbf{P}_{M_2}v_3\mathbf{P}_M\ov{w}dtdx \Big|\\
	\lesssim & \sum_{\substack{M,M_1,M_2,M_3\\
			M_2\leq N_2,M_3\leq N_3\\
			M\sim M_1\gg N_2\vee N_3 }}
			\|\mathbf{P}_{M_1}v\|_{X^{0,\frac{3}{8}}}
		\|\mathbf{P}_{M}w\|_{X^{0,\frac{3}{8}}}	
		 (M_2M_3)^{\frac{1}{2}-\frac{\alpha}{4}}\|\mathbf{P}_{M_2}v_2\|_{X^{0,\frac{3}{8} }}
	\|\mathbf{P}_{M_3}v_3\|_{X^{0,\frac{3}{8}}}\\
	+&\sum_{\substack{M,M_1,M_2,M_3\\
			M_2\leq N_2,M_3\leq N_3\\
			M,M_1\lesssim N_2\vee N_3 } }(M_2M_3)^{\frac{1}{2}-\frac{\alpha}{4}}
	\|\mathbf{P}_{M_1}v\|_{X^{0,\frac{3}{8}}}
	\|\mathbf{P}_{M_2}v_2\|_{X^{0,\frac{3}{8}}}
	\|\mathbf{P}_{M_1}v_3\|_{X^{0,\frac{3}{8}}}
	\|\mathbf{P}_{M}w\|_{X^{0,\frac{3}{8}}}\\
	\lesssim &(N_2N_3)^{\frac{1}{2}-\frac{\alpha}{4}}(\log( N_2\vee N_3))^4 \|v\|_{X^{0,\frac{3}{8}}}\|v_2\|_{X^{0,\frac{3}{8}}}\|v_3\|_{X^{0,\frac{3}{8}}}.
	\end{align*}
	Thus we have
	\begin{align}\label{..}
	&	\Big|\iint v\ov{v}_2v_3\cdot\ov{w}dxdt\Big|\notag \\ \lesssim_{\epsilon} &\|v\|_{X^{0,\frac{3}{8}}}(N_2\vee N_3)^{\epsilon}\min\big\{\|v_2\|_{L_t^4L_x^{\infty}}\|v_3\|_{L_t^4L_x^{\infty}},
	(N_2N_3)^{\frac{1}{2}-\frac{\alpha}{4}}\|v_2\|_{X^{0,\frac{3}{8}}}\|v_3\|_{X^{0,\frac{3}{8}}}  \big\}.
	\end{align}
	When $v_2,v_3$ are both of type (G), the bound $\|v_2\|_{L_t^4L_x^{\infty}}\|v_3\|_{L_t^4L_x^{\infty}}\lesssim (N_2N_3)^{-\frac{\alpha-1}{2}+\epsilon_2}$ is conclusive. When $v_2,v_3$ are both of type (C), we have the bound (choosing $0<\epsilon<\sigma-\delta_0$)
	\begin{align*}
	&\min\big\{\|v_2\|_{L_t^4L_x^{\infty}}\|v_3\|_{L_t^4L_x^{\infty}},
	(N_2N_3)^{\frac{1}{2}-\frac{\alpha}{4}}\|v_2\|_{X^{0,\frac{3}{8}}}\|v_3\|_{X^{0,\frac{3}{8}}}  \big\}\\
	\lesssim &\min\{(N_2N_3)^{-(\alpha-1)+\epsilon_2}(L_2L_3)^{\frac{1}{2}-\nu},(N_2N_3)^{\frac{1}{2}-\frac{\alpha}{4}-(\alpha-1)}(L_2L_3)^{-\nu} \}\\
	\lesssim &(N_2N_3)^{-(\alpha-1)-2s\nu+s+\epsilon_2-(1+2\nu)\sigma},
	\end{align*}
	which is conclusive since \eqref{constraint-operatorbound} holds. Finally we assume that $v_2$ is of type (G) and $v_3$ is of type (C). Using the bilinear Strichartz inequality we have
	\begin{align*}
&	\Big|\iint \mathbf{P}_{M_1}v \ov{v}_2\mathbf{P}_{M_3}v_3\cdot \mathbf{P}_M\ov{w}dxdt \Big|
	\leq \|\mathbf{P}_{M_1}v\mathbf{P}_{M_3}v_3\|_{L_{t,x}^2}\|v_2\|_{L_t^4L_x^{\infty}}\|\mathbf{P}_Mw\|_{L_t^4L_x^2}\\
	\lesssim & (M_1\wedge M_3)^{\frac{1}{2}-\frac{\alpha}{4}}\|\mathbf{P}_{M_1}v\|_{X^{0,\frac{3}{8}}}\|\mathbf{P}_{M_3}v_3\|_{X^{0,\frac{3}{8}}}\|\mathbf{P}_Mw\|_{X^{0,\frac{1}{4}}}\|v_2\|_{L_t^4L_x^{\infty}}.
	\end{align*}
	For the non-zero contributions, we must have $M_3\lesssim N_2\vee N_3$, thus when $M_1\gg N_2\vee N_3$, we must have $M_1\sim M\gg N_2\vee N_3\geq M_3$, hence
	\begin{align*}
	\sum_{\substack{M_1,M_3,M\\
			M\sim M_1\gg N_2\vee N_3\\
			M_3\leq  N_3 } }(M_1\wedge M_3)^{\frac{1}{2}-\frac{\alpha}{4}}\|\mathbf{P}_{M_1}v\|_{X^{0,\frac{3}{8}}}\|\mathbf{P}_{M_3}v_3\|_{X^{0,\frac{3}{8}}}\|\mathbf{P}_Mw\|_{X^{0,\frac{1}{4}}}
	\lesssim & \|v\|_{X^{0,\frac{3}{8}}}N_3^{\frac{1}{2}-\frac{\alpha}{4}}\|v_3\|_{X^{0,\frac{3}{8}}}.
	\end{align*}
The other contribution can be bounded by
	\begin{align*}
& \sum_{\substack{M_1,M_3,M\\
			M,M_1\lesssim N_2\vee N_3,\\M_3\leq N_3 } } (M_1\wedge M_3)^{\frac{1}{2}-\frac{\alpha}{4}}\|\mathbf{P}_{M_1}v\|_{X^{0,\frac{3}{8}}}\|\mathbf{P}_{M_3}v_3\|_{X^{0,\frac{3}{8}}}\|\mathbf{P}_Mw\|_{X^{0,\frac{1}{4}}}
\\	\lesssim & N_3^{\frac{1}{2}-\frac{\alpha}{4}}(\log (N_2\vee N_3))^2\|v\|_{X^{0,\frac{3}{8}}}\|v_3\|_{X^{0,\frac{3}{8}}}.
	\end{align*}
Combining with \eqref{..}, we obtain that
	\begin{align*}
	\Big|\iint v\ov{v}_2v_3\cdot\ov{w}dxdt\Big| \lesssim &\|v\|_{X^{0,\frac{3}{8}}}(\log N_*)^4\|v_2\|_{L_t^4L_x^{\infty}}\min\big\{\|v_3\|_{L_t^4L_x^{\infty}},
	N_3^{\frac{1}{2}-\frac{\alpha}{4}}\|v_3\|_{X^{0,\frac{3}{8}}}  \big\}
	\\
	\lesssim_{\epsilon}
	&N_2^{-\frac{\alpha-1}{2}+\epsilon_2}( N_2\vee N_3)^{\epsilon}\|v\|_{X^{0,\frac{3}{8}}}\min\{N_3^{-(\alpha-1)+\epsilon_2}L_3^{\frac{1}{2}-\nu},N_3^{\frac{1}{2}-\frac{\alpha}{4}-(\alpha-1)}L_3^{-\nu} \}\\
	\lesssim_{\epsilon} &N_2^{-\frac{\alpha-1}{2}+\epsilon_2}N_3^{-(\alpha-1)-2s\nu+s+\epsilon_2-(1+2\nu)\sigma}(N_2\vee N_3)^{\epsilon} \|v\|_{X^{0,\frac{3}{8}}},
	\end{align*}
	which is conclusive since \eqref{constraint-operatorbound} holds.
	The proof of Lemma \ref{lem:mappingXsb} is now complete.
\end{proof}

The proof of (5) and (7) of Proposition \ref{Multilinearkey} is an immediate consequence of the above lemma. Now we prove (3) of Proposition \ref{Multilinearkey}. Consider $\Pi_{N_0}^{\perp}\mathcal{I}\mathcal{N}_3(v_1,v_2,v_3)$ for $v_1,v_2,v_3$ with characterized parameters $(N_1,L_1),(N_2,L_2),(N_3,L_3)$ with $N_0\gg N_{(1)}.$ If the projection $\Pi_{N_0}^{\perp}\mathcal{I}\mathcal{N}_3(v_1,v_2,v_3)$ does not vanish, then at least one of $v_1,v_2,v_3$ is of type (D), say $v_1$. Then we decompose $v_1$ as $\sum_{M}\mathbf{P}_Mv_1$, then for $M\geq N_0(\gg N_{1})$, we have $\|\mathbf{P}_{M}v_1\|_{X^{0,b}}\lesssim M^{-s}$. Applying Lemma \ref{lem:mappingXsb} to $\mathbf{P}_{M}v_1$ and using the triangle inequality, we obtain (3) of Proposition \ref{Multilinearkey}.

\section{Low modulation reduction}

\subsection{Modulation reduction for the estimates of operator kernels}
For given $v_2,v_3$ of type (G),(C),(D), recall that the kernel $\Theta_{kk'}(t,t')$ of the operator $\mathcal{Q}_{3,N}:$
$$ w\mapsto \mathcal{Q}_{3,N}(w):=\mathcal{I}\Pi_{N}\mathcal{N}_3(w,v_2,v_3)
$$ 
is given by \eqref{kernel1} with $\Xi_{kk'}(\mu,\lambda')$ given by \eqref{interkernel1}. In order to prove Proposition \ref{KernelSb} and Proposition \ref{kernelZb}, in this section, we will reduce the estimate of the kernel bounds $\|\Theta\|_{S^{b_1,b,q}}$ and $\|\Theta\|_{Z^{b_1,b,q}}$ to the low-modulation portion which consists of multi-linear expression of discrete sums. Recall that the numerical parameters satisfy $\frac{1}{2}<b_0<b<b_1, \frac{q}{q-1}-2b_1\ll 1$. We will use the notation $L_{\mu_0*}^{r}$ to stand for $L_{|\mu_0|\lesssim N^{\alpha}}^{r}$.
\begin{proposition}\label{modulationreduction1}
Assume that $\widehat{v_2},\widehat{v_3}$ are supported on $|k_j|\lesssim N$. Define $\widetilde{w}_j(\lambda_j,k_j)=\widetilde{v}_j(\lambda_j,k_j)\langle\lambda_j\rangle^{\frac{2b_0}{r_j'}}$, for $j=2,3$ with $r_j=2$ or $r_j=q$.
 Then 
\begin{align}\label{eq:modulationerduction1}
\big\|\Theta_{kk'}(\lambda,\lambda') \big\|_{Z^{b_1,b}} \lesssim &\sup\big\{ N^{50(2b_1-1)} \Upsilon_{N}[y^0;w_2,w_3]:\|y_{kk'}^0(\lambda,\lambda')\|_{L_{\lambda,\lambda'}^2l_{k,k'}^2}\leq 1\big\}, \notag \\
 +&N^{-10}\|\widetilde{w}_2\|_{L_{\lambda}^{r_2}l_k^2}\|\widetilde{w}_3\|_{L_{\lambda}^{r_3}l_k^2} 
\end{align}
and
\begin{align}\label{eq:modulationerduction2}
\big\|\Theta_{kk'}(\lambda,\lambda') \big\|_{S^{b_1,b,q}}  \lesssim  &\sup\big\{ N^{50(2b_1-1)} \Lambda_{N}[y^0;w_2,w_3]: \|y_{kk'}^0(\lambda,\lambda')\|_{l_k^1L_{\lambda}^{q'}L_{\lambda'}^2l_{k'}^2 }\leq 1 \big\} \notag \\ +&N^{-10}\|\widetilde{w}_2\|_{L_{\lambda}^{r_2}l_k^2}\|\widetilde{w}_3\|_{L_{\lambda}^{r_3}l_k^2},
\end{align}
where
\begin{align*}
\Upsilon_{N}[y^0,w_2,w_3]:=\Big\|\sum_{\substack{k,k',k_2,k_3\\(k',k_2,k_3)\in\Gamma(k)\\
	|k|\leq N }}\!\!\!\!\!\!\!\!\! \widehat{\chi}(\mu_0-\Phi_{k',k_2,k_3})y_{kk'}^0(\lambda,\lambda') \ov{\widetilde{w}_2}(\lambda_2,k_2)\widetilde{w}_3(\lambda_3,k_3)  \Big\|_{L_{\lambda,\lambda'}^2L_{\lambda_2}^{r_2}L_{\lambda_3}^{r_3}L_{\mu_0*}^{\frac{4}{2b-1} } }
\end{align*}
and
\begin{align*}
\Lambda_{N}[y^0;w_2,w_3]:=\Big\|\!\!\!\!\!\sum_{\substack{k,k',k_2,k_3 \\(k',k_2,k_3)\in\Gamma(k)\\
		|k|\leq N }}\!\!\!\!\!\!\! \widehat{\chi}(\mu_0-\Phi_{k',k_2,k_3})y_{kk'}^0(\lambda,\lambda')
\cdot\ov{\widetilde{w}_2}(\lambda_2,k_2)\widetilde{w}_3(\lambda_3,k_3)  \Big\|_{L_{\lambda}^{q'}L_{\lambda'}^2L_{\lambda_2}^{r_2}L_{\lambda_3}^{r_3}L_{\mu_0*}^{\frac{2q}{2b-1}} }.
\end{align*}
\end{proposition}
\begin{proof}
We only prove \eqref{eq:modulationerduction1}, since \eqref{eq:modulationerduction2} follows from the similar argument (with possible changes of numerical parameters). By duality, it suffices to estimate
\begin{equation}\label{eq:pf4.1}
\begin{split}
\mathfrak{I}:=\int &\langle\lambda\rangle^{b_1}
\langle\lambda'\rangle^{-b}
\langle\lambda_2\rangle^{-\frac{2b_0}{r_2'}}
\langle\lambda_3\rangle^{-\frac{2b_0}{r_3'}}d\lambda d\lambda d\lambda_2 d\lambda_3\\ \times &\int K(\lambda,\mu)d\mu\!\!\!\!\!\! \cdot
\sum_{\substack{|k|\leq N, |k_2|,|k_3|\lesssim N \\
(k',k_2,k_3)\in\Gamma(k) }}\!\!\!\!\!\!
y_{kk'}^0(\lambda,\lambda')\widehat{\chi}(\mu-\lambda'+\lambda_2-\lambda_3-\Phi_{k',k_2,k_3})\ov{\widetilde{w}_2}(\lambda_2,k_2)\widetilde{w}_3(\lambda_3,k_3).
\end{split}
\end{equation}
From Lemma \ref{DuhamelKernel} and the triangle inequality, we have
\begin{equation}\label{eq:pf4.2}
\mathfrak{I}\lesssim_A \int \frac{1}{\langle\mu\rangle}\Big(\frac{1}{\langle\lambda-\mu\rangle^A}+\frac{1}{\langle\lambda\rangle^A} \Big)\mathcal{T}_{\mu,\lambda,\lambda',\lambda_2,\lambda_3}(w_2,w_3; y)\langle\lambda\rangle^{b_1}
\langle\lambda'\rangle^{-b}
\langle\lambda_2\rangle^{-\frac{2b_0}{r_2'}}
\langle\lambda_3\rangle^{-\frac{2b_0}{r_3'}}d\mu d\lambda d\lambda' d\lambda_2 d\lambda_3,
\end{equation}
where
\begin{equation}\label{eq:pf4.3}
\mathcal{T}_{\mu,\lambda',\lambda,\lambda_2,\lambda_3}(w_2,w_3;y):=\Big|\!\!\!\sum_{\substack{|k|\leq N, |k_2|,|k_3|\lesssim N \\
		(k',k_2,k_3)\in\Gamma(k) }}\!\!\!\!\!\!y_{kk'}^0(\lambda,\lambda')\widehat{\chi}(\mu-\lambda'+\lambda_2-\lambda_3-\Phi_{k',k_2,k_3})\ov{\widetilde{w}_2}(\lambda_2,k_2)\widetilde{w}_3(\lambda_3,k_3)\Big|.
\end{equation}
Here and in the sequel, $\mathcal{T}$ stands for $\mathcal{T}_{\mu,\lambda,\lambda',\lambda_2,\lambda_3}(v_2,v_3;y)$ when there is no risk of confusion.

\vspace{0.3cm}
\noi
$\bullet$ {\bf Contribution from the integration of $\mathcal{T}$	against $\frac{1}{\langle\mu\rangle\langle\lambda\rangle^A}$:}
\\
We split (for fixed $\lambda, \lambda',\lambda_2,\lambda_3$) the integration as $\int_{|\mu- (\lambda'-\lambda_2+\lambda_3)|\gg N^{\alpha}}$ and $\int_{|\mu-(\lambda'-\lambda_2+\lambda_3)|\lesssim N^{\alpha} }$. Note that for non-zero contributions in the summation of $k,k',k_2,k_3$, $|\Phi_{k',k_2,k_3}|\lesssim N^{\alpha}$, we have from the rapid decay of $\widehat{\chi}(\cdot)$ that
$$ |\widehat{\chi}(\mu-\lambda'+\lambda_2-\lambda_3-\Phi_{k',k_2,k_3})|\lesssim_A \langle\mu-(\lambda'-\lambda_2+\lambda_3)\rangle^{-A}
$$
if $|\mu-(\lambda'-\lambda_2+\lambda_3)|\gg N^{\alpha}$,
hence
\begin{equation}\label{eq:pf4.4}
\begin{split}
&\int_{|\mu- (\lambda'-\lambda_2+\lambda_3)|\gg N^{\alpha}} \frac{1}{\langle\mu\rangle\langle\lambda\rangle^A}\cdot\mathcal{T} d\mu\\ \lesssim & \int_{|\mu-(\lambda'-\lambda_2+\lambda_3)|\gg N^{\alpha}}\frac{N\cdot N^{\frac{1}{2}}N^{\frac{1}{2}}d\mu }{\langle\lambda\rangle^A\langle\mu\rangle\langle\mu-(\lambda'-\lambda_2+\lambda_3) \rangle^{-A} }\|\widetilde{w}_2(\lambda_2,\cdot)\|_{l_{k_2}^2}\|\widetilde{w}_3(\lambda_3,\cdot)\|_{l_{k_3}^2}\|y_{kk'}^0(\lambda,\lambda')\|_{l_{k,k'}^2}\\
\lesssim &\langle\lambda\rangle^{-A}N^{2-A\alpha}\|\widetilde{w}_2(\lambda_2,\cdot)\|_{l_{k_2}^2}\|\widetilde{w}_3(\lambda_3,\cdot)\|_{l_{k_3}^2}\|y_{kk'}^0(\lambda,\lambda')\|_{l_{k,k'}^2}.
\end{split}
\end{equation}
Choosing $A=200$, say, and multiplying by $\langle\lambda\rangle^{b_1}
\langle\lambda'\rangle^{-b}
\langle\lambda_2\rangle^{-\frac{2b_0}{r_2'}}
\langle\lambda_3\rangle^{-\frac{2b_0}{r_3'}}$ and integrating in $\lambda,\lambda',\lambda_2,\lambda_3$, this contribution for $\mathfrak{I}$ is bounded by $$N^{-100}\|\widetilde{w}_2\|_{L_{\lambda}^{r_2}l_k^2}\|\widetilde{w}_3\|_{L_{\lambda}^{r_3}l_k^2}\|y_{kk'}^0(\lambda,\lambda')\|_{L_{\lambda,\lambda'}^2l_{k,k'}^2}.$$ The other term can be estimated as
\begin{equation*}
\begin{split}
 \int_{|\mu-( \lambda'-\lambda_2+\lambda_3)|\lesssim N^{\alpha}}\frac{1}{\langle\mu\rangle\langle\lambda\rangle^A}\cdot\mathcal{T}d\mu\lesssim_{b,b_1} &\frac{1 }{\langle\lambda\rangle^A}\big\|\mathcal{T}_{\mu,\lambda',\lambda,\lambda_2,\lambda_3}\big\|_{L_{\mu:|\mu-(\lambda'-\lambda_2+\lambda_3)|\lesssim N^{\alpha}}^{\frac{4}{2b-1}} }.
\end{split}
\end{equation*}
Again, choosing $A=200$ and multiplying by $\langle\lambda\rangle^{b_1}
\langle\lambda'\rangle^{-b}
\langle\lambda_2\rangle^{-\frac{2b_0}{r_2'}}
\langle\lambda_3\rangle^{-\frac{2b_0}{r_3'}}$ and integrating in $\lambda,\lambda',\lambda_2,\lambda_3$, this contribution for $\mathfrak{I}$ is bounded by the first term of the right side of \eqref{eq:modulationerduction1}. 
\\

\noi
$\bullet$ {\bf Contribution from the integration of $\mathcal{T}$ against $\frac{1}{\langle\mu\rangle\langle\lambda-\mu\rangle^A}$:}
\\
As the previous case, we split (for fixed $\lambda, \lambda',\lambda_2,\lambda_3$) the integration in $\int_{|\mu|\gg |\lambda'-\lambda_2+\lambda_3|+N^{\alpha}}$ and $\int_{|\mu|\lesssim |\lambda'-\lambda_2+\lambda_3|+N^{\alpha} }$. Using Cauchy-Schwartz, we have
$$ \int \frac{\|y_{k,k'}^0(\lambda,\lambda')\|_{l_{k,k'}^2}\langle\lambda\rangle^{b_1} }{\langle\lambda-\mu\rangle^{A}}d\lambda \lesssim \|y_{k,k'}^{0}(\lambda,\lambda')\|_{L_{\lambda}^2l_{k,k'}^2}\langle\mu\rangle^{b_1},
$$
 then by similar manipulations as in \eqref{eq:pf4.4}, the contribution from the region $|\mu-( \lambda'-\lambda_2+\lambda)|\gg N^{\alpha}$ yields the second term of the right side of \eqref{eq:modulationerduction1} as an error. The main contribution comes from the region $|\mu- (\lambda'-\lambda_2+\lambda_3)|\lesssim N^{\alpha}$. 
  We further split the integration of $\mathfrak{I}$:
$$ \int \langle\lambda'\rangle^{-b}\langle\lambda_2\rangle^{-\frac{2b}{r_2'}}\langle\lambda_3\rangle^{-\frac{2b}{r_3'}}d\lambda'd\lambda_2 d\lambda_3\Big(\int_{\substack{|\lambda-\mu|\ll N^{\alpha} \\
|\mu- (\lambda'-\lambda_2+\lambda_3)\lesssim N^{\alpha} } }+\int_{\substack{|\lambda-\mu|\gtrsim N^{\alpha} \\
|\mu- (\lambda'-\lambda_2+\lambda_3)\lesssim N^{\alpha} } }  \Big)\frac{\langle\lambda\rangle^{b_1}}{\langle\mu\rangle\langle\lambda-\mu\rangle^A}\cdot\mathcal{T}d\mu d\lambda.
$$
Taking $A=200$ and applying H\"older for the integration in $\mu$ and using Lemma \ref{convolution}, we have
\begin{align}\label{highmodu} \int_{\substack{|\lambda-\mu|\ll N^{\alpha}\\
|\mu-( \lambda'-\lambda_2+\lambda_3)|\lesssim N^{\alpha} }}\frac{\langle\lambda\rangle^{b_1}}{\lg\lambda-\mu\rg^A\lg\mu\rg}\cdot\mathcal{T}d\mu d\lambda \lesssim \int_{|\lambda-(\lambda'-\lambda_2+\lambda_3)|\lesssim N^{\alpha}}\!\!\!\big\| \mathcal{T}_{\mu,\lambda',\lambda,\lambda_2,\lambda_3}\big\|_{L_{\mu:|\mu-(\lambda'-\lambda_2+\lambda_3)|\lesssim N^{\alpha} }^{\frac{4}{2b-1 }} }\!\!\!\!\!\!\!\!\!\!\!\!\!\!\!\!\!\!\!\!\!\!\!\!\lg\lambda\rg^{b_1-1} d\lambda.
\end{align} 
Alternatively, applying Cauchy-Schwartz for the integration in $\lambda$ and using Lemma \ref{convolution}, we have
\begin{align}\label{lowmodu}
\int_{\substack{|\lambda-\mu|\ll N^{\alpha}\\
		|\mu-( \lambda'-\lambda_2+\lambda_3)|\lesssim N^{\alpha} }}\frac{\langle\lambda\rangle^{b_1}}{\lg\lambda-\mu\rg^A\lg\mu\rg}\cdot\mathcal{T}d\mu d\lambda \lesssim \int_{|\mu-(\lambda'-\lambda_2+\lambda_3) |\lesssim N^{\alpha} } \frac{\|\mathcal{T}_{\mu,\lambda',\lambda,\lambda_2,\lambda_3}\|_{L_{\lambda}^2}}{\langle\mu\rangle^{1-b_1}}d\mu.
\end{align}
Multiplying by $\lg\lambda'\rg^{-b}\lg\lambda_2\rg^{-\frac{2b_0}{r_2'}}\lg\lambda_3\rg^{-\frac{2b_0}{r_3'}}\mathbf{1}_{|\lambda'-\lambda_2+\lambda_3|\leq N^{100}}$ to the left side of \eqref{highmodu} and integrating in $\lambda',\lambda_2,\lambda_3$, we have
\begin{align*}
&\int \lg\lambda'\rg^{-b}\lg\lambda_2\rg^{-\frac{2b_0}{r_2'}}\lg\lambda_3\rg^{-\frac{2b_0}{r_3'}}\mathbf{1}_{|\lambda'-\lambda_2+\lambda_3|\leq N^{100}}\cdot (\text{l.h.s. of } \eqref{highmodu} )d\lambda' d\lambda_2 d\lambda_3\\
\lesssim & \int_{|\lambda|\lesssim N^{100}}\big\| \mathcal{T}_{\mu,\lambda',\lambda,\lambda_2,\lambda_3}\big\|_{L_{\mu:|\mu-(\lambda'-\lambda_2+\lambda_3)|\lesssim N^{\alpha} }^{\frac{4}{2b-1 }} }\frac{1}{\lg\lambda\rg^{1-b_1}}\lg\lambda'\rg^{-b}\lg\lambda_2\rg^{-\frac{2b_0}{r_2'}}\lg\lambda_3\rg^{-\frac{2b_0}{r_3'}}d\lambda d\lambda' d\lambda_2 d\lambda_3\\
\lesssim &N^{100(b_1-\frac{1}{2})}\big\| \mathcal{T}_{\mu,\lambda',\lambda,\lambda_2,\lambda_3}\big\|_{L_{\lambda',\lambda}^2L_{\lambda_2}^{r_2}L_{\lambda_3}^{r_3}L_{\mu:|\mu-(\lambda'-\lambda_2+\lambda_3)|\lesssim N^{\alpha} }^{\frac{4}{2b-1 }} }.
\end{align*}
Multiplying by $\lg\lambda'\rg^{-b}\lg\lambda_2\rg^{-\frac{2b_0}{r_2'}}\lg\lambda_3\rg^{-\frac{2b_0}{r_3'}}\mathbf{1}_{|\lambda'-\lambda_2+\lambda_3|> N^{100}}$ to the left side of \eqref{lowmodu}, we obtain that
\begin{align*}
&\int \lg\lambda'\rg^{-b}\lg\lambda_2\rg^{-\frac{2b_0}{r_2'}}\lg\lambda_3\rg^{-\frac{2b_0}{r_3'}}\mathbf{1}_{|\lambda'-\lambda_2+\lambda_3|> N^{100}}\cdot (\text{r.h.s. of }\eqref{lowmodu})d\lambda' d\lambda_2 d\lambda_3\\
\lesssim &\int_{|\mu|\sim N^{100}}\frac{\|\mathcal{T}_{\mu,\lambda',\lambda,\lambda_2,\lambda_3}\|_{L_{\lambda}^2}}{\langle\mu\rangle^{1-b_1}}\lg\lambda'\rg^{-b}\lg\lambda_2\rg^{-\frac{2b_0}{r_2'}}\lg\lambda_3\rg^{-\frac{2b_0}{r_3'}} d\mu d\lambda' d\lambda_2 d\lambda_3\\
\leq &N^{-100(1-b_1)}\|\mathcal{T}_{\mu,\lambda',\lambda,\lambda_2,\lambda_3}\|_{L_{\lambda'}^2L_{\lambda_2}^{r_2}L_{\lambda_3}^{r_3}L_{\mu}^1L_{\lambda}^2}.
\end{align*}
By definition and Cauchy-Schwartz,
\begin{align*}
&N^{-100(1-b_1)}\|\mathcal{T}_{\mu,\lambda',\lambda,\lambda_2,\lambda_3}\|_{L_{\lambda'}^2L_{\lambda_2}^{r_2}L_{\lambda_3}^{r_3}L_{\mu}^1L_{\lambda}^2}\\ \leq &N^{-100(1-b_1)} \Big\|\sum_{|k|,|k'|,|k_2|,|k_3|\lesssim N } \|\widehat{\chi}\|_{L_{\mu}^1}\|y_{k,k'}^0(\lambda,\lambda')\|_{L_{\lambda}^2} \ov{\widehat{w}_2}(\lambda_2,k_2)\widehat{w}_3(\lambda_3,k_3) \Big\|_{L_{\lambda'}^2L_{\lambda_2}^{r_2}L_{\lambda_3}^{r_3}}\\
\lesssim & N^{2-100(1-b_1)}\|y_{k,k'}^0(\lambda,\lambda')\|_{l_{k,k'}^2L_{\lambda,\lambda'}^2}\|\widetilde{w}_2(\lambda_2,k_2)\|_{L_{\lambda_2}^{r_2}l_{k_2}^2}
\|\widetilde{w}_3(\lambda_3,k_3)\|_{L_{\lambda_3}^{r_3}l_{k_3}^2},
\end{align*}
hence it can be bounded by the second error term of the right side of \eqref{eq:modulationerduction1}.

For the integration over $|\lambda-\mu|\gtrsim N^{\alpha}$, we further split it in three parts: $|\lambda|>2|\mu|, |\lambda|< \frac{1}{2}|\mu|$ and $\frac{1}{2}|\mu|\leq|\lambda|\leq 2|\mu|$. For $|\lambda|< \frac{1}{2}|\mu|$, we have $|\mu|\sim |\lambda-\mu|\gtrsim N^{\alpha}$,thus
$$
 \int_{\substack{|\lambda-\mu|\gtrsim N^{\alpha},
	|\lambda|<\frac{1}{2}	|\mu|,\\ |\mu-(\lambda'-\lambda_2+\lambda_3)|\lesssim N^{\alpha} }}\frac{\langle\lambda\rangle^{b_1}}{\lg\lambda-\mu\rg^A\lg\mu\rg}\cdot\mathcal{T}d\lambda d\mu \lesssim \int_{|\mu|\gtrsim N^{\alpha}}\frac{\|\mathcal{T}_{\mu,\lambda',\lambda,\lambda_2,\lambda_3}\|_{L_{\lambda}^2} }{\langle\mu \rangle^{A-b_1} } d\mu.
$$ 
Multiplying by $\langle \lambda'\rangle^{-b}\langle \lambda_2\rangle^{-\frac{2b_0}{r_2'}}\langle \lambda_3\rangle^{-\frac{2b_0}{r_3'}}$ and integrating in $\lambda',\lambda_2,\lambda_3$ then using Cauchy-Schwartz and Minkowski, the above term is bounded by the second error term of the right side of \eqref{eq:modulationerduction1}. Similarly, for the case $|\lambda|>2 |\mu|$, we have $|\lambda|\sim |\lambda-\mu|\gtrsim N^{\alpha}$. This gives us an error term as the second term of the right side of \eqref{eq:modulationerduction1}, provided that $A$ is chosen large enough.  Finally, for the case $\frac{1}{2}|\mu|\leq |\lambda|\leq 2|\mu|$, using H\"older and Lemma \ref{convolution} (and we write $\lg\lambda-\mu\rg^{A}\gtrsim \lg\lambda-\mu\rg^{A/2}N^{A\alpha/2}$ for $A\gg 1$), we have 
\begin{align*}
& \int_{\substack{|\lambda-\mu|\gtrsim N^{\alpha}\\
\frac{1}{2}|\mu|\leq |\lambda|\leq 2|\mu|
  } } \frac{\lg\lambda\rg^{b_1}}{\lg\lambda-\mu\rg^{A}\lg\mu\rg }\mathcal{T}\mathbf{1}_{|\mu-(\lambda'-\lambda_2+\lambda_3)|\lesssim N^{\alpha}}d\lambda d\mu\\
\lesssim & N^{-A\alpha/2}\int_{|\mu-(\lambda'-\lambda_2+\lambda_3)\lesssim N^{\alpha} }\big\|\mathbf{1}_{|\mu|\sim |\lambda|}\mathcal{T}_{\mu,\lambda',\lambda,\lambda_2,\lambda_3} \big\|_{L_{\lambda}^{2 }} \cdot\frac{d\mu}{\langle\mu\rangle^{1-b_1}}\\
\lesssim &N^{-A\alpha/2}\big\|\langle \mu\rangle^{1-b_1}\big\|_{L_{\mu}^3}\big\|\mathbf{1}_{|\mu|\sim |\lambda|}\mathcal{T}_{\mu,\lambda',\lambda,\lambda_2,\lambda_3}\big\|_{L_{\mu}^{3/2}L_{\lambda}^2}.
\end{align*}
Multiplying by $\lg\lambda'\rg^{-b}\lg\lambda_2\rg^{-\frac{2b_0}{r_2'}}\lg\lambda_3\rg^{-\frac{2b_0}{r_3'}}$, integrating in $\lambda',\lambda_2,\lambda_3$ and using Cauchy-Schwartz, this contribution can be controlled by
\begin{align*}
&N^{-\frac{A\alpha}{2}}\Big\|\!\!\sum_{\substack{|k|,|k'|,|k_2|,|k_3|\lesssim N} }\!\!\!\!\!\!\widehat{\chi}(\mu-\lambda'+\lambda_2-\lambda_3-\Phi_{k',k_2,k_3}) y_{k,k'}^0(\lambda,\lambda')
\ov{\widetilde{w}_2}(\lambda_2,k_2)\widetilde{w}_3(\lambda_3,k_2) \Big\|_{L_{\lambda'}^2L_{\lambda_2}^{r_2}L_{\lambda_3}^{r_3}L_{\mu}^{\frac{3}{2}}L_{\lambda}^2 }\\
\leq & N^{-\frac{A\alpha}{2}}
\Big\|\sum_{|k|,|k'|,|k_2|,|k_3|\lesssim N }\|\widehat{\chi}\|_{L_{\mu}^{\frac{3}{2}}}\|y_{k,k'}^0(\lambda,\lambda')\|_{L_{\lambda}^2}\ov{\widetilde{w}_2}(\lambda_2,k_2)\widetilde{w}_3(\lambda_3,k_2)  \Big\|_{L_{\lambda'}^2L_{\lambda_2}^{r_2}L_{\lambda_3}^{r_3}}\\
\lesssim &N^{-A\alpha/2+2}\|y_{k,k'}^0(\lambda,\lambda')\|_{l_{k,k'}^2L_{\lambda,\lambda'}^2}\|\widetilde{w}_2(\lambda_2,k_2)\|_{L_{\lambda_2}^{r_2}l_{k_2}^2}
\|\widetilde{w}_3(\lambda_3,k_3)\|_{L_{\lambda_3}^{r_3}l_{k_3}^2},
\end{align*}
and it can be controlled by the second error term of the right side of \eqref{eq:modulationerduction1}). This completes the proof of Proposition \ref{modulationreduction1}.
\end{proof}

\subsection{Modulation reduction for the trilinear estimates}
Assume that $v_j\in X^{0,b_0}\cap X_{\infty,q}^{0,\frac{2b_0}{q'}}$ such that supp$(\widetilde{v}_j)\subset\{|k_j|\lesssim N_j \}$, for $j=1,2,3$. Let $\widetilde{w}_j^{(r_j)}=\langle\lambda\rangle^{\frac{2b_0}{r_j'}}\widetilde{v}_j$, and without loss of generality, we assume that $v_j=\chi_1(t)v_j$, for $j=1,2,3$, where $\chi_1\in C_c^{\infty}(\R)$. Let $\chi(t)$ be another time cut-off function such that $\chi\chi_1=\chi_1$.
\begin{proposition}\label{modulationreduction:trilinear}
Adapting to the notations above, we have for any $\epsilon>0$,
\begin{equation}\label{eq:modulationreduction3}
\begin{split}
&\|\mathcal{N}_3(v_1,v_2,v_3)\|_{X^{0,b_1-1}}\lesssim_{\epsilon} N_{(1)}^{-100}\prod_{j=1}^3\|v_j\|_{X^{0,b}}\\
+ &\Big\|\!\sum_{\substack{(k_1,k_2,k_3)\in\Gamma(k)\\
|k_j|\lesssim N_j,j=1,2,3 } }\!\widehat{\chi}(\mu_0-\Phi_{k_1,k_2,k_3})\widetilde{w}^{(r_1)}_1(\lambda_1,k_1)\ov{\widetilde{w}^{(r_2)}_2}(\lambda_2,k_2)\widetilde{w}_3^{(r_3)}(\lambda_3,k_3) \Big\|_{L_{\lambda_1}^{r_1}L_{\lambda_2}^{r_2}L_{\lambda_3}^{r_3}L_{|\mu_0|\lesssim N_{(1)}^{\alpha}}^{\frac{2}{2b_1+2\epsilon-1}} l_k^2}.
\end{split}
\end{equation}	
\end{proposition}
\begin{proof}
Since there is no significant importance of the conjugate bar on $\widetilde{v}_2$, we will omit it in the proof. By duality, for $v=\chi(t)v\in X^{0,1-b_1}$, $\|\lg\lambda\rg^{1-b_1}\widetilde{v}(\lambda,k)\|_{L_{\lambda}^2l_k^2}\lesssim 1$, we need to estimate
\begin{equation}\label{eq:pf4.5}
\begin{split}
&\sum_{\substack{|k|\lesssim N_{(1)}, |k_j|\leq N_j, j=1,2,3\\
(k_1,k_2,k_3)\in \Gamma(k) }}\int \prod_{j=1}^3\widehat{v}_j(\tau_j,k_j)d\tau_j\int \widehat{v}(\tau_4,k)\widehat{\chi}(\tau_1-\tau_2+\tau_3-\tau_4)d\tau_4\\
=&\sum_{\substack{|k|\lesssim N_{(1)}, |k_j|\leq N_j, j=1,2,3\\
		(k_1,k_2,k_3)\in \Gamma(k) }}\int\Big(\prod_{j=1}^3\langle\lambda_j\rangle^{-\frac{2b_0}{r_j'}}\widetilde{w}_j^{(r_j)}(\lambda_j,k_j)d\lambda_j\Big)\cdot \widetilde{v}(\lambda,k)\widehat{\chi}(\lambda_1-\lambda_2+\lambda_3-\lambda-\Phi_{k_1,k_2,k_3})d\lambda.
\end{split}
\end{equation}
By Cauchy-Schwartz, \eqref{eq:pf4.5} is bounded by
\begin{equation}\label{eq:pf4.6}
\begin{split}
\int \Big(\prod_{j=1}^3\langle\lambda_j\rangle^{-\frac{2b_0}{r_j'}}d\lambda_j\Big)\|\lg\lambda\rg^{1-b_1}\widetilde{v}\|_{L_{\lambda}^2l_k^2}\cdot\|\lg\lambda\rg^{-(1-b_1)}\mathcal{M}_{\lambda,\lambda_1,\lambda_2,\lambda_3}(v_1,v_2,v_3)\|_{L_{\lambda}^2l_k^2},
\end{split}
\end{equation}
where
\begin{equation*}
\mathcal{M}_{\lambda,\lambda_1,\lambda_2,\lambda_3}(v_1,v_2,v_3)\!:=
\!\!\!\!\!\!\!\sum_{\substack{|k_j|\leq N_j, j=1,2,3\\
		(k_1,k_2,k_3)\in \Gamma(k) }}\!\!\!\!\!\!\widetilde{w}^{(r_1)}_1(\lambda_1,k_1)\widetilde{w}_2^{(r_2)}(\lambda_2,k_2)\widetilde{w}^{r_3}_3(\lambda_3,k_3)\widehat{\chi}(\lambda-\lambda_1+\lambda_2-\lambda_3-\Phi_{k_1,k_2,k_3}).
\end{equation*}
To simplify the notation, we denote $\mathcal{M}_{\lambda,\lambda_1,\lambda_2,\lambda_3}(v_1,v_2,v_3)$ simply by $\mathcal{M}$, when there is no risk of confusing. By H\"older, we have
$$ \|\lg\lambda\rg^{-(1-b_1)}\mathcal{M}\|_{L_{\lambda}^2l_k^2}\leq \|\lg\lambda\rg^{-(1-b_1)}\|_{L_{\lambda}^{\small{\frac{1}{1-b_1-\epsilon}}}}\|\mathcal{M}\|_{L_{\lambda}^{\small{\frac{2}{2b_1+2\epsilon-1}}}l_k^2}\lesssim_{\epsilon}\|\mathcal{M}\|_{L_{\lambda}^{\small{\frac{2}{2b_1+2\epsilon-1}}}l_k^2}.
$$ 
For fixed $\lambda_1,\lambda_2,\lambda_3$, we split the region of integration (in $\lambda$) in $|\lambda-\lambda_1+\lambda_2-\lambda_3|\gg N_{(1)}^{\alpha}$ and $|\lambda-\lambda_1+\lambda_2-\lambda_3|\lesssim N_{(1)}^{\alpha}$. For $|\lambda-\lambda_1+\lambda_2-\lambda_3|\gg N_{(1)}^{\alpha},$ using the rapid decay of $\widehat{\chi}(\cdot)$, we have 
$$ |\widehat{\chi}(\lambda-\lambda_1+\lambda_2-\lambda_3+\Phi_{k_1,k_2,k_3})|\lesssim_A\lg\lambda-\lambda_1+\lambda_2-\lambda_3\rg^{-A}\mathbf{1}_{|\lambda-\lambda_1+\lambda_2-\lambda_3|\gg N_{(1)}^{\alpha}}.
$$
By taking $A\gg 1$, we have
$$ \|\mathcal{M}\|_{L_{\lambda}^{\small{\frac{2}{2b_1+2\epsilon-1} }}l_k^2(|\lambda-\lambda_1+\lambda_2-\lambda_3|\gg N_{(1)}^{\alpha}) }\lesssim_{\epsilon} N_{(1)}^{-100}\prod_{j=1}^3\|\widetilde{w}_j(\lambda_j,k_j)\|_{l_{k_j}^2},
$$
provided that $A\gg 1$.
Hence this contribution in \eqref{eq:pf4.5} is an error and can be bounded by the first term on the right side of \eqref{eq:modulationreduction3}. The other term can be bounded by
\begin{equation*}
\begin{split}
\|\mathcal{M}\|_{L_{\lambda}^{\small{\frac{2}{2b_1+2\epsilon-1} }}l_k^2(|\lambda-\lambda_1+\lambda_2-\lambda_3|\lesssim N_{(1)}^{\alpha})} 
\end{split}
\end{equation*} 
as desired.
This completes the proof of Proposition \ref{modulationreduction:trilinear}.
\end{proof}


\section{Multilinear estimate for the kernel $\mathcal{P}_{N,L}^{+}$}\label{section:kernel}

The goal of this section is to prove Proposition \ref{KernelSb} and Proposition \ref{kernelZb}. 
Recall that  $\Theta_{k,k_1}^{N_1,L}(t,t_1)$ is the kernel of the operator
$$ \mathcal{P}_{N_1,L}^+:=\chi_T(t)\Pi_{N_1}\mathcal{I}\big[\mathcal{N}_3\big(\Pi_{N_1}\cdot, \Pi_Lv_L^{\#},\Pi_Lv_L^{\#} \big)-\Pi_{N_1}\mathcal{N}_3\big(\Pi_{N_1}\cdot,\Pi_{\frac{L}{2}} v_{\frac{L}{2}}^{\#},\Pi_{\frac{L}{2}}v_{\frac{L}{2}}^{\#} \big) \big]
$$
where $L<N_1^{1-\delta}$. Since we will only estimate the kernel restricted to $|k|,|k_1|\geq \frac{N_1}{4}$, by abusing the notation, we will sometimes regard $\Theta_{k,k_1}^{N_1,L}(t,t_1)$ as $\Theta_{k,k_1}(t,t_1)^{N_1,L}\mathbf{1}_{|k|,|k_1|\sim N_1}$. By decomposing $v_L^{\#}$ as sums of terms $v_j$ of type (D),(G) or (C) with corresponding characterized parameters $(N_2,L_2)$, we can write $\Theta_{kk_1}^{N_1,L}(t,t_1)\mathbf{1}_{|k|,|k_1|\sim N_1}$ as the sum of kernels of the linear combination of operators
$$\Pi_{N_1}\mathcal{I}\mathcal{N}_3(\cdot,\Pi_{L}v_2,\Pi_Lv_3),
$$
where $v_2,v_3$ are of type (G),(C) or (D) with characterized parameters $(N_2,L_2), (N_3,L_3)$, with respectively, satisfying
$ N_2\vee N_3\sim L. 
$ To prove Proposition \ref{KernelSb} and Proposition \ref{kernelZb}, we will first provide estimates for each single piece and then sum them up.

\subsection{Notational simplifications}
We fix another numerical parameter (remember that $q$ is reserved in Remark \ref{numerical}):
$$ q_0:=\frac{2}{2b_1-1} \text{ or } \frac{q}{2b_1-1}
$$
in this section, which will be clear in different contexts. The importance is that $q_0\gg 1$. As in the previous section, we will write $L_{\mu_0*}^{r}$ to stand for $L_{|\mu_0|\lesssim N_1^{\alpha}}^{r}$. Let $\mathcal{C}_j$ be the $\sigma$-algebra $\mathcal{B}_{\leq L_j}$ which is independent of the $\sigma$-algebra generated by $\{g_{k_j^*}:|k_j^*|\sim N_j\}$, we denote by $\mathbb{E}^{\mathcal{C}_j}[\cdot]=\mathbb{E}[\cdot|\mathcal{C}_j].$

Before doing the estimates, we observe that modulo terms of $$O_{\epsilon}(N_1^{-\frac{1}{\epsilon}+10})\|y_{k,k_1}^0(\lambda,\lambda_1)\|_{l_{k,k_1}^2}\|\widetilde{w}_2(\lambda_2,k_2)\|_{l_{k_2}^2}\|\widetilde{w}_3(\lambda_3,k_3)\|_{l_{k_3}^2} ,$$ we may replace $\widehat{\chi}_0(\mu_0-\Phi_{k_1,k_2,k_3})$ by $\mathbf{1}_{\Phi_{k_1,k_2,k_3}=\mu_0+O(N_1^{\epsilon})}$ for any $\epsilon>0$ and we denote by
$$ S_{k_1,k_2,k_3}:=\mathbf{1}_{\Phi_{k_1,k_2,k_3}=\mu_0+O(N_1^{\epsilon})}\mathbf{1}_{k_2\neq k_1,k_3}.
$$
Note that $S_{k_1,k_2,k_3}$ depends on $\mu_0$, but we will always have uniform estimates for $S_{k_1,k_2,k_3}$ in $\mu_0$. Therefore, we will not mention this dependence explicitly.
Here $\epsilon>0$ is another free (small) parameter which will be fixed later, according to different contexts.
Furthermore, when $v_j$ is of type $(C)$, we may replace it by
$$ \widetilde{v}_j(\lambda_j,k_j)=\sum_{|k_j^*|\sim N_j}\mathbf{1}_{|k_j-k_j^*|\leq L_jN_j^{\epsilon}}\widetilde{h}_{k_jk_j^*}^{N_jL_j}(\lambda_j)\frac{g_{k_j^*}(\omega)}{[k_j^*]^{\frac{\alpha}{2}}},
$$  
since the contribution coming from $\mathbf{1}_{|k_j-k_j^*|>L_jN_j^{\epsilon}}$ is bounded by
$$ N_j^{-\kappa\epsilon}\Big\|\lg\lambda_j\rg^b\Big\lg\frac{|k_j-k_j^*|}{L_j}\Big\rg^{\kappa}\widetilde{h}_{k_jk_j^*}^{N_j,L_j}(\lambda_j)  \Big\|_{L_{\lambda_j}^2l_{k_j,k_j^*}^2}
$$
which by \eqref{almostlocal} is much smaller than the main contribution\footnote{The far-diagonal part $|k_j-k_j^*|>L_jN_j^{\epsilon}$ can be easily treated by the deterministic estimate \eqref{outputA1-1}. However we need to be cautious when $v_2$ and $v_3$ are both of type (C). In that case, we may slightly change the constraint $\mathbf{1}_{|k_j-k_j^*|\leq L_jN_j^{\epsilon}}$ by $\mathbf{1}_{|k_j-k_j^*|\leq L_j(N_2\vee N_3)^{\epsilon}} $.  } from $\mathbf{1}_{|k_j-k_j^*|\leq L_jN_j^{\epsilon}},$ provided that $\epsilon^{10}<\kappa^{-1}$. In particular, we may assume that $|k_j|\sim |k_j^*|\sim N_j$. Recall that type (C) terms satisfy \eqref{eq(C):hypothesis1}, \eqref{eq(C):hypothesis2},\eqref{eq(C):hypothesis3}. Slightly different from previous sections, we will denote by $w_j^{(r_j)}(\lambda_j,k_j)=\lg\lambda_j\rg^{\frac{2b_0}{r_j'}}\wt{v}_j(\lambda_j,k_j)$ for $r_j\in\{2,q\}$. In order to clean the exposition, we introduce the following notations:
\begin{align}\label{newnameKernel}
\Upsilon_{L_2,L_3}^{N_1,N_2,N_3}(r_2,r_3):=\Big\|\!\!\!\!\!\!\sum_{\substack{k_1,k_2,k_3,k,|k_1|\sim N_1\\
		(k_1,k_2,k_3)\in \Gamma(k)\\ 
		|k_j|\leq N_j, j=2,3 } }\!\!\!\!\!\!\!\!\!\!\!\! S_{k_1,k_2,k_3}y_{k,k_1}^0(\lambda,\lambda_1)\ov{w}^{(r_2)}_2(\lambda_2,k_2)w^{(r_3)}_3(\lambda_3,k_3)  \Big\|_{L_{\lambda,\lambda_1}^2L_{\lambda_2}^{r_2}L_{\lambda_3}^{r_3}L_{\mu_0*}^{q_0} },
\end{align}
and
\begin{align}\label{newnameKernel'}
\Xi_{L_2,L_3}^{N_1,N_2,N_3}(r_2,r_3):=\Big\|\!\!\!\!\!\!\sum_{\substack{k_1,k_2,k_3,k,|k_1|\sim N_1\\
		(k_1,k_2,k_3)\in \Gamma(k)\\ 
		|k_j|\leq N_j, j=2,3 } }\!\!\!\!\!\!\!\!\!\!\!\! S_{k_1,k_2,k_3}y_{k,k_1}^0(\lambda,\lambda_1)\ov{w}^{(r_2)}_2(\lambda_2,k_2)w^{(r_3)}_3(\lambda_3,k_3)  \Big\|_{L_{\lambda}^{q'}L_{\lambda_1}^2L_{\lambda_2}^{r_2}L_{\lambda_3}^{r_3}L_{\mu_0*}^{q_0} },
\end{align}
where $v_2,v_3$ are of type (G), (C) or (D) with characterized pairs $(N_2,L_2), (N_3,L_3)$, with respectively. When $v_j$ is of type (G) or (C), we have the freedom to choose $r_j=2$ or $r_j=q$, while if $v_j$ is of type (D) (equivalently, $L_j=2L_{N_j}$), it forces $r_j=2$. 
We will fix $N_2,N_3\lesssim N_1^{1-\delta}$ such that $N_2\vee N_3\sim L$. According to Proposition \ref{modulationreduction1}, modulo an $N^{-A}$ error, we have
\begin{align} 
&\|\Theta_{kk_1}^{N_1,L}(\lambda,\lambda')\mathbf{1}_{|k_1|\sim N_1}\|_{Z^{b,b}}\lesssim N_1^{50(2b_1-1)}\!\!\!\!\!\!\!\!\!\!\!\!\!\!\!\!\sum_{\substack{N_2,N_3<N_1^{1-\delta}\\
		L_2\leq 2L_{N_2},
		L_3\leq 2L_{N_3}\\
	N_2\vee N_3\sim L }}\!\!\!\!\!\!\!\!\!\!\!\min\big\{\Upsilon_{L_2,L_3}^{N_1,N_2,N_3}(r_2,r_3): (r_2,r_3)\in I^{N_2}_{L_2}\times I^{N_3}_{L_3}  \big\},\label{Sec8-bound0}\\
&\|\Theta_{kk_1}^{N_1,L}(\lambda,\lambda')\mathbf{1}_{|k_1|\sim N_1}\|_{S^{b,b,q}}\lesssim N_1^{50(2b_1-1)}\!\!\!\!\!\!\!\!\!\!\!\!\!\!\!\!\sum_{\substack{N_2,N_3<N_1^{1-\delta}\\
		L_2\leq 2L_{N_2},
		L_3\leq 2L_{N_3} \\
	N_2\vee N_3\sim L }}\!\!\!\!\!\!\!\!\!\!\!\min\big\{\Xi_{L_2,L_3}^{N_1,N_2,N_3}(r_2,r_3): (r_2,r_3)\in I^{N_2}_{L_2}\times I^{N_3}_{L_3}  \big\},\label{Sec8-bound1}
\end{align}
where the index set $I^{N_j}_{L_j}=\{2,q\}$
if $L_j\leq L_{N_j}$ (type (G) or (C)) and $I_{L_j}^{N_j}=\{2\}$ if $L_j=2L_{N_j}$ (type (D)).

\subsection{Algorithms and reductions}
In order to clean up the arguments and to emphasize the point, we describe several algorithms to be used that reduce the analysis to the multi-linear summation.

Several algorithms to estimate the sum of the multi-linear expression
$$ \mathcal{M}_{N_1,N_2,N_3}(y;a_2,a_3):=\Big|\sum_{\substack{k_1,k_2,k_3,k,|k_1|\sim N_1\\
(k_1,k_2,k_3)\in \Gamma(k)\\
|k_j|\leq N_j,j=2,3 } } y_{k,k_1}S_{k_1,k_2,k_3}a_2(k_2)a_3(k_3)\Big| 
$$
are at our proposal:\\

\noi
$\bullet${\bf Algorithm A1: Deterministic estimates}

We may assume that $N_3\leq N_2$. Using Cauchy-Schwartz, we have
\begin{align}\label{A1-1}
 \mathcal{M}_{N_1,N_2,N_3}(y;a_2,a_3)
 \leq &\|y_{k_1-k_2+k_3,k_1}a_3(k_3)\|_{l_{k_1,k_2,k_3}^2}\|a_2(k_2)S_{k_1,k_2,k_3}\|_{l_{k_1,k_2,k_3}^2}\notag \\
 \leq &\|y_{k,k_1}\|_{l_{k,k_1}^2}\|a_3\|_{l_{k_3}^2}\Big\||a_2(k_2)|^2\sum_{|k_1|\sim N_1}S_{k_1,k_2,k_3} \Big\|_{l_{k_2,k_3}^1}^{\frac{1}{2}} \notag \\
 \lesssim &N_1^{\epsilon}\big(N_1^{1-\frac{\alpha}{2}}+N_3^{\frac{1}{2}}\big)\|y_{k,k_1}\|_{l_{k,k_1}^2}
 \|a_{2}\|_{l_{k_2}^2}
 \|a_{3}\|_{l_{k_3}^2},
\end{align}
where we use the counting estimate
$$ \sum_{k_1}S_{k_1,k_2,k_3}\lesssim N_1^{\epsilon}\Big( 1+\frac{N_1^{2-\alpha}}{\lg k_2-k_3\rg}\Big)
$$
for fixed $k_2,k_3$ such that $ |k_2|\leq N_2\ll N_1, |k_3|\leq N_3\ll N_1$. Alternatively, we have
\begin{align*}
\mathcal{M}_{N_1,N_2,N_3}(y;a_2,a_3)\leq &\|y_{k_1-k_2+k_3,k_1}a_3(k_3)\|_{l_{k_1,k_2,k_3}^2}\|a_2(k_2)S_{k_1,k_2,k_3}\|_{l_{k_1,k_2,k_3}^2}\\
\leq &\|y_{k,k_1}\|_{l_{k,k_1}^2}\|a_3\|_{l_{k_3}^2}\|a_2\|_{l_{k_2}^{q}}\|S_{k_1,k_2,k_3}\|_{l_{k_3}^2l_{k_2}^{\frac{2q}{q-2}}l_{k_1}^2}\\
\lesssim &N_1^{\epsilon}N_3^{\frac{1}{2}}(N_2^{\frac{1}{2}-\frac{1}{q}}+N_1^{1-\frac{\alpha}{2}} )\|y_{k,k_1}\|_{l_{k,k_1}^2}\|a_3\|_{l_{k_3}^2}\|a_2\|_{l_{k_2}^{q}}.
\end{align*}
Choosing in particular $\epsilon=\kappa^{-0.01}$, replacing $y_{kk_1}, a_2, a_3$ by $y_{kk_1}^0(\lambda,\lambda_1), w_2^{(2)}(\lambda_2,k_2), w_3^{(2)}(\lambda_3,k_3)$ or  $y_{kk_1}^0(\lambda,\lambda_1), w_2^{(q)}(\lambda_2,k_2),$ $w_3^{(2)}(\lambda_3,k_3)$ with respectively and then taking $L_{\lambda,\lambda_1}^2L_{\lambda_2}^{2}L_{\lambda_3}^{2}L_{\mu_0*}^{q_0}$ or $L_{\lambda,\lambda_1}^2L_{\lambda_3}^{2}L_{\lambda_2}^{q}L_{\mu_0*}^{q_0}$, we obtain (note that $\alpha q_0^{-1}<\kappa^{-0.09}$):
\begin{lemme}\label{lemA1-123}
Assume that $N_1\gg N_2, N_3$, then we have
\begin{align}\label{outputA1-1}
\Upsilon_{L_2,L_3}^{N_1,N_2,N_3}(2,2)\lesssim N_1^{\kappa^{-0.1}}(N_1^{1-\frac{\alpha}{2}}+(N_2\wedge N_3)^{\frac{1}{2}})\|y_{k,k_1}^0(\lambda,\lambda_1)\|_{L_{\lambda,\lambda_1}^2l_{k,k_1}^2}\|v_2\|_{X^{0,b_0}}\|v_3\|_{X^{0,b_0}}.
\end{align}
Moreover, if $N_3\leq N_2$, we have alternatively
\begin{align}\label{outputA1-2}
\Upsilon_{L_2,L_3}^{N_1,N_2,N_3}(q,2) \lesssim N_1^{\kappa^{-0.1}}N_3^{\frac{1}{2}}(N_2^{\frac{1}{2}}+N_2^{\frac{1}{q}}N_1^{1-\frac{\alpha}{2}})\|y_{kk_1}^0(\lambda,\lambda_1)\|_{L_{\lambda,\lambda_1}^2l_{k,k_1}^2}\|v_2\|_{X_{\infty,q}^{0,\frac{2b_0}{q'}}}\|v_3\|_{X^{0,b_0}},
\end{align}
if $N_2\leq N_3$, 
\begin{align}\label{outputA1-3}
\Upsilon_{L_2,L_3}^{N_1,N_2,N_3}(2,q) \lesssim N_1^{\kappa^{-0.1}}N_2^{\frac{1}{2}}(N_3^{\frac{1}{2}}+N_3^{\frac{1}{q}}N_1^{1-\frac{\alpha}{2}})\|y_{kk_1}^0(\lambda,\lambda_1)\|_{L_{\lambda,\lambda_1}^2l_{k,k_1}^2}\|v_2\|_{X^{0,b_0}}\|v_3\|_{X_{\infty,q}^{0,\frac{2b_0}{q'}}}.
\end{align}
\end{lemme}
\begin{remarque}
We will use \eqref{outputA1-1} when $v_2,v_3$ are of type (D) or type (C) with large $L_j$. The estimate \eqref{outputA1-2} or \eqref{outputA1-3} is useful when $v_2,v_3$ are all of type (G) or (C) and the $L_j$ parameter of one type (C) term is greater than the $N_j$ parameter of the other. Note that this is the case where we are no not able to exploit Wiener chaos.  
\end{remarque}

Similarly, we have
\begin{align}\label{A1-2}
\mathcal{M}_{N_1,N_2,N_3}(y;a_2,a_3)\leq & \|y_{k,k+k_2-k_3}a_3(k_3)\|_{l_k^1l_{k_2,k_3}^2}\|a_2(k_2)S_{k+k_2-k_3,k_2,k_3}\|_{l_k^{\infty}l_{k_2,k_3}^2}\notag \\
\lesssim &\|y_{k,k_1}\|_{l_k^1l_{k_1}^2}\|a_2\|_{l_{k_2}^2}\|a_3\|_{l_{k_3}^2},
\end{align}
where we use the counting estimate
$$\sup_{|k_2|\leq N_2}\sum_{|k_3|\leq N_3}S_{k+k_2-k_3,k_2,k_3}\lesssim 1,
$$
for fixed $k$, since $|k|\sim N_1\gg N_2,N_3$. Alternatively, when $N_3\leq N_2$, we have
\begin{align*}
\mathcal{M}_{N_1,N_2,N_3}(y;a_2,a_3)\leq &\|y_{k,k+k_2-k_3}a_3(k_3)\|_{l_k^1l_{k_2,k_3}^2}\|a_2(k_2)S_{k+k_2-k_3,k_2,k_3}\|_{l_k^{\infty}l_{k_2,k_3}^2}\\
\leq &\|y_{kk_1}\|_{l_k^1l_{k_1}^2}\|a_3(k_3)\|_{l_{k_3}^2}\|a_2\|_{l_{k_2}^{q}}\|S_{k+k_2-k_3,k_2,k_3}\|_{l_{k_3}^2l_{k_2}^{\frac{2q}{q-2}} }\\
\lesssim &N_3^{\frac{1}{2}}\|y_{kk_1}\|_{l_k^1l_{k_1}^2}\|a_3(k_3)\|_{l_{k_3}^2}\|a_2\|_{l_{k_2}^{q}}.
\end{align*}
Substituting the bounds \eqref{A1-1}, \eqref{A1-2} with 
$y_{k,k_1}=y^0_{k,k_1}(\lambda,\lambda_1)$, $a_2(k_2)=w_2^{(r_2)}(\lambda_2,k_2),$ and $ a_3(k_3)=w_3^{(r_3)}(\lambda_3,k_3)$ and then taking $L_{\lambda}^{q'}L_{\lambda_1}^2L_{\lambda_2}^{2}L_{\lambda_3}^{2}L_{\mu_0*}^{q_0} $ norm or
$L_{\lambda}^{q'}L_{\lambda_1}^2L_{\lambda_2}^{q}L_{\lambda_3}^{2}L_{\mu_0*}^{q_0}$ (after switching the order of $L^p$ spaces by Minkowski if necessary), we have proved:
\begin{lemme}\label{lemA1-456}
	Assume that $N_1\gg N_2,N_3$, then we have
\begin{align}\label{outputA1-4}
\Xi_{L_2,L_3}^{N_1,N_2,N_3}(2,2)\lesssim \|y_{k,k_1}^0(\lambda,\lambda_1)\|_{l_k^1L_{\lambda}^{q'}L_{\lambda_1}^2l_{k_1}^2}\|v_2\|_{X^{0,b_0}}\|v_3\|_{X^{0,b_0}}.
\end{align}
Alternatively, when $N_3\leq N_2$ we have
\begin{align}\label{outputA1-5}
\Xi_{L_2,L_3}^{N_1,N_2,N_3}(q,2)\lesssim N_1^{\kappa^{-0.1}}N_2^{\frac{1}{q}}N_3^{\frac{1}{2}}\|y_{k,k_1}^0(\lambda,\lambda_1)\|_{l_k^1L_{\lambda}^{q'}L_{\lambda_1}^2l_{k_1}^2}\|v_2\|_{X_{\infty,q}^{0,\frac{2b_0}{q'}}}\|v_3\|_{X^{0,b_0}},
\end{align}
and when $N_2\leq N_3$, we have
\begin{align}\label{outputA1-6}
\Xi_{L_2,L_3}^{N_1,N_2,N_3}(2,q)\lesssim N_1^{\kappa^{-0.1}}N_3^{\frac{1}{q}}N_2^{\frac{1}{2}}\|y_{k,k_1}^0(\lambda,\lambda_1)\|_{l_k^1L_{\lambda}^{q'}L_{\lambda_1}^2l_{k_1}^2}\|v_2\|_{X^{0,b_0}}\|v_3\|_{X_{\infty,q}^{0,\frac{2b_0}{q'}}}.
\end{align}
\end{lemme}
\noi
$\bullet${\bf Algorithm A2: One random input }

Recall that $w_j^{(r_j)}(\lambda_j,k_j)=\lg\lambda_j\rg^{\frac{2b_0}{r_j'}}\wt{v}_j(\lambda_j,k_j)$.
According to $v_2$ is of type (C),(G) or $v_3$ is of type (C),(G), for fixed $\lambda,\lambda_1,\lambda_2,\lambda_3,\mu_0$, we have 
\begin{lemme}\label{sum21random}
	We have
	\begin{align}
	&\mathcal{M}_{N_1,N_2,N_3}(y;w_2^{(2)},w_3^{(q)})\leq \Big[\sup_{k_2}|\sigma_{k_2k_2}^{(3)}|+\Big(\sum_{k_2,k_2':k_2\neq k_2' }|\sigma_{k_2k_2'}^{(3)}|^2\Big)^{\frac{1}{2}} \Big]^{\frac{1}{2}}\|y^0_{k,k_1}(\lambda,\lambda_1)\|_{l_{k_1,k}^2}\|w_2(\lambda_2,k_2)\|_{l_{k_2}^2},\label{eq:sum21random1}\\
	&\mathcal{M}_{N_1,N_2,N_3}(y;w_2^{(q)},w_3^{(2)})\leq \Big[\sup_{k_3}|\sigma_{k_3k_3}^{(2)}|+\Big(\sum_{k_3,k_3':k_3\neq k_3'}|\sigma_{k_3k_3'}^{(2)}|^2\Big)^{\frac{1}{2}} \Big]^{\frac{1}{2}}\|y^0_{k,k_1}(\lambda,\lambda_1)\|_{l_{k_1,k}^2}\|w_3(\lambda_3,k_3)\|_{l^2}\label{eq:sum21random2},
	\end{align}
	where 
	\begin{align}\label{lambda23}
	& \sigma_{k_2k_2'}^{(3)}=\sum_{\substack{k_1,k	
			:k_1\neq k,\\|k_1|\sim |k|\sim N_1 } }w_3^{(q)}(\lambda_3,k+k_2-k_1)\ov{w_3^{(q)}}(\lambda_3,k+k_2'-k_1)S_{k_1,k_2,k+k_2-k_1}S_{k_1,k_2',k+k_2'-k_1},\\
	&\sigma_{k_3k_3'}^{(2)}=\sum_{\substack{k_1,k	
			:k_1\neq k,\\|k_1|\sim |k|\sim N_1 } }\ov{w_2^{(q)}}(\lambda_2,k_1+k_3-k)w_2^{(q)}(\lambda_2,k_1+k_3'-k)S_{k_1,k_1+k_3-k,k_3}S_{k_1,k_1+k_3'-k,k_3'},
	\end{align}
	Note that the matrix elements $\sigma_{k_jk_j'}^{(j)}$ are functions of $\lambda_j,\mu_0$. 
\end{lemme}

\begin{proof}
	We only prove \eqref{eq:sum21random1} since the proof for \eqref{eq:sum21random2} is similar. Consider the operator (depending on $\lambda_2,\mu_0$):
	$$ \mathcal{G}_2: l_{k,k_1}^2\rightarrow l_{k_2}^2: y_{k,k_1}\mapsto \sum_{k_1,k}y_{k,k_1}\widetilde{w}_3(\lambda_3,k+k_2-k_1)S_{k_1,k_2,k+k_2-k_1},
	$$
	where in the summation, we do not display the constraints $k_1\neq k, |k_1|\sim |k|\sim N$. By Cauchy-Schwartz, 
	\begin{align*}
 \mathcal{M}_{N_1,N_2,N_3}(y;w_2^{(2)}, w_3^{(q)})=&\Big|\sum_{k_2}\mathcal{G}_2(y^0)(k_2)\ov{\widetilde{w}_2}(k_2)\Big|=|(\mathcal{G}_2(y^0),\widetilde{w}_2)_{l^2}|\\ \leq &\|\mathcal{G}_2\|_{l_{k_1,k}^2\rightarrow l_{k_2}^2}\|y^0_{k_1,k}\|_{l_{k_1,k}^2}\|\widetilde{w}_2\|_{l_{k_2}^2}.
	\end{align*}
	Note that $\|\mathcal{G}_2\|_{l_{k_1,k}^2\rightarrow l_{k_2}^2}^2=\|\mathcal{G}_2\mathcal{G}_2^*\|_{l_{k_2}^2\rightarrow l_{k_2}^2}$, and one verifies that matrix element of $\mathcal{G}_2\mathcal{G}_2^*$ is exactly
	$ \sigma_{k_2k_2'}^{(2)}.
	$ Hence by Lemma \ref{matrix-bound}, we obtain \eqref{eq:sum21random1}. The proof of Lemma \ref{sum21random} is complete.
\end{proof}
Denote by $h_{k_jk_j^*}^{N_jL_j(r_j)}(\lambda_j,k_j)=\lg\lambda_j\rg^{\frac{2b_0}{r_j'}}\wt{h}_{k_jk_j^*}^{N_jL_j}(\lambda_j,k_j)$, and when there is no risk of confusing, we will denote simply by $h_{k_jk_j^*}^{(r_j)}(\lambda_j,k_j)$. When $v_3$ is of type (G) or (C), we have
\begin{align*} 
&\Upsilon_{L_2,L_3}^{N_1,N_2,N_3}(2,q)\\ \leq&\big(\|\sigma_{k_2k_2}^{(3)}(\lambda_3,\mu_0)\|_{L_{\lambda_3}^{\frac{q}{2}}L_{\mu_0*}^{\frac{q_0}{2}}l_{k_2}^{q_0}}^{\frac{1}{2}}+\|\sigma_{k_2k_2'}^{(3)}(\lambda_3,\mu_0)\mathbf{1}_{k_2\neq k_2'}\|_{L_{\lambda_3}^{\frac{q}{2}}L_{\mu_0*}^{\frac{q_0}{2}}l_{k_2,k_2'}^2 }^{\frac{1}{2}} \big)\|y_{kk_1}^0(\lambda,\lambda_1)\|_{L_{\lambda,\lambda_1}^2l_{k,k_1}^2}\|v_2\|_{X^{0,b_0}}
\end{align*}
and when $v_2$ is of type (G) or (C), we have
\begin{align*} 
&\Upsilon_{L_2,L_3}^{N_1,N_2,N_3}(q,2)\\ \leq&\big(\|\sigma_{k_3k_3}^{(2)}(\lambda_2,\mu_0)\|_{L_{\lambda_2}^{\frac{q}{2}}L_{\mu_0*}^{\frac{q_0}{2}}l_{k_3}^{q_0}}^{\frac{1}{2}}+\|\sigma_{k_3k_3'}^{(2)}(\lambda_2,\mu_0)\mathbf{1}_{k_3\neq k_3'}\|_{L_{\lambda_2}^{\frac{q}{2}}L_{\mu_0*}^{\frac{q_0}{2}}l_{k_3,k_3'}^2 }^{\frac{1}{2}} \big)\|y_{kk_1}^0(\lambda,\lambda_1)\|_{L_{\lambda,\lambda_1}^2l_{k,k_1}^2}\|v_3\|_{X^{0,b_0}}.
\end{align*}
Note that
$$ \sigma_{k_2k_2'}^{(3)}=\sum_{\substack{k_1,k:k_1\neq k\\
|k_1|\sim |k|\sim N_1 } }S_{k_1,k_2,k+k_2-k_1}S_{k_1,k_2',k+k_2'-k_1}\!\!\!\!\!\sum_{\substack{|k_3^*|,|k_3'^*|\sim N_3\\ |k_3-k_3^*|\leq L_3N_3^{\epsilon}\\
|k_3'-k_3'^*|\leq L_3N_3^{\epsilon}
} }h_{k+k_2-k_1,k_3^*}^{(q)}(\lambda_3)\ov{h}_{k+k_2'-k_1,k_3'^*}^{(q)}(\lambda_3)
\frac{g_{k_3^*}\ov{g}_{k_3'^*} }{[k_3^*]^{\frac{\alpha}{2}}[k_3^*]^{\frac{\alpha}{2}}  },
$$
$$ \sigma_{k_3k_3'}^{(2)}=\sum_{\substack{k_1,k:k_1\neq k\\
		|k_1|\sim |k|\sim N_1 } }S_{k_1,k_1+k_3-k,k_3}S_{k_1,k_1+k_3'-k,k_3'}\!\!\!\!\!\sum_{\substack{|k_2^*|,|k_2'^*|\sim N_2\\ |k_2-k_2^*|\leq L_2N_2^{\epsilon}\\
		|k_2'-k_2'^*|\leq L_2N_2^{\epsilon}
} }h_{k_1+k_3-k,k_3^*}^{(q)}(\lambda_3)\ov{h}_{k_1+k_3'-k,k_3'^*}^{(q)}(\lambda_3)
\frac{g_{k_2^*}\ov{g}_{k_2'^*} }{[k_2^*]^{\frac{\alpha}{2}}[k_2^*]^{\frac{\alpha}{2}}  }.
$$
If we are able to establish the estimates:
\begin{align}\label{inputA2-1}
 \sup_{|k_2|\sim N_2}\Big(\mathbb{E}^{\mathcal{C}_3}[|\sigma_{k_2k_2}^{(3)}|^2]\Big)^{\frac{1}{2}}+\Big(\sum_{k_2\neq k_2'}\mathbb{E}^{\mathcal{C}_3}[|\sigma_{k_2k_2'}^{(3)}|^2]\Big)^{\frac{1}{2}}\leq \Lambda_3(N_1,N_2,N_3,L_2,L_3)\|h_{k_3k_3^*}^{(q)}(\lambda_3)\|_{l_{k_3}^{\infty}l_{k_3^*}^2}^2,
\end{align}
and
\begin{align}\label{inputA2-2}
\sup_{|k_3|\sim N_3}\Big(\mathbb{E}^{\mathcal{C}_2}[|\sigma_{k_3k_3}^{(2)}|^2]\Big)^{\frac{1}{2}}+\Big(\sum_{k_3\neq k_3'}\mathbb{E}^{\mathcal{C}_2}[|\sigma_{k_3k_3'}^{(2)}|^2]\Big)^{\frac{1}{2}}\leq \Lambda_2(N_1,N_2,N_3,L_2,L_3)\|h_{k_2k_2^*}^{(q)}(\lambda_2)\|_{l_{k_2}^{\infty}l_{k_2^*}^2}^2,
\end{align}
then from the estimate
$$ \|h_{k_jk_j^*}^{(q)}(\lambda_j)\|_{L_{\lambda_j}^ql_{k_j}^{\infty}l_{k_j^*}^2}\leq \|h_{k_jk_j^*}^{(q)}(\lambda_j)\|_{L_{\lambda_j}^ql_{k_j}^{q}l_{k_j^*}^2}\lesssim N_j^{\frac{1}{q}}
\|h_{k_jk_j^*}^{(q)}(\lambda_j)\|_{l_{k_j}^{\infty}L_{\lambda_j}^ql_{k_j^*}^2}
$$ 
and the large deviation property (Corollary \ref{largedeviation}), we deduce that
outside a set of probability $<\mathrm{e}^{-N_1^{\theta}R}$:
\begin{align}\label{outputA2-1}
\Upsilon_{L_2,L_3}^{N_1,N_2,N_3}(2,q)\lesssim RN_1^{\frac{\alpha}{q_0}+\theta}N_3^{\frac{1}{q}}\Lambda_3(N_1,N_2,N_3,L_2,L_3)^{\frac{1}{2}}\|y_{kk_1}^0(\lambda,\lambda_1)\|_{L_{\lambda,\lambda_1}^2l_{k,k_1}^2}\|v_2\|_{X^{0,b_0}}\|h^{N_3L_3}\|_{S^{b_0,q}},
\end{align}
and
\begin{align}\label{outputA2-2}
\Upsilon_{L_2,L_3}^{N_1,N_2,N_3}(q,2)\lesssim N_1^{\frac{\alpha}{q_0}+\theta}N_2^{\frac{1}{q}}\Lambda_2(N_1,N_2,N_3,L_2,L_3)^{\frac{1}{2}}\|y_{kk_1}^0(\lambda,\lambda_1)\|_{L_{\lambda,\lambda_1}^2l_{k,k_1}^2}\|v_3\|_{X^{0,b_0}}\|h^{N_2L_2}\|_{S^{b_0,q}}.
\end{align}
The algorithm to estimate $\Xi_{L_2,L_3}^{N_1,N_2,N_3}(2,q)$ and $\Xi_{L_2,L_3}^{N_1,N_2,N_3}(2,q)$ is similar. Denote by
$$ \sigma_{k,k_1;k_2}^{(3)}=\sum_{\substack{k_3^*:|k_3^*|\sim N_3\\
|k_3-k_3^*|\leq L_3N_3^{\epsilon} } }
h_{k+k_2-k_1,k_3^*}^{(q)}(\lambda_3)\frac{g_{k_3^*}}{[k_3^*]^{\frac{\alpha}{2}}}S_{k_1,k_2,k+k_2-k_1}   
$$
and
$$ \sigma_{k,k_1;k_3}^{(2)}=\sum_{\substack{k_2^*:|k_2^*|\sim N_2\\
		|k_2-k_2^*|\leq L_2N_2^{\epsilon} } }
h_{k_1+k_3-k,k_2^*}^{(q)}(\lambda_2)\frac{g_{k_2^*}}{[k_2^*]^{\frac{\alpha}{2}}}S_{k_1,k_1+k_3-k,k_3}.
$$
Then by Lemma \ref{matrix-bound2} and Corollary \ref{largedeviation}, if we are able to establish the estimates:
\begin{align}\label{inputA2-3}
&\sup_{k,k_1}\sum_{k_2}\mathbb{E}^{\mathcal{C}_3}[|\sigma_{k,k_1;k_2}^{(3)}|^2]+\!\!\!\!\!\sup_{k,k':|k|,|k'|\sim N_1}\Big(\!\!\!\!\sum_{\substack{k_1,k_1'\\
(k,k_1)\neq (k',k_1') } }\!\!\!\!\!\! \mathbb{E}^{\mathcal{C}_3}\Big[\Big|\sum_{k_2}\sigma_{k',k_1';k_2}^{(3)}\ov{\sigma}_{k,k_1;k_2}^{(3)}
\Big|^2\Big]
 \Big)^{\frac{1}{2}} \notag \\
 \leq & \wt{\Lambda}_3(N_1,N_2,N_3,L_2,L_3)\|h_{k_3k_3^*}^{(q)}(\lambda_3)\|_{l_{k_3}^{\infty}l_{k_3^*}^2}^2,
\end{align}
and
\begin{align}\label{inputA2-4}
&\sup_{k,k_1}\sum_{k_3}\mathbb{E}^{\mathcal{C}_2}[|\sigma_{k,k_1;k_3}^{(2)}|^2]+\sup_{k,k':|k|,|k'|\sim N_1}\Big(\!\!\!\!\sum_{\substack{k_1,k_1'\\
		(k,k_1)\neq (k',k_1') } } \!\!\!\!\!\!\mathbb{E}^{\mathcal{C}_2}\Big[\Big|\sum_{k_3}\sigma_{k',k_1';k_3}^{(2)}\ov{\sigma}_{k,k_1;k_3}^{(2)}
\Big|^2\Big]
\Big)^{\frac{1}{2}} \notag \\
\leq & \wt{\Lambda}_2(N_1,N_2,N_3,L_2,L_3)\|h_{k_2k_2^*}^{(q)}(\lambda_2)\|_{l_{k_2}^{\infty}l_{k_2^*}^2}^2,
\end{align}
then outside a set of probability $<\mathrm{e}^{-N_1^{\theta}R}$, we have
\begin{align}\label{outputA2-3}
\Xi_{L_2,L_3}^{N_1,N_2,N_3}(2,q)\lesssim RN_1^{\frac{\alpha}{q_0}+\theta}N_3^{\frac{1}{q}}\widetilde{\Lambda}_3(N_1,N_2,N_3,L_2,L_3)^{\frac{1}{2}}\|y_{kk_1}^0(\lambda,\lambda_1)\|_{l_k^1L_{\lambda}^{q'}L_{\lambda_1}^2l_{k_1}^2 }\|v_2\|_{X^{0,b_0}}\|h^{N_3L_3}\|_{S^{b_0,q}}
\end{align}
and
\begin{align}\label{outputA2-4}
\Xi_{L_2,L_3}^{N_1,N_2,N_3}(q,2)\lesssim RN_1^{\frac{\alpha}{q_0}+\theta}N_2^{\frac{1}{q}}\widetilde{\Lambda}_2(N_1,N_2,N_3,L_2,L_3)^{\frac{1}{2}}\|y_{kk_1}^0(\lambda,\lambda_1)\|_{l_k^1L_{\lambda}^{q'}L_{\lambda_1}^2l_{k_1}^2 }\|v_3\|_{X^{0,b_0}}\|h^{N_2L_2}\|_{S^{b_0,q}}.
\end{align}

\noi
$\bullet${\bf Algorithm A3: Two random inputs}

We only use this algorithm to deal with the case where $v_2$ and $v_3$ are both of type (G) or (C) with characterized parameters $(N_2,L_2)$ and $(N_3,L_3)$ satisfying $(L_2\vee L_3)\ll (N_2\wedge N_3)$. Assume without loss of generality that $N_2\leq N_3$, we have 
\begin{align*}
|\Upsilon_{L_2,L_3}^{N_1,N_2,N_3}(2,q)|\leq \|y_{k,k_1}^0(\lambda,\lambda_1)\|_{l_{k,k_1}^2}\|\mathbf{1}_{k_1\neq k}G_{k_1,k}\|_{l_{k,k_1}^2},
\end{align*}
and
\begin{align*}
|\Xi_{L_2,L_3}^{N_1,N_2,N_3}(2,q)|\leq \|y_{k,k_1}^0(\lambda,\lambda_1)\|_{l_k^1l_{k_1}^2}\|\mathbf{1}_{k_1\neq k}G_{k_1,k}\|_{l_k^{\infty}l_{k_1}^2},
\end{align*}
where
\begin{align}\label{Gkk1}
 G_{k,k_1}(\lambda_2,\lambda_3,\mu_0):=\sum_{k_3}S_{k_1,k_1-k+k_3,k_3}\sum_{k_2^*,k_3^*}h_{k_1-k+k_3,k_2^*}^{(q)}(\lambda_2)h_{k_3k_3^*}^{(2)}(\lambda_3)\frac{\ov{g}_{k_2^*}g_{k_3^*} }{[k_2^*]^{\frac{\alpha}{2}}[k_3^*]^{\frac{\alpha}{2}}}.
\end{align}
Note that $h_{k_2k_2^*}^{(q)}$ and $ h_{k_3k_3^*}^{(2)}$ are $\mathcal{C}_{L_2\vee L_3}$ measurable, independent of $\{g_{k_2^*}(\omega):|k_2^*|\sim N_2 \}$ and $\{g_{k_3^*}(\omega):|k_3^*|\sim N_3 \}$. If we are able to show that
\begin{align}\label{inputA3-1}
\sum_{k_1,k:k_1\neq k}\mathbb{E}^{\mathcal{C}}[|G_{k_1,k}|^2]\leq \Lambda_{2,3}(N_1,N_2,N_3,L_2,L_3)\|h_{k_3k_3^*}^{(q)}\|_{l_{k_3}^{\infty}l_{k_3^*}^2}^2\|h_{k_2k_2^*}^{(2)}\|_{l_{k_2k_2^*}^2}^2
\end{align}
and
\begin{align}\label{inputA3-2}
\sup_{|k|\sim N_1}\sum_{k_1}\mathbb{E}^{\mathcal{C}}[|G_{k_1,k}|^2]\leq \wt{\Lambda}_{2,3}(N_1,N_2,N_3,L_2,L_3)\|h_{k_3k_3^*}^{(q)}\|_{l_{k_3}^{\infty}l_{k_3^*}^2}^2\|h_{k_2k_2^*}^{(2)}\|_{l_{k_2k_2^*}^2}^2
\end{align}
where $\mathcal{C}=\mathcal{C}_{L_2\vee L_3}$,
then from Corollary \ref{largedeviation}, outside a set of probability $<\mathrm{e}^{-N_1^{\theta}R}$, we have
\begin{align}\label{outputA3-1}
|\Upsilon_{L_2,L_3}^{N_1,N_2,N_3}(2,q)|\lesssim RN_1^{\frac{\alpha}{q_0}+\theta}N_3^{\frac{1}{q}}\Lambda_{2,3}(N_1,N_2,N_3,L_2,L_3)^{\frac{1}{2}}\|h^{N_3L_3}\|_{S^{b_0,q}}\|h^{N_2L_2}\|_{Z^{b_0}}
\end{align}
and
\begin{align}\label{outputA3-2}
|\Xi_{L_2,L_3}^{N_1,N_2,N_3}(2,q)|\lesssim RN_1^{\frac{\alpha}{q_0}+\frac{1}{q}+\theta}N_3^{\frac{1}{q}}\widetilde{\Lambda}_{2,3}(N_1,N_2,N_3,L_2,L_3)^{\frac{1}{2}}\|h^{N_3L_3}\|_{S^{b_0,q}}\|h^{N_2L_2}\|_{Z^{b_0}}.
\end{align}

We remark that when $v_2$ is of type (G) or $v_2,v_3$ are both of type (G) (note that $N_2\leq N_3$), we should instead estimate $\Upsilon_{L_2,L_3}^{N_1,N_2,N_3}(q,q)$ and $\Xi_{L_2,L_3}^{N_1,N_2,N_3}(q,q)$ and the norms on the right side of \eqref{outputA3-1} and \eqref{outputA3-2} should be modified by $\|h^{N_2L_2}\|_{S^{b_0,q}}\|h^{N_3L_3}\|_{S^{b_0,q}}$.

\subsection{Implementing the algorithms}

In this subsection, we compute $$\Upsilon_{L_2,L_3}^{N_1,N_2,N_3}\text{ and }\; \Xi_{L_2,L_3}^{N_1,N_2,N_3}$$
where\footnote{For fixed $v_2,v_3$ of characteristic parameters $(L_2,N_2), (L_3,N_3)$, we are free to choose $r_2,r_3\in\{2,q\}$, meaning that we specify \emph{in a priori} the norms of $h^{N_2,L_2}$ and $h^{N_3,L_3}$ to be used. However, we are not able to use both norms of $h^{N_2,L_2}, h^{N_3,L_3}$ in \emph{a single} multi-linear expression. }
$$ \Upsilon_{L_2,L_3}^{N_1,N_2,N_3}:=\min\{\Upsilon_{L_2,L_3}^{N_1,N_2,N_3}(r_2,r_3):(r_2,r_3)\in \{2,q\}\times \{2,q\} \}$$ and
$$\Xi_{L_2,L_3}^{N_1,N_2,N_3}:= \min\{\Xi_{L_2,L_3}^{N_1,N_2,N_3}(r_2,r_3):(r_2,r_3)\in \{2,q\}\times \{2,q\} \}
$$
for all possible $N_1,N_2,N_3$ such that $N_1\gg N_2,N_3$ and $L_j<N_j^{1-\delta},j=2,3$. This will be done by executing the {\bf Algorithm A1}-{\bf Algorithm A3} that we have just described.

\noi
$\bullet${\bf Two random inputs}

We begin with the case that $v_2,v_3$ are both of type (G) or (C). 
\begin{lemme}\label{2random-Zb} 
We normalize $\|y_{kk_1}^0(\lambda,\lambda_1)\|_{L_{\lambda,\lambda_1}^2l_{k,k_1}^2}=1$. Assume that $v_2,v_3$ are both of type (G) or (C) with characterized parameters $(N_2,L_2),(N_3,L_3)$ satisfying $L_2\vee L_3\ll (N_2\wedge N_3)$. Then outside a set of probability $<\mathrm{e}^{-N_1^{\theta}R}$, we have:
\begin{align*}
&\mathrm{(i)}\text{ If }N_2\ll N_3\ll N_1 \text{  and $v_2$ is of type (C), then }\\ 
& \boxed{|\Upsilon_{L_2,L_3}^{N_1,N_2,N_3}(2,q)|\lesssim RN_1^{\theta+3\kappa^{-0.1}}(N_2N_3)^{-\frac{\alpha}{2}+\epsilon_1}N_2^{1-\frac{\alpha}{2}}(N_3^{\frac{1}{2}}+N_1^{1-\frac{\alpha}{2}})(L_2\wedge L_3)^{\frac{1}{2}}(L_2L_3)^{-\nu}}
\\ 
&\mathrm{(ii)}\text{ If }N_3\ll N_2\ll N_1 \text{  and $v_3$ is of type (C), then }\\ 
& \boxed{|\Upsilon_{L_2,L_3}^{N_1,N_2,N_3}(q,2)|\lesssim RN_1^{\theta+3\kappa^{-0.1}}(N_2N_3)^{-\frac{\alpha}{2}+\epsilon_1}N_3^{1-\frac{\alpha}{2}}(N_2^{\frac{1}{2}}+N_1^{1-\frac{\alpha}{2}})(L_2\wedge L_3)^{\frac{1}{2}}(L_2L_3)^{-\nu}}
\\ 
&\mathrm{(iii)}\text{ If }N_2\ll N_3\ll N_1 \text{  and $v_2$ is of type (G), then }\\ 
&\boxed{ |\Upsilon_{L_2,L_3}^{N_1,N_2,N_3}(q,q)|\lesssim RN_1^{\theta+3\kappa^{-0.1}}(N_2N_3)^{-\frac{\alpha}{2}+\epsilon_1}N_2^{\frac{1}{2}}(N_3^{\frac{1}{2}}+N_1^{1-\frac{\alpha}{2}})L_3^{-\nu}}\\
&\mathrm{(iv)}\text{ If }N_3\ll N_2\ll N_1 \text{  and $v_3$ is of type (G), then }\\ 
& \boxed{|\Upsilon_{L_2,L_3}^{N_1,N_2,N_3}(q,q)|\lesssim RN_1^{\theta+3\kappa^{-0.1}}(N_2N_3)^{-\frac{\alpha}{2}+\epsilon_1}N_3^{\frac{1}{2}}(N_2^{\frac{1}{2}}+N_1^{1-\frac{\alpha}{2}})L_2^{-\nu}}\\
&\mathrm{(v)}\text{ If }N_2\sim N_3\ll N_1 \text{and at least one of $v_2,v_3$ is of type (C), then }\\ 
&\boxed{ |\Upsilon_{L_2,L_3}^{N_1,N_2,N_3}(2,q)|\lesssim RN_1^{\theta+3\kappa^{-0.1}}N_1^{1-\frac{\alpha}{2}}N_2^{-\frac{7(\alpha-1)}{4}+2\epsilon_1}(L_2L_3)^{-\nu}}\\
&\mathrm{(vi)}\text{ If }N_2\sim N_3\ll N_1 \text{and  $v_2,v_3$ are both of type (G), then }\\ 
&\boxed{ |\Upsilon_{L_2,L_3}^{N_1,N_2,N_3}(q,q)|\lesssim RN_1^{\theta+3\kappa^{-0.1}}N_2^{2\epsilon_1}(N_2^{-(\alpha-1)}+N_1^{1-\frac{\alpha}{2}}N_2^{\frac{1}{2}-\alpha} )}
\end{align*}
\end{lemme}

\begin{proof}
Since in the regime $N_2,N_3\ll N_1$, there is no significant difference between the second and the third place in the multi-linear expression, so without loss of generality, we may assume that $N_2\leq N_3$.
By executing {\bf Algorithm A3}, it suffices to estimate the input constant $\Lambda_{2,3}(N_1,N_2,N_3,L_2,L_3)$ in \eqref{outputA3-1}. Since in this step, we do not operate on modulation variables $\lambda,\lambda_j,\mu_0$, we will omit these variables below to simplify the notation. Recall that the multi-linear expression $G_{k,k_1}$ is given by \eqref{Gkk1}, we have
\begin{equation}\label{situation1:formal}
\begin{split}
\sum_{k_1\neq k}\mathbb{E}^{\mathcal{C}}[|G_{k_1,k}|^2]\leq \sum_{\substack{k_1,k_2,k_2',k_3,k_3'\\k_2-k_2'=k_3-k_3'} }\sum_{ \substack{k_2^*,k_3^*,k_2'^*,k_3'^*\\ |k_2'-k_2'^*|<L_2N_3^{\epsilon}\\|k_3'-k_3'^*|<L_3N_3^{\epsilon}
	  } }&\ov{h}_{k_2k_2^*}^{(2)}h_{k_2'k_2'^*}^{(2)}h_{k_3k_3^*}^{(q)}\ov{h}_{k_3'k_3'^*}^{(q)}S_{k_1,k_2,k_3}S_{k_1,k_2',k_3'}\\
\times&\mathbb{E}[\ov{g}_{k_2^*}g_{k_2'^*}g_{k_3^*}\ov{g}_{k_3'^*} ]\cdot\frac{1}{[k_2^*]^{\frac{\alpha}{2}}[k_2'^*]^{\frac{\alpha}{2}}[k_3^*]^{\frac{\alpha}{2}}[k_3'^*]^{\frac{\alpha}{2}}  }.
\end{split}
\end{equation}
Note that the only non-zero contributions in the summation are $k_2^*=k_2'^*, k_3^*=k_3'^*$ and $k_2^*=k_3^*, k_2'^*=k_3'^*$. Denote by $\mathcal{S}_{1}$ the contribution from $k_2^*=k_2'^*$ and $k_3^*=k_3'^*$ and by $\mathcal{S}_{2}$ the contribution from $k_2^*=k_3^*$ and $k_2'^*=k_3'^*$. Note that $\mathcal{S}_2=0$ if $N_2\ll N_3$ or $v_2,v_3$ are both of type (G).

\underline{Estimate for $\mathcal{S}_1$:} Denote by $a_2(k_2)=\|h_{k_2k_2^*}^{(2)}\|_{l_{k_2^*}^2}$ and $a_3(k_3)=\|h_{k_3k_3^*}^{(q)}\|_{l_{k_3^*}^2}$. For this contribution we must have $|k_2-k_2'|\lesssim L_2N_3^{\epsilon}$ and $|k_3-k_3'|\lesssim L_3N_3^{\epsilon}$. Since $k_2-k_2'=k_3-k_3'$, we may change the constraint to $|k_2-k_2'|=|k_3-k_3'|\lesssim (L_2\wedge L_3)N_3^{\epsilon}$. By Cauchy-Schwartz, we have
\begin{align*}
\mathcal{S}_1\lesssim &(N_2N_3)^{-\alpha}\sum_{\substack{k_1,k_2,k_2',k_3,k_3'\\
k_2-k_2'=k_3-k_3'
 } }\mathbf{1}_{|k_2-k_2'|\lesssim (L_2\wedge L_3)N_3^{\epsilon}}
S_{k_1,k_2,k_3}S_{k_1,k_2',k_3'}a_2(k_2)a_2(k_2')a_3(k_3)a_3(k_3')\\
\lesssim &(N_2N_3)^{-\alpha}\|a_3\|_{l_{k_3}^{\infty}}^2\sum_{k_2,k_2',k_3}a_2(k_2)a_2(k_2')N_1^{\epsilon}\mathbf{1}_{|k_2-k_2'|\lesssim (L_2\wedge L_3)N_3^{\epsilon}}\Big(1+\frac{N_1^{2-\alpha}}{\lg k_2-k_3\rg}\Big),
\end{align*}
where to the last inequality, we use the counting bound $\sum_{k_1}S_{k_1,k_2,k_3}\lesssim N_1^{\epsilon}\Big(1+\frac{N_1^{2-\alpha}}{\lg k_2-k_3\rg}\Big)$.
Using Young's convolution inequality to the sum $\sum_{k_2,k_2'}a_2(k_2)a_2(k_2')\mathbf{1}_{|k_2-k_2'|\lesssim (L_2\wedge L_3)N_3^{\epsilon}}$, we have
\begin{align*}
\mathcal{S}_1\lesssim (N_2N_3)^{-\alpha}(L_2\wedge L_3)N_1^{2\epsilon}(N_3+N_1^{2-\alpha}) \|a_2\|_{l_{k_2}^2}^2\|a_3\|_{l_{k_3}^{\infty}}^2.
\end{align*}
When $v_2$ is of type (G), we could estimate $a_2(k_2)$ by $\|a_2\|_{l_{k_2}^{\infty}}$ and
$$ \mathcal{S}_1\lesssim (N_2N_3)^{-\alpha}\|a_2\|_{l^{\infty}}^2\|a_3\|_{l^{\infty}}^2N_3^{\epsilon}(N_2N_3+N_1^{2-\alpha}N_2).
$$

\underline{ Estimate for $\mathcal{S}_2$:}  Note that $\mathcal{S}_2\neq 0$ only if $N_2\sim N_3$ and at least one of $v_2,v_3$ is of type (C). Therefore, without loss of generality, we may assume that $v_2$ is of type (C). Moreover, we have the constraint $|k_2-k_3|=|k_2'-k_3'|\lesssim (L_2\vee L_3)N_3^{\epsilon}$. By Cauchy-Schwartz,
\begin{align*}
\mathcal{S}_2\lesssim &(N_2N_3)^{-\alpha}\|a_3\|_{l_{k_3}^{\infty}}^2\sum_{\substack{k_1,k_2,k_2',k_3,k_3'\\
k_2-k_2'=k_3-k_3' } }\mathbf{1}_{|k_2-k_3|\lesssim (L_2\vee L_3)N_3^{\epsilon}}S_{k_1,k_2,k_3}S_{k_1,k_2',k_3'}a_2(k_2)a_2(k_2')\\
\lesssim &(N_2N_3)^{-\alpha}\|a_3\|_{l_{k_3}^{\infty}}^2 \sum_{k_1,k_2,k_3}\mathbf{1}_{|k_2-k_3|\lesssim (L_2\vee L_3)N_3^{\epsilon}}S_{k_1,k_2,k_3}a_2(k_2)N_1^{\epsilon}\|a_2\|_{l_{k_2'}^2}\Big(1+\frac{N_2^{1-\frac{\alpha}{2}}}{\lg k_2-k_3\rg^{\frac{1}{2}}}\Big),
\end{align*}
where to the last inequality we use Cauchy-Schwartz for the sum in $k_2'$ and the counting bound
\begin{align}\label{counting} 
 \sum_{k_2'}S_{k_1,k_2',k_3+k_2'-k_2}\lesssim N_1^{\epsilon}\Big(1+\frac{N_2^{2-\alpha}}{\lg k_2-k_3\rg}\Big).
\end{align}
Using again $\sum_{k_1}S_{k_1,k_2,k_3}\lesssim N_1^{\epsilon}\Big(1+\frac{N_1^{2-\alpha}}{\lg k_2-k_3\rg}\Big)$, we have
\begin{align*}
\mathcal{S}_2\lesssim & (N_2N_3)^{-\alpha}N_1^{2\epsilon}\|a_3\|_{l_{k_3}^{\infty}}^2\|a_2\|_{l_{k_2}^2}\!\!\!\!\!\sum_{\substack{k_2,k_3 \\
|k_2-k_3|\leq (L_2\vee L_3)N_3^{\epsilon} } }\!\!\!\!\!\!\!\!\!\!\!\!\!\!a_2(k_2)\Big(1+\frac{N_2^{1-\frac{\alpha}{2}}}{\lg k_2-k_3\rg^{\frac{1}{2}}}+\frac{N_1^{2-\alpha}}{\lg k_2-k_3\rg}+\frac{N_1^{2-\alpha}N_2^{1-\frac{\alpha}{2}}}{\lg k_2-k_3\rg^{\frac{3}{2}}}\Big)\\
\lesssim &
(N_2N_3)^{-\alpha}N_1^{3\epsilon}
((L_2\vee L_3)N_2^{\frac{1}{2}}+(L_2\vee L_3)^{\frac{1}{2}}N_2^{\frac{3}{2}-\frac{\alpha}{2}}+N_1^{2-\alpha}N_2^{\frac{3}{2}-\frac{\alpha}{2}} )\|a_3\|_{l_{k_3}^{\infty}}^2\|a_2\|_{l_{k_2}^2}^2\\
\leq & (N_2N_3)^{-\alpha}N_1^{3\epsilon}\cdot N_1^{2-\alpha}N_2^{\frac{3-\alpha}{2}}\|a_3\|_{l_{k_3}^{\infty}}^2\|a_2\|_{l_{k_2}^2}^2\\
\lesssim & N_1^{2-\alpha+3\epsilon}N_2^{\frac{3-5\alpha}{2}}\|a_3\|_{l_{k_3}^{\infty}}^2\|a_2\|_{l_{k_2}^2}^2,
\end{align*}
since $\alpha<\frac{4}{3}$ and $N_2\sim N_3$. 

In summary, if at least one of $v_2,v_3$ is of type (C), when $N_2\ll N_3$, we obtain the upper bound
$$ \Lambda_{2,3}(N_1,N_2,N_3,L_2,L_3)=(N_2N_3)^{-\alpha}(L_2\wedge L_3)N_1^{2\epsilon}(N_3+N_1^{2-\alpha}),
$$ 
and when $N_2\sim N_3$, we obtain the upper bound
$$\Lambda_{2,3}(N_1,N_2,N_3,L_2,L_3)=N_1^{2-\alpha+3\epsilon}N_2^{\frac{3-5\alpha}{2}}.
$$
By choosing $4\epsilon<\kappa^{-0.1}$, plugging into \eqref{outputA3-1} and using the corresponding $Z^{b_0}$ and $S^{b_0,q}$ bounds of $h^{N_jL_j}$ (note that $\frac{\alpha}{q_0},\frac{1}{q}<\kappa^{-0.1}$), the proof of Lemma \ref{2random-Zb} is complete.

\end{proof}

For $\Xi_{L_2,L_3}^{N_1,N_2,N_3}$, we have:
\begin{lemme}\label{2random-Sb} 
	We normalize $\|y_{kk_1}^0(\lambda,\lambda_1)\|_{l_k^1L_{\lambda}^{q'}L_{\lambda_1}^2l_{k_1}^2}=1$. Assume that $v_2,v_3$ are both of type (G) or (C) with characterized parameters $(N_2,L_2),(N_3,L_3)$ satisfying $L_2\vee L_3\ll (N_2\wedge N_3)$. Then outside a set of probability $<\mathrm{e}^{-N_1^{\theta}R}$:
	\begin{align*}
	&\mathrm{(i)}\text{ If }N_2\wedge N_3\ll N_2\vee N_3\ll N_1, \text{ then }\\ 
	&\boxed{ |\Xi_{L_2,L_3}^{N_1,N_2,N_3}(q,q)|\lesssim RN_1^{\theta+3\kappa^{-0.1}}(N_2\wedge N_3)^{-\frac{\alpha-1}{2}}(N_2\vee N_3)^{-\frac{\alpha}{2}+2\epsilon_1}(L_2\wedge L_3)^{\frac{1}{2}}(L_2L_3)^{-\nu}}\\
	&\mathrm{(ii)}\text{ If }N_2\sim N_3\ll N_1 \text{and at least one of $v_2,v_3$ is of type (C), then }\\ 
	&\boxed{ |\Xi_{L_2,L_3}^{N_1,N_2,N_3}(2,q)|\lesssim RN_1^{\theta+3\kappa^{-0.1}}N_2^{-\frac{7(\alpha-1)}{4}+2\epsilon_1}(L_2L_3)^{-\nu}}\\
	&\mathrm{(iii)}\text{ If }N_2\sim N_3\ll N_1 \text{and  $v_2,v_3$ are both of type (G), then }\\ 
	&\boxed{ |\Xi_{L_2,L_3}^{N_1,N_2,N_3}(q,q)|\lesssim RN_1^{\theta+3\kappa^{-0.1}}N_2^{\frac{1}{2}-\alpha+2\epsilon_1}}
	\end{align*}
\end{lemme}

\begin{proof}
Without loss of generality, we may assume that $N_2\leq N_3$. First we assume that $N_2\ll N_3$. From {\bf Algorithm A3}, it suffices to estimate, for fixed $|k|\in N_1$, the expression
 \begin{align}\label{esperance} 
 \sum_{k_1:k_1\neq k}\mathbb{E}^{\mathcal{C}}[|G_{k_1,k}|^2]\sim &(N_2N_3)^{-\alpha}\sum_{\substack{k_2,k_2',k_3,k_3'\\
 k_2-k_2'=k_3-k_3' } }\!\!\!\!\!\!S_{k+k_2-k_3,k_2,k_3}S_{k+k_2'-k_3',k_2',k_3'}\notag \\
\times &\sum_{\substack{k_2^*,k_3^*,k_2'^*,k_3'^* } }\prod_{j=2}^3\mathbf{1}_{\substack{|k_j-k_j^*|<L_jN_3^{\epsilon}\\
|k_j'-k_j'^*|<L_jN_3^{\epsilon} }}\cdot \ov{h}_{k_2k_2^*}^{(q)}h_{k_2'k_2'^*}^{(q)}h_{k_3k_3^*}^{(q)}\ov{h}_{k_3'k_3'^*}^{(q)}\cdot\mathbb{E}[\ov{g}_{k_2^*}g_{k_2'^*}g_{k_3^*}\ov{g}_{k_3'^*} ].
 \end{align}
Since $N_2\ll N_3$ and $|k_j^*|\sim N_j$, the only non-zero contribution comes from $k_2^*=k_2'^*$ and $k_3^*=k_3'^*$.  Consequently, since $k_2-k_2'=k_3-k_3'$,  $|k_j-k_j'|\lesssim (L_2\wedge L_3)N_3^{\epsilon}, j=2,3$. Denote by $a_{j}(k_j)=\|h_{k_jk_j^*}^{(q)}\|_{l_{k_j^*}^2}$, by Cauchy-Schwartz, the quantity above can be bounded by
\begin{align*}
& (N_2N_3)^{-\alpha}\|a_2\|_{l_{k_2}^{\infty}}^2\|a_3\|_{l_{k_3}^{\infty}}^2\sum_{k_2,k_2' }\mathbf{1}_{|k_2-k_2'|\lesssim (L_2\wedge L_3)N_3^{\epsilon}}\sum_{k_3}S_{k+k_2-k_3,k_2,k_3}\\
\lesssim & (N_2N_3)^{-\alpha}\|a_2\|_{l_{k_2}^{\infty}}^2\|a_3\|_{l_{k_3}^{\infty}}^2\cdot N_3^{\epsilon}(L_2\wedge L_3)N_2=N_2^{-(\alpha-1)}N_3^{-\alpha+\epsilon}(L_2\wedge L_3)\|a_2\|_{l_{k_2}^{\infty}}^2\|a_3\|_{l_{k_3}^{\infty}}^2.
\end{align*}
Note that this estimate remains true if $N_2\sim N_3$ and $v_2,v_3$ are both of type (G). By choosing $\epsilon<\kappa^{-0.1}$ and noticing that $\frac{\alpha}{q_0},\frac{1}{q}<\kappa^{-0.1}$, we have proved (i) and (iii).

Next we assume that $N_2\sim N_3$, then by {\bf Algorithm 3}, to estimate $\Xi_{L_2,L_3}^{N_1,N_2,N_3}(2,q)$, it suffices to calculate \eqref{esperance} by changing $h_{k_2k_2^*}^{(q)}, h_{k_2'k_2'^*}^{(q)}$ to $h_{k_2k_2^*}^{(2)}, h_{k_2'k_2'^*}^{(2)}$. Denote by $\mathcal{S}_1$ the contribution from $k_2^*=k_2'^*, k_3^*=k_3'^*$ and by $\mathcal{S}_2$ the contribution from $k_2^*=k_3^*, k_2'^*=k_3'^*$. 

\underline{Estimate of $\mathcal{S}_1$:} by Cauchy-Schwartz, we have
\begin{align*}
\mathcal{S}_1\lesssim & (N_2N_3)^{-\alpha}\|a_3\|_{l_{k_3}^{\infty}}^2\sum_{k_2,k_2' }a_2(k_2)a_2(k_2')\mathbf{1}_{|k_2-k_2'|<(L_2\wedge L_3)N_3^{\epsilon}}\sum_{k_3}S_{k+k_2-k_3,k_2,k_3}\\
\lesssim & (N_2N_3)^{-\alpha}\|a_3\|_{l_{k_3}^{\infty}}^2\|a_2\|_{l_{k_2}^{2}}^2(L_2\wedge L_3)N_3^{\epsilon},
\end{align*}
where to the last inequality, we used Schur's test and the fact that 
$\sum_{k_3}S_{k+k_2-k_3,k_2,k_3}\lesssim1$.

\underline{Estimate of $\mathcal{S}_2$:} by Cauchy-Schwartz and \eqref{counting}, we have
\begin{align*}
\mathcal{S}_1\lesssim & (N_2N_3)^{-\alpha}\|a_3\|_{l_{k_3}^{\infty}}^2\sum_{k_2,k_2',k_3 }a_2(k_2)a_2(k_2')\mathbf{1}_{|k_2-k_3|\lesssim (L_2\vee L_3)N_3^{\epsilon}}S_{k+k_2-k_3,k_2,k_3}S_{k+k_2-k_3,k_2',k_3+k_2'-k_2}\\
\lesssim & (N_2N_3)^{-\alpha}\|a_3\|_{l_{k_3}^{\infty}}^2\sum_{k_2,k_3} a_2(k_2)\mathbf{1}_{|k_2-k_3|\lesssim(L_2\vee L_3)N_3^{\epsilon} }S_{k+k_2-k_3,k_2,k_3}\|a_2(k_2')\|_{l_{k_2'}^2}N_1^{\epsilon}\Big(1+\frac{N_2^{1-\frac{\alpha}{2}}}{\lg k_2-k_3\rg^{\frac{1}{2}}}\Big)\\
\lesssim &(N_2N_3)^{-\alpha}N_1^{\epsilon}\|a_3\|_{l_{k_3}^{\infty}}^2\|a_2\|_{l_{k_2}^2}\sum_{k_2}a_2(k_2)\sum_{k_3}S_{k+k_2-k_3,k_2,k_3}\Big(1+\frac{N_2^{1-\frac{\alpha}{2}}}{\lg k_2-k_3\rg^{\frac{1}{2}}}\Big)\\
\lesssim &N_1^{\epsilon}(N_2N_3)^{-\alpha}\|a_3\|_{l_{k_3}^{\infty}}^2\|a_2\|_{l_{k_2}^2}^2\cdot N_2^{\frac{1}{2}}\cdot N_2^{1-\frac{\alpha}{2}}\\
\lesssim & N_1^{\epsilon}N_2^{\frac{3-5\alpha}{2}}\|a_3\|_{l_{k_3}^{\infty}}^2\|a_2\|_{l_{k_2}^2}^2,
\end{align*}
since $\sum_{k_3}S_{k+k_2-k_3,k_2,k_3}\lesssim 1$.
Implementing {\bf Algorithm 3}, the proof of Lemma \ref{2random-Sb} is complete.
\end{proof}

It remains to deal with the case where $L_2\vee L_3\gtrsim N_2\wedge N_3$, say $N_2\leq N_3$ and $L_3\gtrsim N_2$. Then from Lemma \ref{lemA1-123} and Lemma \ref{lemA1-456}, we have
\begin{lemme}\label{2random-Zb'}
We normalize $\|y_{kk_1}^0(\lambda,\lambda_1)\|_{L_{\lambda,\lambda_1}^2l_{k,k_1}^2}=1.$ Assume that $v_2,v_3$ are both of type (G) or (C) with characterized parameters $(N_2,L_2), (N_3,L_3)$ satisfying $L_2\vee L_3\gtrsim N_2\wedge N_3$. Then
$$
\boxed{ |\Upsilon_{L_2,L_3}^{N_1,N_2,N_3}|\lesssim (N_2\wedge N_3)^{1-\frac{\alpha}{2}-\nu+\epsilon_1}(N_2\vee N_3)^{-\frac{\alpha}{2}+\frac{2}{q}+\epsilon_2}((N_2\vee N_3)^{\frac{1}{2}}+N_1^{1-\frac{\alpha}{2}}) }
$$ 	
\end{lemme}
\begin{proof}
	We apply \eqref{outputA1-2} or \eqref{outputA1-3}, according to $N_3\leq N_2$ or $N_2\leq N_3$.
\end{proof}
\begin{lemme}\label{2random-Sb'}
	We normalize $\|y_{kk_1}^0(\lambda,\lambda_1)\|_{l_k^{1}L_{\lambda}^{q'}L_{\lambda_1}^2l_{k_1}^2}=1.$ Assume that $v_2,v_3$ are both of type (G) or (C) with characterized parameters $(N_2,L_2), (N_3,L_3)$ satisfying $L_2\vee L_3\gtrsim N_2\wedge N_3$. Then
	$$
	\boxed{ |\Xi_{L_2,L_3}^{N_1,N_2,N_3}|\lesssim (N_2\wedge N_3)^{1-\frac{\alpha}{2}-\nu+\epsilon_1}(N_2\vee N_3)^{-\frac{\alpha}{2}+\frac{1}{q}+\epsilon_2}}
	$$ 	
\end{lemme}
\begin{proof}
We apply \eqref{outputA1-5} or \eqref{outputA1-6}, according to $N_3\leq N_2$ or $N_2\leq N_3$.
\end{proof}
\noi
$\bullet${\bf One random input }
Now we deal with the case where one of $v_2,v_3$ is of type (G) or (C) and the other is of type (D).

\begin{lemme}\label{Case2formallemma1}
Assume that $\|y_{kk_1}^0(\lambda,\lambda_1)\|_{L_{\lambda,\lambda_1}^2l_{k,k_1}^2 }=1$.	Assume that one of $v_2,v_3$ is of type (G) or (C) and the other is of type (D). Then outside a set of probability $<\mathrm{e}^{-N_1^{\theta}R},$ the following estimates hold:
	\begin{itemize}
		\item[(i)] If $N_2\ll N_3\ll N_1$ and $v_3$ is of type (C) or (G), we have
		\begin{equation*}
		\begin{split}
	&\boxed{|\Upsilon_{L_2,L_3}^{N_1,N_2,N_3}(2,q)| 
		\lesssim RN_1^{\theta+3\kappa^{-0.1} }N_3^{-\frac{\alpha}{2}+\epsilon_1}N_2^{\frac{1}{4}-s}L_3^{\frac{1}{4}-\nu}\big(N_3^{\frac{1}{2}}+N_1^{1-\frac{\alpha}{2}} \big)	}
		\end{split}
		\end{equation*}
		\item[(ii)] If $N_3\ll N_2\ll N_1$ and $v_2$ is of type (G) or (C), we have
			\begin{equation*}
		\begin{split}
	&\boxed{|\Upsilon_{L_2,L_3}^{N_1,N_2,N_3}(q,2) |
		\lesssim RN_1^{\theta+3\kappa^{-0.1} }N_2^{-\frac{\alpha}{2}+\epsilon_1}N_3^{\frac{1}{4}-s}L_2^{\frac{1}{4}-\nu}\big(N_2^{\frac{1}{2}}+N_1^{1-\frac{\alpha}{2}} \big)	}
		\end{split}
		\end{equation*}
		\item[(iii)] In any of these situations:
		\begin{itemize}
			\item $N_2\sim N_3\ll N_1$;
			\item $N_2\ll N_3\ll N_1$ and $v_3$ is of type (D);
			\item $N_3\ll N_2\ll N_1$ and $v_2$ is of type (D);
		\end{itemize}
	we have
		\begin{align*}
	&\boxed{|\Upsilon_{L_2,L_3}^{N_1,N_2,N_3}(2,2)| \lesssim RN_1^{\theta+3\kappa^{-0.1} }\big[N_1^{1-\frac{\alpha}{2}}+(N_2\wedge N_3)^{\frac{1}{2}}\big](N_{2}\vee N_3)^{-s}(N_2\wedge N_3)^{-\frac{\alpha-1}{2}}}
		\end{align*}
	\end{itemize}
\end{lemme}

\begin{proof}
	Note that (iii) follows directly from Lemma \ref{lemA1-123}, it suffices to prove (i) and (ii). The situations (i) and (ii) are similar, so we only prove (i). Recall that in this case, $N_2\ll N_3$, $v_2$ is of type (D) and $v_3$ is of type (G) or (C).

	By executing {\bf Algorithm A2}, we need to estimate (see \eqref{inputA2-1}) \footnote{To save the notation, in the summation below, we implicitly insert the constraint $|k_2|,|k_2'|\leq N_2, |k_3|,|k_3'|,|m_3|,|m_3'|\sim N_3.$}
	$$ \underbrace{\Big(\sum_{k_2\neq k_2'}\mathbb{E}^{\mathcal{C}_3}[|\sigma_{k_2k_2'}^{(3)}|^2]\Big)^{\frac{1}{2}}}_{\mathrm{I}}
	+\sup_{|k_2|\sim N_2}\underbrace{\Big(\mathbb{E}^{\mathcal{C}_3}[|\sigma_{k_2k_2}^{(3)}|^2]\Big)^{\frac{1}{2}}}_{\mathrm{II}}.
	$$
	We calculate
	\begin{align*}
	\mathrm{I}^2=\sum_{\substack{k_2,k_2',k_1,k_3,m_1,m_3,k_3',m_3'\\
			k_2\neq k_2'\\
			k_3-k_3'=m_3-m_3'=k_2-k_2' } }&S_{k_1,k_2,k_3}S_{k_1,k_2',k_3'}
	S_{m_1,k_2,m_3}S_{m_1,k_2',m_3'}\\
	\times \sum_{\substack{|k_3^*|\sim N_3, |k_3'^*|\sim N_3\\ 
			|m_3^*|\sim N_3, |m_3'^*|\sim N_3  }  }&h_{k_3k_3^*}^{(q)}\ov{h}_{k_3'k_3'^*}^{(q)}\ov{h}_{m_3m_3^*}^{(q)}h_{m_3'm_3'^*}^{(q)}\cdot\frac{ \mathbb{E}[g_{k_3^*}\ov{g}_{k_3'^*}\ov{g}_{m_3^*}g_{m_3'^*}] }{[k_3^*]^{\frac{\alpha}{2}}
		[k_3'^*]^{\frac{\alpha}{2}}
		[m_3^*]^{\frac{\alpha}{2}}
		[m_3'^*]^{\frac{\alpha}{2}} },
	\end{align*}
	where in the last line, the range of summation in $k_3^*,k_3'^*,m_3^*,m_3'^*$ satisfies:
	$$|k_3-k_3^*|\leq L_3N_3^{\epsilon},\;
	|k_3'-k_3'^*|\leq L_3N_3^{\epsilon},\;
	|m_3-m_3^*|\leq L_3N_3^{\epsilon},\;
	|m_3'-m_3'^*|\leq L_3N_3^{\epsilon}.
	$$
	As usual, to simplify the notation, we omit the dependence on $\lambda_3,\mu_0$ and denote simply $h_{k_3k_3^*}^{(q)}=\lg\lambda_3\rg^{\frac{2b_0}{q'}}\widetilde{h}_{k_3k_3^*}^{N_3L_3}(\lambda_3)$. 
Using the independence,
	\begin{equation}\label{EI}
	\begin{split}
	\mathrm{I}^2\lesssim &N_3^{-2\alpha}\sum_{\substack{k_2,k_2',k_1,k_3,m_1,m_3,k_3',m_3'\\
			k_2\neq k_2'\\
			k_2-k_2'=k_3-k_3'=m_3-m_3'   }  }S_{k_1,k_2,k_3}S_{k_1,k_2',k_3'}
	S_{m_1,k_2,m_3}S_{m_1,k_2',m_3'}\\
	\times &\mathbb{E}\Big[\underbrace{\mathbf{1}_{|k_3-k_3'|\lesssim L_3N_3^{\epsilon}}\sum_{k_3^*,m_3^*} |h_{k_3k_3^*}^{(q)}\ov{h}_{k_3'k_3^*}^{(q)}\ov{h}_{m_3m_3^*}^{(q)}h_{m_3'm_3^*}^{(q)}| }_{\mathrm{I}_{1}}+\underbrace{\mathbf{1}_{|k_3-m_3|\lesssim L_3N_3^{\epsilon}}\sum_{k_3^*,k_3'^*}|h_{k_3k_3^*}^{(q)}\ov{h}_{m_3k_3^*}^{(q)}\ov{h}_{k_3'k_3'^*}^{(q)}h_{m_3'k_3'^*}^{(q)} }_{\mathrm{I}_{2}}| \Big].
	\end{split}
	\end{equation}
	In the arguments below, we will not display the expectation $\mathbb{E}$ since we will not use the random feature of the coefficients $h_{k_jk_j^*}$.

	Using Cauchy-Schwartz, we have
	$$ \mathrm{I}_1\leq \mathbf{1}_{|k_3-k_3'|\lesssim L_3N_3^{\epsilon}}\|h_{k_3k_3^*}^{(q)}\|_{l_{k_3}^{\infty}l_{k_3^*}^2}^4,\quad \mathrm{I}_2\leq \mathbf{1}_{|k_3-m_3|\lesssim L_3N_3^{\epsilon}}\|h_{k_3k_3^*}^{(q)}\|_{l_{k_3}^{\infty}l_{k_3^*}^2}^4.
	$$
	Plugging into \eqref{EI}, the summation corresponding to the contribution $\mathrm{I}_1$ can be bounded by
	\begin{align}\label{bianhao1}
	&N_3^{-2\alpha}\|h_{k_3k_3^*}^{(q)}\|_{l_{k_3}^{\infty}l_{k_3^*}^2}^4\!\!\!\!\sum_{\substack{k_2,k_2'\\0<|k_2-k_2'|\lesssim L_3N_3^{\epsilon} }}\!\!\!\!\Big(\sum_{k_1,k_3:k_3\neq k_2}S_{k_1,k_2,k_3}S_{k_1,k_2',k_3+k_2'-k_2}\Big)^2 \notag\\
	\lesssim &N_3^{-2\alpha}\|h_{k_3k_3^*}^{(q)}\|_{l_{k_3}^{\infty}l_{k_3^*}^2}^4\!\!\!\! \sum_{\substack{k_2,k_2'\\0<|k_2-k_2'|\lesssim L_3N_3^{\epsilon} }}
	\!\!\!\!\!\!\!\!(N_3+N_1^{2-\alpha}N_3^{\epsilon})^2\notag\\
	\lesssim & N_3^{-2\alpha}N_1^{2\epsilon}L_3N_3^{\epsilon}N_2\big[N_3^2+N_1^{2(2-\alpha)}N_3^{2\epsilon}\big]\|h_{k_3k_3^*}^{(q)}\|_{l_{k_3}^{\infty}l_{k_3^*}^2}^4.
	\end{align}
	since
	$$ \sum_{k_3:k_3\neq k_2}\sum_{k_1}S_{k_1,k_2,k_3}\lesssim \sum_{k_3}N_1^{\epsilon}\Big(1+\frac{N_1^{2-\alpha}}{\lg k_2-k_3\rg}\Big)\lesssim N_1^{\epsilon}\big(N_3+N_1^{2-\alpha}N_3^{\epsilon}\big),
	$$ 
	thanks to the fact that $|k_1|\sim N_1\gg |k_2|,|k_3|$.
	Next we consider the summation corresponding to the contribution $\mathrm{I}_2$ which can be bounded by
	\begin{align}\label{EI2}
	&N_3^{-2\alpha}\|h_{k_3k_3^*}^{(q)}\|_{l_{k_3}^{\infty}l_{k_3^*}^2}^4
	\sum_{\substack{ k_2,k_2',k_1,m_1,k_3,m_3\\k_2\neq k_2' \\
			|k_3-m_3|\lesssim L_3N_3^{\epsilon} } }
	S_{k_1,k_2,k_3}S_{k_1,k_2',k_3+k_2'-k_2 }S_{m_1,k_2,m_3}S_{m_1,k_2',m_3+k_2'-k_2}.
	\end{align}
	Since $|k_2|,|k_2'|\leq N_2\ll N_3\sim |k_3|$ 	and
	$ |\partial_{k_2'}\Phi(k_1,k_2',k_3+k_2'-k_2)|\sim N_3^{\alpha-1}>1$, we have, for fixed $k_1,k_3,k_2$,
	$ \sum_{k_2'}S_{k_1,k_2',k_3+k_2'-k_2}\lesssim N_1^{\epsilon}.$
	Thus
	\begin{align*}
	&
	\sum_{\substack{ k_2,k_2',k_1,m_1,k_3,m_3\\k_2\neq k_2'  } }\mathbf{1}_{|k_3-m_3|\lesssim L_3N_3^{\epsilon} }
	S_{k_1,k_2,k_3}S_{k_1,k_2',k_3+k_2'-k_2 }S_{m_1,k_2,m_3}S_{m_1,k_2',m_3+k_2'-k_2}
	\\
	\lesssim & N_1^{\epsilon}\sum_{\substack{k_2,k_3,m_3, \\
			|k_3-m_3|\lesssim L_3N_3^{\epsilon} } }\sum_{k_1,m_1}S_{k_1,k_2,k_3}S_{m_1,k_2,m_3}\\
	\lesssim & \sum_{\substack{k_2,k_3,m_3:k_2\neq k_3,m_3\\
			|k_3-m_3|\lesssim L_3N_3^{\epsilon}  } } \!\!\!\!\!\!\!N^{3\epsilon}\Big(1+\frac{N_1^{2-\alpha}}{\lg k_2-k_3\rg}\Big)\Big(1+\frac{N_1^{2-\alpha}}{\lg k_2-m_3\rg}\Big)\\ \lesssim &N_1^{2\epsilon}\big[N_1^{2(2-\alpha)}N_2+ L_3N_2N_3+L_3N_2N_1^{2-\alpha} \big],
	\end{align*}
	which is smaller than the upper bound \eqref{bianhao1} of $\mathrm{I}_1$.
	Therefore, we have
	\begin{align*}
	\mathrm{I}^{\frac{1}{2}}\lesssim &N_3^{-\frac{\alpha}{2}}\|h_{k_3k_3^*}^{(q)}(\lambda_3)\|_{l_{k_3}^{\infty}l_{k_3^*}^2 }N_1^{2\epsilon}N_2^{\frac{1}{4}}L_3^{\frac{1}{4}}\cdot(N_3^{\frac{1}{2}}+N_1^{1-\frac{\alpha}{2}} ).
	\end{align*}

	Next we estimate $\mathrm{II}^2$. For fixed $k_2$, we have
	\begin{align*}
	&\mathbb{E}^{\mathcal{C}_3}[|\sigma_{k_2k_2}^{(3)}|^2]\\=&\mathbb{E}^{\mathcal{C}_3}\Big[\Big(\sum_{k_1,k:k_1\neq k}|w^{(q)}_3(k+k_2-k_1)|^2S_{k_1,k_2,k+k_2-k_1}\Big)^2 \Big]\\
	=&\mathbb{E}^{\mathcal{C}_3}\Big[\sum_{\substack{k_1,k_3:k_1,k_3\neq k_2\\
			m_1,m_3:m_1,m_3\neq k_2 } }S_{k_1,k_2,k_3}S_{m_1,k_2,m_3}\Big|\sum_{\substack{|k_3^*|\sim N_3\\|k_3-k_3^*|\lesssim L_3N_3^{\epsilon} }} h_{k_3k_3^*}^{(q)}\frac{g_{k_3^*}}{[k_3^*]^{\frac{\alpha}{2}}} \Big|^2\Big|\sum_{\substack{|m_3^*|\sim N_3\\|m_3-m_3^*|\lesssim L_3N_3^{\epsilon} }} h_{m_3m_3^*}^{(q)}\frac{g_{m_3^*}}{[k_3^*]^{\frac{\alpha}{2}}} \Big|^2 \Big]\\
	\sim & N_3^{-2\alpha}\|h_{k_3k_3^*}^{(q)}\|_{l_{k_3}^{\infty}l_{k_3^*}^2}^4\Big(\sum_{k_3,k_1:k_1,k_3\neq k_2}S_{k_1,k_2,k_3}\Big)^2\\
	\lesssim & N_3^{-2\alpha}\|h_{k_3k_3^*}^{(q)}\|_{l_{k_3}^{\infty}l_{k_3^*}^2}^4\Big(\sum_{k_3:k_3\neq k_2}N_1^{\epsilon}\Big(1+\frac{N_1^{2-\alpha}}{\lg k_3-k_2\rg}\Big)\Big)^2\lesssim N_3^{-2\alpha}N_1^{4\epsilon}\big[N_1^{2(2-\alpha)}+N_3^2\big]\|h_{k_3k_3^*}^{(q)}\|_{l_{k_3}^{\infty}l_{k_3^*}^2}^4.
	\end{align*}
	Hence
	$$ \mathrm{II}^{\frac{1}{2}}\lesssim \big[N_1^{1-\frac{\alpha}{2}+\epsilon}+N_3^{\frac{1}{2}}N_1^{\epsilon}\big]N_3^{-\frac{\alpha}{2}}\|h_{k_3k_3^*}^{(q)}(\lambda_3)\|_{l_{k_3}^{\infty}l_{k_3^*}^2},
	$$
	which is smaller than the upper bound \eqref{bianhao1} of $\mathrm{I}^{\frac{1}{2}}$. 
	Plugging into the output \eqref{outputA2-1} and choosing $\epsilon\ll \kappa^{-0.1}$, we complete the proof of Lemma \ref{Case2formallemma1}.
\end{proof}
For the $\Xi_{L_2,L_3}^{N_1,N_2,N_3}$ norm, we have:
\begin{lemme}\label{Case2formallemma1'}
Assume that $\|y_{kk_1}^0(\lambda,\lambda_1)\|_{l_k^{1}L_{\lambda}^{q'}L_{\lambda_1}^2l_{k_1}^2 }=1$.	Assume that one of $v_2,v_3$ is of type (G) or (C) and the other is of type (D). Then outside a set of probability $<\mathrm{e}^{-N_1^{\theta}R},$ the following estimates hold:
	\begin{itemize}
		\item[(i)] If $N_2\ll N_3\ll N_1$ and $v_3$ is of type (G) or (C) , we have
		\begin{equation*}
		\begin{split}
	&\boxed{|\Xi_{L_2,L_3}^{N_1,N_2,N_3}(2,q)|\lesssim RN_1^{\theta+3\kappa^{-0.1}}N_3^{-\frac{\alpha}{2}}N_2^{\frac{1}{4}-s}\cdot N_3^{\epsilon_1}L_3^{-\nu}}
		\end{split}
		\end{equation*}
		\item[(ii)]If $N_3\ll N_2\ll N_1$ and $v_2$ is of type (G) or (C) , we have
		\begin{equation*}
		\begin{split}
	&\boxed{ |\Xi_{L_2,L_3}^{N_1,N_2,N_3}(q,2)|\lesssim  RN_1^{\theta+3\kappa^{-0.1}}N_2^{-\frac{\alpha}{2}}N_3^{\frac{1}{4}-s}\cdot N_2^{\epsilon_1}L_2^{-\nu}}
		\end{split}
		\end{equation*}	
		\item[(iii)] In any of these situations:	\begin{itemize}
			\item $N_2\sim N_3\ll N_1$;
			\item $N_2\ll N_3\ll N_1$ and $v_3$ is of type (D);
			\item $N_3\ll N_2\ll N_1$ and $v_2$ is of type (D);
		\end{itemize} then
		\begin{align*}
&\boxed{ |\Xi_{L_2,L_3}^{N_1,N_2,N_3}(2,2)|\lesssim 
 RN_1^{\theta+3\kappa^{-0.1}}(N_2\vee N_3)^{-s}(N_2\wedge N_3)^{-\frac{\alpha-1}{2}}  } 
		\end{align*}
	\end{itemize}
\end{lemme}




\begin{proof}

For similar reasons, it suffices to treat the situation (i) where $v_2$ is deterministic and $N_3\gg N_2$.   Executing {\bf Algorithm A2}, the key point is to control
	$$ \sup_{k,k_1}\underbrace{\Big(\sum_{k_2}\mathbb{E}^{\mathcal{C}_3}|\sigma_{k,k_1;k_2}^{(3)}|^2\Big)}_{\mathrm{I}},\quad \sup_{k,k'}\underbrace{\Big(\sum_{\substack{k_1,k_1'\\
				(k,k_1)\neq (k',k_1') } } \mathbb{E}^{\mathcal{C}_3}\Big|\sum_{k_2}\sigma_{k',k_1';k_2}^{(3)}\ov{\sigma}_{k,k_1;k_2}^{(3)} \Big|^2 \Big)^{\frac{1}{2}} }_{\mathrm{II}},
	$$
	where
	$$ \sigma_{k,k_1;k_2}^{(3)}=\sum_{\substack{k_3^*:|k_3^*|\sim N_3 \\ |k_3-k_3^*|\leq L_3N_3^{\epsilon} } }h_{k+k_2-k_1,k_3^*}^{(q)}\frac{g_{k_3^*}}{[k_3^*]^{\frac{\alpha}{2}}}S_{k_1,k_2,k+k_2-k_1}. 
	$$
	Using the independence and Cauchy-Schwartz
	\begin{align*}
	\mathrm{I}=&\sum_{\substack{k_2,k_3^*,k_3^{'*}\\
			|k+k_2-k_1-k_3^*|<L_3N_3^{\epsilon}\\ |k+k_2-k_1-k_3^{'*}|<L_3N_3^{\epsilon} }}\!\!\!\!\!\!\!\!\! \mathbb{E}^{\mathcal{C}_3}\Big[h_{k+k_2-k_1,k_3^*}^{(q)}\ov{h}_{k+k_2-k_1,k_3^{'*}}^{(q)}\frac{g_{k_3^* }\ov{g}_{k_3^{'*}} }{|k_3^*|^{\frac{\alpha}{2}}|k_3^{'*}|^{\frac{\alpha}{2}}  }S_{k_1,k_2,k+k_2-k_1} \Big]\\
	&\quad\quad\quad\lesssim N_1^{\epsilon}N_3^{-\alpha}\|h_{k_3k_3^*}^{N_3L_3}\|_{l_{k_3}^{\infty}l_{k_3^*}^2}^2,
	\end{align*}
	where we used
	$ \sum_{k_2}S_{k_1,k_2,k+k_2-k_1}\lesssim N_1^{\epsilon}.
	$
	Next,
	\begin{align*}
	\mathrm{II}^2=\sum_{\substack{k_1,k_1',k_2,m_2\\
			(k,k_1)\neq (k_1,k_1') } }&S_{k_1,k_2,k+k_2-k_1}S_{k_1',k_2,k'+k_2-k_1'}S_{k_1,m_2,k+m_2-k_1}S_{k_1',m_2,k'+m_2-k_1'}\\
	\times\sum_{k_3^*,k_3^{'*}, m_3^*,m_3^{'*} }&h_{k+k_2-k_1,k_3^*}\ov{h}_{k'+k_2-k_1',k_3^{'^*}} \ov{h}_{k+m_2-k_1,m_3^*}h_{k'+m_2-k_1',m_3^{'*}} \frac{\mathbb{E}[ g_{k_3^*}\ov{g}_{k_3^{'*}} \ov{g}_{m_3^*}g_{m_3^{'*}}] }{|k_3^*|^{\frac{\alpha}{2}}|k_3^{'*}|^{\frac{\alpha}{2}}|m_3^*|^{\frac{\alpha}{2}}|m_3^{'*}|^{\frac{\alpha}{2}}  } .
	\end{align*}
	The non-zero contributions are $k_3^*=k_3^{'*}, m_3^*=m_3^{'*}$ and $k_3^*=m_3^*, k_3^{'*}=m_3^{'*}$.
	
	For either the case $k_3^*=k_3^{'*}, m_3^*=m_3^{'*}$ or  $k_3^*=m_3^*, k_3^{'*}=m_3^{'*}$, using Cauchy-Schwartz, we can estimate these contributions by
	\begin{align*}
	&N_3^{-2\alpha}\|h_{k_3k_3^*}^{(q)}\|_{l_{k_3}^{\infty}l_{k_3^*}^2 }^4\sum_{\substack{k_1,k_1',k_2,m_2\\ (k,k_1)\neq (k',k_1') } } S_{k_1,k_2,k+k_2-k_1}S_{k_1',k_2,k'+k_2-k_1'}S_{k_1,m_2,k+m_2-k_1}S_{k_1',m_2,k'+m_2-k_1'}\\
	\leq &N_3^{-2\alpha}\|h_{k_3k_3^*}^{(q)}\|_{l_{k_3}^{\infty}l_{k_3^*}^2 }^4\sum_{k_3,k_3',k_2,m_2}S_{k+k_2-k_3,k_2,k_3}S_{k'+k_2-k_3',k_2,k_3'} S_{k+k_2-k_3, m_2, m_2+k_3-k_2 }S_{k'+k_2-k_3',m_2,m_2+k_3'-k_2}.
	\end{align*}
	Since $N_1\gg N_3\gg N_2$, we take the inner sum in the order
	$\sum_{k_2}\sum_{k_3}\sum_{k_3'}\sum_{m_2}$. From
	$$\sum_{m_2}S_{k+k_2-k_3,m_2,m_2+k_3-k_2}\lesssim N_1^{\epsilon}$$
	and
	$$ \sum_{k_3}S_{k+k_2-k_3,k_2,k_3}\sum_{k_3'}S_{k'+k_2-k_3',k_2,k_3'}\lesssim 1,
	$$
	finally we obtain that
	$$ \mathrm{II}^{\frac{1}{4}}\lesssim N_1^{\epsilon}N_3^{-\frac{\alpha}{2}}N_2^{\frac{1}{4}}\|h_{k_3k_3^*}^{(q)}\|_{l_{k_3}^{\infty}l_{k_3^*}^2}.
	$$
By implementing the output \eqref{outputA2-3} and choosing $\epsilon\ll \kappa^{-0.1}$, we complete the proof of Lemma \ref{Case2formallemma1'}. 
\end{proof}


\subsection{Kernel estimates}
To finish the proof of Proposition \ref{kernelZb} and Proposition \ref{KernelSb}, it suffices to estimate the right side of 
\eqref{Sec8-bound0} and \eqref{Sec8-bound1}.
Note that we ignore the error terms appearing in Proposition \ref{modulationreduction1} here, as they are accompanied with some large negative power of $N_1$ which is negligible. 
\vspace{0.3cm}

\noi
$\bullet$ {\bf Proof of Proposition \ref{KernelSb}:} By symmetry, we may assume that $N_2\sim L$ is a fixed parameter and $N_3\leq N_2, L_3<N_3^{1-\delta}$, $L_2<N_2^{1-\delta}$ and $N_1>L^{\frac{1}{1-\delta}}$. First we note that, for fixed $(N_1,N_2,N_3,L_1,L_2,L_3)$, each time using Lemma \ref{2random-Zb}, Lemma \ref{2random-Sb}, Lemma \ref{2random-Zb'}, Lemma \ref{2random-Sb'} and Lemma \ref{Case2formallemma1}, Lemma \ref{Case2formallemma1'}, we should delete a set of probability $<\mathrm{e}^{-RN_1^{\theta}}$. Therefore, the probability of all the exceptional sets is bounded by $$\sum_{N_3\leq L, L_1,L_2,L_3\leq L}\sum_{N_1>L^{\frac{1}{1-\delta}}} O(\mathrm{e}^{-RN_1^{\theta}})=O(\mathrm{e}^{-cRL^{\theta}}),\quad 0<c<1.$$
Therefore, outside a set of probability $<\mathrm{e}^{-cRL^{\theta}}$, we may assume that all the estimates in Lemma \ref{2random-Zb}, Lemma \ref{2random-Sb}, Lemma \ref{2random-Zb'}, Lemma \ref{2random-Sb'} and Lemma \ref{Case2formallemma1}, Lemma \ref{Case2formallemma1'} hold.
To finish the proof, we only need to estimate the sum over $N_3\leq N_2, L_3<N_3^{1-\delta}$ and $L_2<N_2^{1-\delta}$. To simplify the notation, we first pull out all possible small powers of $N_1, N_2, N_3$ like $$N_1^{\kappa^{-1}}, N_1^{\frac{\alpha}{q_0}+\frac{1}{q}}, N_1^{100(b-\frac{1}{2})}, N_2^{\epsilon_1}, N_2^{\epsilon_2}, N_3^{\epsilon_1}, N_3^{\epsilon_2}$$
in a unified one\footnote{The only caution is that the inductive small numbers $\epsilon_1,\epsilon_2$ should not fall on $N_1$. } $N_1^{100(\kappa^{-1}+b-\frac{1}{2}+\frac{1}{q})}N_2^{10(\epsilon_1+\epsilon_2)}$ and \emph{ignore this small power in all the estimates below within this section}. Recall the notation
$$ \Xi_{L_2,L_3}^{N_1,N_2,N_3}=\min\{\Xi_{L_2,L_3}^{N_1,N_2,N_3}(r_2,r_3): (r_2,r_3)\in\{2,q\}\times\{2,q \} \}.
$$
We split the sum into
\begin{align}\label{4terms}
&\underbrace{\sum_{\substack{N_3,L_2,L_3\\ N_3\leq N_2 \\
L_2=2L_{N_2}, L_3=2L_{N_3}   } }\Xi_{L_2,L_3}^{N_1,N_2,N_3} }_{\mathrm{I}}
+\!\!\!\!\!\!\!\underbrace{\sum_{\substack{N_3,L_2,L_3\\N_3\leq N_2 \\
			L_2<2L_{N_2}, L_3<2L_{N_3}   } }\!\!\!\!\!\!\!\Xi_{L_2,L_3}^{N_1,N_2,N_3}}_{\mathrm{II}}\\
	+&\!\!\!\!\!\!\underbrace{\sum_{\substack{N_3,L_2,L_3\\N_3\leq N_2 \\
				L_2=2L_{N_2}, L_3<2L_{N_3}   } }\!\!\!\!\!\! \Xi_{L_2,L_3}^{N_1,N_2,N_3} }_{\mathrm{III}}
		+\underbrace{\sum_{\substack{N_3,L_2,L_3\\N_3\leq N_2 \\
					L_2<2L_{N_2}, L_3=2L_{N_3}   } }\!\!\!\!\!\!\!
				\Xi_{L_2,L_3}^{N_1,N_2,N_3}}_{\mathrm{IV}} \notag.
\end{align} 
Since all the inputs in $\mathrm{I}$ are of type (D), by Lemma \ref{lemA1-456},
\begin{align}\label{sumI'}
\mathrm{I}\lesssim \sum_{N_3: N_3\leq N_2}N_2^{-s}N_3^{-s}\sim N_2^{-s}. 
\end{align}
The inputs in $\mathrm{II}$ are all of type (G) or (C), by Lemma \ref{2random-Sb} and \ref{2random-Sb'}, we have
\begin{align*}
\mathrm{II}\lesssim & \sum_{N_3:N_3\ll N_2}\sum_{L_2,L_3\ll N_3} N_3^{-\frac{\alpha-1}{2}}N_2^{-\frac{\alpha}{2}}L_3^{\frac{1}{2}-\nu}L_2^{-\nu}+\sum_{N_3:N_3\ll N_2}\sum_{\substack{L_3\ll N_3\\
N_3\lesssim L_2\ll N_2 } }N_3^{1-\frac{\alpha}{2}-\nu}N_2^{-\frac{\alpha}{2}}\\
+&\sum_{N_3:N_3\sim N_2}\sum_{\substack{L_2\ll N_2\\ L_3\ll N_3
 } } N_2^{-\frac{7(\alpha-1)}{4}}(L_2L_3)^{-\nu}+N_2^{\frac{1}{2}-\alpha}
\lesssim  \log(N_2)N_2^{-(\alpha-1)-\nu}+N_2^{-\frac{7(\alpha-1)}{4}}.
\end{align*}
Next, applying Lemma \ref{Case2formallemma1'}, we deduce that
\begin{align*}
\mathrm{III}+\mathrm{IV}\lesssim N_2^{-s}+N_2^{-\frac{\alpha}{2}+\frac{1}{4}-s}\log N_2\lesssim N_2^{-s}.
\end{align*}
Therefore, since $\nu< \min\{s,\frac{7(\alpha-1)}{4}\}$ and $L=N_2$,
we have proved Proposition \ref{KernelSb}.

\noi
$\bullet$ {\bf Proof of Proposition \ref{kernelZb}:} 
As before, we ignore a unified factor $N_1^{100\kappa^{-0.1}}N_2^{10(\epsilon_1+\epsilon_2)}$ and split the sum into $\mathrm{I}+\mathrm{II}+\mathrm{III}+\mathrm{IV}$ as \eqref{4terms}. From Lemma \ref{lemA1-123},
$$ \mathrm{I}\lesssim N_1^{1-\frac{\alpha}{2}}N_2^{-s},\quad \text{ if } N_2<N_1^{2-\alpha}
$$
and
$$ \mathrm{I}\lesssim N_1^{1-\frac{\alpha}{2}}N_2^{-s}+N_2^{\frac{1}{2}-2s}<N_1^{1-\frac{\alpha}{2}}N_2^{-s}
\quad \text{ if } N_1^{2-\alpha}\leq N_2(<N_1^{1-\delta}),\text{ since } \alpha<\frac{4}{3}.
$$
To estimate $\mathrm{II}$, we apply Lemma \ref{2random-Zb} and Lemma \ref{2random-Zb'} and obtain that
\begin{align*}
\mathrm{II}\lesssim & \sum_{N_3:N_3\ll N_2}\sum_{\substack{ L_3,L_2\ll N_3\\
L_3>\frac{1}{2} } } (N_2N_3)^{-\frac{\alpha}{2}}N_3^{1-\frac{\alpha}{2}}(N_2^{\frac{1}{2}}+N_1^{1-\frac{\alpha}{2}})L_2^{\frac{1}{2}-\nu}L_3^{-\nu}\\
+&\sum_{N_3:N_3\ll N_2}\sum_{\substack{ L_2\ll N_3
		 } } 
	 (N_2N_3)^{-\frac{\alpha}{2}}N_3^{\frac{1}{2}}(N_2^{\frac{1}{2}}+N_1^{1-\frac{\alpha}{2}})L_2^{-\nu}\\
+&\sum_{N_3:N_3\ll N_2}\sum_{\substack{L_2\ll N_2,L_3\ll N_3\\
L_2\vee L_3\gtrsim N_3 } } N_3^{1-\frac{\alpha}{2}-\nu}N_2^{-\frac{\alpha}{2}}(N_2^{\frac{1}{2}}+N_1^{1-\frac{\alpha}{2}} )\\
+&\sum_{N_3:N_3\sim N_2}\sum_{L_2\ll N_2, L_3\ll N_3} N_1^{1-\frac{\alpha}{2}}N_2^{-\frac{7(\alpha-1)}{4}}(L_2L_3)^{-\nu}+(N_2^{-(\alpha-1)}+N_1^{1-\frac{\alpha}{2}}N_2^{\frac{1}{2}-\alpha}).
\end{align*}
Then direct computation gives
$$ \mathrm{II}\lesssim \log(N_2)N_1^{1-\frac{\alpha}{2}}N_2^{-\frac{\alpha-1}{2}-\nu}+N_1^{1-\frac{\alpha}{2}}N_2^{-\frac{7(\alpha-1)}{4}}.
$$
From Lemma \ref{Case2formallemma1}, 
\begin{align*}
\mathrm{III}\lesssim &\!\!\!\!\!\!\!\!\!\sum_{\substack{N_3, L_2,L_3\\
N_3\leq N_2\\
L_2=2L_{N_2},\; L_3<N_3^{1-\delta} } }\!\!\!\!\!\!\!\!\!(N_1^{1-\frac{\alpha}{2}}+N_3^{\frac{1}{2}})N_2^{-s}N_3^{-\frac{\alpha-1}{2}}
\lesssim  N_1^{1-\frac{\alpha}{2}}N_2^{-s}+N_2^{1-\frac{\alpha}{2}-s}\log(N_2)\lesssim N_1^{1-\frac{\alpha}{2}}N_2^{-s},
\end{align*}
since $L\sim N_2<N_1^{1-\delta}$.
Similarly,
\begin{align*}
\mathrm{IV}\lesssim & \sum_{N_3:N_3\ll N_2}\sum_{\substack{L_2,L_3\\
L_2< N_2^{1-\delta},\; L_3=2L_{N_3} }} N_2^{-\frac{\alpha}{2}}N_3^{\frac{1}{4}-s}L_2^{\frac{1}{4}-\nu}(N_1^{1-\frac{\alpha}{2}}+N_2^{\frac{1}{2}} )\\
+& \sum_{N_3:N_3\sim N_2}\sum_{\substack{L_2,L_3\\
		L_2< N_2^{1-\delta},\; L_3=2L_{N_3} }} (N_1^{1-\frac{\alpha}{2}}+N_3^{\frac{1}{2}})N_2^{-s}N_3^{-\frac{\alpha-1}{2}}\\
\lesssim &(N_1^{1-\frac{\alpha}{2}}+N_2^{\frac{1}{2}} )N_2^{-\frac{\alpha}{2}+\frac{1}{2}-\nu-s}+N_2^{1-\frac{\alpha}{2}-s}\log(N_2)\lesssim N_1^{1-\frac{\alpha}{2}}N_2^{-s}.
\end{align*}
Thus
$$ \mathrm{III}+\mathrm{IV}\lesssim N_1^{1-\frac{\alpha}{2}}N_2^{-s}.
$$
Therefore, since $\nu\leq \min\{s,\frac{7(\alpha-1)}{4}\}-100(\epsilon_1+\epsilon_2)$ and $L=N_2$,
we have proved Proposition \ref{kernelZb}.

\section{Reductions and algorithms for the tri-linear estimates}\label{reductionsection}

\subsection{Reduction on the Fourier supports of type (D) and (C) terms }
Before turning to concrete estimates, we will make some reductions, just in order to clean up the notations and the arguments below. We first reduce the estimate to the case where type (D) terms are localized in the Fourier space, in order to apply Proposition \ref{modulationreduction:trilinear}.  This can be seen as follows: if some $v_j$ is of type (D), we will decompose it as $$\Pi_{N_{(1)}^{10}}v_j+\Pi_{N_{(1)}^{10}}^{\perp}v_j.
$$ 
Then by the bilinear Strichartz inequality, the contributions in $\mathcal{N}_3(v_1,v_2,v_3)$ when we replace $v_j$ by $\Pi_{N_{(1)}^{10}}^{\perp}v_j$ is negligible. Indeed, by duality, to estimate $\iint v_1\ov{v}_2v_3\ov{v}dxdt$, we split each $v_i$ into dyadic pieces and we have to estimate
$$ \sum_{M_1,M_2,M_3,M}\iint \mathbf{P}_{M_1}v_1\cdot\mathbf{P}_{M_2}\ov{v_2}\cdot \mathbf{P}_{M_3}v_3\cdot \mathbf{P}_M\ov{v}dxdt
$$
where at least one of $M_1,M_2,M_3$ is greater than $N_{(1)}^{10}$ and the corresponding function $v_j$ is of type (D). In particular, if for some $k$, $M_k=\max\{M_1,M_2,M_3\}$, then $M_k> N_{(1)}^{10}$ and $v_k$ must be of type (D), since each $v_j$ of type (G) or (C) has Fourier support $|k_j|\leq N_j\leq N_{(1)}$. By H\"older and the bilinear Strichartz inequality, we can bound the sum by
$$ \sum_{\substack{ M_1,M_2,M_3,M\\
		M_{(1)\geq N_{(1)}^{10}  }} } (M_{(2)}M_{(3)})^{\frac{1}{2}-\frac{\alpha}{4}} \|\mathbf{P}_{M_1}v_1\|_{X^{0,\frac{3}{8}}}
\|\mathbf{P}_{M_2}v_2\|_{X^{0,\frac{3}{8}}}
\|\mathbf{P}_{M_3}v_3\|_{X^{0,\frac{3}{8}}}
\|\mathbf{P}_{M}v\|_{X^{0,\frac{3}{8}}}.
$$ 
Note that if $v_{(j)}, j=2,3$ is not of type (D), then $M_{(j)}\leq N_{(1)}, j=2,3$, the dyadic summation converges and is bounded by $N_{(1)}^{-10s}$, which is negligible. Therefore, from now on, we always assume that a term $v_j$ of type (D) has Fourier support $|k_j|\leq N_{(1)}^{10}$, hence the modulation reduction, Proposition \ref{modulationreduction:trilinear} is always applicable. 

Next, we claim that without loss of generality, 
\begin{itemize}
	\item[{\bf(H)}] for $v_j$ of type (D) or (C), we may further assume that $$\mathrm{supp}_{k_j}(\widehat{v}_j)\subset \{k_j:|k_j|\sim N_j \} $$
\end{itemize}
	and our goal is to prove an estimate of the form
	\begin{align}\label{enhancedslightly}
	\|\mathcal{I}\mathcal{N}(v_1,v_2,v_3)\|_{X^{0,b_1}}\lesssim RN_{(1)}^{-s-\delta_0}.
	\end{align}
Indeed, if $v_j$ is of type (D), we can decompose it as
$$ v_j=\sum_{M_j\leq N_{(1)}^{10} }\mathbf{P}_{M_j}v_j.
$$  
By our assumption that $v_j$ has characterized frequency $N_j$, we have
$$ \|\mathbf{P}_{M_j}v_j\|_{X^{0,b}}\lesssim M_j^{-s}, \text{ if } M_j\geq N_j;\quad \|\mathbf{P}_{M_j}v_j\|_{X^{0,b}}\lesssim N_j^{-s}, \text{ if } M_j<N_j.
$$
Then to estimate the $X^{0,1-b_1}$ norm of $\mathcal{N}(v_1,v_2,v_3)$, we can replace such $v_j$ by $\mathbf{P}_{M_j}v_j$ and then sum over every $M_j\leq N_{(1)}^{10}$. Note that the dyadic sum of $N_j^{-s}$ over $M_j\leq N_j$ only contributes $N_j^{-s}\log(N_j)$ and the dyadic sum of $M_j^{-s}$ over $M_j>N_j$ contributes $N_j^{-s}$, finally the logarithmic loss and the loss from the small powers $N_{(1)}^{100(\frac{1}{q}+b_1-0.5+\theta+\kappa^{-0.1})}$ will be compensate by $N_{(1)}^{-\delta_0}$, thanks to Remark \ref{numerical}.  Similarly, if $v_j$ is of type (C), we decompose it as
$$ v_j=\sum_{|k_j|\sim N_j}\mathrm{e}_{k_j}\sum_{\substack{|k_j^*|\sim N_j,\\
		|k_j-k_j^*|\leq L_jN_j^{\epsilon} }  } h_{k_jk_j^*}^{N_jL_j}(t)\frac{g_{k_j^*}(\omega) }{[k_j^*]^{\frac{\alpha}{2}}}+\sum_{|k_j|\leq N_j}\mathrm{e}_{k_j}\sum_{\substack{|k_j^*|\sim N_j,\\
		|k_j-k_j^*|> L_jN_j^{\epsilon} }  } h_{k_jk_j^*}^{N_jL_j}(t)\frac{g_{k_j^*}(\omega) }{[k_j^*]^{\frac{\alpha}{2}}},
$$
where $\epsilon<\kappa^{-0.1}$.
The second term is negligible in the estimate, compared with the first term on the right side. Indeed, we can treat the second term as a function of type (D)\footnote{We then perform the same Littlewood-Paley decomposition as we just did for the true type (D) terms. } with the $X^{0,b}$ bound  $N_j^{10}\cdot N_j^{1-\kappa\epsilon}\ll N_j^{-10}$, thanks to \eqref{almostlocal}. From these discussions, we always assume {\bf(H)} in the sequel and proceed to prove \eqref{enhancedslightly}.

\subsection{Reduction to the corresponding dyadic summations}

Except for the high-high-high interactions with at least one term of type (D), we can reduce the matter to several modes of summation, depending on how many random structures we want to exploit.

First, applying Proposition \ref{modulationreduction:trilinear},
\begin{align*}
\|\mathcal{I}\mathcal{N}_3(v_1,v_2,v_3)\|_{X^{0,b_1}}\lesssim  &N_{(1)}^{-100}\prod_{j=1}^3\|v_j\|_{X^{0,b}}
\\+&\Big\|\mathcal{M}_{L_1,L_2,L_3}^{N_1,N_2,N_3}(\lambda_1,\lambda_2,\lambda_3,\mu_0,k) \Big\|_{L_{\lambda_1}^{r_1}L_{\lambda_2}^{r_2}L_{\lambda_3}^{r_3}L_{\mu_0}^{q_0}(|\mu_0|\lesssim N_{(1)}^{\alpha})l_k^2 },
\end{align*}
where
$$ \mathcal{M}_{L_1,L_2,L_3}^{N_1,N_2,N_3}=\sum_{\substack{(k_1,k_2,k_3)\in\Gamma(k) } }\widehat{\chi}(\mu_0-\Phi_{k_1,k_2,k_3})\widetilde{w}_1(\lambda_1,k_1)\ov{\widetilde{w}}_2(\lambda_2,k_2)\widetilde{w}_3(\lambda_3,k_3)  ,
$$
$q_0=\frac{1}{b_1-0.5}$, and $\widetilde{w}_j^{(r_j)}(\lambda_j,k_j)=\langle\lambda_j\rangle^{\frac{2b_0}{r_j'}}\widetilde{v}_j(\lambda_j,k_j)$ for $j=1,2,3.$  We can ignore the error $N_{(1)}^{-100}\prod_{j=1}^3\|v_j\|_{X^{0,b_0}}$ and concentrate to the estimate of $\mathcal{M}_{L_1,L_2,L_3}^{N_1,N_2,N_3}$. As before, we may replace $\widehat{\chi}_0(\mu_0-\Phi_{k_1,k_2,k_3})$ by $\mathbf{1}_{\Phi_{k_1,k_2,k_3}=\mu_0+O(N_{(1)}^{\epsilon})}.$ 
Moreover, we denote simply by $h_{k_jk_j^*}^{(r_j)}=h_{k_jk_j^*}^{N_jL_j,(r_j)}:=\langle\lambda_j\rangle^{\frac{2b}{r_j'}}\widetilde{h}_{k_jk_j^*}^{N_jL_j}(\lambda_j)$ when there is no risk of confusing. Recall also the notations:
$$ S_{k_1,k_2,k_3}:=\mathbf{1}_{k_2\neq k_1,k_3}\mathbf{1}_{\Phi_{k_1,k_2,k_3}=\mu_0+O(N_{(1)}^{\epsilon})},\quad \|\cdots\|_{L_{\mu_0*}^{q_0}}:=\|\cdots\|_{L_{\mu_0}^{q_0}(|\mu_0|\lesssim N_{(1)}^{\alpha}) }
$$


\noi
$\bullet$ {\bf Algorithm 1. Prototype: $v_1,v_2,v_3$ are all of type (G) or (C) }

Our algorithm in this case can be described as follows:\\
	(a) Denote by $\mathcal{C}=\mathcal{B}_{\leq \max\{L_1,L_2,L_3\} }$, then $\mathcal{M}_{L_1,L_2,L_3}^{N_1,N_2,N_3}$ is a tri-linear expression of Gaussian variables. When the coefficients are $\mathcal{C}-$measurable and are independent of all Gaussians in the expansion, we can apply Corollary \ref{largedeviation} to deduce that outside a set of probability $<\mathrm{e}^{-N_{(1)}^{\theta}R^{\frac{2}{3}}}$,
			$$ \|\mathcal{M}_{L_1,L_2,L_3}^{N_1,N_2,N_3} \|_{L_{\lambda_1}^qL_{\lambda_2}^{q}L_{\lambda_3}^{q}L_{\mu_0*}^{q_0}l_k^2 }\leq RN_{(1)}^{\theta}\|(\mathbb{E}^{\mathcal{C}}[|\mathcal{M}_{L_1,L_2,L_3}^{N_1,N_2,N_3}|^2])^{\frac{1}{2}} \|_{L_{\lambda_1}^qL_{\lambda_2}^{q}L_{\lambda_3}^{q}L_{\mu_0*}^{q_0}l_k^2 }.
			$$
	(b)
	As the crucial step, for fixed $\lambda_1,\lambda_2,\lambda_3, |\mu_0|\lesssim N_{(1)}^{\alpha}$, we need to establish the following estimate:
	$$ \mathbb{E}^{\mathcal{C}}[\|\mathcal{M}_{L_1,L_2,L_3}^{N_1,N_2,N_3} \|_{l_k^2}^2 ]\leq K(N_1,N_2,N_3;L_1,L_2,L_3)\prod_{j=1}^3\|h_{k_jk_j^*}^{N_jL_j,(q)}(\lambda_j)\|_{l_{k_j}^{\infty}l_{k_j^*}}^2.
	$$
	(c)
	Using the embedding $l^{\infty}\hookrightarrow l^q$, taking the square root of the output of the step (b) and then taking the $L_{\lambda_1}^2L_{\lambda_2}^{q}L_{\lambda_3}^{q}L_{\mu_0*}^{q_0}$ norm, we obtain that
	\begin{align*}
	 \hspace{1.2cm} \|\mathcal{M}_{L_1,L_2,L_3}^{N_1,N_2,N_3} \|_{L_{\lambda_1}^qL_{\lambda_2}^qL_{\lambda_3}^qL_{\mu_0*}^{q_0}L_{\omega}^2 } \leq N_{(1)}^{\frac{\alpha}{q_0}+\frac{3}{q}}K(N_1,N_2,N_3,L_1,L_2,L_3)^{\frac{1}{2}}\prod_{j=1}^3\|h^{N_jL_j}\|_{S^{b,q}}.
	\end{align*}
	From this algorithm, in practice,  only the step (b) is not robust. By abusing the notations a bit, we can forget the modulation variable and ignore all the small powers of $N_{1},N_2,N_3$ and write each $\widetilde{w}_j(\lambda_j,k_j)$ (with characterized parameters $(N_j,L_j)$ ) simply as $a_j(k_j)=a_j(k_j)\mathbf{1}_{|k_j|\sim N_j}$ and assume that:\\ 
	\begin{itemize}
		\item If $a_j$ is of type (C),
	\begin{align}\label{boundCsimplified}
	 \|h_{k_jk_j^*}^{N_jL_j}\|_{l_{k_j}^{\infty}l_{k_j^*}^2}\leq L_j^{-\nu};
		\end{align} 
\item If $a_j$ is of type (D),
	\begin{align}\label{boundDsimplified}
     \|a_j(k_j)\|_{l_{k_j}^2}\lesssim N_j^{-s};
		\end{align}
	\item 
	If $a_j$ is of type (G),
	\begin{align}\label{boundGsimplified}
	 \|a_j(k_j)\|_{l_{k_j}^2}\lesssim N_j^{-\frac{\alpha-1}{2}}.
		\end{align}
			\end{itemize}
	Moreover, to make the notations cleaner, we will ignore all the small powers (in terms of $\frac{1}{q},\frac{1}{q_0}\sim b_1-0.5,\epsilon<\kappa^{-0.1},\epsilon_1,\epsilon_2,\delta_0,\cdots$) of $N_j$ and 
	finally we multiply by a unified factor $$N_{(1)}^{\frac{\alpha}{q_0}+\frac{3}{q}+\theta+100\kappa^{-0.1}+3\epsilon_1+3\epsilon_2}$$ to the output. 
	
	In summary, by {\bf Algorithm 1}, we need to establish an estimate
	\begin{align}\label{Alg1}
	\mathbb{E}[|\mathcal{U}_{L_1,L_2,L_3}^{N_1,N_3,N_3}|^2]\leq K_1(N_1,N_2,N_3),
	\end{align}
	where
	$$ \mathcal{U}_{L_1,L_2,L_3}^{N_1,N_2,N_3}=\Big(\sum_{k}\Big|\sum_{(k_1,k_2,k_3)\in\Gamma(k)} a_1(k_1)\ov{a}_2(k_2)a_3(k_3) \Big|^2\Big)^{\frac{1}{2}},
	$$
	with $a_1,a_2,a_3$ having characterized parameters $(N_1,L_1),(N_2,L_2),(N_3,L_3)$, with respectively, satisfying corresponding estimates \eqref{boundCsimplified},\eqref{boundDsimplified} and \eqref{boundGsimplified}. Then the output is
	$$ 
	\boxed{\text{Output of Algorithm 1}\leq RN_{(1)}^{\frac{\alpha}{q_0}+\frac{3}{q}+\theta+100(\kappa^{-0.1}+\epsilon_1+\epsilon_2)}K_1(N_1,N_2,N_3)^{\frac{1}{2}}}
	$$

\noi
$\bullet$ {\bf Algorithm 2. Prototype: at least two of $v_1,v_2,v_3$ are of type (G) or (C)  }

We may assume that $v_1,v_2$ are of type (G) or (C) and we denote by $\mathcal{C}=\mathcal{B}_{\leq (L_1\vee L_2)}$. The algorithm can be described as follows:
\\
(a)
	$$ \|\mathcal{M}_{L_1,L_2,L_3}^{N_1,N_2,N_3}\|_{L_{\lambda_1}^qL_{\lambda_2}^qL_{\lambda_3}^2L_{\mu_0*}^{q_0}l_{k}^2 }\leq \|\mathcal{G}_3\mathcal{G}_3^*\|_{L_{\lambda_1,\lambda_2}^{\frac{q}{2}}L^{\frac{q_0}{2}}_{\mu_0*}\mathcal{L}(l_{k_3}^2)}^{\frac{1}{2}} \|\widetilde{w}_3(\lambda_3,k_3)\|_{l_{k_3}^2},
	$$
	where the kernel of the random operator $\mathcal{G}_3\mathcal{G}_3^*$ (depending on $\lambda_1,\lambda_2$ and $\mu_0$) is given by
	\begin{align}\label{randomopkernel1}
	\sigma_{k,k'}^{(3)}=\sum_{k_2,k_2',k_3}&\ov{\widetilde{w}}_1(\lambda_1,k'+k_2'-k_3)\widetilde{w}_1(\lambda_1,k+k_2-k_3)\widetilde{w}_2(\lambda_2,k_2')\ov{\widetilde{w}}_2(\lambda_2,k_2)\\\times &S_{k'+k_2'-k_3,k_2',k_3}S_{k+k_2-k_3,k_2,k_3}. \notag
	\end{align}
	Then by Lemma \ref{matrixboundoffdiagonal},
	$$ \|\mathcal{G}_3\mathcal{G}_3^*\|_{l_{k'}^2\rightarrow l_k^2}\leq L\sup_{k,k':|k-k'|< L}|\sigma_{k,k'}^{(3)}|+\Big(\sum_{k,k':|k-k'|\geq L}|\sigma_{k,k'}^{(3)}|^2 \Big)^{\frac{1}{2}}.
	$$
	(b)
	Since $|\sigma^{(3)}_{k,k'}|^2$ is a bilinear expression of Gaussian variables, when the coefficients are independent of these Gaussians in the expansion, by Corollary \ref{largedeviation}, outside a set of probability $<\mathrm{e}^{-N_1^{\theta}R}, $ we have
			$$ \Big\|\Big(\sum_{k,k':|k-k'|\geq L}\!\!\!\!\!\!\!\!|\sigma_{k,k'}^{(3)}|^2 \Big)^{\frac{1}{2}} \Big\|_{L_{\lambda_1,\lambda_2}^{\frac{q}{2}}L_{\mu_0*}^{\frac{q_0}{2}} }\leq RN_1^{\theta} \Big\|\Big(\sum_{k,k':|k-k'|\geq L}\!\!\!\!\!\!\mathbb{E}^{\mathcal{C}}[|\sigma_{k,k'}^{(3)}|^2] \Big)^{\frac{1}{2}} \Big\|_{L_{\lambda_1,\lambda_2}^{\frac{q}{2}}L_{\mu_0*}^{\frac{q_0}{2}} },
			$$ 	
			and
			$$ \big\|\sup_{k,k':|k-k'|< L}|\sigma_{k,k'}^{(3)}|^2 \big\|_{L_{\lambda_1,\lambda_2}^{\frac{q}{2}}L_{\mu_0*}^{\frac{q_0}{2}} }\leq RN_1^{\theta} 
			\big\|\sup_{k,k':|k-k'|< L}\mathbb{E}^{\mathcal{C}}[|\sigma_{k,k'}^{(3)}|^2] \big\|_{L_{\lambda_1,\lambda_2}^{\frac{q}{2}}L_{\mu_0*}^{\frac{q_0}{2}}}.
			$$
				The main step is to establish the estimates
	$$ \sum_{k,k':|k-k'|\geq L}\mathbb{E}^{\mathcal{C}}[|\sigma_{k,k'}^{(3)}|^2] \leq K(N_1,N_2,N_3,L_1,L_2,L_3)\prod_{j=1}^2\|h_{k_jk_j^*}^{N_jL_j,(q)}\|_{l_{k_j}^{\infty}l_{k_2^*}^2}^2
	$$
	and
	$$ \sup_{k,k':|k-k'|<L}\mathbb{E}^{\mathcal{C}}[|\sigma_{k,k'}^{(3)}|^2 ]\leq K'(N_1,N_2,N_3,L_1,L_2,L_3)\prod_{j=1}^2\|h_{k_jk_j^*}^{N_jL_j,(q)}\|_{l_{k_j}^{\infty}l_{k_2^*}^2}^2.
	$$
(c) Having the bounds in the step (b), we deduce that outside a set of probability $<\mathrm{e}^{-N_1^{\theta}R}$, 
	\begin{align*}
	 &\|\mathcal{G}_3\mathcal{G}_3^*\|_{L_{\lambda_1,\lambda_2}^{\frac{q}{2}}L_{\mu_0*}^{\frac{q_0}{2}}\mathcal{L}(l_{k_3}^2)}^{\frac{1}{2}}\\ \leq &RN_1^{\theta+\frac{\alpha}{q_0}+\frac{2}{q} }\big(L^{\frac{1}{2}}K'(N_1,N_2,N_3,L_1,L_2,L_3)^{\frac{1}{4}}+K(N_1,N_2,N_3,L_1,L_2,L_3)^{\frac{1}{4}} \big)\prod_{j=1}^2\|h^{N_jL_j}\|_{S^b}^{\frac{1}{2}}.
	\end{align*}
	Again, since only the estimates for the expectations in the step (b) are not robust, by abusing the notations a bit, we may write each $\widetilde{w}_j(\lambda_j,k_j)$ (with characterized parameters $(N_j,L_j)$ ) simply as $a_j(k_j)=a_j(k_j)\mathbf{1}_{|k_j|\sim N_j}$ and making the same assumptions \eqref{boundCsimplified},\eqref{boundDsimplified} and \eqref{boundGsimplified} for type (C), (D) and (G) terms as in {\bf Algorithm 1}, with respectively. Again, to make the notations cleaner, we will ignore all the small powers (in terms of $\epsilon<\kappa^{-0.1},\epsilon_1,\epsilon_2,\delta_0,\frac{1}{q},\frac{1}{q_0}\sim b_1-0.5$) of $N_j$ and 
	finally we multiply by a unified factor $N_{(1)}^{\frac{\alpha}{q_0}+\frac{2}{q}+\theta+100\kappa^{-0.1}+3\epsilon_1+3\epsilon_2}$ to the output.
	
	Similarly, if $v_1,v_3$ are of type (G) or (C), we denote by $\mathcal{C}=\mathcal{B}_{\leq (L_1\vee L_3)}$ we will apply the above algorithm to the operator $\mathcal{G}_2\mathcal{G}_2^*$ with matrix elements
	\begin{align}\label{randomopkernel2}
	\sigma_{k,k'}^{(2)}=\sum_{k_3,k_3',k_2}&\ov{\widetilde{w}}_1(\lambda_1,k'+k_2-k_3')\widetilde{w}_1(\lambda_1,k+k_2-k_3)\widetilde{w}_3(\lambda_3,k_3')\ov{\widetilde{w}}_3(\lambda_3,k_3)\\\times &S_{k'+k_2-k_3
		',k_2,k_3'}S_{k+k_2-k_3,k_2,k_3}. \notag
	\end{align}
	In summary, by {\bf Algorithm 2}, 
	we need to establish the bound:
	$$ \sum_{j=2}^3\mathbb{E}^{\mathcal{C}}\Big[\sum_{k,k':|k-k'|\geq L}|\eta_{k,k'}^{(j)}|^2 \Big]+\sup_{k,k':|k-k'|<L}\mathbb{E}^{\mathcal{C}}[|\eta_{k,k'}^{(j)}|^2] \leq K_2(N_1,N_2,N_3),
	$$
	where
\begin{align}\label{etakk3} 
 \eta_{k,k'}^{(3)}=&\sum_{k_2,k_2',k_3}\ov{a}_1(k'+k_2'-k_3)a_1(k+k_2-k_3)a_2(k_2')\ov{a}_2(k_2)S_{k'+k_2'-k_3,k_2',k_3}S_{k+k_2-k_3,k_2,k_3},
\end{align}
	\begin{align}\label{etakk2}  \eta_{k,k'}^{(2)}=\sum_{k_3,k_2',k_2}\ov{a}_1(k'+k_2-k_3')a_1(k+k_2-k_3)a_3(k_3')\ov{a}_3(k_3)S_{k'+k_2-k_3',k_2,k_3'}S_{k+k_2-k_3,k_2,k_3}.
	\end{align}
	Then 
	$$ \boxed{\text{ Output of Algorithm 2}\leq RN_1^{\frac{\alpha}{q_0}+\frac{2}{q}+\theta+100(\kappa^{-0.1}+\epsilon_1+\epsilon_2)}K_2(N_1,N_2,N_3)^{\frac{1}{4}}\|a_2\|_{l^2}}
	$$
	or
	$$\boxed{\text{ Output of Algorithm 2}\leq RN_1^{\frac{\alpha}{q_0}+\frac{2}{q}+\theta+100(\kappa^{-0.1}+\epsilon_1+\epsilon_2)}K_2(N_1,N_2,N_3)^{\frac{1}{4}}\|a_3\|_{l^2}.}
	$$


\noi
$\bullet${\bf Algorithm 3. Prototype: at least one term of type (G) or (C) }
Without loss of generality, we assume that $v_1$ is of type (G) or (C). By Cauchy-Schwartz, we have
\begin{align*}
&\|\mathcal{M}_{L_1,L_2,L_3}^{N_1,N_2,N_3}\|_{L_{\lambda_1}^qL_{\lambda_2,\lambda_3}^2L_{\mu_0*}^{q_0}l_k^2 }\\ \leq & \|\widetilde{w}_1(\lambda_1,k_1)\|_{L_{\lambda_1}^ql_{k_1}^{\infty}}
\Big\|\Big(\sum_{|k|\leq N_{(1)}}
\Big|\sum_{\substack{(k_1,k_2,k_3)\in\Gamma(k) }  } \ov{\widetilde{w}}_2(\lambda_2,k_2)\widetilde{w}_3(\lambda_3,k_3)     
\Big|^2 
\Big)^{\frac{1}{2}} \Big\|_{L_{\lambda_2,\lambda_3}^2L_{\mu_0*}^{q_0} }.
\end{align*}
From the embedding $l^{\infty}\hookrightarrow l^q$, H\"older's inequality and Lemma \ref{paracontrolregularityX}, we deduce that, outside a set of probability $<\mathrm{e}^{-cN^{\theta}R^2}$, 
$$\|\widetilde{w}_1(\lambda_1,k_1)\|_{L_{\lambda_1}^2l_{k_1}^{\infty}}\leq N_1^{\frac{1}{q}}\|\widetilde{w}_1(\lambda_1,k_1)\|_{l_{k_1}^{\infty}L_{\lambda_1}^q}\lesssim  N_1^{-\frac{\alpha}{2}+\epsilon_1+\frac{1}{q}}L^{-\nu}.$$ Again, since the key step is to estimate the discrete sum, and the $L_{\lambda_2,\lambda_3}^2L_{\mu_0*}^{q_0}$ will only contribute a $N_0^{\frac{\alpha}{q_0}}$ factor, we may write each $\widetilde{w}_j(\lambda_j,k_j)$ (with characterized parameters $(N_j,L_j)$ ) simply as $a_j(k_j)=a_j(k_j)\mathbf{1}_{|k_j|\sim N_j}$ and assume that:
\begin{itemize}
	\item
 If $a_j$ is of type (C),
\begin{align}\label{boundCsimplified'}
 \|a_j(k_j)\|_{l_{k_j}^{\infty}}\leq N_j^{-\frac{\alpha}{2}}L_j^{-\nu},
\text{ and }
 \|h_{k_jk_j^*}^{N_jL_j}\|_{l_{k_j}^{\infty}l_{k_j^*}^2}\leq L_j^{-\nu}
\end{align}
\item If $a_j$ is of type (D),
\begin{align}\label{boundDsimplified'}
 \|a_j(k_j)\|_{l_{k_j}^2}\lesssim N_j^{-s}
\end{align}
\item If $a_j$ is of type (G),
\begin{align}\label{boundGsimplified'}
 \|a_j(k_j)\|_{l_{k_j}^2}\lesssim N_j^{-\frac{\alpha-1}{2}},\; \|a_j(k_j)\|_{l_{k_j}^{\infty}}\leq N_j^{-\frac{\alpha}{2}}.
\end{align}
\end{itemize}
Similarly, to make the notations cleaner, we will ignore all the small powers (in terms of $\epsilon<\kappa^{-0.1},\epsilon_1,\epsilon_2,\delta_0,\frac{1}{q},\frac{1}{q_0}\sim b_1-0.5$) of $N_j$ and 
finally we multiply by $N_{(1)}^{\frac{\alpha}{q_0}+\frac{1}{q}+\theta+100\epsilon+3\epsilon_1+3\epsilon_2}$ to the output.

In summary, by {\bf Algorithm 3}, it suffices to establish the bound
$$ \mathcal{U}_{L_1,L_2,L_3}^{N_1,N_2,N_3}\leq K_3(N_1,N_2,N_3), 
$$ 
where
$$ \mathcal{U}_{L_1,L_2,L_3}^{N_1,N_2,N_3}=\Big(\sum_{k}\Big|\sum_{(k_1,k_2,k_3)\in\Gamma(k)} a_1(k_1)\ov{a}_2(k_2)a_3(k_3) \Big|^2\Big)^{\frac{1}{2}},
$$
with $a_1,a_2,a_3$ having characterized parameters $(N_1,L_1),(N_2,L_2),(N_3,L_3)$, with respectively, satisfying corresponding estimates \eqref{boundCsimplified'},\eqref{boundDsimplified'} and \eqref{boundGsimplified'}. Then the output is
$$ 
\boxed{\text{Output of Algorithm 3}\leq RN_{(1)}^{\frac{\alpha}{q_0}+\frac{1}{q}+\theta+100(\kappa^{-0.1}+\epsilon_1+\epsilon_2)}K_3(N_1,N_2,N_3).}
$$


\noi
$\bullet${\bf Algorithm 4. Prototype: all of type (D) }

In this case, the simple algorithm is to obtain an estimate of the type
$$ \|\mathcal{M}_{L_1,L_2,L_3}^{N_1,N_2,N_3}\|_{l_k^2}\leq K(N_1,N_2,N_3,L_1,L_2,L_3)\prod_{j=1}^3\|\widetilde{w}_j(\lambda_j,k_j)\|_{l_{k_j}^2} 
$$ 
and then take the $L_{\lambda_1,\lambda_2,\lambda_3}^2L_{\mu_0*}^{q_0}$ norm.

In summary, by {\bf Algorithm 4}, it suffices to establish the bound
$$ \mathcal{U}_{L_1,L_2,L_3}^{N_1,N_2,N_3}\leq K_4(N_1,N_2,N_3), 
$$ 
where
$$ \mathcal{U}_{L_1,L_2,L_3}^{N_1,N_2,N_3}=\Big(\sum_{k}\Big|\sum_{(k_1,k_2,k_3)\in\Gamma(k)} a_1(k_1)\ov{a}_2(k_2)a_3(k_3) \Big|^2\Big)^{\frac{1}{2}},
$$
with $a_1,a_2,a_3$ having characterized parameters $(N_1,L_1),(N_2,L_2),(N_3,L_3)$, with respectively, satisfying corresponding estimates \eqref{boundDsimplified'}. Then the output is
$$\boxed{  
\text{Output of Algorithm 4}\leq RN_{(1)}^{\frac{\alpha}{q_0}+\theta+100(\kappa^{-0.1}+\epsilon_1+\epsilon_2)}K_4(N_1,N_2,N_3).}
$$

\begin{remarque}
	Note that Algorithm 3 and Algorithm 4 are purely deterministic, the only difference is that the upper bound $K_3(N_1,N_2,N_3)$ is formed by one $l_{k_j}^{\infty}l_{k_j^*}^2$ norm and two $l_{k_jk_j^*}^2$ norms, while the upper bound $K_4(N_1,N_2,N_3)$ is formed by three $l_{k_jk_j^*}^2$ norms.
\end{remarque}


\section{Tri-linear estimates 1: high-high-high interactions}\label{sec:high-high-high}

In this section, we will prove (8) of Proposition \ref{Multilinearkey} and (1) of Proposition \ref{Multilinearkey} in the case $N_1\sim N_2\sim N_3$.
\subsection{Diagonal terms}
For $v_1,v_2,v_3$ of type (G),(C) or (D) with characterized parameters $(N_1,L_1),(N_2,L_2)$ and $(N_3,L_3)$, note that
\begin{align*}
 &\mathcal{F}_{t,x}(\chi(t)\mathcal{N}_0(v_1,v_2,v_3))(\lambda+|k|^{\alpha},k)\\=&\int \wh{\chi}(\lambda-\lambda_1+\lambda_2-\lambda_3)\wt{v}_1(\lambda_1,k)\ov{\wt{v}}_2(\lambda_2,k)\wt{v}_3(\lambda_3,k)d\lambda_1 d\lambda_2 d\lambda_3.
 \end{align*}
Our goal is to show that
\begin{equation}\label{goal-diagonal}
\begin{split} \mathcal{M}:=&\Big\|\lg\lambda\rg^{b_1-1}\int \wh{\chi}(\lambda-\lambda_1+\lambda_2-\lambda_3)\wt{v}_1(\lambda_1,k)\ov{\wt{v}}_2(\lambda_2,k)\wt{v}_3(\lambda_3,k)d\lambda_1 d\lambda_2 d\lambda_3 \Big\|_{L_{\lambda}^2l_n^2}\\ \lesssim &N_{(1)}^{-s}N_{(2)}^{-\delta_0}.
\end{split}
\end{equation}
Note that we will omit the similar argument to treat $\Pi_{N_0}^{\perp}\mathcal{I}\mathcal{N}_0(v_1,v_2,v_3)$ when $N_0\gg N_{(1)}$.
For $r_j\in\{2,q\}$, we denote by $$V_j^{(r_j)}(\lambda_j,k)=\lg\lambda_j\rg^{\frac{2b_0}{r_j'}}\wt{v}_j(\lambda_j,k),\; f_{j}(\lambda_j)=\|V_j^{(2)}(\lambda_j,k)\|_{l_{k_j}^2},\text{ and } \; g_j(\lambda_j)=\|V_j^{(q)}(\lambda_j,k)\|_{l_k^{q}}.$$ Note that when $v_j$ is of type (G) or (C), the spatial Fourier support is constraint on $|k_j|\leq N_j$, thus
\begin{align}\label{boundgj}
\|g_j(\lambda_j)\|_{L_{\lambda_j}^q}\leq N_j^{\frac{1}{q}}\|v_j\|_{X_{\infty,q}^{0,\frac{2b_0}{q'}}}. 
\end{align}
\vspace{0.3cm}
\noi
$\bullet${\bf Case 1: at least one of $v_1,v_2,v_3$ is of type (D)}

First we assume that $v_1,v_2,v_3$ are all of type (D). From the embedding $l^2\hookrightarrow l^{\infty}$, H\"older's inequality and Lemma \ref{convolution}, we have
\begin{align*}
\mathcal{M}\lesssim & \int \frac{f_1(\lambda_1)f_2(\lambda_2)f_3(\lambda_3)}{\lg\lambda_1-\lambda_2+\lambda_3\rg^{1-b_1}\lg\lambda_1\rg^{b_0}\lg\lambda_2\rg^{b_0}\lg\lambda_3\rg^{b_0}  }d\lambda_1 d\lambda_2 d\lambda_3\\
\lesssim &\prod_{j=1}^3\|f_j(\lambda_j)\|_{L_{\lambda_j}^2}=\prod_{j=1}^3\|v_j\|_{X^{0,b_0}}\lesssim (N_1N_2N_3)^{-s},
\end{align*}
which is conclusive.

Next, assume that exact one of $v_1,v_2,v_3$ is of type (G) or (C), say, $v_1$ is of type (G) or (C), then from the same argument together with the embedding $l^q\hookrightarrow l^{\infty}$, we have
\begin{align*}
\mathcal{M}=&\Big\|\lg\lambda\rg^{b_1-1}\int \wh{\chi}(\lambda-\lambda_1+\lambda_2-\lambda_3)\frac{V_1^{(q)}(\lambda_1,k)\ov{V}_2^{(2)}(\lambda_2,k)V_3^{(2)}(\lambda_3,k)  }{\lg\lambda_1\rg^{\frac{2b_0}{q'}}\lg\lambda_2\rg^{b_0}\lg\lambda_3\rg^{b_0} }d\lambda_1 d\lambda_2 d\lambda_3 \Big\|_{L_{\lambda}^2l_k^2}\\
\lesssim & \int \frac{g_1(\lambda_1)f_2(\lambda_2)f_3(\lambda_3)  }{\lg\lambda_1-\lambda_2+\lambda_3 \rg^{1-b_1}\lg\lambda_1\rg^{\frac{2b_0}{q'}}\lg\lambda_2\rg^{b_0}\lg\lambda_3\rg^{b_0} }d\lambda_1 d\lambda_2 d\lambda_3\\
\lesssim & \|g_1(\lambda_1)\|_{L_{\lambda_1}^q}\|f_2(\lambda_2)\|_{L_{\lambda_2}^2}\|f_3(\lambda_3)\|_{L_{\lambda_3}^2}\\ \leq &N_1^{\frac{1}{q}}\|v_1\|_{X_{\infty,q}^{0,\frac{2b_0}{q'}}}\|v_2\|_{X^{0,b_0}}\|v_3\|_{X^{0,b_0}}\lesssim N_1^{-\frac{\alpha}{2}+\frac{1}{q}+\epsilon_2}N_2^{-s}N_3^{-s},
\end{align*}
which is conclusive.

If there are exact two of $v_1,v_2,v_3$ of type (G) or (C) and the other is of type (D), the estimate is similar and we obtain that $$\mathcal{M}\lesssim N_{(1)}^{-s}(N_{(2)}N_{(3)})^{-\frac{\alpha}{2}+\frac{1}{q}+\epsilon_2},
$$
which is also conclusive.



\vspace{0.3cm}
\noi
$\bullet${\bf Case 2: $v_1,v_2,v_3$ are all of type (G) or (C)}

Without loss of generality, we may assume that $N_1\geq N_2\geq N_3$ and we will put the $X^{0,b}$ norm on $v_3$. From the same manipulation as in the previous cases, we have
\begin{align*}
\mathcal{M}\lesssim &\int \frac{g_1(\lambda_1)g_2(\lambda_2)f_3(\lambda_3)}{\lg\lambda_1-\lambda_2+\lambda_3 \rg^{1-b_1}\lg\lambda_1\rg^{\frac{2b_0}{q'}}\lg\lambda_2\rg^{\frac{2b_0}{q'}}\lg\lambda_3\rg^{b_0} }d\lambda_1\lambda_2 d\lambda_3\\
\lesssim & \|g_1(\lambda_1)\|_{L_{\lambda_1}^q}
\|g_2(\lambda_2)\|_{L_{\lambda_2}^q}
\|f_3(\lambda_3)\|_{L_{\lambda_3}^2}\\
\leq & (N_1N_2)^{\frac{1}{q}}\|v_1\|_{X_{\infty,q}^{0,\frac{2b_0}{q'}}} 
\|v_2\|_{X_{\infty,q}^{0,\frac{2b_0}{q'}}}
\|v_2\|_{X^{0,b_0}}\lesssim (N_1N_2)^{-\frac{\alpha}{2}+\frac{1}{q}+\epsilon_2}N_3^{-s},
\end{align*}
which is conclusive. The proof of (8) of Proposition \ref{Multilinearkey} is complete.

\subsection{Non-diagonal terms}

We assume that $N_1\sim N_2\sim N_3\sim N$. By Lemma \ref{localizationXsb} and the $X^{0,b}$-mapping property (Lemma \ref{lem:mappingXsb}), outside a set of probability $<\mathrm{e}^{-N^{\theta}R^2},$ we have
$$ \|\mathcal{I}\mathcal{N}_3(v_1,v_2,v_3)\|_{X^{0,b_1}}\lesssim RN^{-\delta_0}\|v_j\|_{X^{0,\frac{3}{8}}}
$$
for all $j\in\{1,2,3\}$, provided that the range of $\alpha$ satisfies
\begin{align}\label{constraint1}
\alpha-1+2\nu s>s.
\end{align}
 Therefore, if at least one $v_j$ is of type (D), the right side can be bounded by $RN^{-s-\delta_0}$. If at least one $v_j$ is of type (C), we have the upper bound $$RN^{-\delta_0-2(\alpha-1)+2\epsilon_1}<RN^{-s-\delta_0},$$ which is also conclusive from the hypothesis on the range of $\alpha$.

 \vspace{0.3cm}

The only case left to treat alternatively is that when $v_1,v_2,v_3$ are all of type (G). From {\bf Algorithm 1} in Section \ref{reductionsection}, it suffices to estimate
\begin{align*}
\mathbb{E}\big[|\mathcal{U}_{L_1,L_2,L_3}^{N_1,N_2,N_3}|^2\big]=&\sum_{k}\mathbb{E}\Big[\Big|\sum_{\substack{(k_1,k_2,k_3)\in\Gamma(k)\\
|k_j|\sim N_j, j=1,2,3 }}\frac{g_{k_1}\ov{g}_{k_2}g_{k_3} }{[k_1]^{\frac{\alpha}{2}}[k_2]^{\frac{\alpha}{2}}[k_3]^{\frac{\alpha}{2}}  } \Big|^2\Big].
\end{align*}
Recall that if the output of this estimate if $K_1(N_1,N_2,N_3)$, then
outside of set of probability $<\mathrm{e}^{-N_1^{\theta}R^{\frac{2}{3}}}$,
$$ \|\mathcal{N}_3(v_1,v_2,v_3)\|_{X^{0,b_1}}\lesssim RN^{\frac{\alpha}{q_0}+\theta+100(\kappa^{-0.1}+\epsilon_1+\epsilon_2)}K_1(N_1,N_2,N_3)^{\frac{1}{2}}.
$$ 
By expanding the square and using the independence, we have 
\begin{align*}
\mathbb{E}\big[|\mathcal{U}_{L_1,L_2,L_3}^{N_1,N_2,N_3}|^2 \big]\sim &(N_1N_2N_3)^{-\alpha}\sum_{\substack{k,k_1,k_2,k_3\\
(k_1,k_2,k_3)\in\Gamma(k) } } \prod_{j=1}^3\mathbf{1}_{|k_j|\sim N_j}\\
\lesssim &(N_1N_2N_3)^{-\alpha}\sum_{k,k_2,k_3}S_{k+k_2-k_3,k_2,k_3}\mathbf{1}_{|k+k_2-k_3|\sim N_1}\prod_{j=2}^3\mathbf{1}_{|k_j|\sim N_j} \\
\lesssim &(N_1N_2N_3)^{-\alpha}\sum_{k_2,k_3}N_1^{\epsilon}(1+\frac{N_1^{2-\alpha}}{\lg k_2-k_3\rg})\lesssim N_1^{-\alpha-2(\alpha-1)+\epsilon},
\end{align*}
thus $K_1(N_1,N_2,N_3)^{\frac{1}{2}}\sim N_1^{-\frac{\alpha}{2}-(\alpha-1)+\epsilon}$, which is conclusive by choosing $\epsilon<\kappa^{-0.1},$ say.


In summary, the proof of (1) of Proposition \ref{Multilinearkey} in the case $N_1\sim N_2\sim N_3$ is complete.




\section{High-high-low interactions}
In this section, we will prove (1) of Proposition \ref{Multilinearkey} in the case $N_{(1)}\sim N_{(2)}\gg N_{(3)}$.
More precisely, we will finish the estimate of  $\|\mathcal{M}_{L_1,L_2,L_3}^{N_1,N_2,N_3}\|_{L_{\lambda_1}^{r_1}L_{\lambda_2}^{r_2}L_{\lambda_3}^{r_3}L_{\mu_0*}^{q_0}l_k^2}$  by executing one of {\bf Algorithm} 1 to {\bf Algorithm }4, according to different triples of characterized parameters $(N_j,L_j), j=1,2,3$. Moreover, we will ignore all the small powers of $N_j$ in the estimates by assuming that
$$S_{k_1,k_2,k_3}=\mathbf{1}_{\Phi_{k_1,k_2,k_3}=\mu_0+O(1)}
\text{ and } h_{k_jk_j^*}^{N_jL_j}=h_{k_jk_j^*}^{N_jL_j}\mathbf{1}_{|k_j-k_j^*|\leq L_j},$$ 
since they are all compensated by a unified factor $N_{(1)}^{\frac{\alpha}{q_0}+\frac{3}{q}+\theta+100(\kappa^{-1}+\epsilon_1+\epsilon_2)}$.
To save the notation, in this section, all the sums for $k,k_j$ are taken for $|k_j|\sim N_j, |k|\lesssim N_{(1)}$ without declaration, when there is no risk of confusing.

\subsection{The case $N_1\sim N_2\gg N_3$}
First we deal with\\
\noi
$\bullet$ {\bf Case A-1: $a_1,a_2$ are all of type (D)}

We execute {\bf Algorithm 4}. By the triangle inequality and Cauchy-Schwartz,
\begin{align*}
|\mathcal{U}_{L_1,L_2,L_3}^{N_1,N_2,N_3}|^2\leq & \Big(\sum_k\sum_{(k_1,k_2,k_3)\in\Gamma(k)} |a_1(k_1)|^2|a_3(k_3)|^2\mathbf{1}_{k_1k_2<0}
\Big)\cdot
\Big(\sup_k\sum_{(k_1,k_2,k_3)\in\Gamma(k)}
|a_2(k_2)|^2\mathbf{1}_{k_1k_2<0}
\Big)\\
+&\Big(\sum_{k}\sum_{(k_1,k_2,k_3)\in\Gamma(k)} |a_1(k_1)|^2|a_3(k_3)|^2\mathbf{1}_{k_1k_2\geq 0}
\Big)\cdot
\Big(\sup_k\sum_{(k_1,k_2,k_3)\in\Gamma(k)}
|a_2(k_2)|^2\mathbf{1}_{k_1k_2\geq 0}
\Big).
\end{align*}
Since $N_3\ll N_1\sim N_2$, for fixed $k$ and $k_2$,
$$ \sum_{(k_1,k_2,k_3)\in\Gamma(k)}|a_2(k_2)|^2=\sum_{k_2}|a_2(k_2)|^2\sum_{k_3}S_{k+k_2-k_3,k_2,k_3}\lesssim \|a_2\|_{l^2}^2.
$$
When $k_1k_2<0$, for fixed $k_1,k_3$,
$$ |\partial_{k_2}\Phi_{k_1,k_2,k_3}|=\alpha\big|\mathrm{sign}(k_2)|k_2|^{\alpha-1}+\mathrm{sign}(k_2-k_1-k_3)|k_2-k_1-k_3|^{\alpha-1} \big|
\gtrsim N_1^{\alpha-1},
$$
hence 
$$ \sum_k\sum_{(k_1,k_2,k_3)\in\Gamma(k)}|a_1(k_1)|^2|a_3(k_3)|^2\mathbf{1}_{k_1k_2<0}=\sum_{k_1,k_3}|a_1(k_1)|^2|a_3(k_3)|^2\sum_{k_2}S_{k_1,k_2,k_3}\mathbf{1}_{k_1k_2<0}\lesssim \|a_1\|_{l^2}^2\|a_3\|_{l^2}^2.
$$
When $k_1k_2\geq 0$, if $\mathrm{sign}(k_2)=\mathrm{sign}(k_2-k_1-k_3)$, then $|\partial_{k_2}\Phi_{k_1,k_2,k_3}|\gtrsim N_1^{\alpha-1}$. Otherwise, we have for fixed $k_1,k_3$, by Lemma \ref{counting2ndorder},
$$ \sum_{k_2}S_{k_1,k_2,k_3}\mathbf{1}_{k_2(k_1+k_3-k_2)>0}\lesssim  N_2^{1-\frac{\alpha}{2}}\sim N_1^{1-\frac{\alpha}{2}}.
$$
Therefore,
$$ \sum_{k}\sum_{(k_1,k_2,k_3)\in \Gamma(k)}|a_1(k_1)|^2|a_3(k_3)|^2\mathbf{1}_{k_1k_2\geq 0}\lesssim N_1^{1-\frac{\alpha}{2}}\|a_1\|_{l^2}^2\|a_3\|_{l^2}^2.
$$
We have proved:
\begin{proposition}\label{CaseA-1}
When $N_1\sim N_2\gg N_3$ and $a_1, a_2$ are of type (D), the output of the Case A-1 is bounded by
$$ \mathcal{U}_{L_1,L_2,L_3}^{N_1,N_2,N_3}\lesssim N_1^{\frac{1}{2}-\frac{\alpha}{4}}N_1^{-2s}N_3^{-\frac{\alpha-1}{2}}.
$$
\end{proposition}
Note that the power $N_3^{-\frac{\alpha-1}{2}}$ comes from the worst case when $a_3$ is of type (G).

\vspace{0.3cm}

\noi
$\bullet${\bf Case A-2:} Exact one of $a_1,a_2$ is of type (G) or (C), the other is of type (D).

First we assume that $a_1$ is of type (G) or (C) and $a_2$ is of type (D). By Cauchy-Schwartz,
\begin{align*}
|\mathcal{U}_{L_1,L_2,L_3}^{N_1,N_2,N_3}|^2\leq &\|a_1\|_{l^{\infty}}^2\Big(\!\!\!\!\!\!\sum_{\substack{k,k_2,k_3\\|k_2|\sim N_2 \\
|k+k_2-k_3|\sim N_1   }}\!\!\!\!\!\!\!|a_3(k_3)|^2S_{k+k_2-k_3,k_2,k_3} \Big)\sup_{k}\!\!\!\!\!\!\!\sum_{\substack{ k_2,k_3\\
|k_3|\sim N_3\\|k+k_2-k_3|\sim N_1 } }\!\!\!\!\!\!|a_2(k_2)|^2S_{k+k_2-k_3,k_2,k_3}\\
\lesssim &\|a_1\|_{l^{\infty}}^2\|a_2\|_{l^2}^2\sum_{k_1,k_2,k_3}|a_3(k_3)|^2S_{k_1,k_2,k_3}\mathbf{1}_{|k_2|\sim N_2}\mathbf{1}_{|k_1|\sim N_1},
\end{align*}
since for fixed $k,k_2$
$$ \sum_{k_3}S_{k+k_2-k_3,k_2,k_3}\mathbf{1}_{|k_3|\sim N_3}\mathbf{1}_{|k+k_2-k_3|\sim N_1}\lesssim 1,
$$
due to the fact that $N_3\ll N_1$. Now for fixed $k_2,k_3$, if $k_1(k_1-k_2+k_3)\geq 0$, we have
$$ |\partial_{k_1}\Phi_{k_1,k_2,k_3}|\sim \big|\mathrm{sign}(k_1)|k_1|^{\alpha-1}-\mathrm{sign}(k_1-k_2+k_3)|k_1-k_2+k_3|^{\alpha-1} \big|\gtrsim N^{\alpha-2}|k_2-k_3|,
$$
otherwise $|\partial_{k_1}\Phi_{k_1,k_2,k_3}|\gtrsim N_1^{\alpha-1}, $
hence
$$ \sum_{k_1}S_{k_1,k_2,k_3}\mathbf{1}_{|k_1|\sim N_1,|k_3|\sim N_3}\mathbf{1}_{|k_2|\sim N_2}\lesssim 1+\frac{N_1^{2-\alpha}}{\lg k_2-k_3\rg}. 
$$
Therefore, modulo a factor $N_1^{\epsilon}$, we have
$$ |\mathcal{U}_{L_1,L_2,L_3}^{N_1,N_2,N_3}|\lesssim \|a_1\|_{l^{\infty}}\|a_2\|_{l^2}\|a_3\|_{l^2}\big(N_2^{\frac{1}{2}}+N_1^{1-\frac{\alpha}{2}} \big)\lesssim N_1^{-\frac{\alpha-1}{2}-s}L_1^{-\nu}N_3^{-\frac{\alpha-1}{2}}.
$$
Next we assume that $a_1$ is of type (D) and $a_2$ is of type (G) or (C). Repeating the argument above, we have
\begin{align*}
 |\mathcal{U}_{L_1,L_2,L_3}^{N_1,N_2,N_3}|^2\leq &\|a_2\|_{l^{\infty}}^2\Big(\!\!\!\!\!\!\sum_{\substack{k,k_1,k_3\\
|k_1|\sim N_1\\
|k_1+k_3-k|\sim N_2 } }\!\!\!\!\!\!|a_3(k_3)|^2S_{k_1,k_1+k_3-k,k_3} \Big)\sup_k\!\!\!\!\!\!\sum_{\substack{k_1,k_3\\
|k_3|\sim N_3\\
|k_1+k_3-k|\sim N_2 } }\!\!\!\!\!\!|a_1(k_1)|^2S_{k_1,k_1+k_3-k,k_3}\\
\lesssim &\|a_1\|_{l^2}^2\|a_2\|_{l^{\infty}}^2
\sum_{\substack{k_1,k_2,k_3 } }|a_3(k_3)|^2S_{k_1,k_2,k_3}\mathbf{1}_{|k_1|\sim N_1}\mathbf{1}_{|k_2|\sim N_2}\\
\lesssim &\|a_1\|_{l^2}^2\|a_2\|_{l^{\infty}}^2\|a_3\|_{l^2}^2(N_2+N_1^{2-\alpha})\lesssim \|a_1\|_{l^2}^2\|a_2\|_{l^{\infty}}^2\|a_3\|_{l^2}^2N_2,
\end{align*} 
thanks to the fact that $N_3\ll N_2\sim N_1$. 
Therefore, we have proved:
\begin{proposition}\label{CaseA-2}
	Assume that $N_1\sim N_2\gg N_3$  and exact one of $a_1, a_2$ is of type (G) or (C) and the other is of type (D).  Then modulo a factor $N_1^{\epsilon}$, $\epsilon<\kappa^{-0.1}$, the output of the Case A-1 is bounded by
	$$ \mathcal{U}_{L_1,L_2,L_3}^{N_1,N_2,N_3}\lesssim N_{1}^{-\frac{\alpha-1}{2}-s}L_1^{-\nu}N_3^{-\frac{\alpha-1}{2}},
	$$
	if $a_1$ is of type (C) or (G) and $a_2$ is of type (D). If $a_1$ is of type (D) and $a_2$ is of type (C) or (G), modulo a factor $N_1^{\epsilon}$, $\epsilon<\kappa^{-0.1}$, we have
	$$ \mathcal{U}_{L_1,L_2,L_3}^{N_1,N_2,N_3}\lesssim N_{1}^{-\frac{\alpha-1}{2}-s}L_2^{-\nu}N_3^{-\frac{\alpha-1}{2}}.
	$$
\end{proposition}

\noi
$\bullet${\bf Case A-3: $a_1,a_2$ are all of type (G) or (C) }

If at least one of $a_1,a_2$ is of type (C), we first execute  {\bf Algorithm 4}. Without loss of generality, we may assume that $a_1$ is of type (C) and $L_1\geq L_2$. From the same argument as for {\bf Case A-2}, we have
$$ |\mathcal{U}_{L_1,L_2,L_3}^{N_1,N_2,N_3}|\lesssim N_2^{\frac{1}{2}}\|a_1\|_{l^{\infty}}\|a_2\|_{l^2}\|a_3\|_{l^2}. 
$$
If $a_2$ is of type (G), we have
\begin{align}\label{A-3Al4-1} N_2^{\frac{1}{2}}\|a_1\|_{l^{\infty}}\|a_2\|_{l^2}\|a_3\|_{l^2}\lesssim N_1^{-(\alpha-1)}L_1^{-\nu}N_3^{-\frac{\alpha-1}{2}},
\end{align}
since the worst case is that $a_3$ is of type (G). If $a_2$ is of type (C), we have
\begin{align}\label{A-3Al4-2}
 N_2^{\frac{1}{2}}\|a_1\|_{l^{\infty}}\|a_2\|_{l^2}\|a_3\|_{l^2}\lesssim N_1^{-\frac{3(\alpha-1)}{2}}L_1^{-\nu}L_2^{-\nu}N_3^{-\frac{\alpha-1}{2}}.
\end{align}
In order to get a better bound when $L_1\vee L_2$ is relatively small, we need to execute {\bf Algorithm 2}. In what follows, we do not distinguish the type (G) and (C), since we will only use the $S^{b,q}$ norm of $h^{N_jL_j}$, which does not make any difference between type (G) and (C) terms. Fix $L=10(L_1\vee L_2)$, recall the definition \eqref{etakk3} of $\eta_{k,k'}^{(3)}$, we have 
\begin{align*}
&\sum_{k,k':|k-k'|\geq L}|\eta_{k,k'}^{(3)}|^2\\=&
\sum_{\substack{k,k'\\
|k-k'|\geq L }}\Big|\sum_{\substack{(k_1,k_2,k_3)\in\Gamma(k)\\
(k_1',k_2',k_3)\in\Gamma(k')  \\ k_1^*,k_1'^*,k_2^*,k_2'^*
}} h_{k_1k_1^*}\ov{h}_{k_1'k_1'^*}\ov{h}_{k_2k_2^*}h_{k_{2}'k_2'^*}S_{k_1,k_2,k_3}S_{k_1',k_2',k_3}\frac{g_{k_1^*}\ov{g}_{k_1'^*}\ov{g}_{k_2^*}g_{k_2'^*} }{[k_1^*]^{\frac{\alpha}{2}}
[k_1'^*]^{\frac{\alpha}{2}}
[k_2^*]^{\frac{\alpha}{2}}
[k_2'^*]^{\frac{\alpha}{2}} } \Big|^2.
\end{align*}
To simplify the summation before expanding the square, we split the inner sum $\sum_{k_1^*,k_1'^*,k_2^*,k_2'^* }$ into six groups:
\begin{itemize}
	\item[(1)] $k_1^*=k_1'^*,k_2^*=k_2'^*$;
	\item[(2)] $k_1^*=k_2^*,k_1'^*=k_2'^*$;
	\item[(3)] $k_1^*=k_1'^*, k_2^*\neq k_2'^*$;
	\item[(4)] $k_1^*\neq k_1'^*,k_2^*=k_2'^*$;
	\item[(5)] $k_1^*=k_2^*, k_1'^*\neq k_2'^*$ or $k_1'^*=k_2'^*, k_1^*\neq k_2^*$;
	\item[(6)] No pairings in $k_1^*,k_1'^*,k_2^*,k_2'^*$ (in the sense that $k_1^*\neq k_1'^*,k_2^*$ and $ k_2^*\neq k_2'^*, k_1^*$).
\end{itemize}
Taking the expectation, the contributions from ($j$) ($j\in\{1,2,3,4,5,6\}$) can be bounded by
\begin{align*}
\mathcal{C}_{j}:=&N_1^{-4\alpha}\sum_{\substack{ k,k'\\
|k-k'|\geq L } }\sum_{\substack{(k_1,k_2,k_3)\in \Gamma(k)\\ (k_1',k_2',k_3)\in\Gamma(k')  } }\sum_{ \substack{
		(m_1,m_2,m_3)\in\Gamma(k)\\ (m_1',m_2',m_3)\in\Gamma(k')  }} \prod_{j=1}^2 \|h_{k_jk_j^*}\|_{l_{k_j^*}^2}\|h_{m_jm_j^*}\|_{l_{m_j^*}^2}\|h_{k_j'k_j'^*}\|_{l_{k_j'^*}^2}\|h_{m_j'm_j'^*}\|_{l_{m_j'^*}^2}\\
	&\hspace{5cm}\times S_{k_1,k_2,k_3}S_{k_1',k_2',k_3}S_{m_1,m_2,m_3}S_{m_1',m_2',m_3}\mathbf{1}_{\mathcal{A}_j}
\end{align*}
where $\mathcal{A}_j$ is the index set of $k_1,k_2,k_3,k_1',k_2',m_1,m_2,m_3,m_1',m_2',k,k' 
$ defining by the constraints:
\begin{align*}
&\mathcal{A}_1:=\{|k_i-k_i'|\leq 2L_i, |m_i-m_i'|\leq 2L_i, i=1,2  \};\\
&\mathcal{A}_2:=\{|k_1-k_2|,|m_1-m_2|\leq 2L_1\vee L_2; |k_1'-k_2'|,|m_1'-m_2'|\leq 2L_1\vee L_2 \};\\
&\mathcal{A}_3:=\{|k_1-k_1'|,|m_1-m_1'|\leq 2L_1, |k_2-m_2|,|k_2'-m_2'|\leq 2L_2  \};\\
&\mathcal{A}_4:=\{|k_1-m_1|,|k_1'-m_1'|\leq 2L_1, |k_2-k_2'|,|m_2-m_2'|\leq 2L_2 \};\\
&\mathcal{A}_5:=\{|k_1-k_2|,|m_1-m_2|\leq 2L_1\vee L_2, |k_1'-m_1'|\leq 2L_1,|k_2'-m_2'|\leq 2L_2 \}\\
&\hspace{0.7cm}\cup \{|k_1'-k_2'|,|m_1'-m_2'|\leq 2L_1\vee L_2, |k_1-m_1|\leq 2L_1,|k_2-m_2|\leq 2L_2 \};\\
&\mathcal{A}_6:=\{|k_1-m_1|,|k_1'-m_1'|\leq 2L_1, |k_2-m_2|,|k_2'-m_2'|\leq 2L_2 \}\\
&\hspace{0.7cm}\cup \{|k_1-m_2'|,|k_2'-m_1|\leq 2L_1\vee L_2,|k_1'-m_1'|\leq 2L_1, |k_2-m_2|\leq 2L_2  \};\\
&\hspace{0.72cm}\cup\{|k_1'-m_2|,|k_2-m_1'|\leq 2L_1\vee L_2,|k_1-m_1|\leq 2L_1, |k_2'-m_2'|\leq 2L_2 \}\\
&\hspace{0.72cm}\cup\{|k_1-m_2'|,|k_2-m_1'|,||k_1'-m_2|,|k_2'-m_1|\leq 2L_1\vee L_2  \}.
\end{align*}
Note that under the constraint $|k-k'|\geq L=10(L_1\vee L_2)$, $\mathcal{A}_1=\mathcal{A}_2=\emptyset$, hence $\mathcal{C}_1=\mathcal{C}_2=0$. Indeed, on $\mathcal{A}_1$,
$$ |k-k'|=|(k_1-k_2+k_3)-(k_1'-k_2'+k_3)|\leq |k_1-k_1'|+|k_2-k_2'|\leq 2L_1+2L_2<L.
$$
On $\mathcal{A}_2$,
$$ |k-k'|\leq |k_1-k_2|+|k_1'-k_2'|\leq 4(L_1\vee L_2)<L.
$$
On $\mathcal{A}_3$, $|k-k'+k_2-k_2'|\leq 2L_1, |k-k'+m_2-m_2'|\leq 2L_1$ and $|k_2-m_2|\leq 2L_2, |k_2'-m_2'|\leq 2L_2$, we have
\begin{align*}
\mathcal{C}_3\lesssim  &N_1^{-4\alpha}\prod_{j=1}^2\|h_{k_jk_j^*}^{N_jL_j}\|_{l_{k_j}^{\infty}l_{k_j^*}^{2}}^4\cdot\!\!\!\!\!\sum_{\substack{k,k',k_2,k_2',k_3\\
m_2,m_2',m_3 } }\!\!\!\!\!S_{k+m_2-m_3,m_2,m_3}S_{k'+m_2'-m_3,m_2',m_3}
S_{k+k_2-k_3,k_2,k_3}S_{k'+k_2'-k_3,k_2',k_3}\\
&\hspace{5cm}\times\mathbf{1}_{\substack{|k-k'+k_2-k_2'|\leq 2L_1\\
|k-k'+m_2-m_2'|\leq 2L_1 } }\mathbf{1_{\substack{|k_2-m_2|\leq 2L_2\\
|k_2'-m_2'|\leq 2L_2 } } }\\
\lesssim &N_1^{-4\alpha}(L_1L_2)^{-4\nu}L_2^2\sum_{\substack{k,k',k_2,k_2',k_3 } }\mathbf{1}_{|k-k'+k_2-k_2'|\leq 2L_1}S_{k+k_2-k_3,k_2,k_3}S_{k'+k_2'-k_3,k_2',k_3},
\end{align*}
where we first use
$$ \sum_{m_3}S_{k+m_2-m_3,m_2,m_3}\lesssim 1
$$
and then sum for $\mathbf{1}_{|k_2-m_2|\leq 2L_2,|k_2'-m_2'|\leq 2L_2}$ over $m_2,m_2'$. Next we sum over $k'$ first and estimate (modulo a small factor $N_1^{\epsilon}$)
\begin{align*}
&\sum_{k,k',k_2,k_2',k_3}\mathbf{1}_{|k-k'+k_2-k_2'|\leq 2L_1}S_{k+k_2-k_3,k_2,k_3}S_{k'+k_2'-k_3,k_2',k_3}\\
\leq &2L_1\sum_{k_2,k_3}\sum_{k_2'} \sum_{k}S_{k+k_2-k_3,k_2,k_3}\\
\lesssim &L_1N_2\sum_{k_2,k_3}\big(1+\frac{N_1^{2-\alpha}}{\lg k_2-k_3\rg}\big)\lesssim L_1N_2(N_2N_3+N_1^{2-\alpha}N_3)\lesssim L_1N_1^2N_3.
\end{align*}
Therefore,
$$ \mathcal{C}_3\lesssim L_1^{1-4\nu}L_2^{2-4\nu}N_1^{2-4\alpha}N_3.
$$ 
Note that in the estimate of $\mathcal{C}_3$, we do not make use of the constraint $|k-k'|\geq L$.

The estimate for $\mathcal{C}_4$ is similar as for $\mathcal{C}_3$, by switching $k_2,k_2',m_2,m_2'$ to $k_1,k_1',m_1,m_1'$. More precisely, 
\begin{align*}
\mathcal{C}_4\lesssim  &N_1^{-4\alpha}\prod_{j=1}^2\|h_{k_jk_j^*}^{N_jL_j}\|_{l_{k_j}^{\infty}l_{k_j^*}^{2}}^4\cdot\!\!\!\!\!\sum_{\substack{k,k',k_1,k_1',k_3\\
		m_1,m_1',m_3 } }\!\!\!\!\!S_{k_1,m_1+m_3-k,m_3}S_{m_1',m_1'+m_3-k',m_3}
S_{k_1,k_1+k_3-k,k_3}S_{k_1',k_1'+k_3-k',k_3}\\
&\hspace{5cm}\times\mathbf{1}_{\substack{|k_1-k_1'+k'-k|\leq 2L_2\\
		|m_1-m_1'+k'-k|\leq 2L_2 } }\mathbf{1_{\substack{|k_1-m_1|\leq 2L_1\\
			|k_1'-m_1'|\leq 2L_1 } } }\\
\lesssim &N_1^{-4\alpha}(L_1L_2)^{-4\nu}L_1^2\sum_{\substack{k,k',k_1,k_1',k_3 } }\mathbf{1}_{|k_1-k_1'+k'-k|\leq 2L_2}S_{k_1,k_1+k_3-k,k_3}S_{k_1',k_1'+k_3-k',k_3},
\end{align*}
where we used
$$ \sum_{m_1,m_1'}\mathbf{1}_{\substack{|m_1-k_1|\leq 2L_1\\
|m_1'-k_1'|\leq 2L_1 } }\sum_{m_3}S_{m_1,m_1+m_3-k,m_3}\lesssim L_1^2.
$$
Next we sum over $k'$ first and estimate (modulo a factor $N_1^{\epsilon}$)
\begin{align*}
&\sum_{k,k',k_1,k_1',k_3 }\mathbf{1}_{|k_1-k_1'+k'-k|\leq 2L_2}S_{k_1,k_1+k_3-k,k_3}S_{k_1',k_1'+k_3-k',k_3}\\
\lesssim &L_2\sum_{k,k_3}\sum_{k_1'}\sum_{k_1}S_{k_1,k_1+k_3-k,k_3}
\lesssim L_2N_1\sum_{k,k_3}\big(1+\frac{N_1^{2-\alpha}}{\lg k-k_3\rg}\big)\lesssim L_2N_1^2N_3.  
\end{align*}
Thus
$$ \mathcal{C}_4\lesssim L_2^{1-4\nu}L_1^{2-4\nu}N_1^{2-4\alpha}N_3.
$$

To estimate $\mathcal{C}_5$, by symmetry, 
\begin{align*}
\mathcal{C}_5\lesssim &N_1^{-4\alpha}\prod_{j=1}^2\|h_{k_jk_j^*}^{N_jL_j}\|_{l_{k_j}^{\infty}l_{k_j^*}^{2}}^4\cdot\!\!\!\!\!\sum_{\substack{k,k',k_2,k_2',k_3\\
		m_2,m_2',m_3 } }\!\!\!\!\!S_{k+m_2-m_3,m_2,m_3}S_{k'+m_2'-m_3,m_2',m_3}
S_{k+k_2-k_3,k_2,k_3}S_{k'+k_2'-k_3,k_2',k_3}\\
&\hspace{5cm}\times\mathbf{1}_{\substack{|k-k_3|\leq 2L_1\vee L_2\\
		|k-m_3|\leq 2L_1\vee L_2 } }\mathbf{1}_{|k_2'-m_2'|\leq 2L_2}.
\end{align*}
First we sum over $m_2$ and obtain that
$$ \sum_{m_2}S_{k+m_2-m_3,m_2,m_3}\mathbf{1}_{|k-m_3|\leq 2L_1\vee L_2}\lesssim \frac{N_1^{2-\alpha}}{\lg k-m_3\rg}\mathbf{1}_{|k-m_3|\leq 2L_1\vee L_2},
$$ 
since $L_1\vee L_2\ll N_1^{2-\alpha}$. Next we sum over $m_2'$ and then $m_3$ to obtain that
\begin{align*}
\mathcal{C}_5\lesssim &N_1^{-4\alpha}(L_1L_2)^{-4\nu}L_2\sum_{k,k',k_2,k_2',k_3}N_1^{2-\alpha}\mathbf{1}_{|k-k_3|\leq 2L_1\vee L_2}S_{k+k_2-k_3,k_2,k_3}S_{k'+k_2'-k_3,k_2',k_3}.
\end{align*}
Using
$$ \mathbf{1}_{|k-k_3|\leq 2L_1\vee L_2}\sum_{k_2}S_{k+k_2-k_3,k_2,k_3}\lesssim \frac{N_1^{2-\alpha}}{\lg k-k_3\rg}\mathbf{1}_{|k-k_3|\leq 2L_1\vee L_2}
$$
and
$$ \sum_{k'}S_{k'+k_2'-k_3,k_2',k_3}\lesssim 1+\frac{N_1^{2-\alpha}}{\lg k_2'-k_3\rg},
$$
we have (modulo a factor $N_1^{\epsilon}$)
\begin{align*}
\mathcal{C}_5\lesssim &N_1^{-4\alpha+2-\alpha}L_1^{-4\nu}L_2^{1-4\nu}\sum_{k,k_2',k_3}\mathbf{1}_{|k-k_3|\leq 2L_1\vee L_2}\frac{N_1^{2-\alpha}}{\lg k-k_3\rg}\big(1+\frac{N_1^{2-\alpha}}{\lg k_2'-k_3\rg}\big)\\
\lesssim &N_1^{-4\alpha+2(2-\alpha)}L_1^{-4\nu}L_2^{1-4\nu}\Big(N_2\sum_{k,k_3}\frac{1}{\lg k-k_3\rg}+\sum_{k_3}\sum_{k}\frac{1}{\lg k-k_3\rg}\sum_{k_2'}\frac{N_1^{2-\alpha}}{\lg k_2'-k_3\rg} \Big)\\
\lesssim_{\epsilon} &N_1^{-4\alpha+2(2-\alpha)}L_1^{-4\nu}L_2^{1-4\nu}(N_2N_3+N_1^{2-\alpha}N_3).
\end{align*}
Thus
$$\mathcal{C}_5\lesssim N_1^{5-6\alpha}N_3L_1^{-4\nu}L_2^{1-4\nu}.
$$
Next we estimate $\mathcal{C}_6$. We observe that whatever the pairing is, we can always do the sum $\sum_{m_2,m_2',m_3}$ first to obtain a factor $L_1L_2^2+L_1^2L_2$, hence
\begin{align*}
\mathcal{C}_6\lesssim &N_1^{-4\alpha}\prod_{j=1}^2\|h_{k_jk_j^*}^{N_jL_j}\|_{l_{k_j}^{\infty}l_{k_j^*}^{2}}^4\cdot (L_1^2L_2+L_1L_2^2)\sum_{k,k',k_2,k_2',k_3}S_{k+k_2-k_3,k_2,k_3}S_{k'+k_2'-k_3,k_2',k_3}.
\end{align*}
Then we estimate
\begin{align*}
\sum_{k,k',k_3}\sum_{k_2,k_2'}S_{k+k_2-k_3,k_2,k_3}S_{k'+k_2'-k_3,k_2',k_3}
\lesssim \sum_{k,k',k_3}\big(1+\frac{N_1^{2-\alpha}}{\lg k-k_3\rg}\big)\big(1+\frac{N_1^{2-\alpha}}{\lg k'-k_3\rg}\big)
\lesssim &N_1^2N_3.
\end{align*}
Hence
$$ \mathcal{C}_6\lesssim N_1^{2-4\alpha}N_3(L_1\wedge L_2)^{1-4\nu}(L_2\vee L_2)^{2-4\nu}.
$$
We remark that the estimates for $\mathcal{C}_3,\mathcal{C}_4,\mathcal{C}_5,\mathcal{C}_6$ do not use the constraint $|k-k'|\geq L$. Hence it remains to estimate
\begin{align*}
\mathcal{C}'_j=&LN_1^{-4\alpha}\sup_{\substack{ k,k'\\
		|k-k'|<L } } \sum_{\substack{(k_1,k_2,k_3)\in \Gamma(k)\\ (k_1',k_2',k_3)\in\Gamma(k')  } }\sum_{ \substack{
		(m_1,m_2,m_3)\in\Gamma(k)\\ (m_1',m_2',m_3)\in\Gamma(k')  }} \prod_{j=1}^2 \|h_{k_jk_j^*}\|_{l_{k_j^*}^2}\|h_{m_jm_j^*}\|_{l_{m_j^*}^2}\|h_{k_j'k_j'^*}\|_{l_{k_j'^*}^2}\|h_{m_j'm_j'^*}\|_{l_{m_j'^*}^2}\\
&\hspace{5cm}\times S_{k_1,k_2,k_3}S_{k_1',k_2',k_3}S_{m_1,m_2,m_3}S_{m_1',m_2',m_3}\mathbf{1}_{\mathcal{A}_j}
\end{align*}
for $j=1,2$.

For $\mathcal{C}_1'$, we bound it by (with $k,k'$ fixed)
\begin{align*}
\mathcal{C}_1'\lesssim  &LN_1^{-4\alpha}\prod_{j=1}^2\|h_{k_jk_j^*}^{N_jL_j}\|_{l_{k_j}^{\infty}l_{k_j^*}^{2}}^4\cdot\!\!\!\!\!\sum_{\substack{k_2,k_2',k_3\\
		m_2,m_2',m_3 } }\!\!\!\!\!S_{k+m_2-m_3,m_2,m_3}S_{k'+m_2'-m_3,m_2',m_3}
\\
&\hspace{5cm}\times S_{k+k_2-k_3,k_2,k_3}S_{k'+k_2'-k_3,k_2',k_3} \mathbf{1}_{\substack{|k_2-k_2'|\leq 2L_2\\
|m_2-m_2'|\leq 2L_2 } }
\\
\lesssim &	L\cdot L_2^2N_1^{-4\alpha}(L_1L_2)^{-4\nu}\sum_{k_2,m_2,k_3,m_3}S_{k+k_2-k_3,k_2,k_3}S_{k+m_2-m_3,m_2,m_3}\\
\lesssim &(L_1\vee L_2)^{3-4\nu}(L_1\wedge L_2)^{-4\nu}N_1^{-4\alpha}\sum_{k_2,k_3,m_2}S_{k+k_2-k_3,k_2,k_3}\\
\lesssim &(L_1\vee L_2)^{3-4\nu}(L_1\wedge L_2)^{-4\nu}N_1^{-4\alpha}\cdot N_2(N_3+N_2^{2-\alpha})\\
\lesssim &(L_1\vee L_2)^{3-4\nu}(L_1\wedge L_2)^{-4\nu}\big(N_1^{1-4\alpha}N_3+N_1^{2-5\alpha} \big). 
\end{align*}
Next, for fixed $k,k'$ (and without loss of generality we assume that $L_2\geq L_1$),
\begin{align*}
\mathcal{C}_2'\lesssim  &LN_1^{-4\alpha}\prod_{j=1}^2\|h_{k_jk_j^*}^{N_jL_j}\|_{l_{k_j}^{\infty}l_{k_j^*}^{2}}^4\cdot\!\!\!\!\!\sum_{\substack{k_2,k_2',k_3\\
		m_2,m_2',m_3 } }\!\!\!\!\!S_{k+m_2-m_3,m_2,m_3}S_{k'+m_2'-m_3,m_2',m_3}
\\
&\hspace{3cm}\times S_{k+k_2-k_3,k_2,k_3}S_{k'+k_2'-k_3,k_2',k_3} \mathbf{1}_{\substack{|k-k_3|\leq 2(L_1\vee L_2)\\
		|k-m_3|\leq 2(L_1\vee L_2) } }
	\mathbf{1}_{\substack{|k'-k_3|\leq 2(L_1\vee L_2)\\
			|k'-m_3|\leq 2(L_1\vee L_2) } }\\
\lesssim &LN_1^{-4\alpha}(L_1L_2)^{-4\nu}\sum_{k_3,m_3}\mathbf{1}_{\substack{|k-k_3|\leq 2(L_1\vee L_2)\\
		|k-m_3|\leq 2(L_1\vee L_2) } }
\mathbf{1}_{\substack{|k'-k_3|\leq 2(L_1\vee L_2)\\
		|k'-m_3|\leq 2(L_1\vee L_2) } }\\
&\hspace{1.5cm} \times\sum_{k_2}S_{k+k_2-k_3,k_2,k_3}\sum_{k_2'}S_{k'+k_2'-k_3,k_2',k_3}
\sum_{m_2}S_{k+m_2-m_3,m_2,m_3}
	\sum_{m_2'}S_{k'+m_2'-m_3,m_2',m_3}\\
	\lesssim& (L_1\vee L_2)^{1-4\nu}(L_1\wedge L_2)^{-4\nu}N_1^{-8(\alpha-1)},
\end{align*}
where we use the facts that
$$ \sum_{k_2}S_{k+k_2-k_3,k_2,k_3}\mathbf{1}_{|k-k_3|\leq 2(L_1\vee L_2)}\lesssim \frac{N_1^{2-\alpha}}{\lg k-k_3\rg},$$
$$
\sum_{k_2'}S_{k'+k_2'-k_3,k_2,k_3}\mathbf{1}_{|k'-k_3|\leq 2(L_1\vee L_2)}\lesssim \frac{N_1^{2-\alpha}}{\lg k'-k_3\rg}
$$
and
$$ \sum_{m_2}S_{k+m_2-m_3,m_2,m_3}\mathbf{1}_{|k-m_3|\leq 2(L_1\vee L_2)}\lesssim \frac{N_1^{2-\alpha}}{\lg k-m_3\rg},$$
$$
\sum_{m_2'}S_{k'+m_2'-m_3,m_2,m_3}\mathbf{1}_{|k'-m_3|\leq 2(L_1\vee L_2)}\lesssim \frac{N_1^{2-\alpha}}{\lg k'-m_3\rg},
$$
since $L_1\vee L_2\ll N_1^{2-\alpha}$. Hence
$$ \mathcal{C}_2'\lesssim (L_1\vee L_2)^{1-4\nu}(L_1\wedge L_2)^{-4\nu}N_1^{-8(\alpha-1)}.
$$ Note that here the treatment is different, compared with $\mathcal{C}_1'$, due to the different type of pairing. In summary, we have
\begin{align*}
\mathcal{C}_1'+\mathcal{C}_2'+\sum_{j=3}^6\mathcal{C}_j\lesssim 
 &(L_1\wedge L_2)^{1-4\nu}(L_1\vee L_2)^{2-4\nu}N_3\big[N_1^{2-4\alpha}+(L_1L_2)^{-1}N_1^{5-6\alpha}+
N_1^{1-4\alpha}\big(\frac{L_1\vee L_2}{L_1\wedge L_2}\big) \big]\\
+&(L_1\wedge L_2)^{-4\nu}(L_1\vee L_2)^{1-4\nu}N_1^{-8(\alpha-1)}.
\end{align*}
Combining the estimates above and \eqref{A-3Al4-1},\eqref{A-3Al4-2}, we obtain that:
\begin{proposition}\label{CaseA-3}
Assume that $N_1\sim N_2\gg N_3$. Then outside a set of probability $<\mathrm{e}^{-N_1^{\theta}R}$, modulo factors $N^{\epsilon}$, $\epsilon<\kappa^{-0.1}$, we have:
\item[(i)]  If $a_1,a_2$ are both of type (C), then
	$$ |\mathcal{U}_{L_1,L_2,L_3}^{N_1,N_2,N_3}|\leq CRN_1^{-\frac{3(\alpha-1)}{2}}N_3^{-\frac{\alpha-1}{2}}(L_1L_2)^{-\nu},
	$$
	and
	\begin{align*}
	|\mathcal{U}_{L_1,L_2,L_3}^{N_1,N_2,N_3}|\leq &CRN_1^{-\frac{3(\alpha-1)}{2}}N_3^{-\frac{\alpha-1}{2}}(L_1L_2)^{\frac{1}{2}-\nu}\big[N_3^{\frac{1}{4}}N_1^{-\frac{1}{4}}(L_1\vee L_2)^{-\frac{1}{4}}(L_1\wedge L_2)^{-\frac{1}{2}}+N_3^{\frac{1}{4}}N_1^{-\frac{5}{4}+\frac{\alpha}{2}}\frac{(L_1\vee L_2)^{\frac{1}{4}}}{(L_1\wedge L_2)^{\frac{1}{2}}} \big]\\
	+&CRN_1^{-\frac{3(\alpha-1)}{2}}N_3^{-\frac{\alpha-1}{2}}(L_1L_2)^{\frac{1}{2}-\nu}N_3^{\frac{1}{4}}N_1^{-(1-\frac{\alpha}{2})}(L_1\wedge L_2)^{-\frac{1}{4}}\\
	+&CR(L_1L_2)^{-\nu}(L_1\vee L_2)^{\frac{1}{4}}N_1^{-2(\alpha-1)}N_3^{-\frac{\alpha-1}{2}}.
	\end{align*}
\item[(ii)]  If $a_1$ is of type (C) and $a_2$ is of type (G) (or $a_1$ is of type (G), $a_2$ is of type (C)), then
		$$ |\mathcal{U}_{L_1,L_2,L_3}^{N_1,N_2,N_3}|\leq RN_1^{-(\alpha-1)}N_3^{-\frac{\alpha-1}{2}}L_1^{-\nu},
	$$
	and
	\begin{align*}
	|\mathcal{U}_{L_1,L_2,L_3}^{N_1,N_2,N_3}|\leq &CRN_1^{-\frac{3(\alpha-1)}{2}}N_3^{-\frac{\alpha-1}{2}}L_1^{-\nu}\big(N_3^{\frac{1}{4}}N_1^{-\frac{1}{4}}L_1^{\frac{1}{4}}+N_3^{\frac{1}{4}}N_1^{-\frac{5}{4}+\frac{\alpha}{2}}L_1^{\frac{3}{4}} \big)\\
	+&CRN_1^{-\frac{3(\alpha-1)}{2}}N_3^{-\frac{\alpha-1}{2}}L_1^{-\nu}N_3^{\frac{1}{4}}N_1^{-(1-\frac{\alpha}{2})}L_1^{\frac{1}{2}} +CRL_1^{\frac{1}{4}-\nu}N_1^{-2(\alpha-1)}N_3^{-\frac{\alpha-1}{2}}.
	\end{align*}
	Note that when $a_1$ is of type (G) and $a_2$ is of type (C), the estimates above true by switching $L_1$ and $L_2$.
	\item[(iii)] If $a_1,a_2$ are both of type (G), then
	$$ |\mathcal{U}_{L_1,L_2,L_3}^{N_1,N_2,N_3}|\leq CRN_1^{\frac{5}{4}-\frac{3\alpha}{2}}N_3^{\frac{3}{4}-\frac{\alpha}{2}}+RN_1^{-2(\alpha-1)}N_3^{-\frac{\alpha-1}{2}}<CRN_1^{-2(\alpha-1)}.
	$$
Finally, in whatever the situation, we always have, modulo $N_1^{\epsilon}, \epsilon<\kappa^{-0.1}$,
\begin{align}\label{numerology3} |\mathcal{U}_{L_1,L_2,L_3}^{N_1,N_2,N_3}|\leq CR\max\big\{N_1^{-(\frac{3}{2}+2\nu)(\alpha-1)},N_1^{-(\alpha-1)-\nu},N_1^{-\frac{3(\alpha-1)}{2}-\frac{2\nu}{3} }, N_1^{-(1+4\nu)(\alpha-1)}\big\}.
\end{align}
\end{proposition}

\begin{proof}
We only need to justify the last assertion \eqref{numerology3}. First we assume that $a_1,a_2$ are both of type (C) and without loss of generality, $L_1\geq L_2$. Note that, for $A>0, B>0$,
$$ \min\{1,A+B\}\leq \min\{1,A\}+\min\{1,B\}.
$$
Then from the two inequalities of (i), we have a rougher bound
\begin{align*}  |\mathcal{U}_{L_1,L_2,L_3}^{N_1,N_2,N_3}|\leq &CRN_1^{-\frac{3(\alpha-1)}{2}}N_3^{-\frac{\alpha-1}{2}}(L_1L_2)^{-\nu}\min\big\{1,(L_1L_2)^{\frac{1}{2}}\big(\frac{N_3}{N_1}\big)^{\frac{1}{4}}+L_1\frac{N_3^{\frac{1}{4}}}{N_1^{\frac{5}{4}-\frac{\alpha}{2}}} \big\}\\
+&CRN_1^{-\frac{3(\alpha-1)}{2}}N_3^{-\frac{\alpha-1}{2}}(L_1L_2)^{-\nu}\min\{1, L_1^{\frac{1}{4}}N_1^{-\frac{\alpha-1}{2}} \}.
\end{align*}
When $(L_1L_2)^{\frac{1}{2}}<N_1^{\frac{1}{4}}N_3^{-\frac{1}{4}}$,
we have
\begin{align*}
 |\mathcal{U}_{L_1,L_2,L_3}^{N_1,N_2,N_3}|\leq &CRN_1^{-\frac{3(\alpha-1)}{2}}N_3^{-\frac{\alpha-1}{2}}\big((L_1L_2) ^{\frac{1}{2}-\nu}N_3^{\frac{1}{4}}N_1^{-\frac{1}{4} }+L_1^{1-\nu}N_3^{\frac{1}{4}}N_1^{\frac{\alpha}{2}-\frac{5}{4}} \big)\\
 \leq &CRN_1^{-\frac{3(\alpha-1)}{2}-\frac{\nu}{2}}N_3^{\frac{\nu}{2}-\frac{\alpha-1}{2}}+CRN_1^{-\frac{3(\alpha-1)}{2}-\frac{5}{4}+\frac{\alpha}{2}}N_3^{-\frac{\alpha-1}{2}+\frac{1}{4}}\cdot (N_1^{\frac{1}{4}}N_3^{-\frac{1}{4}})^{2(1-\nu)} \\
 \leq &CRN_1^{-\frac{3(\alpha-1)}{2}-\frac{\nu}{2}}N_3^{\frac{\nu}{2}-\frac{\alpha-1}{2}},
\end{align*}
since $L_1^{\frac{1}{2}}<N_1^{\frac{1}{4}}N_3^{-\frac{1}{4}}$ and 
$$ N_1^{-\frac{3(\alpha-1)}{2}-\frac{5}{4}+\frac{\alpha}{2}}N_3^{-\frac{\alpha-1}{2}+\frac{1}{4}}\cdot (N_1^{\frac{1}{4}}N_3^{-\frac{1}{4}})^{2(1-\nu)}=N_1^{-\alpha+\frac{3}{4}-\frac{\nu}{2}}N_3^{\frac{1}{4}-\frac{\alpha}{2}+\frac{\nu}{2}}<N_1^{-\frac{3(\alpha-1)}{2}-\frac{\nu}{2}}N_3^{\frac{\nu}{2}-\frac{\alpha-1}{2}}.
$$
When $(L_1L_2)^{\frac{1}{2}}\geq N_1^{\frac{1}{4}}N_3^{-\frac{1}{4}},$ we have simply
$$ N_1^{-\frac{3(\alpha-1)}{2}}N_3^{-\frac{\alpha-1}{2}}(L_1L_2)^{-\nu}\leq CRN_1^{-\frac{3(\alpha-1)}{2}-\frac{\nu}{2}}N_3^{\frac{\nu}{2}-\frac{\alpha-1}{2}}.
$$
Similarly, 
$$ N_1^{-\frac{3(\alpha-1)}{2}}N_3^{-\frac{\alpha-1}{2}}(L_1L_2)^{-\nu}\min\{1,L_1^{\frac{1}{4}}N_1^{-\frac{\alpha-1}{2}} \}\leq N_1^{-(\frac{3}{2}+2\nu)(\alpha-1)}N_3^{-\frac{\alpha-1}{2}}.
$$
Therefore, for whatever $L_1,L_2$ and $N_3\ll N_1$, we have the bound
$$ |\mathcal{U}_{L_1,L_2,L_3}^{N_1,N_2,N_3}|\leq CR N_1^{-\frac{3(\alpha-1)}{2}-\frac{\nu}{2}}N_3^{\frac{\nu}{2}-\frac{\alpha-1}{2}}+CRN_1^{-(\frac{3}{2}+2\nu)(\alpha-1)}N_3^{-\frac{\alpha-1}{2}}.
$$
Note that $\nu>(\alpha-1)$, we have
$$ |\mathcal{U}_{L_1,L_2,L_3}^{N_1,N_2,N_3}|\leq CRN_1^{-2(\alpha-1)}+CRN_1^{-(\frac{3}{2}+2\nu)(\alpha-1)}.
$$

Next we assume that $a_1$ is of type (C) and $a_2$ is of type (G), from (ii) we have
\begin{align*}
|\mathcal{U}_{L_1,L_2,L_3}^{N_1,N_2,N_3}|\leq &CRN_1^{-(\alpha-1)}N_3^{-\frac{\alpha-1}{2}}L_1^{-\nu}\min\{1, L_1^{\frac{1}{4}}N_3^{\frac{1}{4}}N_1^{-\frac{1}{4}-\frac{\alpha-1}{2} } \}\\+&CRN_1^{-(\alpha-1)}N_3^{-\frac{\alpha-1}{2}}L_1^{-\nu} \min\{1,N_3^{\frac{1}{4}}N_1^{-\frac{5}{4}+\frac{\alpha}{2}-\frac{\alpha-1}{2}}L_1^{\frac{3}{4}} \}\\
+&CRN_1^{-(\alpha-1)}N_3^{-\frac{\alpha-1}{2}}L_1^{-\nu}\min\{1, L_1^{\frac{1}{2}}N_3^{\frac{1}{4}}N_1^{-(1-\frac{\alpha}{2})-\frac{\alpha-1}{2}}   \}
 \\
+&CRN_1^{-(\alpha-1)}N_3^{-\frac{\alpha-1}{2}}L_1^{-\nu}\min\{1,L_1^{\frac{1}{4}} N_1^{-(\alpha-1)} \}.
\end{align*}
Note that
$$ N_1^{-(\alpha-1)}N_3^{-\frac{\alpha-1}{2}}L_1^{-\nu}\min\{1,L_1^{\frac{1}{4}}N_3^{\frac{1}{4}}N_1^{-\frac{1}{4}-\frac{\alpha-1}{2}} \}\leq N_1^{-(\alpha-1)-\nu(1+2(\alpha-1))}N_3^{\nu-\frac{\alpha-1}{2}}\leq N_1^{-(\frac{3}{2}+2\nu)(\alpha-1)},
$$
since $\nu>\alpha-1$ and thus $N_3^{\frac{\nu}{2}-\frac{\alpha-1}{2}}<N_1^{\frac{\nu}{2}-\frac{\alpha-1}{2}}$. Next, 
\begin{align*}
 N_1^{-(\alpha-1)}N_3^{-\frac{\alpha-1}{2}}L_1^{-\nu}\min\{1, N_3^{\frac{1}{4}}N_1^{-\frac{5}{4}+\frac{\alpha}{2}-\frac{\alpha-1}{2}}L_1^{\frac{3}{4}} \}\leq &N_1^{-(\alpha-1)-\nu}N_3^{\frac{\nu}{3}-\frac{\alpha-1}{2}}\\ \leq &\max\{ N_1^{-(\alpha-1)-\nu}, N_1^{-\frac{3(\alpha-1)}{2}-\frac{2\nu}{3}}\}.
\end{align*}
For the third term, we have
\begin{align*}
N_1^{-(\alpha-1)}N_3^{-\frac{\alpha-1}{2}}L_1^{-\nu}\min\{1, L_1^{\frac{1}{2}}N_3^{\frac{1}{4}}N_1^{-(1-\frac{\alpha}{2})-\frac{\alpha-1}{2}} \}\leq &N_1^{-(\alpha-1)-\nu}N_3^{\frac{\nu}{2}-\frac{\alpha-1}{2}}\\ \leq &N_1^{-\frac{3}{2}(\alpha-1)-\frac{\nu}{2}}<N_1^{-(\frac{3}{2}+2\nu)(\alpha-1)},
\end{align*}
since $\alpha<\frac{5}{4}$.
Finally,
$$ N_1^{-(\alpha-1)}N_3^{-\frac{\alpha-1}{2}}L_1^{-\nu}\min\{1,L_1^{\frac{1}{4}}N_1^{-(\alpha-1)} \}\leq N_1^{-(1+4\nu)(\alpha-1)}N_3^{-\frac{\alpha-1}{2}}\leq N_1^{-(1+4\nu)(\alpha-1)}.
$$
Therefore, we have
$$ |\mathcal{U}_{L_1,L_2,L_3}^{N_1,N_2,N_3}|\leq CR\max\big\{N_1^{-(\frac{3}{2}+2\nu)(\alpha-1)},N_1^{-(\alpha-1)-\frac{3}{4}\nu},N_1^{-(1+4\nu)(\alpha-1)}\big\}.
$$
This completes the proof of Proposition \ref{CaseA-3}.
\end{proof}

\subsection{The case $N_1\sim N_3\gg N_2$}
For the high-low-high interactions, first we deal with\\
\noi
$\bullet${\bf Case A-4: At least one of $a_1,a_3$ is of type (D)}

We execute {\bf Algorithm 4}.
By inserting the indicator $\mathbf{1}_{k_1\cdot k_3\geq 0}$ and $\mathbf{1}_{k_1\cdot k_2<0}$, we have
\begin{align*}
|\mathcal{U}_{L_1,L_2,L_3}^{N_1,N_2,N_3}|^2\lesssim &\Big(\sum_{|k|\leq N_1}\sum_{(k_1,k_2,k_3)\in\Gamma(k)}|a_1(k_1)|^2|a_2(k_2)|^2\mathbf{1}_{k_1k_3\geq 0}\Big)\cdot
\Big(\sup_{k}\sum_{(k_1,k_2,k_3)\in\Gamma(k)}|a_3(k_3)|^2\Big)\\
+&\Big(\sum_{|k|\leq N_1}\sum_{(k_1,k_2,k_3)\in\Gamma(k)}|a_1(k_1)|^2|a_2(k_2)|^2\mathbf{1}_{k_1k_3< 0}\Big)\cdot
\Big(\sup_{k}\sum_{(k_1,k_2,k_3)\in\Gamma(k)}|a_3(k_3)|^2\Big).
\end{align*}
When $k_1k_3\geq 0$, for fixed $k_1,k_2$, we have
$$ |\partial_{k_3}\Phi_{k_1,k_2,k_3}|=\alpha\big|\mathrm{sign}(k_3)|k_3|^{\alpha-1}-\mathrm{sign}(k_1-k_2+k_3)|k_1-k_2+k_3|^{\alpha-1} \big|\gtrsim N_1^{\alpha-1},
$$ 
hence
$$ \sum_{|k|\leq N_1}\sum_{(k_1,k_2,k_3)\in\Gamma(\ov{k}) }|a_1(k_1)|^2|a_2(k_2)|^2\mathbf{1}_{k_1k_3\geq 0}\lesssim \|a_1\|_{l^2}^2\|a_2\|_{l^2}^2.
$$
Since $|k+k_2-k_3|=|k_1|\gg |k_2|$, for fixed $k,k_3$,
$$ |\partial_{k_2}\Phi(k+k_2-k_3,k_2,k_3)|=\alpha\big|\mathrm{sign}(k+k_2-k_3)|k+k_2-k_3|^{\alpha-1}-\mathrm{sign}(k_2)|k_2|^{\alpha-1} \big|\sim N_1^{\alpha-1}.
$$ 
Thus
$$\Big(\sum_{|k|\leq N_1}\sum_{(k_1,k_2,k_3)\in\Gamma(\ov{k})}|a_1(k_1)|^2|a_2(k_2)|^2\mathbf{1}_{k_1k_3\geq 0}\Big)\cdot
\Big(\sup_{k}\sum_{(k_1,k_2,k_3)\in\Gamma(\ov{k})}|a_3(k_3)|^2\mathbf{1}_{k_1\cdot k_3\geq 0}\Big)\lesssim \prod_{j=1}^3\|a_j\|_{l^2}^2.
$$
When $k_1k_3<0$, if sign$(k_3)\neq$ sign$(k_1-k_2+k_3)$, $|\partial_{k_3}\Phi_{k_1,k_2,k_3}|\gtrsim N_1^{\alpha-1}$. If sign$(k_3)=$sign$(k_3-(-k_1+k_2))$, we must have
$ |k_3|>|k_1-k_2|.
$ 
Without loss of generality, we may assume that $k_3>0$, hence
$$ |\partial_{k_3}\Phi_{k_1,k_2,k_3}|=\alpha(\alpha-1)\big|\int_{k_3-k_2+k_1}^{k_3}|x|^{\alpha-2}dx \big|\gtrsim \frac{|k_1-k_2|}{|k_3|^{2-\alpha}}\sim N_1^{\alpha-1}.
$$
Therefore,
$$\Big(\sum_{|k|\leq N_1}\sum_{(k_1,k_2,k_3)\in\Gamma(\ov{k})}|a_1(k_1)|^2|a_2(k_2)|^2\mathbf{1}_{k_1k_3< 0}\Big)\cdot
\Big(\sup_{k}\sum_{(k_1,k_2,k_3)\in\Gamma(\ov{k})}|a_3(k_3)|^2\mathbf{1}_{k_1\cdot k_3\geq 0}\Big)\lesssim \prod_{j=1}^3\|a_j\|_{l^2}^2.
$$
Therefore, we have proved:
\begin{proposition}\label{CaseA-4}
Assume that $N_1\sim N_3\gg N_2$ and at least one of $a_1,a_2$ is of type (D), then
$$ \mathcal{U}_{L_1,L_2,L_3}^{N_1,N_2,N_3}\lesssim \prod_{j=1}^3\|a_j\|_{l^2}\lesssim N_1^{-\frac{\alpha-1}{2}-s}N_3^{-\frac{\alpha-1}{2}}.
$$
\end{proposition}

\noi
$\bullet${\bf $a_1,a_3$ are both of type (G) or (C)  }

This situation is similar to {\bf Case A-3} and we can obtain the same upper bound. For this reason, we omit the details. Finally, we remark that from the choice of parameters $\nu, s$ in Remark \ref{numerical}, under the constraint $\alpha>\alpha_0$, there exists $\sigma>0$, sufficiently small, such that (1) of Proposition \ref{Multilinearkey} holds in the situation $N_{(1)}\sim N_{(2)}\gg N_{(3)}$.
\section{High-low-low interactions}

In this section, we finish the proof of Proposition \ref{Multilinearkey} by showing (2), (4) and the regime  $N_{(1)}\gg N_{(2)}\gtrsim N_{(1)}^{1-\delta}$ or $N_{(1)}\gg N_{(2)}$ and $N_2\sim N_{(1)}$ for (1). 

\noi
$\bullet${\bf Case B-1: $N_1\gg N_2, N_3$, $a_1$ is of type (D) and $N_{(2)}\gtrsim N_{(1)}^{1-\delta}$}

In this case, we execute {\bf Algorithm 4}. By Cauchy-Schwartz,
\begin{align*}
|\mathcal{U}_{L_1,L_2,L_3}^{N_1,N_2,N_3}|^2\leq & \Big(\sum_{k_1,k_2,k_3}S_{k_1,k_2,k_3}|a_1(k_1)|^2|a_2(k_2)|^2 \Big)\cdot\sup_{k}\Big(\sum_{k_2,k_3}|a_3(k_3)|^2S_{k+k_2-k_3,k_2,k_3} \Big).
\end{align*}
Since $N_1\gg N_2, N_3$,
$$ |\partial_{k_2}\Phi(k+k_2-k_3,k_2,k_3)|\gtrsim N_1^{\alpha-1},\quad |\partial_{k_3}\Phi_{k_1,k_2,k_3}|\gtrsim N_1^{\alpha-1},
$$
we have
$$ \sum_{k_2}S_{k+k_2-k_3,k_2,k_3}\lesssim 1,\quad \sum_{k_3}S_{k_1,k_2,k_3}\lesssim 1,
$$
hence
$$ |\mathcal{U}_{L_1,L_2,L_3}^{N_1,N_2,N_3}|\lesssim \|a_1\|_{l^2}\|a_2\|_{l^2}\|a_3\|_{l^2}.
$$
Therefore, we have proved:
\begin{proposition}\label{CaseB-1}
Assume that $N_1\gg N_2,N_3$, $N_{(2)}\gtrsim N_{(1)}^{1-\delta}$ and $a_1$ is of type (D), then
$$ |\mathcal{U}_{L_1,L_2,L_3}^{N_1,N_2,N_3}|\lesssim \prod_{j=1}^3\|a_j\|_{l^2}\lesssim N_{(1)}^{-s-\frac{\alpha-1}{2}(1-\delta)}, 
$$
where the factor $N^{-\frac{\alpha-1}{2}}$ comes from the worst case when $a_{(2)}$ is of type (G).
\end{proposition}

\noi
$\bullet${\bf Case B-2: $N_1\gg N_2, N_3$, $a_1$ is of type (G) or (C), $a_{(2)}$ is of type (D) and $N_{(2)}\gtrsim N_{(1)}^{1-\delta}$}

The estimate in this case is the same as {\bf Case B-1}, and we summarize as follows:
\begin{proposition}\label{CaseB-2}
	Assume that $N_1\gg N_2,N_3$, $N_{(2)}\gtrsim N_{(1)}^{1-\delta}$ and $a_1$ is of type (G) or (C), then
	\begin{align}\label{numerology3'}
	 |\mathcal{U}_{L_1,L_2,L_3}^{N_1,N_2,N_3}|\lesssim \prod_{j=1}^3\|a_j\|_{l^2}\lesssim N_{(1)}^{-(1-\delta)s-\frac{\alpha-1}{2}}, 
		\end{align}
	where the factor $N^{-\frac{\alpha-1}{2}}$ comes from the worst case when $a_{1}$ is of type (G).
\end{proposition}

\noi
$\bullet${\bf Case B-3: $N_1\gg N_2, N_3$, $a_{(1)},a_{(2)}$ are both of type (G) or (C) and $N_{(2)}\gtrsim N_{(1)}^{1-\delta}$}

If $a_{(1)}, a_{(2)}$ are both of type (C), then we can apply the same argument as for {\bf Case B-1} to obtain that
\begin{align}\label{numerology3''}
 |\mathcal{U}_{L_1,L_2,L_3}^{N_1,N_2,N_3}|\lesssim \prod_{j=1}^3\|a_j\|_{l^2}\lesssim N_{(1)}^{-2(\alpha-1)(1-\delta)},
\end{align}
which is conclusive. Therefore we may assume that at least one of $a_{(1)}, a_{(2)}$ is of type (G). If $a_{(1)}$ is of type (C) and $L_{(1)}\gtrsim  N_{(2)}(\gtrsim N_{(1)}^{1-\delta})$, we have the deterministic bound
\begin{align}\label{numerology3'''} |\mathcal{U}_{L_1,L_2,L_3}^{N_1,N_2,N_3}|\lesssim \prod_{j=1}^3\|a_j\|_{l^2}\lesssim N_{(1)}^{-(\alpha-1)}L_{(1)}^{-\nu}N_{(2)}^{-\frac{\alpha-1}{2}}\lesssim N_{(1)}^{-\frac{3(\alpha-1)(1-\delta)}{2}-\nu}
\end{align}
which is also conclusive. Now we assume that $L_{(1)}\ll N_{(2)}$, then 
 this situation is essentially the same as {\bf Case A-3}. Revisiting all the analysis for {\bf Case A-3} and {\bf Case A-4}, the only difference here is that we should replace $N_{(2)}\sim N_{(1)}$ by $N_{(2)}\sim N_{(1)}^{1-\delta}$. All the outputs of the summations like $\sum_{k_2}S_{k_1,k_2,k_3}, \sum_{k_3}S_{k_1,k_2,k_3}, \sum_{k_1}S_{k_1,k_2,k_3}$ and $\sum_{k}S_{k+k_2-k_3,k_2,k_3}$ are smaller than the case $N_{(1)}\sim N_{(2)}\gg N_{(3)}$, up to a loss of small power $N_{(1)}^{\delta}$. We omit the details.


\noi
$\bullet${\bf Case B-4: $N_2\gg N_1, N_3$, $a_2$ is of type (G) or (C)}

We execute {\bf Algorithm 3}. By Cauchy-Schwartz,
\begin{align*}
|\mathcal{U}_{L_1,L_2,L_3}^{N_1,N_2,N_3}|^2\leq \Big(\sum_{k_1,k_2,k_3}|a_1(k_1)|^2|a_2(k_2)|^2S_{k_1,k_2,k_3} \Big)\cdot 
\sup_k\Big(\sum_{k_1,k_3}|a_3(k_3)|^2S_{k_1,k_1+k_3-k,k_3} \Big).
\end{align*}
Note that for fixed $k,k_2$, $|\partial_{k_1}\Phi(k_1,k_1+k_3-k,k_3)|\gtrsim N_2^{\alpha-1}$ since $N_1\ll N_2$, we have
$$ \sum_{k_1}S_{k_1,k_1+k_3-k,k_3}\lesssim 1.
$$
For fixed $k_1,k_3$, by Lemma \ref{counting2ndorder},
$$ \sum_{k_2}S_{k_1,k_2,k_3}\lesssim N_2^{1-\frac{\alpha}{2}}.
$$
Thus we have
\begin{align*}
|\mathcal{U}_{L_1,L_2,L_3}^{N_1,N_2,N_3}|\lesssim N_2^{\frac{1}{2}-\frac{\alpha}{4}}\|a_1\|_{l^2}\|a_2\|_{l^{\infty}}\|a_3\|_{l^2}\leq CRN_1^{-\frac{\alpha-1}{2}}N_3^{-\frac{\alpha-1}{2}}N_2^{\frac{1}{2}-\frac{3\alpha}{4}}\leq CRN_{(1)}^{-\frac{\alpha}{4}-\frac{\alpha-1}{2}}.
\end{align*}
We have proved:
\begin{proposition}\label{CaseB-4}
Assume that $N_2\gg N_1,N_3$ and $a_2$ is of type (G) or (C). We have
$$ |\mathcal{U}_{L_1,L_2,L_3}^{N_1,N_2,N_3}|\leq CRN_{(1)}^{-\frac{\alpha}{4}-\frac{\alpha-1}{2}}.
$$
\end{proposition}
\vspace{0.3cm}

\noi
$\bullet${\bf Case B-5:
Projective terms $\Pi_{N_1}^{\perp}\mathcal{I}\mathcal{N}_3(v_1,v_2,v_3)$ and $a_1$ is of type (G) or (C)}

This time we denote slight differently by
$$ \wt{\mathcal{U}}_{L_1,L_2,L_3}^{N_1,N_2,N_3}:=\Big(\sum_{|k|>N_1}\Big|\sum_{\substack{(k_1,k_2,k_3)\in\Gamma(k)\\
|k_j|\lesssim N_j,j=1,2,3 } } a_1(k_1)\ov{a}_2(k_2)a_3(k_3) \Big|^2 \Big)^{\frac{1}{2}}.
$$
The key point here is that the range of $k_1$ in the summation satisfies
$$ N_1-(N_2\vee N_3)\leq N_1-|k_2-k_3|<|k_1|<N_1+|k_2-k_3|<N_1+(N_2\vee N_3),
$$
hence for fixed $k_2,k_3$, the range of $k$ (or $k_1$) is at most $2|k_2-k_3|$. We have the following improved counting bound (again we ignore small powers of $N_1$ in the definition of $S_{k_1,k_2,k_3}$):
\begin{align}\label{improvedcounting} \sum_{|k|>N_1}S_{k+k_2-k_3,k_2,k_3}\lesssim \min\{\lg k_2-k_3\rg,\big(1+\frac{N_1^{2-\alpha}}{\lg k_2-k_3\rg}\big) \}.
\end{align}
Now let us first execute {\bf Algorithm 1}. Without loss of generality, we assume that $N_3\leq N_2$. By Cauchy-Schwartz, we have
\begin{align*}
|\wt{\mathcal{U}}_{L_1,L_2,L_3}^{N_1,N_2,N_3}|^2\leq &\sum_{\substack{k,k_2,k_3\\|k|>N_1 }}|a_1(k+k_2-k_3)|^2|a_2(k_2)|^2S_{k+k_2-k_3,k_2,k_3} \cdot\sup_{|k|>N_1}\sum_{k_2,k_3 } |a_3(k_3)|^2S_{k+k_2-k_3,k_2,k_3}\\
\lesssim & \|a_1\|_{l_{k_1}^{\infty}}^2\|a_3\|_{l_{k_3}^2}^2\sum_{k_2,k_3}|a_2(k_ 2)|^2\min\{\lg k_2-k_3\rg,1+\frac{N_1^{2-\alpha}}{\lg k_2-k_3\rg} \}, 
\end{align*}
where we used \eqref{improvedcounting} and the fact that $\sum_{k_2}S_{k+k_2-k_3,k_2,k_3}\lesssim 1$ in the last step. Note that $\lg k_2-k_3\rg<1+\frac{N_1^{2-\alpha}}{\lg k_2-k_3\rg}$ only if $\lg k_2-k_3\rg\lesssim N_1^{1-\frac{\alpha}{2}}$. Hence if $N_2\ll N_1^{1-\frac{\alpha}{2}}$, we can bound $|\widetilde{\mathcal{U}}_{L_1,L_2,L_3}^{N_1,N_2,N_3}|^2$ by $N_2N_3\|a_1\|_{l^{\infty}}^2\|a_2\|_{l^2}^2\|a_3\|_{l^2}^2$. When $N_2\gtrsim N_1^{1-\frac{\alpha}{2}}$, we can split the sum of $k_3$ into $\lg k_3-k_2\rg\leq N_1^{1-\frac{\alpha}{2}}$ and $\lg k_3-k_2\rg>N_1^{1-\frac{\alpha}{2}}$.The sum (over $k_2,k_3$) for the former case can be bounded by
$$ \|a_2\|_{l^2}^2\min\{N_3N_1^{1-\frac{\alpha}{2}},N_1^{2-\alpha}
\log(N_2) \},
$$
while the sum for the later case can be bounded by
$$ \|a_2\|_{l^2}^2(N_3+\min\big\{N_1^{2-\alpha}\log(N_3),N_3N_1^{1-\frac{\alpha}{2}} \big\} ).
$$
Therefore, we have
\begin{align*}
|\widetilde{\mathcal{U}}_{L_1,L_2,L_3}^{N_1,N_2,N_3}|^2\lesssim \|a_1\|_{l^{\infty}}^2\|a_2\|_{l^2}^2\|a_3\|_{l^2}^2(N_3+\min\{N_1^{2-\alpha}+N_3N_1^{1-\frac{\alpha}{2}} \} )\log(N_2).
\end{align*}

\begin{lemme}\label{Projective-1}
Assume that $N_2,N_3<N_1^{1-\delta}$. Then modulo small powers of $N_1$, we have:
\begin{align*}
&\mathrm{(1)}\quad  |\wt{\mathcal{U}}_{L_1,L_2,L_3}^{N_1,N_2,N_3}| \lesssim (N_2N_3)^{\frac{1}{2}}\|a_1\|_{l^{\infty}}\|a_2\|_{l^2}\|a_3\|_{l^2},\quad \text{ if }N_2\vee N_3\ll N_1^{1-\frac{\alpha}{2}};\\
&\mathrm{(2)}\quad |\wt{\mathcal{U}}_{L_1,L_2,L_3}^{N_1,N_2,N_3}| \lesssim \big[(N_2\wedge N_3)^{\frac{1}{2}}+N_1^{\frac{1}{2}-\frac{\alpha}{4}}\min\{(N_2\wedge N_3)^{\frac{1}{2}},N_1^{\frac{1}{2}-\frac{\alpha}{4}} \} \big] \|a_1\|_{l^{\infty}}\|a_2\|_{l^2}\|a_3\|_{l^2},\\ &\hspace{3cm}\text{ if }N_2\vee N_3\gtrsim N_1^{1-\frac{\alpha}{2}}.
\end{align*}
\end{lemme}
Consequently, we have:
\begin{corollaire}\label{projective-2}
Assume that $N_2,N_3<N_1^{1-\delta}$. Then modulo small powers of $N_1$:
\begin{itemize}
	\item If $a_2,a_3$ are both of type (D), then 
	\begin{align}\label{numerology4}
	 |\wt{\mathcal{U}}_{L_1,L_2,L_3}^{N_1,N_2,N_3}|\lesssim N_1^{-\frac{\alpha-1}{2}-\frac{\alpha^2}{8}}
	\end{align}
	\item If $L_1\gtrsim N_2\wedge N_3$, then 
	\begin{align}\label{numerology5}
	|\wt{\mathcal{U}}_{L_1,L_2,L_3}^{N_1,N_2,N_3}|\lesssim N_1^{-(\alpha-1+\nu)(1-\frac{\alpha}{2})-(\alpha-1)}+N_1^{-\frac{3(\alpha-1)}{2}-\nu}.
		\end{align}
\end{itemize}	
In particular, when $\alpha>\alpha_0$, for sufficiently small free parameter $\sigma>0$, we have
$$ |\widetilde{\mathcal{U}}_{L_1,L_2,L_3}^{N_1,N_2,N_3}|\lesssim N_1^{-s-\delta_0}.
$$
\end{corollaire}
\begin{proof}
Assume that $a_2,a_3$ are both of type (D). From Lemma \ref{Projective-1}, when $N_2\vee N_3\ll N_1^{1-\frac{\alpha}{2}}$, we have $$|\wt{\mathcal{U}}_{L_1,L_2,L_3}^{N_1,N_2,N_3}|\lesssim (N_2N_3)^{\frac{1}{2}-s}N_1^{-\frac{\alpha}{2}}\leq N_1^{-\frac{\alpha}{2}}(N_2\vee N_3)^{1-2s}\ll N_1^{-\frac{\alpha^2}{4}},$$
since $s=\frac{1}{2}-\frac{\alpha}{4}+\sigma>\frac{1}{2}-\frac{\alpha}{4}$.
Now we assume without loss of generality that $N_3\leq N_2, N_2\gtrsim N_1^{1-\frac{\alpha}{2}}$. If $N_3\leq N_1^{1-\frac{\alpha}{2}}$, we have
\begin{align*} 
|\wt{\mathcal{U}}_{L_1,L_2,L_3}^{N_1,N_2,N_3}|\lesssim &N_3^{\frac{1}{2}}N_1^{\frac{1}{2}-\frac{\alpha}{4}}\cdot N_1^{-\frac{\alpha}{2}}N_2^{-s}N_3^{-s}\leq N_3^{\frac{1}{2}-s}N_1^{\frac{1}{2}(1-\frac{\alpha}{2})-s(1-\frac{\alpha}{2})}N_1^{-\frac{\alpha}{2}}\\
\leq & N_1^{(\frac{1}{2}-s)(1-\frac{\alpha}{2})+(\frac{1}{2}-s)(1-\frac{\alpha}{2})}N_1^{-\frac{\alpha}{2}}\leq N_1^{-\frac{\alpha^2}{4}+\delta_0}N_2^{-\delta_0}.
\end{align*}
When $N_3>N_1^{1-\frac{\alpha}{2}}$, we have
\begin{align*}
|\wt{\mathcal{U}}_{L_1,L_2,L_3}^{N_1,N_2,N_3}|\lesssim & (N_3^{\frac{1}{2}}+N_1^{1-\frac{\alpha}{2}} )\cdot N_1^{-\frac{\alpha}{2}}N_2^{-s}N_3^{-s}\leq (N_2^{\frac{1}{2}-2s}+N_1^{1-\frac{\alpha}{2} }N_2^{-s} )N_1^{-\frac{\alpha}{2}}\\  \leq &(N_1^{-\frac{\alpha-1}{2}-2s}+N_1^{-\frac{\alpha-1}{2}-\frac{\alpha^2}{8}} ).
\end{align*}
Note that when $\alpha>\alpha_0$, $\frac{\alpha-1}{2}+\frac{\alpha^2}{8}>s$ for very small free numerical parameter $\sigma>0$, thus we obtain the first inequality.

 Next we assume that $L_1\gtrsim N_2\wedge N_3$, from 
$\|a_1\|_{l^{\infty}}\lesssim N_1^{-\frac{\alpha}{2}}L_1^{-\nu}\lesssim N_1^{-\frac{\alpha}{2}}(N_2\wedge N_3)^{-\nu}, 
$ we have, if $N_2\vee N_3\ll N_1^{1-\frac{\alpha}{2}}$,
$$ |\wt{\mathcal{U}}_{L_1,L_2,L_3}^{N_1,N_2,N_3}|\lesssim (N_2N_3)^{\frac{1}{2}}N_1^{-\frac{\alpha}{2}}(N_2\wedge N_3)^{-\nu}\|a_2\|_{l^2}\|a_3\|_{l^2}.
$$
In the worst case, $\|a_j\|_{l^2}\lesssim N_j^{-\frac{\alpha-1}{2}}$ for $j=2,3$, we have
$$|\wt{\mathcal{U}}_{L_1,L_2,L_3}^{N_1,N_2,N_3}|\lesssim (N_2\vee N_3)^{2-\alpha-\nu}N_1^{-\frac{\alpha}{2}}<N_1^{2-\frac{3\alpha}{2}-(\alpha+\nu)(1-\frac{\alpha}{2}) }, $$
since $2-\alpha-\nu>0$.
Now we assume that $N_2\gtrsim N_1^{1-\frac{\alpha}{2}}$ and $N_3\leq N_2$. If $N_3\leq N_1^{1-\frac{\alpha}{2}}$, we have
\begin{align*}
|\wt{\mathcal{U}}_{L_1,L_2,L_3}^{N_1,N_2,N_3}|\lesssim &
N_3^{\frac{1}{2}}N_1^{\frac{1}{2}(1-\frac{\alpha}{2})}\cdot N_1^{-\frac{\alpha}{2}}N_2^{-\frac{\alpha-1}{2}}N_3^{-\nu-\frac{\alpha-1}{2}}\leq N_1^{(1-\frac{\alpha}{2}-\nu-\frac{\alpha-1}{2} )(1-\frac{\alpha}{2}) +\frac{1}{2}(1-\frac{\alpha}{2}) }\cdot N_1^{-\frac{\alpha}{2}}\\
\leq &N_1^{-(\alpha+\nu)(1-\frac{\alpha}{2})+2-\frac{3\alpha}{2}}=N_1^{-(\alpha-1+\nu)(1-\frac{\alpha}{2})-(\alpha-1)},
\end{align*}
since $1-\frac{\alpha}{2}-\nu>0$.
 When $N_3>N_1^{1-\frac{\alpha}{2}}$, we have
 \begin{align}
 |\wt{\mathcal{U}}_{L_1,L_2,L_3}^{N_1,N_2,N_3}|\lesssim &(N_3^{\frac{1}{2}}+N_1^{1-\frac{\alpha}{2}} )\cdot N_1^{-\frac{\alpha}{2}}N_2^{-\frac{\alpha-1}{2}}N_3^{-\nu-\frac{\alpha-1}{2}}\notag \\ \leq &N_3^{\frac{3}{2}-\alpha-\nu}N_1^{-\frac{\alpha}{2}}+N_1^{-(\alpha-1)}N_2^{-\frac{\alpha-1}{2}}N_3^{-\nu-\frac{\alpha-1}{2}}\notag  \\
 \lesssim & N_1^{-\nu-\frac{3(\alpha-1)}{2}}+N_1^{-(\alpha-1+\nu)(1-\frac{\alpha}{2})-(\alpha-1)}, \notag
 \end{align}
 where to the last step we used $N_3^{\frac{3}{2}-\alpha-\nu}<N_1^{\frac{3}{2}-\alpha-\nu}$.  
When $\alpha>\alpha_0$, we have $(\alpha-1+\nu)(1-\frac{\alpha}{2})+(\alpha-1)>s+\delta_0$ and $\frac{3(\alpha-1)}{2}+\nu>s+\delta_0$. This completes the proof of Lemma \ref{projective-2}.
\end{proof}
To deal with other situations, we need to execute other Algorithms.

\noi
$\bullet${\bf Subcase B-5(a): Exact one of $a_2,a_3$ is of type (G) or (C) and the other is of type (D) }

Without loss of generality, we may assume that $a_2$ is of type (G) or (C) and $a_3$ is of type (D), since in the regime $N_1\gg N_2,N_3$, the second and third positions in $\mathcal{N}_3(\cdot,\cdot,\cdot)$ are similar. By Corollary \ref{projective-2}, it suffices to consider the case $L_1\ll N_2\wedge N_3\leq N_2$. By implementing {\bf Algorithm 2}, it suffices to estimate the quantity
$$ \sum_{k,k':k\neq k'} \mathbb{E}^{\mathcal{C}}[|\sigma_{kk'}^{(3)}|^2]+\sup_{k}\mathbb{E}^{\mathcal{C}}[|\sigma_{kk}^{(3)}|^2],
$$
and then take the square root of the obtained upper bound, where $\mathcal{C}$ is the $\sigma$-algebra generated by $\mathcal{B}_{L_1\vee L_2}$,
\begin{align*}
\sigma_{kk'}^{(3)}=&\sum_{\substack{k_1,k_1',k_2,k_2',k_3\\
k_1=k+k_2-k_3\\
k_1'=k'+k_2'-k_3  }}\ov{a}_1(k_1')a_1(k_1)a_2(k_2')\ov{a}_2(k_2)S_{k_1',k_2',k_3}S_{k_1,k_2,k_3}\mathbf{1}_{|k|,|k'|>N_1}
\end{align*}
and
$$ a_j(k_j)=\sum_{|k_j^*|\sim N_j}\mathbf{1}_{|k_j-k_j^*|\leq L_j}h_{k_jk_j^*}^{(q)}\frac{g_{k_j^*}}{[k_j^*]^{\frac{\alpha}{2}}},\quad j=1,2.
$$

\begin{lemme}\label{SubcaseB5a}
	Assume that $a_1,a_2$ are both of type (G) or (C) and $L_1\ll N_2$. Then by implementing {\bf Algorithm 2}, outside a set of probability $<\mathrm{e}^{-N_{(1)}^{\theta}R }$ and modulo small powers of $N_{(1)}$:
	\begin{itemize} 
	\item If $N_2\vee N_3<N_1^{1-\frac{\alpha}{2}}$, then
	\begin{align*}
	|\wt{\mathcal{U}}_{L_1,L_2,L_3}^{N_1,N_2,N_3}|\lesssim 
	N_1^{-\frac{\alpha}{2}}N_2^{-\frac{\alpha}{2}}L_1^{-\nu}L_2^{\frac{1}{2}-\nu}(N_2\vee N_3)(N_2\wedge N_3)^{\frac{1}{4}}
\|a_3\|_{l^2}
	\end{align*}
\item If $N_2\vee N_3\geq N_1^{1-\frac{\alpha}{2}}$, then
\begin{align*}
\hspace{1.5cm}	|\wt{\mathcal{U}}_{L_1,L_2,L_3}^{N_1,N_2,N_3}|\lesssim 
	&N_1^{-\frac{\alpha}{2}}N_2^{-\frac{\alpha}{2}}L_1^{-\nu}L_2^{\frac{1}{2}-\nu}
	\big[N_2^{\frac{1}{2}}N_3^{\frac{1}{4}}+N_3^{\frac{1}{4}}N_1^{\frac{1}{2}(1-\frac{\alpha}{2}) }\min\{N_2^{\frac{1}{2}},N_1^{\frac{1}{2}(1-\frac{\alpha}{2}) } \}  \big]\|a_3\|_{l^2} \\+&N_1^{-\frac{\alpha}{2}}N_2^{-\frac{\alpha}{2}}L_1^{-\nu}L_2^{\frac{1}{2}-\nu}\cdot N_2^{\frac{1}{2}}N_1^{\frac{1}{2}(1-\frac{\alpha}{2})}\min\{N_3^{\frac{1}{4}},N_1^{\frac{1}{4}(1-\frac{\alpha}{2})} \} \|a_3\|_{l^2} .
	\end{align*}
\end{itemize}
If $a_1,a_3$ are both of type (G) or (C) and $a_2$ is of type (D) and $L_1\ll N_3$, the above estimates hold by switching $N_2$ to $N_3$, $a_3$ to $a_2$ and $L_2$ to $L_3$.
\end{lemme}
\begin{proof}
From our assumption, $a_1(k_1'),a_1(k_1)$ are independent of $a_2(k_2), a_2(k_2')$. Noticing that
\begin{align*}
 |\sigma_{kk'}^{(3)}|^2&=\sum_{\substack{k_2,k_2',k_3\\
 m_2,m_2',m_3 }  } S_{k'+k_2'-k_3,k_2',k_3}S_{k+k_2-k_3,k_2,k_3}S_{k'+m_2'-m_3,m_2',m_3} S_{k+m_2-m_3,m_2,m_3}\\
 &\hspace{1cm}\times\ov{a}_1(k'+k_2'-k_3)a_1(k+k_2-k_3)a_1(k'+m_2'-m_3)\ov{a}_1(k+m_2-m_3)\\ &\hspace{1cm}\times a_2(k_2')\ov{a}_2(k_2)\ov{a}_2(m_2')a_2(m_2).
 \end{align*}
 By using the independence and Cauchy-Schwartz, we have
\begin{align*}
&\sum_{k,k':|k|>N_1, |k'|>N_1 }\mathbb{E}[|\sigma_{kk'}^{(3)}|^2]\\ \lesssim  &(N_1N_2)^{-2\alpha}\sum_{\substack{ k,k',k_2,k_2',k_3\\
m_2,m_2',m_3    } } S_{k'+k_2'-k_3,k_2',k_3}S_{k+k_2-k_3,k_2,k_3}S_{k'+m_2'-m_3,m_2',m_3} S_{k+m_2-m_3,m_2,m_3}\mathbf{1}_{|k|,|k'|>N_1}\\
\times &\prod_{j=1}^2\|h_{k_jk_j^*}^{(q)}\|_{l_{k_j}^{\infty}l_{k_j^*}^2 }^4
 \Big(\mathbf{1}_{|k_2-k_2'|\leq L_2,|m_2-m_2'|\leq L_2 
 }+ \mathbf{1}_{|k_2-m_2|<L_2, |k_2'-m_2'|<L_2
  }
 \Big).
\end{align*}
Denote by $\mathcal{C}_1$ the contribution from the indicator
$
 \mathbf{1}_{|k_2-k_2'|\leq L_2,
|m_2-m_2'|\leq L_2 }.
$
We first sum over $m_3$, using the fact that
$$ \sum_{m_3}S_{k+m_2-m_3,m_2,m_3}S_{k'+m_2'-m_3,m_2',m_3}\lesssim 1,
$$
and then we sum over $k,k'$ by using the inequality \eqref{improvedcounting}. This yields
\begin{align*}
\mathcal{C}_1\lesssim & (N_1N_2)^{-2\alpha}\prod_{j=1}^2\|h_{k_jk_j^*}^{(q)}\|_{l_{k_j}^{\infty}l_{k_j^*}^{2} }^4\sum_{k_2,k_2',k_3,m_2,m_2' }\mathbf{1}_{\substack{|k_2-k_2'|<L_2\\
|m_2-m_2'|<L_2 }} B(k_2,k_3)B(k_2',k_3)\\
\lesssim &(N_1N_2)^{-2\alpha}L_2\prod_{j=1}^2\|h_{k_jk_j^*}^{(q)}\|_{l_{k_j}^{\infty}l_{k_j^*}^{2} }^4\sum_{k_2,k_2',k_3,m_2 }\mathbf{1}_{|k_2-k_2'|<L_2} B(k_2,k_3)B(k_2',k_3),
\end{align*}
where
$$ B(k_2,k_3):=\min\big\{1+\frac{N_1^{2-\alpha}}{\lg k_2-k_3 \rg },\lg k_2-k_3\rg  \big\}. 
$$ 
Now by Schur's test and Cauchy-Schwartz, we have
\begin{align*}
&\sum_{k_3,m_2}\sum_{k_2,k_2'}\mathbf{1}_{|k_2-k_2'|<L_2}B(k_2,k_3)B(k_2',k_3)\\
\leq &L_2N_2\|B(k_2,k_3)\|_{l_{k_2,k_3}^2}^2
\lesssim \begin{cases}
& L_2(N_2\vee N_3)^4(N_2\wedge N_3), \text{ if } N_2\vee N_3<N_1^{1-\frac{\alpha}{2}} \\
&L_2N_2^{2}(N_3+N_1^{2-\alpha}\min\{N_3,N_1^{1-\frac{\alpha}{2}} \} ),\text{ if }N_2\vee N_3\geq N_1^{1-\frac{\alpha}{2}}.
\end{cases}
\end{align*}
Therefore,  
\begin{align}\label{C1}
\mathcal{C}_1\lesssim (N_1N_2)^{-2\alpha}L_1^{-4\nu}L_2^{2-4\nu}\cdot
\begin{cases}
&(N_2\vee N_3)^4(N_2\wedge N_3),\text{ if }N_2\vee N_3<N_1^{1-\frac{\alpha}{2}};\\
& N_2^2N_3+N_1^{2-\alpha}N_2^2\min\{N_3,N_1^{1-\frac{\alpha}{2}} \}, \text{ if } N_2\vee N_3\geq N_1^{1-\frac{\alpha}{2}}. 
\end{cases}
\end{align}
Denote by $\mathcal{C}_2$ the contribution from the indicator
$
\mathbf{1}_{|k_2-m_2|\leq L_2,
		|k_2'-m_2'|\leq L_2 }.
$
Similarly, we first sum over $m_3$ and then $k,k'$ using \eqref{improvedcounting}, we have
\begin{align*}
\mathcal{C}_2\lesssim &(N_1N_2)^{-2\alpha}\prod_{j=1}^2\|h_{k_jk_j^*}^{(q)}\|_{l_{k_j}^{\infty}l_{k_j^*}^2 }^4\sum_{k_2,k_2',k_3,m_2,m_2' } \mathbf{1}_{\substack{|k_2-m_2|<L_2\\ 
|k_2'-m_2'|<L_2} }
B(k_2,k_3)B(k_2',k_3)\\
\lesssim &
(N_1N_2)^{-2\alpha}L_2^2\prod_{j=1}^2\|h_{k_jk_j^*}^{(q)}\|_{l_{k_j}^{\infty}l_{k_j^*}^2 }^4
\sum_{k_2,k_2',k_3}B(k_2,k_3)B(k_2',k_3).
\end{align*}
For fixed $k_3$, we have
\begin{align*}
\sum_{k_2}B(k_2,k_3)\lesssim \begin{cases}
& (N_2\vee N_3)N_2,\text{ if }N_2\vee N_3<N_1^{1-\frac{\alpha}{2}}\\
& N_2+\min\{N_1^{1-\frac{\alpha}{2}}N_2, N_1^{2-\alpha} \},\text{ if } N_2\vee N_3\geq N_1^{1-\frac{\alpha}{2}}.
\end{cases}
\end{align*}
Therefore, modulo small powers,
\begin{align}\label{C2} 
\mathcal{C}_2\lesssim &(N_1N_2)^{-2\alpha}L_1^{-4\nu}L_2^{2-4\nu}\cdot
\begin{cases}
& (N_2\vee N_3)^4(N_2\wedge N_3),\text{ if } N_2\vee N_3<N_1^{1-\frac{\alpha}{2}}\\
& N_2^2N_3+N_3N_1^{2-\alpha}\min\{N_2^2,N_1^{2-\alpha}\},\text{ if }N_2\vee N_3\geq N_1^{1-\frac{\alpha}{2}}.
\end{cases}
\end{align}
 Implementing {\bf Algorithm 2}, the proof of Lemma \ref{SubcaseB5a} is complete.
\end{proof}

\begin{corollaire}\label{Cor5Ba}
	Assume that $N_1\gg N_2,N_3$ and $a_1,a_2$ are of type (G) or (C), $a_3$ is of type (D). Assume that $L_1\ll N_2$, then outside a set of probability $<\mathrm{e}^{-N_{(1)}^{\theta}R }$ and modulo a small power $N_1^{\epsilon},\epsilon<\kappa^{-0.1}$, we have
	\begin{align} |\wt{\mathcal{U}}_{L_1,L_2,L_3}^{N_1,N_2,N_3}|\lesssim &N_1^{-s-2\delta_0}.
 	\end{align}
	The same estimate holds if we switch $a_2$ and $a_3$.
\end{corollaire}

\begin{proof}
The proof is just a numerical computation.
Assume that $a_2$ is of type (G) or (C). 
First we assume that $N_2\vee N_3<N_1^{1-\frac{\alpha}{2}}$. Then by (i) of Lemma \ref{SubcaseB5a}, if $N_3\leq N_2<N_1^{1-\frac{\alpha}{2}}$,
\begin{align}\label{sL2} |\wt{\mathcal{U}}_{L_1,L_2,L_3}^{N_1,N_2,N_3}|\lesssim  &(N_1N_2)^{-\frac{\alpha}{2}}N_3^{-s} L_2^{\frac{1}{2}-\nu}\cdot N_2N_3^{\frac{1}{4}}<N_1^{-\frac{\alpha}{2}}L_2^{\frac{1}{2}-\nu}N_2^{\frac{5}{4}-s-\frac{\alpha}{2}}<L_2^{\frac{1}{2}-\nu}N_1^{-\frac{\alpha}{2}+(\frac{5}{4}-s-\frac{\alpha}{2})(1-\frac{\alpha}{2}) }\notag \\
\lesssim & N_1^{-s}\cdot L_2^{\frac{1}{2}-\nu}N_1^{-\frac{(1-s)\alpha}{2}+(\frac{5}{4}-\frac{\alpha}{2})(1-\frac{\alpha}{2}) }.
\end{align}
 Note that if $N_2\leq N_3<N_1^{1-\frac{\alpha}{2}}$, we have the same upper bound. When $L_2^{\frac{1}{2}-\nu}<N_1^{\frac{(1-s)\alpha}{2}-(\frac{5}{4}-\frac{\alpha}{2})(1-\frac{\alpha}{2})-2\delta_0 }$, this bound is conclusive. When $L_2^{\frac{1}{2}-\nu}\geq N_1^{\frac{(1-s)\alpha}{2}-(\frac{5}{4}-\frac{\alpha}{2})(1-\frac{\alpha}{2})-2\delta_0}$, we alternatively apply (1) of Lemma \ref{Projective-1} to get
\begin{align}\label{bL2} 
|\wt{\mathcal{U}}_{L_1,L_2,L_3}^{N_1,N_2,N_3}|\lesssim &N_1^{-\frac{\alpha}{2}}N_2^{\frac{1}{2}-(\alpha-1)}L_2^{-\nu}N_3^{\frac{1}{2}-s}\leq N_1^{-\frac{\alpha}{2}}L_2^{-\nu}(N_2\vee N_3)^{2-\alpha-s}\notag \\<&L_2^{-\nu}N_1^{-\frac{\alpha}{2}+(2-\alpha-s)(1-\frac{\alpha}{2})}\leq N_1^{-\frac{\alpha}{2}+(2-\alpha-s)(1-\frac{\alpha}{2}) }\cdot N_1^{-\frac{2\nu}{1-2\nu}\cdot \big[\frac{(1-s)\alpha}{2}-(\frac{5}{4}-\frac{\alpha}{2})(1-\frac{\alpha}{2})-2\delta_0 \big] }.
\end{align}
By numerical computation, when $\alpha>1.069 (<\alpha_0)$, we are able to choose sufficiently small $\sigma>0$, such that $|\wt{\mathcal{U}}_{L_1,L_2,L_3}^{N_1,N_2,N_3}|\lesssim N_1^{-s-2\delta_0}$, which is conclusive. 

Next we assume that $N_2\vee N_3\geq N_1^{1-\frac{\alpha}{2}}$. First we deal with the case $N_2,N_3\geq N_1^{1-\frac{\alpha}{2}}$. By (ii) of Lemma \ref{SubcaseB5a}, we have
\begin{align*}
 |\wt{\mathcal{U}}_{L_1,L_2,L_3}^{N_1,N_2,N_3}|\lesssim &\underbrace{N_1^{-\frac{\alpha}{2}}N_2^{-\frac{\alpha}{2}}L_2^{\frac{1}{2}-\nu}N_3^{\frac{1}{4}-s}N_2^{\frac{1}{2}}}_{\mathrm{I}}+
 \underbrace{N_1^{-\frac{\alpha}{2}}N_2^{-\frac{\alpha}{2}}L_2^{\frac{1}{2}-\nu}N_3^{\frac{1}{4}-s}N_1^{1-\frac{\alpha}{2}}}_{\mathrm{II}}\\
 +&\underbrace{N_1^{-\frac{\alpha}{2}}N_2^{-\frac{\alpha}{2}}L_2^{\frac{1}{2}-\nu}N_2^{\frac{1}{2}}N_1^{\frac{3}{4}(1-\frac{\alpha}{2})}N_3^{-s}  }_{\mathrm{III}}.
\end{align*}
Using $L_2\ll N_2$, we have
\begin{align*}
\mathrm{I}\leq &N_1^{-\frac{\alpha}{2}}N_2^{1-\frac{\alpha}{2}-\nu }N_3^{\frac{1}{4}-s}\leq N_1^{-\frac{\alpha}{2}+1-\frac{\alpha}{2}-\nu+\frac{1}{4}-s}<N_1^{-s-2\delta_0},
\end{align*}
for sufficiently small $\sigma>0$, since $\alpha+\min\{\frac{1}{2}-\frac{\alpha}{4},\frac{7(\alpha-1)}{4} \}>\frac{5}{4}$.
Next,
\begin{align*}
\mathrm{II}\leq &N_1^{-(\alpha-1)}N_2^{-\frac{\alpha}{2}}L_2^{\frac{1}{2}-\nu}N_1^{\frac{1}{4}-(\frac{1}{2}-\frac{\alpha}{4}) }=N_1^{-\frac{3(\alpha-1)}{4}}N_2^{-\frac{\alpha}{2}}L_2^{\frac{1}{2}-\nu}<N_1^{-\frac{\alpha}{2}(1-\frac{\alpha}{2})-\frac{3(\alpha-1)}{4}}L_2^{\frac{1}{2}-\nu}.
\end{align*}
\begin{align*}
\mathrm{III}\leq N_1^{-\frac{\alpha}{2}+\frac{3}{4}(1-\frac{\alpha}{2})}L_2^{\frac{1}{2}-\nu}N_2^{-\frac{\alpha-1}{2}}N_3^{-s}\leq N_1^{-(\alpha-1)-\frac{1+\alpha}{4}(1-\frac{\alpha}{2})}L_2^{\frac{1}{2}-\nu}.
\end{align*}
Since $N_1^{-(\alpha-1)-\frac{1+\alpha}{4}(1-\frac{\alpha}{2})}<N_1^{-\frac{\alpha}{2}(1-\frac{\alpha}{2})-\frac{3(\alpha-1)}{4} },$ we have
\begin{align}\label{sL2'} \mathrm{I}+\mathrm{II}+\mathrm{III}\lesssim N_1^{-s-2\delta_0}+L_2^{\frac{1}{2}-\nu}N_1^{-s}N_1^{s-\frac{\alpha}{2}(1-\frac{\alpha}{2})-\frac{3(\alpha-1)}{4}}.
\end{align}
When $L_2^{\frac{1}{2}-\nu}<N_1^{\frac{\alpha}{2}(1-\frac{\alpha}{2})+\frac{3(\alpha-1)}{4}-s-2\delta_0}$, the upper bound \eqref{sL2'} is conclusive. When $L_2^{\frac{1}{2}-\nu}\geq N_1^{\frac{\alpha}{2}(1-\frac{\alpha}{2})+\frac{3(\alpha-1)}{4}-s-2\delta_0},$ we alternatively apply (2) of Lemma \ref{Projective-1} to obtain
\begin{align}\label{bL2'} 
|\wt{\mathcal{U}}_{L_1,L_2,L_3}^{N_1,N_2,N_3}|\lesssim & (N_2^{\frac{1}{2}}+N_1^{1-\frac{\alpha}{2}})N_1^{-\frac{\alpha}{2}}N_2^{-(\alpha-1)}L_2^{-\nu}N_3^{-s}\notag \\
\lesssim & N_1^{-\frac{\alpha}{2}}N_2^{\frac{3}{2}-\alpha}N_3^{-s}L_2^{-\nu}+N_1^{-(\alpha-1)}N_2^{-(\alpha-1)}N_3^{-s}L_2^{-\nu}\notag \\
\leq & N_1^{-(\alpha-1)(2-\frac{\alpha}{2})-s(1-\frac{\alpha}{2})}L_2^{-\nu}<N_1^{-(\alpha-1)(2-\frac{\alpha}{2})-s(1-\frac{\alpha}{2})-\frac{2\nu}{1-2\nu}\cdot \big[\frac{\alpha}{2}(1-\frac{\alpha}{2})+\frac{3(\alpha-1)}{4}-s-2\delta_0\big] }.
\end{align}
By numerical computation, we are able to choose sufficiently small $\sigma>0$, such that the above upper bound can be bounded by $N_1^{-s-2\delta_0}$, provided that $\alpha>1.0698 (<\alpha_0)$.
Thus
$$\mathrm{I}+\mathrm{II}+\mathrm{III}\lesssim 
N_1^{-s-2\delta_0}.
$$ 
Finally we assume that $N_2\vee N_3\geq N_1^{1-\frac{\alpha}{2}}>N_2\wedge N_3$.
 By using (2) of Lemma \ref{SubcaseB5a}, we have
\begin{align*}
|\wt{\mathcal{U}}_{L_1,L_2,L_3}^{N_1,N_2,N_3}|\lesssim & N_1^{-\frac{\alpha}{2}}N_2^{-\frac{\alpha}{2}}L_2^{\frac{1}{2}-\nu}N_3^{\frac{1}{4}-s}\cdot N_2^{\frac{1}{2}}N_1^{\frac{1}{2}(1-\frac{\alpha}{2})} ,\text{ if }N_3<N_1^{1-\frac{\alpha}{2}}\leq N_2,
\end{align*}
and
\begin{align*}
|\wt{\mathcal{U}}_{L_1,L_2,L_3}^{N_1,N_2,N_3}|\lesssim & N_1^{-\frac{\alpha}{2}}N_2^{-\frac{\alpha-1}{2}}L_2^{\frac{1}{2}-\nu}N_3^{-s}\cdot N_3^{\frac{1}{4}}N_1^{\frac{1}{2}(1-\frac{\alpha}{2})} ,\text{ if }N_2<N_1^{1-\frac{\alpha}{2}}\leq N_3.
\end{align*}
If $N_2\wedge N_3<N_1^{1-\frac{\alpha}{2}}\leq N_2\vee N_3$, we have
\begin{align}\label{sL2''}
|\wt{\mathcal{U}}_{L_1,L_2,L_3}^{N_1,N_2,N_3}|\lesssim & N_1^{-\frac{\alpha}{2}+\frac{1}{2}(1-\frac{\alpha}{2})}N_1^{\frac{1}{4}-s}L_2^{\frac{1}{2}-\nu}\leq N_1^{-s }\cdot N_1^{-\frac{3(\alpha-1)}{4}}L_2^{\frac{1}{2}-\nu}.
\end{align}
When $L_2^{\frac{1}{2}-\nu}<N_1^{\frac{3(\alpha-1)}{4}-2\delta_0}$, this bound is conclusive. When $L_2^{\frac{1}{2}-\nu}\geq N_1^{\frac{3(\alpha-1)}{4}-2\delta_0}$, we can alternatively apply (2) of Lemma \ref{Projective-1} to obtain
\begin{align*}
|\wt{\mathcal{U}}_{L_1,L_2,L_3}^{N_1,N_2,N_3}|\lesssim & (N_2\wedge N_3)^{\frac{1}{2}}N_1^{\frac{1}{2}-\frac{\alpha}{4}}\cdot N_1^{-\frac{\alpha}{2}}N_2^{-(\alpha-1)}L_2^{-\nu}N_3^{-s}\\
\leq & N_1^{\frac{1}{2}-\frac{3\alpha}{4}}(N_2\vee N_3)^{-(\alpha-1)}(N_2\wedge N_3)^{\frac{1}{2}-s}L_2^{-\nu},
\end{align*}
since $\alpha<\frac{6}{5}$ and $\alpha-1<s$.
Thus
\begin{align*} |\wt{\mathcal{U}}_{L_1,L_2,L_3}^{N_1,N_2,N_3}|\lesssim N_1^{\frac{1}{2}-\frac{3\alpha}{4}+(\frac{1}{2}-s-(\alpha-1))(1-\frac{\alpha}{2}) -\frac{2\nu}{1-2\nu}\cdot \big[\frac{3(\alpha-1)}{4}-2\delta_0\big] }.
\end{align*}
By numerical computation, when $\alpha>1.0724 (<\alpha_0)$, we are able to choose sufficiently small $\sigma>0$, such that 
$$ |\wt{\mathcal{U}}_{L_1,L_2,L_3}^{N_1,N_2,N_3}|\lesssim N_1^{-s-2\delta_0}.
$$
This completes the proof of Corollary \ref{Cor5Ba}.
\end{proof}

\noi
$\bullet${\bf Subcase B-5(b):  $a_2,a_3$ are both of type (G) or (C) and $L_1\ll N_2\wedge N_3$ }

First we assume that $N_3\gg N_2$ and $L_3\gtrsim N_2$ (similar for the case $N_2\gg N_3$ and $L_2\gtrsim N_3$). In this case, we cannot execute {\bf Algorithm 1} since $h_{k_2k_2^*}^{(q)}$ are not independent of $g_{k_3^*}$. Instead, we apply Lemma \ref{Projective-1}. When $N_3<N_1^{1-\frac{\alpha}{2}}$, we have
\begin{align*}
|\wt{\mathcal{U}}_{L_1,L_2,L_3}^{N_1,N_2,N_3}|\lesssim & (N_2N_3)^{\frac{1}{2}}N_1^{-\frac{\alpha}{2}}N_2^{-\frac{\alpha-1}{2}}N_3^{-(\alpha-1)}L_3^{-\nu}\lesssim N_1^{-\frac{\alpha}{2}}N_2^{1-\frac{\alpha}{2}-\nu}N_3^{\frac{3}{2}-\alpha}\notag\\
\leq & N_3^{\frac{5}{2}-\frac{3\alpha}{2}-\nu}N_1^{-\frac{\alpha}{2}}<N_1^{-\frac{\alpha}{2}+(\frac{5-3\alpha}{2}-\nu)(1-\frac{\alpha}{2})}.
\end{align*}
When $N_3\gg N_2\geq N_1^{1-\frac{\alpha}{2}}$, we have
\begin{align*}
|\wt{\mathcal{U}}_{L_1,L_2,L_3}^{N_1,N_2,N_3}|\lesssim &(N_2^{\frac{1}{2}}+N_1^{1-\frac{\alpha}{2}} )\cdot N_1^{-\frac{\alpha}{2}}N_2^{-\frac{\alpha-1}{2}}N_3^{-(\alpha-1)}L_3^{-\nu}\notag \\
\leq &N_1^{-(2-\frac{\alpha}{2})(\alpha-1)-\nu}+N_1^{-\frac{\alpha}{2}+(\frac{5-3\alpha}{2}-\nu)(1-\frac{\alpha}{2})}.
\end{align*}
When $N_3\geq N_1^{1-\frac{\alpha}{2}}>N_2$, we have
\begin{align*}
|\wt{\mathcal{U}}_{L_1,L_2,L_3}^{N_1,N_2,N_3}|\lesssim & N_1^{\frac{1}{2}(1-\frac{\alpha}{2})}N_2^{\frac{1}{2}}\cdot N_1^{-\frac{\alpha}{2}}N_2^{-\frac{\alpha-1}{2}}N_3^{-(\alpha-1)}L_3^{-\nu}\notag \\
\lesssim & N_1^{-\frac{\alpha-1}{2}-\frac{\alpha}{4}}N_2^{1-\frac{\alpha}{2}-\nu}N_3^{-(\alpha-1)}\leq N_1^{-\frac{\alpha}{2}+(\frac{5-3\alpha}{2}-\nu)(1-\frac{\alpha}{2})}.
\end{align*}
Therefore,
\begin{align}\label{numerology9}
|\wt{\mathcal{U}}_{L_1,L_2,L_3}^{N_1,N_2,N_3}|\lesssim N_1^{-\frac{\alpha}{2}+(\frac{5-3\alpha}{2}-\nu)(1-\frac{\alpha}{2})}+
N_1^{-(2-\frac{\alpha}{2})(\alpha-1)-\nu}.
\end{align}
By numerical computation, when $\alpha>1.0918 (<\alpha_0)$, we are able to choose sufficiently small $\sigma>0$, small enough, such that the right side of \eqref{numerology9} is bounded by $N_1^{-s-2\delta_0}$.

Now we assume that $a_2,a_3$ are both of type (G) or (C) and $L_1\ll N_2\wedge N_3$, $L_2\vee L_3\ll N_2\wedge N_3$. In this case, we execute {\bf Algorithm 1} and the goal is to estimate
\begin{align*}
\mathbb{E}^{\mathcal{C}}[|\wt{\mathcal{U}}_{L_1,L_2,L_3}^{N_1,N_2,N_3}|^2]=&\sum_{|k|>N_1}\sum_{\substack{(k_1,k_2,k_3)\in\Gamma(k)\\
(m_1,m_2,m_3)\in\Gamma(k) } }\mathbb{E}^{\mathcal{C}}[a_1(k_1)\ov{a}_1(m_1)]\cdot \mathbb{E}^{\mathcal{C}}[\ov{a}_2(k_2)a_2(m_2)a_3(k_3)\ov{a}_3(m_3) ],
\end{align*}
where $\mathcal{C}=\mathcal{B}_{\leq \max\{L_1,L_2,L_3\}}$. Here we used the fact that $a_1(k_1),a_1(m_1)$ are independent of $a_j(k_j),a_j(m_j)$ for $j=2,3$ since $N_1\gg N_2,N_3$ and $L_1\ll N_2\wedge N_3$. Using the independence and Cauchy-Schwartz, we have
\begin{align*}
\mathbb{E}^{\mathcal{C}}[|\wt{\mathcal{U}}_{L_1,L_2,L_3}^{N_1,N_2,N_3}|^2]\lesssim & (N_1N_2N_3)^{-\alpha}\prod_{j=1}^3\|h_{k_jk_j^*}^{(q)}\|_{l_{k_j}^{\infty}l_{k_j^*}^2}^2\\ \times &\!\!\!\!\!\!\!\!\sum_{\substack{k,k_2,k_3,m_2,m_3\\|k|>N_1}}  S_{k+k_2-k_3,k_2,k_3}S_{k+m_2-m_3,m_2,m_3} \Big(\mathbf{1}_{\substack{|k_2-m_2|<L_2\\
|k_3-m_3|<L_3} }+
\mathbf{1}_{\substack{|k_2-k_3|<L_2\vee L_3\\
		|m_2-m_3|<L_2\vee L_3 } } \Big).
\end{align*}
To sum the second line of the right side, we first sum over $m_3$ by using $\sum_{m_3}S_{k+m_2-m_3,m_2,m_3}\lesssim 1$ and then sum over $|k|>N_1$ by using \eqref{improvedcounting}. This procedure yields
\begin{align*}
\mathbb{E}^{\mathcal{C}}[|\wt{\mathcal{U}}_{L_1,L_2,L_3}^{N_1,N_2,N_3}|^2]\lesssim &(N_1N_2N_3)^{-\alpha}\prod_{j=1}^3\|h_{k_jk_j^*}^{(q)}\|_{l_{k_j}^{\infty}l_{k_j^*}^2}^2 \sum_{m_2,k_2,k_3}\!\!\!(\mathbf{1}_{|k_2-m_2|<L_2}+\mathbf{1}_{|k_2-k_3|<L_2\vee L_3})B(k_2,k_3),
\end{align*}
where
$$ B(k_2,k_3)=\min\{\lg k_2-k_3\rg,1+\frac{N_1^{2-\alpha}}{\lg k_2-k_3\rg} \}.
$$
To estimate the contribution from $\mathbf{1}_{|k_2-m_2|<L_2}$, we recall that
\begin{align*}
 \sum_{k_2,k_3}B(k_2,k_3) \lesssim
&\begin{cases}
& (N_2\vee N_3)^2(N_2\wedge N_3),\text{ if }N_2\vee N_3<N_1^{1-\frac{\alpha}{2}}\\
& N_2N_3+(N_2\vee N_3)\min\{N_1^{2-\alpha},(N_2\wedge N_3)N_1^{1-\frac{\alpha}{2}} \},\text{ if }N_2\vee N_3\geq N_1^{1-\frac{\alpha}{2}}.
\end{cases}
\end{align*}
To estimate the contribution from $\mathbf{1}_{|k_2-k_3|<L_2\vee L_3}$, we note that if $L_2\vee L_3\leq N_1^{1-\frac{\alpha}{2}}$,
$$ \sum_{k_2,k_3}\mathbf{1}_{|k_2-k_3|<L_2\vee L_3}B(k_2,k_3)\lesssim (L_2\vee L_3)^2(N_2\wedge N_3),
$$
and if $L_2\vee L_3>N_1^{1-\frac{\alpha}{2}}$,
\begin{align*}
 &\sum_{k_2,k_3}\mathbf{1}_{|k_2-k_3|<L_2\vee L_3}B(k_2,k_3)=\sum_{k_2,k_3 } \mathbf{1}_{|k_2-k_3|\leq N_1^{1-\frac{\alpha}{2}}}B(k_2,k_3)+\sum_{k_2,k_3}\mathbf{1}_{N_1^{1-\frac{\alpha}{2}}<|k_2-k_3|<L_2\vee L_3} B(k_2,k_3)\\
\lesssim &N_1^{1-\frac{\alpha}{2}}\min\{(N_2\wedge N_3)N_1^{1-\frac{\alpha}{2}},N_2N_3 \}
+(L_2\vee L_3)(N_2\wedge N_3)+(N_2\wedge N_3)N_1^{2-\alpha}.
\end{align*}
In what follows, we may assume that $N_2\leq N_3$. Therefore, modulo possible small powers of $N_1$, we obtain that
$$
\mathbb{E}^{\mathcal{C}}[|\wt{\mathcal{U}}_{L_1,L_2,L_3}^{N_1,N_2,N_3}|^2] \lesssim \mathrm{I}+\mathrm{II},$$
where
\begin{align*}
 \mathrm{I}=&(N_1N_2N_3)^{-\alpha}(L_1L_2L_3)^{-2\nu}(L_2\vee L_3)\cdot\!\!\!
\begin{cases}
& \!\!\!\!\!(N_2\vee N_3)^2(N_2\wedge N_3),\text{ if } N_3<N_1^{1-\frac{\alpha}{2}}\\
&\!\!\!\!\! N_2N_3+N_3\min\{N_1^{2-\alpha},N_2N_1^{1-\frac{\alpha}{2}} \},\text{ if } N_3\geq N_1^{1-\frac{\alpha}{2}},
\end{cases}
\end{align*}
and
\begin{align*}
\mathrm{II}=(N_1N_2N_3)^{-\alpha}(L_1L_2L_3)^{-2\nu}(L_2\vee L_3)^2N_2N_3,
\end{align*}
if $L_2\vee L_3\leq N_1^{1-\frac{\alpha}{2}}$, while
\begin{align*}
\mathrm{II}=&(N_1N_2N_3)^{-\alpha}(L_1L_2L_3)^{-2\nu}N_2\big[N_1^{1-\frac{\alpha}{2}}\min\{N_2N_1^{1-\frac{\alpha}{2}},N_2N_3 \}
+(L_2\vee L_3)N_2+N_2N_1^{2-\alpha}\big],
\end{align*}
if $L_2\vee L_3>N_1^{1-\frac{\alpha}{2}}$.

Let us first estimate $\mathrm{I}$:
If $ N_3<N_1^{1-\frac{\alpha}{2}}$, using our hypothesis $L_2\vee L_3\ll N_2$, we can bound I by
\begin{align*}
\mathrm{I}\lesssim& N_1^{-\alpha}L_2^{-2\nu}L_3^{-2\nu}(L_2\vee L_3) N_3^{2-\alpha}N_2^{1-\alpha} \notag  
\lesssim  N_1^{-\alpha}N_2^{2-\alpha-2\nu}N_3^{2-\alpha}\leq N_1^{-\alpha} N_3^{2(2-\alpha-\nu)}\notag \\
\leq &N_1^{-\alpha+2(2-\alpha-\nu)(1-\frac{\alpha}{2}) }. 
\end{align*}
In the case $N_3\geq N_1^{1-\frac{\alpha}{2}}$, if $N_3\geq N_2\geq N_1^{1-\frac{\alpha}{2}}$, using $L_2\vee L_3\ll N_2$, we have
\begin{align*}
\mathrm{I}\leq & N_1^{-\alpha}N_2^{-(\alpha-1)}N_3^{-(\alpha-1)}N_2^{1-2\nu}+N_1^{-2(\alpha-1)}N_2^{-\alpha}N_3^{-(\alpha-1)}N_2^{1-2\nu}\\
\leq &N_1^{-\alpha}N_2^{2-\alpha-2\nu}N_3^{-(\alpha-1)}+N_1^{-2(\alpha-1)}N_2^{-(\alpha+2\nu-1)}N_3^{-(\alpha-1)}\\
\leq &N_1^{-\alpha}N_3^{(3-2\nu-2\alpha)}+N_1^{-2(\alpha-1)}N_1^{-(2(\alpha-1)+2\nu)(1-\frac{\alpha}{2}) }.
\end{align*}
If $N_3\geq N_1^{1-\frac{\alpha}{2}}>N_2$, we have
\begin{align*}
\mathrm{I}\leq &N_1^{-\alpha+1-\frac{\alpha}{2}}(N_2N_3)^{-(\alpha-1)}(L_2\vee L_3)^{1-2\nu}\leq N_1^{1-\frac{3\alpha}{2}}N_2^{2-\alpha-2\nu}N_3^{-(\alpha-1)}\\
\leq & N_1^{1-\frac{3\alpha}{2}}N_1^{(2-\alpha-2\nu)(1-\frac{\alpha}{2})}N_1^{-(\alpha-1)(1-\frac{\alpha}{2})}.
\end{align*}
Therefore, we obtain that
\begin{align*}
\mathrm{I}
\lesssim & \begin{cases}
& N_1^{-3(\alpha-1)-2\nu}+N_1^{-\alpha+2(2-\alpha-\nu)(1-\frac{\alpha}{2})},\text{ if }N_3\geq N_2\geq N_1^{1-\frac{\alpha}{2}}\\
&N_1^{-\alpha+2(2-\alpha-\nu)(1-\frac{\alpha}{2}) },\text{ if }N_3\geq N_1^{1-\frac{\alpha}{2}}>N_2,
\end{cases} 
\end{align*}
and in particular,
\begin{align}\label{numerology10}
 \mathrm{I}\lesssim N_1^{-3(\alpha-1)-2\nu}+N_1^{-\alpha+2(2-\alpha-\nu)(1-\frac{\alpha}{2}) }.
\end{align}
When $\alpha>1.1205(<\alpha_0)$, we are able to choose sufficiently small $\sigma>0$, such that the right side of \eqref{numerology10} $<N_1^{-2s-4\delta_0}$.

It remains to estimate II. First we assume that $L_2\vee L_3\leq N_1^{1-\frac{\alpha}{2}}$, since $N_2\geq L_2\vee L_3$, we have
\begin{align*}
\mathrm{II}=&N_1^{-\alpha}(N_2N_3)^{-(\alpha-1)}(L_1L_2L_3)^{-2\nu}(L_2\vee L_3)^2\\
\leq &N_1^{-\alpha}(L_2\vee L_3)^{2(1-\nu)-(\alpha-1)}\leq N_1^{-\alpha+(2(1-\nu)-(\alpha-1))(1-\frac{\alpha}{2})}.
\end{align*}
By numerical computation, when $\alpha>\frac{10}{9}(<\alpha_0)$, we are able to choose sufficiently small $\sigma>0$, such that 
$$ N_1^{-\alpha+2(1-\nu-\frac{\alpha-1}{2})(1-\frac{\alpha}{2})}<N_1^{-2s-4\delta_0},
$$
which is conclusive. Finally we assume that $L_2\vee L_3>N_1^{1-\frac{\alpha}{2}}$, then
\begin{align*}
\mathrm{II}\leq &(N_1N_2N_3)^{-\alpha}(L_2\vee L_3)^{-2\nu}(N_2N_3N_1^{2-\alpha}+(L_2\vee L_3)N_2N_3 )\\
\leq &N_1^{-2(\alpha-1)}(L_2\vee L_3)^{-2\nu-(\alpha-1)}+N_1^{-\alpha}(L_2\vee L_3)^{1-2\nu-(\alpha-1)}\\
\leq & N_1^{-2(\alpha-1)-(2\nu+(\alpha-1))(1-\frac{\alpha}{2})}+N_1^{-\alpha+1-2\nu-(\alpha-1)}\\
\leq & N_1^{-2(\alpha-1)-2(\nu+\frac{\alpha-1}{2} )(1-\frac{\alpha}{2}) }+N_1^{-2(\alpha-1+\nu)}.
\end{align*}
By numerical computation, when $\alpha>\frac{10}{9}(<\alpha_0)$, we are able to choose sufficiently small $\sigma>0$, such that the right hand side of the above inequality is smaller than $N_1^{-2s-4\delta_0}$.

 Implementing {\bf Algorithm 1}.
The proof of tri-linear estimate for {\bf Case B-5} is complete.
In summary, the proof of Proposition \ref{Multilinearkey} is completely finished.
\section*{Appendix 1: Proof of Proposition \ref{iteratedPicard}}

By the triangle inequality and the Wiener chaos estimate, the second assertion in Proposition \ref{iteratedPicard} follows from the expectation bound of $|z_{2k+1}^{\omega}(t,x)|$, hence we will only show that
\begin{align}\label{Appendix1} 
\mathbb{E}[|z_{2k+1}^{\omega}(t,x)|^2]\leq C_0t^{2j}(2j+1)!\big(\frac{(2j-1)!!}{j!}\big)^2.
\end{align} 
Without loss of generality, we take $t>0$. Recall the expression \eqref{expressionz2k+1} and the equation of $z_{2j+1}$, we have the recurrence relation for the coefficient $c_{j}(t,k_1,\cdots,k_{2j+1})$:
\begin{align*}
c_{j}(t,k_1,\cdots,k_{2j+1})=-&\!\!\!\!\!\!\!\sum_{\substack{j_1,j_2,j_3\\
		j_1+j_2+j_3=j-1}}\int_0^t\mathrm{e}^{i(t-t')|k_1-k_2+\cdots-k_{2j}+k_{2j+1}|^{\alpha}}dt'\\
	\times&c_{j_1}(t',k_1,\cdots,k_{2j_1+1})c_{j_2}(t',k_{2j_1+2},\cdots, k_{2j_1+2j_2+2})c_{j_3}(t',k_{2j_2+3}+\cdots +k_{2j+1}).
\end{align*} 
Note that $c_0(t,k_1)=1$ and 
$|c_1(t,k_1)|\leq t$. 
Define recurrently the series $\{\kappa_j\}_{j\geq 0}$ by $\kappa_0=1$ and 
\begin{align}\label{recurrence}
 \kappa_j:=\frac{1}{j}\sum_{\substack{j_1,j_2,j_3\\
j_1+j_2+j_3=j-1 }} \kappa_{j_1}\kappa_{j_2}\kappa_{j_3}.
\end{align}
By induction we deduce that for any $j\geq 0$,
$$ |c_j(t,n_1,\cdots,n_{2j+1})|\leq \kappa_jt^j.
$$ 

Next we determine $\kappa_j$. Considering the power series
$$ f(x)=\sum_{j=0}^{\infty}\kappa_j z^j,
$$
the recurrence relation \eqref{recurrence} implies that
$ f'(z)=f(z)^3, f(0)=0.
$
Solving this ODE we obtain that
$f(z)=(1-2z)^{-\frac{1}{2}}$, hence 
$$ \kappa_j=\frac{(2j-1)!!}{j!}.
$$
Therefore, we obtain that
\begin{align}\label{boundc_j}
 |c_j(t,k_1,\cdots, k_{2j+1})|\leq \kappa_j t^j:=\frac{(2j-1)!!}{j!}t^j.
\end{align}
Thus
\begin{align*}
\mathbb{E}[|z_{2j+1}^{\omega}(t,x)|^2]=&\sum_{\substack{k_1,\cdots,k_{2j+1}\\
m_1,\cdots,m_{2j+1}  }  }
c_j(t,k_1,\cdots,k_{2j+1})\ov{c}_j(t,m_1,\cdots,m_{2j+1})\\
\times & \mathrm{e}_{k_1-k_2+\cdots+k_{2j+1}}\mathrm{e}_{-m_1+m_2-\dots-m_{2j+1}} \mathbb{E}\Big[\prod_{j'=1}^{2j+1}\frac{g^{\iota_{j'}}_{k_{j'}}g^{-\iota_{j'}}_{m_{j'}} }{[k_{j'}]^{\frac{\alpha}{2}}[m_{j'}]^{\frac{\alpha}{2}} }\Big],
\end{align*}
where $\iota_{j'}\in\{+1,-1\}$ and we use the convention $g_k^{+1}=g_j$ and  $g_{k}^{-1}=\ov{g}_k$. Using the independence of Gaussians and \eqref{boundc_j}, we have
\begin{align*}
\mathbb{E}[|z_{2j+1}^{\omega}(t,x)|^2]\leq&\sum_{k_1,\cdots, k_{2j+1}} \kappa_j^2t^{2j} \cdot (2j+1)!\prod_{j'=1}^{2j+1}\frac{1}{[k_{j'}]^{\frac{\alpha}{2}}}\leq C_0t^{2j}(2j+1)!\big(\frac{(2j-1)!!}{j!}\big)^2,
\end{align*} 
where the extra factor $(2j+1)!$ comes from the number of pairings for the index
 $$(k_1,\cdots,k_{2j+1};m_1,\cdots,m_{2j+1}).$$
 This completes the proof of \eqref{Appendix1}.
\section*{Appendix 2: Proof of Lemma \ref{Timelocalization}}

It suffices to prove the first inequality
\begin{align}\label{appendix2-1} \|\chi_T(t)u\|_{X_{p,q}^{0,\gamma}}\lesssim T^{\gamma_1-\gamma}\|u\|_{X_{p,q}^{0,\gamma_1}}
\end{align}
with $u\in X_{q,p}^{0,\gamma}$ satisfying $u|_{t=0}=0$, where $1\leq p\leq \infty$, $0< \gamma<\gamma_1<1+\frac{1}{q'}$ and $1\leq q<\infty$, since we may regard $\|\wt{\Theta}_{kk'}(\lambda,\lambda')\|_{L_{\lambda'}^2l_{k'}^2}$ as
$\wt{u}(\lambda,k)$. We decompose $u=u_1+u_2$, where
$$ \mathcal{F}_{t,x}u_1(\tau,k)=\mathbf{1}_{|\tau-|k|^{\alpha}|\geq\frac{1}{T}}\mathcal{F}_{t,x}u(\tau,k),\quad \mathcal{F}_{t,x}u_2(\tau,k)=\mathbf{1}_{|\tau-|k|^{\alpha}|<\frac{1}{T}}\mathcal{F}_{t,x}u(\tau,k).
$$ 
To prove \eqref{appendix2-1}, it suffices to show that, uniformly in $k\in\Z$,
\begin{align}\label{appendix2-2}
&\mathrm{(a)}\quad \|\lg\tau-|k|^{\alpha}\rg^{\gamma}(\widehat{\chi_T}*\mathcal{F}_{t,x}u_1)(\tau,k)\|_{L_{\tau}^q}\lesssim T^{\gamma_1-\gamma}\|\lg\eta-|k|^{\alpha}\rg^{\gamma_1}\mathcal{F}_{t,x}u(\eta,k) \|_{L_{\eta}^q};\\
&\mathrm{(b)}\quad \|\lg\tau-|k|^{\alpha}\rg^{\gamma}(\widehat{\chi_T}*\mathcal{F}_{t,x}u_2)(\tau,k)\|_{L_{\tau}^q}\lesssim T^{\gamma_1-\gamma}\|\lg\eta-|k|^{\alpha}\rg^{\gamma_1}\mathcal{F}_{t,x}u(\eta,k) \|_{L_{\eta}^q}.
\end{align}
To prove (a), we note that 
$$\lg\tau-|k|^{\alpha}\rg^{\gamma}(\widehat{\chi}_T*\mathcal{F}_{t,x}u_1)(\tau,k)=\int_{\R}K_T^+(\tau-|k|^{\alpha},\omega)\lg\omega\rg^{\gamma_1}\wt{u}(\omega,k)d\omega,
$$
where
$$ K_T^+(\lambda,\omega)=\widehat{\chi}(T(\lambda-\omega))\frac{T\lg\lambda\rg^{\gamma}}{\lg\omega\rg^{\gamma_1}}\mathbf{1}_{|\omega|\geq \frac{1}{T}}.
$$
By Schur's test, it suffices to show that
\begin{align*}
\sup_{\lambda\in\R}\int_{\R}|K_T^+(\lambda,\omega)|d\omega\leq C_1T^{\gamma_1-\gamma},\quad \sup_{\omega\in\R}\int_{\R}|K_T^+(\lambda,\omega)|d\lambda\leq C_2T^{\gamma_1-\gamma }
\end{align*}
with $C_1,C_2$ \emph{independent} of $T$. One can check these two inequalities by direct computation, here we explain it in an informal way. Since $\widehat{\chi}$ is a Schwartz function,  $|\lambda-\omega|$ is essentially bounded by $O(\frac{1}{T})$. Due to the fact that $|\omega|\geq \frac{1}{T}$, $\lambda$ is essentially constraint in the region $|\lambda|\lesssim \frac{1}{T}$. Since the length of the integration is of size $\frac{1}{T}$, we deduce that two integrations are bounded by $O(1)T^{\gamma_1-\gamma}$.

The proof of (b) exploits the cancellation from the condition $u|_{t=0}=0$. By the Fourier inversion formula, we have
$$ \int_{\R}\wt{u}(\omega,k)d\omega=0.
$$
Therefore, with $\lambda=\tau-|k|^{\alpha}$,
\begin{align*}
(\widehat{\chi}_T\ast\mathcal{F}_{t,x}u_2)(\lambda,k)=&\int_{|\omega|<\frac{1}{T}} T\wh{\chi}(T(\lambda-\omega))\wt{u}(\omega,k)d\omega\\
=&\int_{|\omega|<\frac{1}{T}}T\wt{u}(\omega,k)[\wh{\chi}(T(\lambda-\omega))-\wh{\chi}(T\lambda)]d\omega+T\wh{\chi}(T\lambda)\int_{|\omega|\geq \frac{1}{T}}\wt{u}(\omega,k)d\omega
\end{align*}
By H\"older's inequality, 
$$ \int_{|\omega|\geq \frac{1}{T}}|\wt{u}(\omega,k)|d\omega\leq \|\lg\omega\rg^{\gamma_1}\wt{u}(\omega,k)\|_{L_{\omega}^q}\cdot T^{\gamma_1-\frac{1}{q'}},
$$
we have
$$ \Big\|\lg\lambda\rg^{\gamma}T\wh{\chi}(T\lambda)\int_{|\omega|\geq \frac{1}{T}}\wt{u}(\omega,k)d\omega\Big\|_{L_{\lambda}^q}\lesssim T^{\gamma_1-\gamma}\|\lg\omega\rg^{\gamma_1}\wt{u}(\omega,k)\|_{L_{\omega}^q}.
$$
Finally, from the fact that $[\widehat{\chi}(T(\lambda-\omega))-\widehat{\chi}(T\lambda)]\mathbf{1}_{T|\omega|<1}=O(|T\omega|\lg T\lambda\rg^{-100})$, we deduce by H\"older that
$$ \Big|\int_{|\omega|<\frac{1}{T}}T\wt{u}(\omega,k)\widehat{\chi}(T(\lambda-\omega))-\widehat{\chi}(T\lambda)]d\omega \Big|\leq T^2O(\lg T\lambda\rg^{-100})\|\lg\omega\rg^{\gamma_1}\wt{u}(\omega,k)\|_{L_{\omega}^q} \Big\|\frac{|\omega|}{\lg\omega\rg^{\gamma_1}}\Big\|_{L^{q'}(|\omega|\leq \frac{1}{T})}.
$$
Multiplying the right hand side by $\lg\lambda\rg^{\gamma}$ and then taking the $L^q$ norm in $\lambda$, we obtain the desired upper bound $T^{\gamma_1-\gamma}$. This proves (b). The proof of Lemma \ref{Timelocalization} is now complete.

\end{document}